\documentclass{article}
\usepackage{amsmath,amsthm,amssymb}
\usepackage[usenames,dvipsnames]{xcolor}
\usepackage{enumerate}
\usepackage{graphicx}
\usepackage{cite}
\usepackage{comment}
\usepackage{oands}
\usepackage{tikz}
\usepackage{changepage}
\usepackage{bbm}
\usepackage{mathtools}
\usepackage[margin=0.98in]{geometry}
\usepackage[pagewise,mathlines]{lineno}
\usepackage{appendix}
\usepackage{multicol}
\usepackage{microtype}
\usepackage{stmaryrd}

\theoremstyle{plain}
\newtheorem{thm}{Theorem}[section]

\newtheorem{lem}[thm]{Lemma}
\newtheorem{prop}[thm]{Proposition}

\def\@rst #1 #2other{#1}
\newcommand\MR[1]{\relax\ifhmode\unskip\spacefactor3000 \space\fi
  \MRhref{\expandafter\@rst #1 other}{#1}}
\newcommand{\MRhref}[2]{\href{http://www.ams.org/mathscinet-getitem?mr=#1}{MR#2}}

\usepackage[pdftitle={Characterizations of SLE$_\kappa$ for $\kappa \in (4,8)$ on Liouville quantum gravity},
  pdfauthor={Ewain Gwynne and Jason Miller},
colorlinks=true,linkcolor=NavyBlue,urlcolor=RoyalBlue,citecolor=PineGreen,bookmarks=true,bookmarksopen=true,bookmarksopenlevel=2,unicode=true,linktocpage]{hyperref}

\theoremstyle{definition}
\newtheorem{defn}[thm]{Definition}

\numberwithin{equation}{section}

\newcommand{\dsb}{\begin{adjustwidth}{2.5em}{0pt}
\begin{footnotesize}}
\newcommand{\dse}{\end{footnotesize}
\end{adjustwidth}}

\newcommand{\ssb}{\begin{adjustwidth}{2.5em}{0pt}}
\newcommand{\sse}{\end{adjustwidth}}

\newcommand{\aryb}{\begin{eqnarray*}}
\newcommand{\arye}{\end{eqnarray*}}
\def\alb#1\ale{\begin{align*}#1\end{align*}}
\def\allb#1\alle{\begin{align}#1\end{align}}
\newcommand{\eqb}{\begin{equation}}
\newcommand{\eqe}{\end{equation}}
\newcommand{\eqbn}{\begin{equation*}}
\newcommand{\eqen}{\end{equation*}}

\newcommand{\BB}{\mathbbm}
\newcommand{\ol}{\overline}
\newcommand{\ul}{\underline}
\newcommand{\op}{\operatorname}

\newcommand{\frk}{\mathfrak}
\newcommand{\eqD}{\overset{d}{=}}
\newcommand{\ep}{\epsilon}
\newcommand{\rta}{\rightarrow}

\newcommand{\wt}{\widetilde}
\newcommand{\wh}{\widehat} 
\newcommand{\mcl}{\mathcal}

\newcommand{\bdy}{\partial}
\newcommand{\rng}{\mathring}
\newcommand{\srta}{\shortrightarrow}

\newcommand{\SLE}{{\mathrm {SLE}}}
\newcommand{\CLE}{{\mathrm {CLE}}}

\newcommand{\bead}{{\operatorname{b}}}

\newcommand{\nat}{{\operatorname{q}}}

\let\originalleft\left
\let\originalright\right
\renewcommand{\left}{\mathopen{}\mathclose\bgroup\originalleft}
\renewcommand{\right}{\aftergroup\egroup\originalright}

\title{Characterizations of SLE$_{\kappa}$ for $\kappa \in (4,8)$ on\\ Liouville quantum gravity} 
\date{  }
\author{
\begin{tabular}{c} Ewain Gwynne\footnote{\url{ewain@uchicago.edu}}\\[-5pt]\small University of Chicago \end{tabular}
\begin{tabular}{c} Jason Miller\footnote{\url{jpmiller@statslab.cam.ac.uk}}\\[-5pt]\small University of Cambridge \end{tabular}
}

\begin{document}

\maketitle

\begin{abstract}
We prove that SLE$_\kappa$ for $\kappa \in (4,8)$ on an independent $\gamma=4/\sqrt{\kappa}$-Liouville quantum gravity (LQG) surface is uniquely characterized by the form of its LQG boundary length process and the form of the conditional law of the unexplored quantum surface given the explored curve-decorated quantum surface up to each time $t$.  We prove variants of this characterization for both whole-plane space-filling SLE$_\kappa$ on an infinite-volume LQG surface and for chordal SLE$_\kappa$ on a finite-volume LQG surface with boundary.  Using the equivalence of Brownian and $\sqrt{8/3}$-LQG surfaces, we deduce that SLE$_6$ on the Brownian disk is uniquely characterized by the form of its boundary length process and that the complementary connected components of the curve up to each time $t$ are themselves conditionally independent Brownian disks given this boundary length process.
 
The results of this paper are used in another paper by the same authors to show that the scaling limit of percolation on random quadrangulations is given by SLE$_6$ on $\sqrt{8/3}$-LQG with respect to the Gromov-Hausdorff-Prokhorov-uniform topology, the natural analog of the Gromov-Hausdorff topology for curve-decorated metric measure spaces.
\end{abstract}

\medskip


\tableofcontents

\section{Introduction}
\label{sec-intro}
 
\subsection{Overview}
\label{sec-overview}

The Schramm-Loewner evolution ($\SLE$) was introduced by Schramm \cite{schramm0} to describe the scaling limits of the interfaces which arise in discrete two-dimensional models, such as loop-erased random walk, critical percolation, and the uniform spanning tree.  The form of $\SLE$ was derived by Schramm using what is now called its \emph{conformal Markov property}.  This says that if~$\eta$ is an~$\SLE_\kappa$ curve in $\BB H$ from $0$ to~$\infty$ then for each time~$t$ the conditional law of $\eta|_{[t,\infty)}$ given $\eta|_{[0,t]}$ is given by the conformal image of an $\SLE_\kappa$ in $\BB H$ from~$0$ to~$\infty$.  To show that a curve is an $\SLE$ one need only show that this property is satisfied.  Beyond the initial derivation of $\SLE$, this perspective has been very powerful for the purpose of establishing properties of $\SLE$.  Moreover, results of this type have been shown to characterize other $\SLE$-related processes.  For example, it was shown by Sheffield and Werner \cite{shef-werner-cle} that the so-called simple \emph{conformal loop ensembles} ($\CLE$), the loop form of $\SLE_\kappa$ for $\kappa \in (8/3,4]$, are similarly characterized by a variant of the conformal Markov property.

The purpose of the present work is to establish characterizations of various types of $\SLE$ in the spirit of the conformal Markov property but in the context of \emph{Liouville quantum gravity} (LQG). In the special case of SLE$_6$ on $\sqrt{8/3}$-Liouville quantum gravity, we can re-phrase our characterization theorems in terms of the metric space structure of $\sqrt{8/3}$-LQG, as constructed in~\cite{lqg-tbm1,lqg-tbm2,lqg-tbm3}, which gives us characterizations of SLE$_6$ curves on \emph{Brownian surfaces} (generalizations of the Brownian map~\cite{legall-uniqueness,miermont-brownian-map}) which depend only on the metric space structure. Our characterizations will be given in terms of a set of conditions which one would naturally expect any subsequential limit of certain statistical physics models on a random planar map to satisfy.  

Our results play an important role in our proof that the scaling limit of (critical) percolation on random planar maps is given by $\SLE_6$ on $\sqrt{8/3}$-LQG.  This is carried out in the companion paper~\cite{gwynne-miller-perc} (building also on \cite{gwynne-miller-simple-quad,gwynne-miller-sle6}) in which we show that the subsequential limits of percolation on random quadrangulations exist and satisfy the hypotheses of one of our characterization theorems. 

The paper~\cite{gwynne-miller-perc} uses exclusively discrete arguments (based on properties of percolation and random planar maps), so can be read without any knowledge of SLE or LQG. This paper, on the other hand, uses only continuum (SLE/LQG) arguments. Hence, one can think of this paper as providing the continuum input needed to show the convergence of percolation on random planar maps toward SLE$_6$ on $\sqrt{8/3}$-LQG. However, in this paper we also establish characterizations for other variants of SLE$_\kappa$ on $\gamma$-LQG surfaces for $\gamma \in (\sqrt2 , 2)$ and $\kappa  = 16/\gamma^2 \in (4,8)$. We expect that these results may eventually have applications to proving other scaling limit results for other types of statistical mechanics models on random planar maps, e.g., the critical Fortuin-Kasteleyn model~\cite{shef-burger}.

In the remainder of this section, we give a short review of LQG and its relationship to SLE (Section~\ref{sec-lqg-intro}). We then provide informal statements of our main results in the setting of whole-plane space-filling SLE$_{\kappa'}$ (Section~\ref{sec-wpsf-char-intro}), ordinary chordal SLE$_{\kappa'}$ (Section~\ref{sec-bead-char-intro}), and SLE$_6$ on a Brownian surface (Section~\ref{sec-bead-mchar-intro}). We will then give an outline of the rest of the content of the paper in Section~\ref{sec-outline}.

\subsection{Liouville quantum gravity surfaces}
\label{sec-lqg-intro}

Formally, an LQG surface for $\gamma \in (0,2)$ is a random Riemann surface parameterized by a domain $D\subset \BB C$ whose Riemannian metric tensor is $e^{\gamma h(z)} \, dx\otimes dy$, where $h$ is some variant of the Gaussian free field (GFF) on $D$ and $dx\otimes dy$ is the Euclidean metric tensor. For our purposes, we can always assume that $h$ is a \emph{GFF plus a continuous function}, meaning that there is a (possibly random and depending on $h$) continuous function $f$ on $D$ such that $h-f$ is the zero-boundary GFF on $D$ or the whole-plane GFF, as appropriate (see~\cite{shef-gff,ss-contour,ig1,ig4,berestycki-lqg-notes} for more on the GFF).

The above definition does not make rigorous sense since $h$ is a distribution, not a function, so does not take values at points. However, Duplantier and Sheffield~\cite{shef-kpz} showed that one can make rigorous sense of the volume form $\mu_h = ``e^{\gamma h(z)} \, dz"$ associated with an LQG surface as a random measure on $D$ via a regularization procedure. 
This construction is a special case of a more general theory of measures of this form called \emph{Gaussian multiplicative chaos}, which was initiated by Kahane~\cite{kahane}; see~\cite{rhodes-vargas-review,aru-gmc-survey,berestycki-gmt-elementary} for expository works on this theory. 

One can similarly define a random length measure~$\nu_h$ associated with an LQG surface, which is defined on certain curves in $D$ including $\bdy D$ (if $h$ locally looks like a free-boundary GFF near $\bdy D$) and independent Schramm-Loewner evolution~\cite{schramm0} ($\SLE_\kappa$)-type curves for $\kappa = \gamma^2$~\cite{shef-kpz,shef-zipper}. More precisely, one can define $\nu_h$ in the case when $\bdy D$ is piecewise linear segment using a direct regularization procedure, then extend to the boundary of a general domain by conformal covariance (see~\cite[Section 6]{shef-kpz}). One can define $\nu_h$ on an SLE$_\kappa$-type curve using the SLE / GFF coupling results from~\cite{shef-kpz}, or directly as a Gaussian multiplicative chaos measure w.r.t.\ the Minkowski content measure on the curve~\cite{benoist-lqg-chaos}.

The measures $\mu_h$ and $\nu_h$ satisfy a conformal covariance formula~\cite[Proposition 2.1]{shef-kpz}: if $f :   D \rta \wt D$ is a conformal map and  
\eqb \label{eqn-lqg-coord}
\wt h = h\circ f^{-1} + Q\log |(f^{-1})'| ,\quad \op{for} \: Q = \frac{2}{\gamma} + \frac{\gamma}{2} 
\eqe
then $f$ pushes forward $\mu_{ h}$ to $\mu_{\wt h}$ and $\nu_{ h}$ to $\nu_{\wt h}$ (in fact, this holds a.s.\ for all choices of conformal map $f$ simultaneously~\cite{shef-wang-lqg-coord}).  

This leads us to define an LQG surface as an equivalence class of pairs $(D, h)$ consisting of a domain $D \subset \BB C$ and a distribution $h$ on $D$, with two such pairs declared to be equivalent if the distributions are related by a conformal map as in~\eqref{eqn-lqg-coord} which extends to a homeomorphism $D\cup \bdy D \rta \wt D \cup \bdy \wt D$.
In other words, a quantum surface is an equivalence class of measure spaces modulo conformal maps. 
 
The above definition does not require that $D$ be simply connected or even connected. For example, one can make sense of quantum surfaces consisting of a string of beads, each of which is itself a quantum surface homeomorphic to the unit disk. Two such surfaces are illustrated in the left panel of Figure~\ref{fig-thm-illustration}.
 
One can also define quantum surfaces with $k\in\BB N$ marked points as an equivalence class of $k+2$-tuples $(D,h,x_1,\dots,x_k)$ with $x_1,\dots,x_k \in D\cup\bdy D$ with two such $k+2$-tuples declared to be equivalent if there is a conformal map $f $ such that the corresponding fields are related as in~\eqref{eqn-lqg-coord} and $f$ takes the marked points for one $k+2$-tuple to the corresponding marked points of the other. 

If $(D,h,x_1,\dots,x_k)$ is a particular equivalence class representative, we say that $h$ is an \emph{embedding} of the quantum surface into $(D,x_1,\dots,x_k)$. We also define a \emph{sub-surface} of a quantum surface $(D,h)$ to be a surface of the form $(D' , h|_{D'})$ for $D'\subset D$. 

One can also make sense of an LQG surface as a random \emph{metric space}.  This was first done in the special case when $\gamma=\sqrt{8/3}$ in the series of works~\cite{sphere-constructions,tbm-characterization,lqg-tbm1,lqg-tbm2,lqg-tbm3}, which is the only value of $\gamma$ for which we will consider the metric space structure in this paper. See Section~\ref{sec-bead-mchar-intro} for more details. The more recent works~\cite{dddf-lfpp,gm-uniqueness} showed how to define an LQG surface as a metric space for general $\gamma\in (0,2)$, using completely different methods.

LQG surfaces arise as scaling limits of various random planar map models. LQG for $\gamma =\sqrt{8/3}$ corresponds to the scaling limit of uniform random planar maps, and other values of $\gamma$ arise by sampling a planar map with probability proportional to the partition function of an appropriate $\gamma$-dependent statistical mechanics model on the map. For example, weighting by the number of spanning trees corresponds to $\gamma=\sqrt 2$, weighting by the partition function of the Ising model corresponds to $\gamma=\sqrt 3$, and weighting by the number of bipolar orientations corresponds to $\gamma=\sqrt{4/3}$.  So far, scaling limit results for random planar maps toward LQG have been obtained in the Gromov-Hausdorff topology for $\gamma=\sqrt{8/3}$~\cite{legall-uniqueness,miermont-brownian-map} (together with \cite{sphere-constructions,tbm-characterization,lqg-tbm1,lqg-tbm2,lqg-tbm3}) and in the so-called peanosphere sense, which relies on the main theorem of~\cite{wedges}, for all values of $\gamma \in (0,2)$~\cite{shef-burger,kmsw-bipolar,gkmw-burger,lsw-schnyder-wood,bhs-site-perc}. 

Several particular $\gamma$-LQG surfaces, which arise as the scaling limits of random planar maps with various topologies, are defined in~\cite{wedges}. We review the definitions of these surfaces in Section~\ref{sec-wedge}.

It is natural to consider a $\gamma$-LQG surface decorated by an independent $\SLE_\kappa$-type curve for $\kappa = \gamma^2 \in (0,4)$ or $\kappa' = 16/\gamma^2   > 4$.\footnote{Here and throughout this paper we use the imaginary geometry~\cite{ig1,ig2,ig3,ig4} convention of writing $\kappa$ for the SLE parameter when $\kappa \in (0,4)$ and $\kappa' = 16/\kappa$ for the dual parameter.}
One reason why this is natural is that such curve-decorated quantum surfaces describe the scaling limits of statistical mechanics models on random planar maps in the $\gamma$-LQG universality class. Indeed, scaling limits in the peanosphere sense are really statements about the convergence of random planar maps decorated by a space-filling curve toward LQG decorated by space-filling $\SLE$ (as defined in \cite{ig4}). See also~\cite{gwynne-miller-saw} for a scaling limit result for random quadrangulations decorated by a self-avoiding walk toward $\SLE_{8/3}$-decorated $\sqrt{8/3}$-LQG in a variant of the Gromov-Hausdorff topology. 

In the continuum, there are a number of theorems which describe $\gamma$-LQG surfaces decorated by independent $\SLE_\kappa$- or $\SLE_{\kappa'}$-type curves~\cite{shef-zipper,wedges,sphere-constructions}, some of which we discuss just below. 

\subsection{Main results}
\label{sec-results}

\subsubsection{Characterization of space-filling SLE$_{\kappa'}$}
\label{sec-wpsf-char-intro}

\begin{figure}[ht!]
	\begin{center}
		\includegraphics[scale=1]{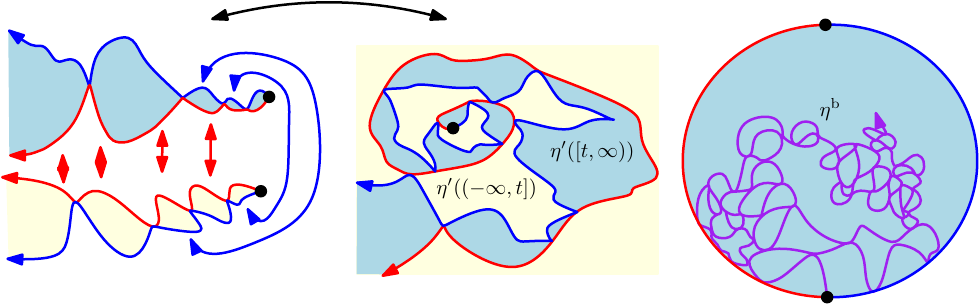}
	\end{center}
	\vspace{-0.025\textheight}
	\caption[SLE curves on LQG surfaces characterized in Chapter~\ref{chap-char}]{\label{fig-thm-illustration} \textbf{Left and middle:} A whole-plane space-filling $\SLE_{\kappa'}$ for $\kappa' \in  (4 , 8)$ on an independent $\gamma$-quantum cone, which is the object characterized in Theorem~\ref{thm-wpsf-char-intro}. The past $\eta'((-\infty, t])$ is shown in yellow and the future $\eta'([t,\infty))$ is shown in blue. Restricting the field $h$ to each of these sets gives two independent beaded quantum surfaces, each of which has the law of a $\frac{3\gamma}{2}$-quantum wedge, which can be conformally welded together according to quantum length along their boundaries to obtain the $\gamma$-quantum cone $(\BB C , h , 0, \infty)$.  In the special case when $\gamma = \sqrt{8/3}$, the $\sqrt{8/3}$-quantum cone admits a metric measure space structure under which it is equivalent to the Brownian plane. In this case, Theorem~\ref{thm-wpsf-char-intro} can be re-phrased as a characterization theorem for whole-plane space-filling $\SLE_6$ on the Brownian plane; see Theorem~\ref{thm-wpsf-mchar}.
	\textbf{Right:} A chordal $\SLE_{\kappa'}$ on an independent bead of a $\frac{3\gamma}{2}$-quantum wedge (i.e., a quantum surface whose law is the same as the conditional law of one of the connected components of the quantum surface parameterized by $\eta'([t,\infty))$ given its quantum area and left/right quantum boundary lengths). 
	This is the setting of our other quantum surface characterization theorem, Theorem~\ref{thm-bead-char-intro}. 
	When $\gamma=\sqrt{8/3}$, a single bead of a $\frac{3\gamma}{2}$-quantum wedge is the same as the quantum disk, which in turn is the same as the Brownian disk when equipped with its $\sqrt{8/3}$-LQG metric and area measure. This gives rise to a characterization of chordal $\SLE_6$ on the Brownian disk; see Theorem~\ref{thm-bead-mchar-intro}. 
	}
\end{figure}

One of the most important relationships between SLE and LQG is the \emph{peanosphere} or \emph{mating of trees} construction of~\cite{wedges}, which we now briefly describe (see Section~\ref{sec-peano-prelim} for a more detailed review). Suppose $(\BB C , h , 0, \infty)$ is a $\gamma$-quantum cone, a particular type of $\gamma$-LQG surface with two marked points which describes the local behavior of any $\gamma$-LQG surface near a point sampled from the $\gamma$-LQG area measure. Let $\eta'$ be a whole-plane space-filling $\SLE_{\kappa'}$ curve from $\infty$ to $\infty$ for $\kappa' = 16/\gamma^2$, independent from $h$. In the case when $\kappa'  \geq 8$, $\eta'$ is just a two-sided variant of chordal $\SLE_{\kappa'}$. In the case when $\kappa' \in (4,8)$ (so ordinary $\SLE_{\kappa'}$ is not space-filling), $\eta'$ is obtained from a two-sided variant of ordinary whole-plane $\SLE_{\kappa'}$ by iteratively filling in the bubbles which it surrounds by $\SLE_{\kappa'}$-type curves. See Section~\ref{sec-wpsf} for a review of the construction and basic properties of space-filling SLE$_{\kappa'}$.

We can parameterize $\eta'$ by $\gamma$-quantum mass with respect to $h$, so that $\mu_h(\eta'([s,t])) = t-s$ for each $s < t$ and $\eta'(0) = 0$. For $t\geq 0$, let $L_t$ (resp.\ $R_t$) be the net change in the $\nu_h$-length of the left (resp.\ right) outer boundary of $\eta'((-\infty,t])$ relative to time 0. Then by~\cite[Theorem~1.9]{wedges}, there is a constant $\alpha = \alpha(\gamma) >0$ such that $Z_t := (L_t, R_t)$ is a pair of correlated Brownian motions with 
\eqb \label{eqn-bm-cov}
\op{Var} L_t = \op{Var} R_t = \alpha|t| \quad\op{and}\quad \op{Cov}(L_t,R_t) = -\alpha \cos\left( \frac{\pi \gamma^2}{4} \right) ,\quad \forall t\in\BB R .
\eqe 
In other words, $\eta'$ is an embedding into $\BB C$ of the space-filling curve on an infinite-volume \emph{peanosphere}, a random curve-decorated topological measure space obtained by gluing together a pair of correlated Brownian motions (see Section~\ref{sec-peanosphere} and Figure~\ref{fig-peano}).

It is also shown in~\cite{wedges} that $\eta'$ possesses a quantum analog of the conformal Markov property. 
To describe this property, we define for each $t\in\BB R$ the past and future curve-decorated quantum surfaces 
\eqb \label{eqn-past-future-surfaces}
\left( \eta'((-\infty,t]) , h|_{\eta'((-\infty,t])} , \eta'|_{(-\infty,t]} \right) \quad \op{and} \quad
\left( \eta'([t,\infty)) , h|_{\eta'([t,\infty))} , \eta'|_{[t,\infty)} \right) .
\eqe 
By~\cite[Theorem~1.9]{wedges}, these two curve-decorated quantum surfaces are independent. Each of the underlying quantum surfaces has the law of a $\frac{3\gamma}{2}$-quantum wedge (equivalently, a weight-$(2-\gamma^2/2)$ quantum wedge)---a quantum surface consisting of a Poissonian string of beads, each of which is itself a finite-volume doubly marked quantum surface homeomorphic to the unit disk $\BB D$ (see~\cite[Definition~4.15]{wedges} or Section~\ref{sec-wedge} for a more detailed description). Moreover, the curves on these quantum surfaces are concatenations of chordal space-filling SLE$_{\kappa'}$ curves, one in each of the beads. Figure~\ref{fig-thm-illustration} illustrates these statements.

It is shown in~\cite[Theorem~1.11]{wedges} that $Z$ a.s.\ determines $(h,\eta')$, modulo scaling and rotation, i.e., there is a (non-explicit) deterministic functional which a.s.\ takes in the Brownian motion $Z$ and outputs the pair $(h,\eta')$, up to rotation and scaling. This theorem does \emph{not} rule out the possibility that there is a different space-filling curve $\wt\eta'$, which has a different law from $\eta'$ and/or is not independent from $h$, with the property that the corresponding left/right boundary length process evolves as a correlated two-dimensional Brownian motion. One of the main results of this paper says that there are no such curves $\wt\eta'$ which also satisfy a version of the above quantum conformal Markov property. In other words, the embedding of the infinite-volume peanosphere into $\BB C$ is canonical in a stronger sense than~\cite[Theorem 1.11]{wedges}. The following is an informal statement of this result.  

\begin{thm}[Whole-plane space-filling SLE$_{\kappa'}$ characterization, informal version] \label{thm-wpsf-char-intro}
Let $\kappa' \in (4,8)$ and $\gamma = 4/\sqrt{\kappa'} \in (\sqrt 2 , 2)$. 
Suppose that $(\wt h, \wt \eta'    )$ is a given coupling where $\wt h$ is an embedding into $(\BB C , 0 , \infty)$ of a $\gamma$-quantum cone and $\wt\eta':\BB R\rta \BB C$ is a random continuous curve (not assumed to be independent from $\wt h$) parameterized by $\gamma$-quantum mass with respect to $\wt h$. Suppose $\wt\eta'$ satisfies the following properties.
\begin{enumerate}
\item (Markov property) For each $t\in\BB R$, the curve-decorated quantum surfaces obtained by restricting $\wt h$ to $\wt\eta'((-\infty,t])$ and to $\wt\eta'([t,\infty))$, resp., and $\wt\eta'$ to $(-\infty,t]$ and $[t,\infty)$, resp.\ (i.e., defined as in~\eqref{eqn-past-future-surfaces} with $(\wt h,\wt\eta')$ in place of $(h,\eta')$) are independent, and the latter quantum surface has the law of a $\frac{3\gamma}{2}$-quantum wedge. \label{item-wpsf-char-wedge-intro}
\item (Left/right boundary length process) The left/right $\gamma$-quantum boundary length process of $\wt\eta'$ with respect to $\wt h$, defined in the same manner as $Z$ above, is a correlated two-dimensional Brownian motion with variances and covariances as in~\eqref{eqn-bm-cov}.   \label{item-wpsf-char-homeo-intro}
\end{enumerate} 
Then $\wt\eta'$ is a whole-plane space-filling SLE$_{\kappa'}$ sampled independently from $\wt h$ then parameterized by $\gamma$-quantum mass with respect to $\wt h$. 
\end{thm}

See Theorem~\ref{thm-wpsf-char} for a more precise and slightly stronger version of Theorem~\ref{thm-wpsf-char-intro}. 

Condition~\ref{item-wpsf-char-wedge-intro} is a quantum version of the conformal Markov property for the pair $(\wt h , \wt\eta')$, which is analogous to the conformal Markov property of ordinary $\SLE_{\kappa'}$.  Unlike in the case of ordinary $\SLE_{\kappa'}$, however, we do \emph{not} assume that the law of the curve-decorated quantum surfaces $(  \wt\eta'([t,\infty) ) , \wt h|_{\wt\eta'([t,\infty) ) }  ,  \wt\eta'|_{[t,\infty)}   )$ is stationary in $t$. We also emphasize that we only assume an independence statement for curve-decorated quantum surfaces, i.e.\ equivalence classes of triples consisting of a domain, field, and curve modulo conformal maps. In particular, $(\wt\eta'([t,\infty)) , \wt h|_{\wt\eta'([t,\infty))}  ,  \wt\eta'|_{[t,\infty)})$ does not determine the particular embedding of the curve $\wt\eta'|_{[t,\infty)}$ into $\BB C$.

The second part of condition~\ref{item-wpsf-char-homeo-intro} enables us to define the $\gamma$-quantum length measure with respect to $\wt h$ on $\bdy\wt\eta'([t,\infty))$ simultaneously for all $t\in\BB R$, so that condition~\ref{item-wpsf-char-homeo-intro} makes sense. This is because any arc of $\bdy\wt\eta'([t,\infty))$ which is at positive distance from $\wt\eta'(t)$ is contained in $\bdy \wt\eta'([s,\infty))$ for some rational $s\geq t$, so we can define the length measure on this arc in terms of the length measure on $\bdy \wt\eta'([s,\infty))$. This definition does not depend on the particular choice of rational $s\geq t$ by the last part of condition~\ref{item-wpsf-char-homeo-intro}. 

Since the left/right boundary length process determines when and where the curve hits itself, condition~\ref{item-wpsf-char-homeo-intro} implies that $\wt\eta'$ has the same topology as a space-filling SLE$_{\kappa'}$ curve parameterized by $\gamma$-quantum mass with respect to an independent $\gamma$-quantum cone, i.e., there is a homeomorphism $\BB C\rta \BB C$ which takes $\wt\eta'$ to such a curve.

\subsubsection{Characterization of chordal SLE$_{\kappa'}$}
\label{sec-bead-char-intro}

Using Theorem~\ref{thm-wpsf-char-intro}, we can deduce a similar characterization theorem for ordinary (non-space filling) chordal SLE$_{\kappa'}$ on a certain finite-volume LQG surface with boundary and two marked points (which correspond to the initial and terminal points of the SLE). The particular type of surface on which chordal SLE$_{\kappa'}$ has nice properties is a single bead of a $\frac{3\gamma}{2}$-quantum wedge, i.e., one of the connected components of the beaded quantum surfaces in the left panel of Figure~\ref{fig-thm-illustration}. A bead of a $\frac{3\gamma}{2}$-quantum wedge can be represented by $(\BB H , h^\bead , 0,\infty)$, where $h^\bead$ is a certain GFF-type distribution on $\BB H$ (the $\bead$ stands for ``bead"). One typically conditions on either the left/right boundary lengths, i.e., the $\nu_{h^\bead}$-lengths of $(-\infty,0]$ and $[0,\infty)$ or on these left/right boundary lengths and the area $\mu_{h^\bead}(\BB H)$ of the surface. 
The definition of a bead of a $\frac{3\gamma}{2}$-quantum surface first appeared in~\cite[Definition 4.15]{wedges}, and is reviewed in Section~\ref{sec-wedge}. However, one does not need to know the precise definition to understand Theorem~\ref{thm-bead-char-intro} or its proof. 

There is another natural type of finite-volume LQG surface with boundary called the \emph{quantum disk}, which is defined in~\cite[Section~4.5]{wedges} and which arises as the scaling limit of appropriate random planar maps with boundary. As above, one can consider a quantum disk with specified boundary length or area and boundary length. In the special case when $\gamma =\sqrt{8/3}$, a bead of a $\frac{3\gamma}{2}$-quantum wedge is the same as a quantum disk with two marked points on its boundary. See Section~\ref{sec-wedge} for a review of the definition of a quantum disk. 

Suppose now that $(\BB H,h^\bead,0,\infty)$ is a bead of a $\frac{3\gamma}{2}$-quantum wedge with specified left/right boundary lengths or specified area and left/right boundary lengths and $\eta^\bead$ is an independent chordal SLE$_{\kappa'}$ from 0 to $\infty$ in $\BB H$.
Since $\eta^\bead$ is not a space-filling curve, one cannot parameterize $\eta^\bead$ by the $\gamma$-quantum mass like we did in the case of space-filling SLE$_{\kappa'}$. One option is to parameterize $\eta^\bead$ by the $\mu_{h^\bead}$-mass of the region it disconnects from $\infty$.

\begin{defn} \label{def-chordal-parameterization}
Suppose $X$ is a topological space equipped with a measure $\mu$, $x\in X$, and $\eta : [0,T] \rta X$ is a curve with $\eta(T) = x$. We say that $\eta$ is \emph{parameterized by the $\mu$-mass it disconnects from $x$} if the following is true. For $t \in [0,T]$, let $U_t$ be the connected component of $X\setminus \eta([0,t])$ containing $x$ and let $K_t = X\setminus U_t$ be the hull generated by $\eta([0,t])$. Then for each $t\in [0,T]$, 
\eqb  \label{eqn-chordal-parameterization}
\eta(t) = \eta\left( \inf\left\{ s\in [0,T] : \mu(K_s) \geq t\right\} \right) .
\eqe
\end{defn}

Note that whenever a curve parameterized by the $\mu$-mass it disconnects from $x$ cuts off some region $U$ from $x$, the curve remains constant on a time interval of length $\mu(U)$ immediately following the disconnection time. One has the following analogue of Theorem~\ref{thm-wpsf-char-intro} for chordal SLE$_{\kappa'}$. 

\begin{thm}[Ordinary SLE$_{\kappa'}$ characterization, informal version] \label{thm-bead-char-intro}
Let $\kappa' \in (4,8)$ and $\gamma = 4/\sqrt{\kappa'} \in (\sqrt 2 , 2)$. 
Suppose $(\wt h^\bead , \wt \eta^\bead)$ is a coupling where $\wt h^\bead$ is an embedding into  $(\BB H,0,\infty)$ of a single bead of a $\frac{3\gamma}{2}$ quantum wedge with specified left/right boundary lengths or specified area and left/right boundary lengths and $\wt\eta^\bead$ is a curve from 0 to $\infty$ in $\BB H$ (not necessarily independent from $\wt h^\bead$) parameterized by the $\mu_{\wt h^\bead}$-mass which it disconnects from $\infty$. 
Assume that the following hypotheses are satisfied.
\begin{enumerate}
\item (Laws of complementary connected components) Suppose $t>0$ and condition on the $\mu_{\wt h^\bead}$-mass and the $\nu_{\wt h^\bead}$-length of each of the connected components of $\BB H\setminus \wt\eta^\bead([0,t])$. Under this conditioning, the quantum surfaces obtained by restricting $\wt h^\bead$ to these connected components are conditionally independent. The surfaces corresponding to bounded connected components are quantum disks and the surface corresponding to the unbounded connected component is a bead of a $\frac{3\gamma}{2}$-quantum wedge. \label{item-bead-char-ind-intro}
\item (Topology and consistency) There is a pair $(h^\bead , \eta^\bead)$ such that $h^\bead\eqD \wt h^\bead$ and $\eta^\bead$ is an SLE$_{\kappa'}$ from 0 to $\infty$ in $\BB H$ sampled independently from $h^\bead$ and parameterized by the $\mu_{  h^\bead}$-mass which it disconnects from $\infty$ such that the following is true.
There is a homeomorphism $\BB H \rta \BB H$ which takes $\eta^\bead$ to $ \wt\eta^\bead $ and also preserves the quantum length measure on the boundary of each complementary connected component of $\eta^\bead$. \label{item-bead-char-homeo-intro}
\end{enumerate}
Then $\wt\eta^\bead$ is a chordal SLE$_{\kappa'}$ curve independent from $\wt h^\bead$. 
\end{thm}

It can be seen from the results of~\cite{wedges} (see Lemma~\ref{lem-bead-sle-as}) that the hypotheses of Theorem~\ref{thm-bead-char-intro} are satisfied in the case when $\wt\eta^\bead$ is in fact a chordal SLE$_{\kappa'}$ curve independent from $\wt h^\bead$. 

We will give a precise statement of Theorem~\ref{thm-bead-char-intro} in Theorem~\ref{thm-bead-char}, which in particular makes precise what we mean by conditioning on the areas and boundary lengths of the connected components of $\BB H\setminus \wt\eta^\bead([0,t])$ using the left/right boundary length process.  The theorem will be deduced from a slightly stronger version of Theorem~\ref{thm-wpsf-char-intro} (Theorem~\ref{thm-wpsf-char}) by restricting attention to a single bead of one of the future quantum surfaces.

Theorem~\ref{thm-bead-char-intro} is closely related to the question of whether SLE$_{\kappa'}$ for $\kappa'\in (4,8)$ is \emph{conformally removable}, i.e., every homeomorphism which is conformal off the range of the path is conformal everywhere. It is known that SLE$_\kappa$ is conformally removable for $\kappa \in (0,4)$~\cite{jones-smirnov-removability}, but it is in general a difficult question to determine whether a fractal carpet like the range of SLE$_{\kappa'}$ for $\kappa'\in (4,8)$ is conformally removable. 

If we knew that SLE$_{\kappa'}$ for $\kappa'\in (4,8)$ were conformally removable, then we would immediately obtain a stronger version of Theorem~\ref{thm-bead-char-intro} where hypothesis~\ref{item-bead-char-ind-intro} is required only to hold for $t=\infty$. Indeed, hypothesis~\ref{item-bead-char-ind-intro} for $t=\infty$ together with the fact that the map in hypothesis~\ref{item-bead-char-homeo-intro} preserves the boundary lengths and areas of the complementary connected components (the areas are encoded by the parameterization of $\wt\eta^\bead$) tells us that the quantum surfaces obtained by restricting $\wt h^\bead$ to the connected components of $\BB H\setminus \wt\eta^\bead$ have the same joint law as the quantum surfaces obtained by restricting $h^\bead$ to the connected components of $\BB  H\setminus \eta^\bead$, where $(\BB H , h^\bead , 0,\infty)$ is a bead of a $\frac{3\gamma}{2}$-quantum wedge with appropriate area and left/right boundary lengths and $\eta^\bead$ is an independent chordal SLE$_{\kappa'}$. This implies that there is a conformal map $\BB H\setminus \eta^\bead \rta \BB H\setminus \wt\eta^\bead$ which preserves areas and left/right boundary lengths. This conformal map can be extended to a homeomorphism $\BB H\rta\BB H$ using condition~\ref{item-bead-char-homeo}, which would then have to be the identity by conformal removability.

Theorem~\ref{thm-bead-char-intro} can therefore be thought of as a weaker version of conformal removability for SLE$_{\kappa'}$. 
The theorem allows us to circumvent dealing with the removability question directly when trying to show that certain random curves are SLE$_{\kappa'}$'s.
See Section~\ref{sec-welding} for further discussion of conformal removability.
See also~\cite{mmq-welding} for a different weaker version of conformal removability for SLE$_{\kappa'}$, which neither implies nor is implied by Theorem~\ref{thm-bead-char-intro}.

\subsubsection{SLE$_6$ on a Brownian surface}
\label{sec-bead-mchar-intro}

It is shown in~\cite{lqg-tbm1,lqg-tbm2,lqg-tbm3} that a $\sqrt{8/3}$-LQG surface admits a metric space structure, i.e., a GFF-type distribution $h$ on a domain $D\subset \BB C$ induces a metric $\frk d_h$ on $D$.
In the special case when $\gamma = \sqrt{8/3}$, we can re-phrase our characterization theorems in terms of this metric space structure, and thereby in terms of so-called Brownian surfaces, which we discuss just below. The metric space version of our chordal SLE characterization theorem allows us to identify the scaling limit of percolation on random quadrangulations in \cite{gwynne-miller-perc}.

The \emph{Brownian map} is a random metric measure space, constructed via a continuum analog of the Schaeffer bijection~\cite{schaeffer-bijection}, which arises as the scaling limit of uniform random planar maps on the sphere~\cite{legall-uniqueness,miermont-brownian-map}. 
A \emph{Brownian surface} is a random metric measure space which locally looks like the Brownian map. Such surfaces include the Brownian plane~\cite{curien-legall-plane}, the Brownian disk~\cite{bet-mier-disk}, and the Brownian half-plane~\cite{gwynne-miller-uihpq,bmr-uihpq}.  
 Certain $\sqrt{8/3}$-LQG surfaces are equivalent as metric measure spaces to these Brownian surfaces. 
\begin{itemize}
\item The quantum sphere is equivalent to the Brownian map.
\item The $\sqrt{8/3}$-quantum cone (which we recall is the surface arising in the peanosphere construction) is equivalent to the Brownian plane.
\item The quantum disk is equivalent to the Brownian disk.
\item The $\sqrt{8/3}$-quantum wedge is equivalent to the Brownian half-plane. 
\end{itemize}
It is shown in~\cite{lqg-tbm3} that the metric measure space structure of a $\sqrt{8/3}$-LQG surface a.s.\ determines its quantum surface structure. Hence the results of~\cite{sphere-constructions,tbm-characterization,lqg-tbm1,lqg-tbm2,lqg-tbm3} can be viewed as endowing a Brownian surface with a canonical conformal structure, and constructing numerous additional Brownian surfaces. 

In particular, these results allow us to make sense of $\SLE_{8/3}$- or $\SLE_6$-type curves on Brownian surfaces (by embedding the surface into a domain in $\BB C$, then drawing an independent SLE curve).  It is natural to expect that such curves respectively arise as the scaling limit of self-avoiding walks and percolation explorations on uniform random planar maps.  In the former case, this was proven in~\cite{gwynne-miller-saw}, building on \cite{gwynne-miller-gluing,gwynne-miller-uihpq} (see also~\cite{gwynne-miller-simple-quad} for the finite-volume case), and in the latter case this is proven in \cite{gwynne-miller-perc}, building on \cite{gwynne-miller-simple-quad,gwynne-miller-sle6} and the present paper.  That is, the results of \cite{gwynne-miller-saw,gwynne-miller-perc} allow us to say that the definitions of $\SLE_{8/3}$ and $\SLE_6$ on Brownian surfaces which come from $\sqrt{8/3}$-LQG are the correct ones because they describe the scaling limits of the corresponding discrete models.

Since the conformal structure of a Brownian surface does not depend on the metric measure space structure in an explicit way, the $\sqrt{8/3}$-LQG metric construction does not yield an explicit description of an $\SLE$ on a Brownian surface which depends only on the metric measure space structure.  The results of the present paper allow us to describe a whole-plane space-filling $\SLE_6$ on the Brownian plane (Theorem~\ref{thm-wpsf-mchar}) or a chordal $\SLE_6$ on the Brownian disk (Theorem~\ref{thm-bead-mchar}) by means of a list of conditions which depend only on the metric measure space structure. 

We give here an informal statement of the disk version of this characterization theorem, since this is the version which is used in~\cite{gwynne-miller-perc}. For the statement, we note that there is a natural way to define a length measure on the boundary of the Brownian disk, which is determined by the metric measure space structure and is equivalent to the $\sqrt{8/3}$-LQG boundary length measure on the boundary of a quantum disk. 

\begin{thm}[Characterization of SLE$_6$ on the Brownian disk, informal version] \label{thm-bead-mchar-intro}
Suppose we are given a random curve-decorated metric measure space $(\wt X^\bead , \wt d^\bead , \wt\mu^\bead ,\wt\eta^\bead)$ where $(\wt X^\bead ,\wt d^\bead ,\wt\mu^\bead)$ is a Brownian disk and $\wt\eta^\bead$ is a random curve between two marked points on its boundary parameterized by the $\wt\mu^\bead$-mass it disconnects from its target point. Assume that the following hypotheses are satisfied.
\begin{enumerate}
\item (Laws of complementary connected components) For each $t > 0$, the connected components of $\wt X\setminus \wt\eta^\bead([0,t])$, each equipped with its internal path metric (i.e., the distance between points is the minimum of the $\wt d^\bead$-lengths of paths contained in the component) and the restriction of $\wt\mu^\bead$ are independent Brownian disks if we condition on their areas and boundary lengths. \label{item-bead-mchar-ind-intro} 
\item (Topology and consistency) There is a doubly marked quantum disk $(\BB D , h^\bead , -i , i)$ and a chordal SLE$_6$ $\eta^\bead$ from $-i$ to $i$ in $\BB D$, sampled independently from $h^\bead$ and parameterized by the $\mu_{h^\bead}$-mass which it disconnects from $i$, such that the following is true. 
There is a homeomorphism $\BB D \rta \wt X^\bead$ which takes $\eta^\bead$ to $\wt\eta^\bead$ also preserves the length measure on the boundary of each complementary connected component of $\eta^\bead$. \label{item-bead-mchar-homeo-intro}
\end{enumerate}
Then $\wt\eta^\bead$ is an independent chordal SLE$_6$ on $\wt X^\bead$.
\end{thm}

For $\gamma=\sqrt{8/3}$, a bead of a $\frac{3\gamma}{2}$-quantum wedge is the same as a quantum disk (this is immediate from the definitions in~\cite[Section 4.4 and 4.5]{wedges}), which in turn is equivalent to a Brownian disk, so Theorem~\ref{thm-bead-mchar-intro} is an exact analog of Theorem~\ref{thm-bead-char-intro} for $\kappa'=6$ but with metric spaces in place of quantum surfaces. See also Theorem~\ref{thm-bead-mchar-nat} for a variant of Theorem~\ref{thm-bead-mchar-intro} where we parameterized by quantum natural time (as defined in~\cite{wedges}) instead of by disconnected area.

Theorem~\ref{thm-bead-mchar-intro} is a key tool for showing that such percolation models converge to $\SLE_6$-decorated $\sqrt{8/3}$-LQG surfaces in the Gromov-Hausdorff-Prokhorov-uniform topology~\cite{gwynne-miller-uihpq}, the natural analog of the Gromov-Hausdorff topology for curve-decorated metric measure spaces. 
The reason for this is that the conditions in Theorem~\ref{thm-bead-mchar-intro} are possible to verify for subsequential limits of percolation models on random planar maps.

Indeed, as explained in~\cite{gwynne-miller-perc}, a chordal percolation interface on, e.g., a uniform quadrangulation with simple boundary can be explored one face at a time by the peeling procedue~\cite{angel-peeling}. This allows one to show that the complementary connected components of the interface are independent quadrangulations with simple boundary (which is a discrete analog of condition~\ref{item-bead-mchar-ind-intro}). Furthermore, one can explicitly describe the law of the discrete left/right boundary length process and compute its scaling limit in the Skorokhod topology, which eventually allows us to relate the topology of a subsequential scaling limit of a percolation interface to that of an SLE$_6$.

The above considerations together with Theorem~\ref{thm-bead-mchar-intro} allow us to identify the scaling limit of percolation explorations on uniform random quadrangulations with simple boundary as chordal $\SLE_6$ on the Brownian disk in~\cite{gwynne-miller-perc}. One can also prove similar results for other percolation models on random planar maps using the same strategy, such as site percolation on a uniform triangulation. The results of~\cite{gwynne-miller-perc} in this last case are used in a subsequent series of works by Holden, Sun, and various coauthors to show that uniform triangulations drawn in the plane via the so-called \emph{Cardy embedding} converge to $\sqrt{8/3}$-LQG~\cite{hs-cardy-embedding}.

One way to think about the relationship between Theorem~\ref{thm-bead-mchar-intro} and the results of~\cite{gwynne-miller-perc} is in terms of metric gluing. It is shown in~\cite{gwynne-miller-gluing} that the metric space quotient of two independent quantum wedges identified according to the $\sqrt{8/3}$-quantum length measure on their boundaries is a different $\sqrt{8/3}$-quantum wedge decorated by an SLE$_{8/3}$-type curve. In other words, the $\sqrt{8/3}$-LQG metric is compatible with the conformal welding operations of~\cite{wedges} in the setting when the gluing interface is simple. This statement is what allows us to identify the scaling limit of the self-avoiding walk on random quadrangulations with SLE$_{8/3}$ in~\cite{gwynne-miller-gluing}.

When attempting to identify the scaling limit of percolation on random planar maps, it is natural to ask whether there is an analog of the above metric gluing statements with SLE$_6$ curves in place of SLE$_{8/3}$ curves. In other words, if we have an SLE$_6$ curve on an independent Brownian disk, can we recover the Brownian disk as the metric space quotient of the internal path metrics on the complementary connected components of the curve, glued together according to the natural identification of points on their boundaries? Proving this statement is similar in spirit to proving conformal removability of SLE$_6$ (see~\cite[Section 1.4]{gwynne-miller-gluing} for further discussion of the relationship between conformal removability and metric gluing), and we expect it to be quite difficult. Theorem~\ref{thm-bead-mchar-intro} allows us to identify a subsequential scaling limit of percolation on random planar maps as SLE$_6$ without the need for these sorts of metric gluing statements. 
 

\bigskip

\noindent{\bf Acknowledgements} We thank two anonymous referees for helpful comments on an earlier version of this paper. J.M.\ thanks Institut Henri Poincar\'e for support as a holder of the Poincar\'e chair, during which part of this work was completed.

\subsection{Outline}
\label{sec-outline}

In this subsection we provide a moderately detailed overview of the content of the remainder of the paper and the proofs of our main results. See Figure~\ref{fig-outline} for a schematic illustration of how the sections fit together.

\begin{figure}[ht!]
	\begin{center}
		\includegraphics[scale=.8]{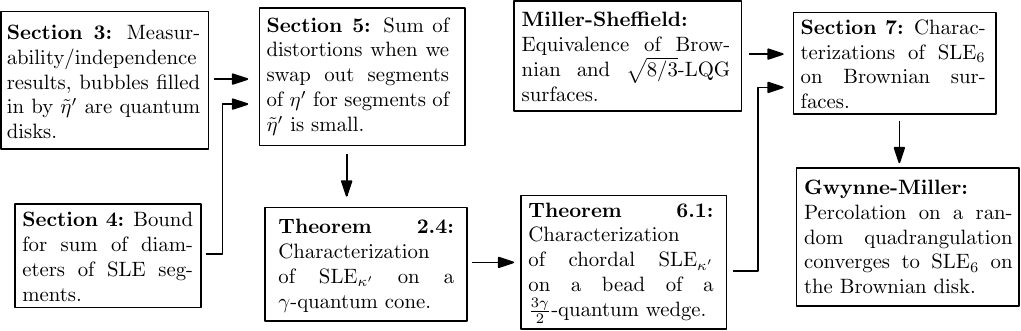}
	\end{center}
	\vspace{-0.025\textheight}
	\caption[Schematic outline of the proof]{\label{fig-outline} Schematic illustration of how the results in this paper fit together. The reader may want to skip Sections 3.2-3.5 on a first read and refer back to the various results as they are used in Section 5. 
	}
\end{figure}

We start in Section~\ref{sec-prelim} by reviewing some background on LQG surfaces, space-filling $\SLE_{\kappa'}$, and the relationships between them. 
We also state a stronger, more precise version of Theorem~\ref{thm-wpsf-char-intro}, namely Theorem~\ref{thm-wpsf-char}.

The next three sections are devoted to the proof of Theorem~\ref{thm-wpsf-char}. The basic idea of the proof is as follows. Let $(\wt h , \wt\eta')$ be the given field/space-filling curve pair from Theorem~\ref{thm-wpsf-char-intro}. Also let $(h,\eta')$ be the (a.s.\ unique by~\cite[Theorem 1.11]{wedges}) pair consisting of an embedding into $(\BB C,0,\infty)$ of a $\gamma$-quantum cone and an independent space-filling SLE$_{\kappa'}$ curve parameterized by $\mu_h$-mass which has the same left/right quantum boundary length process $Z$ as $(\wt h , \wt\eta')$. Then there is a homeomorphism $\BB C\rta\BB C$ which takes $\eta'$ to $\wt\eta'$. We will show that this homeomorphism is the identity map. 

To accomplish this, we will replace quantum surfaces parameterized by $\mu_h$-mass-$1/n$ segments of $\eta'$ with quantum surfaces parameterized by $\mu_{\wt h}$-mass-$1/n$ segments of $\wt\eta'$ one at a time. At each stage of the construction, we get a homeomorphism $\BB C\rta\BB C$ which takes the embedding of one surface to the embedding of the next surface and is conformal on the complement of the $\mu_{ h}$-mass-$1/n$ segment we have just replaced. We will show that the sum of the deviations of these homeomorphisms from the identity map tends to zero when we send the size of the segments we are swapping out to 0. Composing all of the homeomorphisms and sending the size of the segments to zero then shows that the aforementioned map $\BB C\rta\BB C$ which takes $\eta'$ to $\wt\eta'$ must be the identity. 
Our argument has some similarities with the proof of~\cite[Theorem~1.11]{wedges}, which also involves bounding the distortion of the conformal map between two quantum surfaces which differ only in a small segment of a space-filling curve. 

In Section~\ref{sec-disk-law}, we will define several objects in terms of the $\gamma$-quantum cone/space-filling $\SLE_{\kappa'}$ pair $(h,\eta')$  as well as their counterparts with the given pair $(\wt h , \wt\eta')$ in place of $(h,\eta')$. These objects include the beaded quantum surfaces $\mcl S_{a,b}$ (resp.\ $\wt{\mcl S}_{a,b}$) obtained by restricting $h$ (resp.\ $\wt h$) to $\eta'([a,b])$ (resp.\ $\wt\eta'([a,b])$) as well as non-space-filling curves $\eta_{a,b}$ (resp.\ $\wt\eta_{a,b}$) obtained from $\eta'|_{[a,b]}$ (resp.\ $\wt\eta'|_{[a,b]}$) by skipping the intervals of time during which it is filling in a bubble.
We will then prove several statements about the laws of these objects which build on the hypotheses of Theorem~\ref{thm-wpsf-char}. In particular, we will establish the following. 
\begin{itemize}
\item The quantum surface $\mcl S_{a,b}^0$ parameterized by the bubbles cut out by $\eta_{a,b}$ and the quantum surface $\wt{\mcl S}_{a,b}^0$ parameterized by the bubbles cut out by $\wt\eta_{a,b}$ have the same law. More precisely, both of these surfaces are collections of independent quantum disks if we condition on the areas and boundary lengths of their interior connected components (Propositions~\ref{prop-sle-disk} and~\ref{prop-general-disk}).
\item More generally, for $t\in [a,b]$ the joint law of the quantum surface parameterized by the bubbles cut out by $\eta_{a,b}|_{[0,t]}$ and the quantum surface parameterized by the region traced by $\eta'$ after it finishes filling in all of these bubbles is the same as the joint law of the analogous pair of quantum surfaces defined in terms of $(\wt h , \wt\eta')$ (Proposition~\ref{prop-partial-surface-law}).  
\end{itemize} 
These statements will be proven by relating the quantum surfaces we are interested in to other quantum surfaces whose law we know either from the results of~\cite{wedges} or the hypotheses of Theorem~\ref{thm-wpsf-char}. The main purpose of the above statements is that they will allow us to (a) show that each of the overall quantum surfaces in the aforementioned curve-swapping procedure is a $\gamma$-quantum cone and (b) show that the homeomorphism  at each stage of the procedure is the identity map provided the corresponding segment of $\eta'$ does not intersect $\eta_{a,b}$ (i.e., it is entirely contained in one of the bubbles cut out by $\eta_{a,b}$). 

In Section~\ref{sec-diam-sum}, we will prove for given $n\in\BB N$ an estimate for the sum of the squared diameters of the time length-$1/n$ segments of the curve $\eta'|_{[a,b]}$ which intersect the $\SLE_{\kappa'}$-type curve $\eta_{a,b}$ of Section~\ref{sec-disk-law} (Proposition~\ref{prop-diam-sum}).  In particular, we will show that the expectation of this sum tends to zero faster than some negative power of $n$.  (This is the step in the argument that does not extend to $\kappa' \geq 8$ because in this case the sum of the squared diameters of the time length-$1/n$ curve segments will typically be of constant order.)  Combined with the distortion estimates described in~\cite[Section~9]{wedges} (see in particular~\cite[Lemma~9.6]{wedges}) this estimate will eventually allow us to bound how much the conformal maps in the curve-swapping argument mentioned above deviate from the identity.

The argument of Section~\ref{sec-diam-sum} proceeds by way of two other results which are of independent interest. The first of these estimates is a KPZ-type bound for the Lebesgue measure of the $\ep$-neighborhood of a set $X  \subset \BB C$ which is independent from $h$ in terms of the expected number of $\ep$-length intervals needed to cover $(\eta')^{-1}(X)$. This is a variant of~\cite[Theorem~1.1]{ghm-kpz} but with expected Minkowski dimension instead of Hausdorff dimension. The second estimate is a bound for the expected number of  $\ep$-length intervals needed to cover the set of times $t \in [a,b]$ which are not contained in any $\pi/2$-cone interval for $Z$ in $[a,b]$---which is the pre-image of $\eta_{a,b}$ under $\eta'$ (Proposition~\ref{prop-lqg-dim-S}).  We will also introduce in Section~\ref{sec-stability} a regularity condition which is needed in order to apply Proposition~\ref{prop-diam-sum} in the next section. 

Section~\ref{sec-swapping} contains the curve-swapping argument which will eventually lead to a proof of Theorem~\ref{thm-wpsf-char-intro}/Theorem~\ref{thm-wpsf-char}.  The main step is to show that for fixed $a,b\in\BB R$ with $a<b$, the quantum surface $\mcl S_{a,b}$ parameterized by the space-filling $\SLE_{\kappa'}$ segment $\eta'([a,b])$ has the same law as the surface $\wt{\mcl S}_{a,b}$ parameterized by the corresponding segment $\wt\eta'([a,b])$ of our candidate curve (Proposition~\ref{prop-inc-agree}).  To prove this, set $t_{n,k} := a + \frac{k}{n}(b-a)$ for $k\in [0,n]_{\BB Z}$. We will define for each $n\in\BB N$ and each $k\in [0,n]_{\BB Z}$ a quantum surface $\mathring{\mcl C}_{n,k}$, which has the law of a $\gamma$-quantum cone, decorated by a non-space-filling curve $\rng\eta_{n,k}$ (an analog of $\eta_{a,b}$) and a space-filling curve $\rng\eta'_{n,k}$ (an analog of $\eta'$) with the following properties.\footnote{Note that the ring notation has nothing to do with the topological interior; it is just chosen to be visually distinct from the tilde notation.}
\begin{itemize}
\item The sub-surface of $\rng{\mcl C}_{n,0}$ (resp.\ $\rng{\mcl C}_{n,n}$) parameterized by the curve segment $\rng\eta'_{n,0}([a,b])$ (resp.\ $\rng \eta'_{n,n}([a,b])$) has the same law as $\wt{\mcl S}_{a,b}$ (resp.\ $\mcl S_{a,b}$). 
\item The triples $(\mathring{\mcl C}_{n,k}, \rng\eta_{n,k}, \rng\eta'_{n,k})$ are all topologically equivalent. In fact, there exists for each $k \in [1,n]_{\BB Z}$ a homeomorphism  $f_{n,k} : \rng{\mcl C}_{n,k} \rta\rng{\mcl C}_{n,k-1}$ which satisfies $f_{n,k} \circ \rng\eta_{n,k} = \rng\eta_{n,k-1}$ and $f_{n,k} \circ \rng\eta_{n,k}' = \rng\eta'_{n,k-1}$ and which is conformal on $\BB C\setminus \rng\eta_{n,k}([t_{n,k-1} , t_{n,k}])$ where $t_{n,k} = a + \frac{k}{n}(b-a)$ (Lemma~\ref{lem-inc-homeo}). 
\item The map $f_{n,k}$ is the identity map if $\rng\eta_{n,k}'([t_{n,k-1} , t_{n,k}])$ does not intersect $\rng\eta_{n,k}$. 
\end{itemize}
These objects are illustrated in Figure~\ref{fig-swapping-maps}.

The triples $(\mathring{\mcl C}_{n,k}, \rng\eta_{n,k}, \rng\eta'_{n,k})$ will be constructed from $(\wt h, \wt\eta')$ by (roughly speaking) the following procedure, which is illustrated in Figure~\ref{fig-swapping-def}. 
\begin{itemize}
\item First impose a conformal structure on the closure of the union of the bubbles cut out by the curve $\wt\eta_{a,b}$ run up to time $t_{n,k}$ in the same manner that one imposes a conformal structure on the bubbles cut out by $\eta_{a,b}$ run up to time $t_{n,k}$ to get the surface $\mcl S_{a,b}$. The reason we are able to impose a conformal structure in this way is that we know that the bubbles cut out by $\wt\eta_{a,b}$ are quantum disks by the results of Section~\ref{sec-disk-law}. 
\item The first step gives us a quantum surface, which we conformally weld to an independent ``past" $\frac{3\gamma}{2}$-quantum wedge and the ``future" $\frac{3\gamma}{2}$-quantum wedge parameterized by $\wt\eta'([t_{n,k},\infty))$ to get $\rng{\mcl C}_{n,k}$. 
\item The curve $\rng\eta_{n,k}$ is the concatenation of the gluing interface of the bubbles and the future curve $\wt\eta_{a,b}|_{[t_{n,k} , b]}$. The curve $\rng\eta'_{n,k}$ is obtained by filling in the bubbles cut out by $\rng\eta_{n,k}$ with conformal images of segments of $\wt\eta'$, then concatenating the resulting curve with the conformal image of $\wt\eta'|_{[t_{n,k},\infty)}$.  
\end{itemize}
The results of Section~\ref{sec-disk-law} will tell us that the above itemized list of conditions is satisfied for this construction.

We will then use the results of Section~\ref{sec-diam-sum} together with the distortion estimate~\cite[Lemma~9.6]{wedges} to show that the sum of the deviations of the maps $f_{n,k}$ from the identity tends to 0 as $n\rta\infty$. Composing these maps will then give us the desired equality in law $\mcl S_{a,b} \eqD \wt{\mcl S}_{a,b}$. 
Theorem~\ref{thm-wpsf-char} is obtained by applying this equality in law to the quantum surfaces $\{ \mcl S_{(j-1)\ep , j\ep} \}_{j\in\BB N}$ and $\{ \wt{\mcl S}_{(j-1)\ep , j\ep}\}_{j\in\BB N}$ to get that the joint laws of the collections of points $\{\eta'(j\ep)\}_{j\in\BB Z}$ and $\{\wt\eta'(j\ep)\}_{j\in\BB Z}$ are the same, then sending $\ep \rta 0$. 

In the last two sections of the paper we will deduce our other characterization theorems as consequences of Theorem~\ref{thm-wpsf-char}. 
In Section~\ref{sec-chordal}, we will state and prove a precise version of Theorem~\ref{thm-bead-char-intro} (Theorem~\ref{thm-bead-char}). 
The idea of the proof is to identify the surface $(\BB H , \wt h^\bead , 0, \infty)$ with one of the beads of the future $\frac{3\gamma}{2}$-quantum wedge obtained by restricting $\wt h$ to $\wt h|_{\wt\eta'([0,\infty)}$ in the setting of Theorem~\ref{thm-wpsf-char-intro}. 

In Section~\ref{sec-metric-char}, we will restrict attention to the special case when $\kappa'=6$ and prove metric space versions of Theorems~\ref{thm-wpsf-char-intro} and~\ref{thm-bead-char-intro} using the equivalence of Brownian and $\sqrt{8/3}$-LQG surface~\cite{lqg-tbm2} plus the fact that the metric measure space structure and the quantum surface structure of a $\sqrt{8/3}$-LQG surface a.s.\ determine each other~\cite{lqg-tbm3}. 

Appendix~\ref{sec-index} contains an index of some commonly used symbols.

\section{Preliminaries}
\label{sec-prelim}

In this section we will review the definitions of the various objects involved in the statements and proofs of our main results, including LQG surfaces, whole-plane space-filling $\SLE_{\kappa'}$, and the peanosphere construction, and provide references to where more details can be found.  
We also state in Section~\ref{sec-wpsf-char} a stronger version of Theorem~\ref{thm-wpsf-char-intro}, whose proof will occupy most of this paper and which will be used to deduce our other results.

Throughout this paper we fix parameters
\eqb \label{eqn-gamma-kappa}
\gamma \in (0,2) ,\quad \kappa = \gamma^2 \in (0,4),\quad \op{and}\quad \kappa' = \frac{16}{\gamma^2} \in (4,8) .
\eqe
We will often restrict attention to the case when $\gamma \in (\sqrt 2 , 2)$ (equivalently $\kappa \in (2,4)$ and $\kappa' \in (4,8)$) and we will occasionally further restrict to the special case where $\gamma = \sqrt{8/3}$ (equivalently $\kappa = 8/3$ and $\kappa' = 6$), since this is the only case where a metric on $\gamma$-LQG has been constructed.

\subsection{Basic notation}
\label{sec-basic}

Here we record some basic notation which we will use throughout this paper. 

\noindent
We write $\BB N$ for the set of positive integers and $\BB N_0 = \BB N\cup \{0\}$. 
\vspace{6pt}

\noindent
For $a < b \in \BB R$, we define the discrete intervals $[a,b]_{\BB Z} := [a, b]\cap \BB Z$ and $(a,b)_{\BB Z} := (a,b)\cap \BB Z$.
\vspace{6pt}

\noindent
If $a$ and $b$ are two quantities, we write $a\preceq b$ (resp.\ $a \succeq b$) if there is a constant $C>0$ (independent of the parameters of interest) such that $a \leq C b$ (resp.\ $a \geq C b$). We write $a \asymp b$ if $a\preceq b$ and $a \succeq b$.
\vspace{6pt}

\noindent
If $a$ and $b$ are two quantities which depend on a parameter $x$, we write $a = o_x(b)$ (resp.\ $a = O_x(b)$) if $a/b \rta 0$ (resp.\ $a/b$ remains bounded) as $x \rta 0$, or as $x\rta\infty$, depending on context. We write $a= o_x^\infty(b)$ if $a = o_x(b^s)$ for every $s\in \BB R$. 
\vspace{6pt}

\noindent
Unless otherwise stated, all implicit constants in $\asymp, \preceq$, and $\succeq$ and $O_x(\cdot)$ and $o_x(\cdot)$ errors involved in the proof of a result are required to depend only on the auxiliary parameters that the implicit constants in the statement of the result are allowed to depend on.  
\vspace{6pt}

\noindent
For $A\subset\BB C$ and $r  > 0$, we write $B_r(A)$ for the set of points lying at Euclidean distance less than $r$ from $A$. For a singleton we abbreviate $B_r(\{z\}) = B_r(z)$.

\subsection{Liouville quantum gravity surfaces}
\label{sec-lqg-prelim}

In this subsection we review some properties of Liouville quantum gravity surfaces, as defined in Section~\ref{sec-lqg-intro}.  
Throughout, we fix $\gamma, \kappa$, and $\kappa'$ as in~\eqref{eqn-gamma-kappa} and we let $Q = 2/\gamma +\gamma/2$ be as in~\eqref{eqn-lqg-coord}. 
 
\subsubsection{Quantum disks, cones, and wedges} 
\label{sec-wedge}

We will have occasion to consider several particular types of quantum surfaces which are defined in~\cite{wedges}.
Here we give a brief review of these particular types of quantum surfaces, with references to where more information can be found. 

A \emph{quantum disk} is a finite-volume $\gamma$-quantum surface with boundary $(\BB D , h)$ defined in~\cite[Definition~4.21]{wedges}, which can be taken to have fixed boundary length or fixed boundary length and area.  A \emph{singly (resp.\ doubly) marked quantum disk} is a quantum disk together with one (resp.\ two) marked points in $\bdy \BB D$ sampled uniformly (and independently) from the $\gamma$-LQG boundary length measure $\nu_h$. Note that the marked points in~\cite[Definition~4.21]{wedges} correspond to the points $\pm\infty$ in the infinite strip. These points are shown to be sampled uniformly and independently from $\nu_h$ when one conditions on the underlying quantum surface structure in~\cite[Proposition~A.8]{wedges}. 

For $\alpha < Q$, an \emph{$\alpha$-quantum cone} is a doubly-marked quantum surface $(\BB C , h , 0, \infty)$ defined precisely in~\cite[Section~4.3]{wedges} which can be obtained from a whole-plane GFF plus $-\alpha\log|\cdot|$ by ``zooming in near~$0$", so as to fix the additive constant in a canonical way.
In this paper we will be especially interested in the $\gamma$-quantum cone (i.e., $\alpha=\gamma$).
One reason why the case when $\alpha =\gamma$ is special is that if $h$ is some variant of the GFF on a domain $D\subset \BB C$, then near a typical point sampled from the $\gamma$-quantum measure $\mu_h$, the field $h$ locally looks like a whole-plane GFF plus $-\gamma\log|\cdot|$~\cite[Proposition~3.4]{shef-kpz}. 
Hence a $\gamma$-quantum cone describes the behavior of a quantum surface at a quantum typical point.  
It is sometimes convenient to parameterize the set of quantum cones by a different parameter, called the \emph{weight}, which is defined to be
\eqb \label{eqn-cone-weight}
\frk w = 2 \gamma \left( Q - \alpha\right) .
\eqe
We note that the $\gamma$-quantum cone has weight $4-\gamma^2$. 
The reason why the weight parameter is convenient is that it is additive under the gluing and cutting operations for quantum wedges and quantum cones studied in~\cite{wedges}. We will say more about these operations in Section~\ref{sec-welding} just below. We refer to~\cite[Remark 1.3]{wedges} for some discussion and references concerning the connection between the formula~\eqref{eqn-cone-weight} and its half-plane analog~\eqref{eqn-wedge-weight} with earlier predictions for Brownian intersection exponents. In particular, we note that the ``boundary quantum scaling exponent" $\Delta$ arising in the KPZ formula is related to $\frk w$ by $\frk w - 2 = \gamma^2 \Delta$~\cite[Table 1.1]{wedges}. Similar comments apply for the weight of a quantum cone, as defined below. 

For $\alpha \leq Q$, an \emph{$\alpha$-quantum wedge} is a doubly-marked quantum surface $(\BB H , h , 0, \infty)$ defined in~\cite[Section~4.2]{wedges} which can be obtained from a free-boundary GFF on $\BB H$ plus $-\alpha\log|\cdot|$ by ``zooming in near $0$" (so as to fix the additive constant in a canonical way). Quantum wedges in the case when $\alpha \leq Q$ are called \emph{thick wedges} because they describe surfaces homeomorphic to $\BB H$.  

One can also define an \emph{$\alpha$-quantum wedge} for $\alpha \in (Q , Q + \gamma/2)$. In this regime, the surface is no longer homeomorphic to $\BB H$ and is instead described by an ordered Poissonian collection of doubly-marked quantum surfaces, each with the topology of the disk (the two marked points correspond to the points $\pm\infty$ in the infinite strip in~\cite[Definition~4.15]{wedges}). The individual surfaces are called \emph{beads} of the quantum wedge. Instead of looking at a whole quantum wedge, one can consider a single bead of an $\alpha$-quantum wedge conditioned on its left right quantum boundary lengths (i.e., the quantum lengths of the boundary arcs separating the two marked points) or with fixed area and fixed left and right quantum boundary lengths. Such a bead is a sample from the regular conditional distribution of the intensity measure of the Poisson point process of beads conditioned on the left/right boundary lengths or area and left/right boundary lengths of the surface, which gives a probability measure. See~\cite[Section~4.4]{wedges} for more details. Quantum wedges in the case when $\alpha \in (Q , Q +\gamma/2)$ are called \emph{thin wedges}. 

The \emph{weight} of an $\alpha$-quantum wedge for $\alpha  < Q+\gamma/2$ is defined by
\eqb \label{eqn-wedge-weight}
\frk w = \gamma \left( \frac{\gamma}{2} + Q - \alpha \right)  .
\eqe 

The main type of quantum wedge in which we will be interested in this paper is the $\frac{3\gamma}{2}$-quantum wedge, which has weight $2-\gamma^2/2$.  This wedge is thin for $\gamma \in (\sqrt2 , 2)$, which is the main regime in which we will work.  Roughly speaking, the reason why we are interested in the $\frac{3\gamma}{2}$-quantum wedge is that it can be obtained from a $\gamma$-quantum cone by ``cutting it in half" by a pair of $\SLE_{\kappa}$-type curves (see Section~\ref{sec-wpsf}). A $\frac{3\gamma}{2}$-quantum wedge in a neighborhood of either of its marked points locally looks like a $(\frac{4}{\gamma}- \frac{\gamma}{2})$-quantum wedge in a neighborhood of its first marked point, in the sense of local absolute continuity. The latter quantum wedge is thick, and appears in~\cite[Theorem~1.18]{wedges}. 
 
It follows from the definitions in~\cite[Section~4.4 and~4.5]{wedges} that in the special case when $\gamma = \sqrt{8/3}$, a doubly-marked quantum disk has the same law as a single bead of a $\sqrt{6}$-quantum wedge.  More precisely, this follows since the Bessel dimension $\delta = 4/\gamma^2$ from Definition~\cite[Definition~4.15]{wedges} with $\alpha = 3\gamma/2$ coincides with the Bessel dimension $\delta = 3-4/\gamma^2$ from Definition~\cite[Section~4.5]{wedges} when $\gamma=\sqrt{8/3}$.

\subsubsection{Conformal welding and conformal removability}
\label{sec-welding}

As alluded to in Section~\ref{sec-wedge}, the reason for introducing the weight parameter is that it is invariant under various cutting and gluing operations for quantum surfaces.  Suppose $\frk w^-,\frk w^+ > 0$ and $\frk w = \frk w^- + \frk w^+$. It is shown in~\cite[Theorem~1.2]{wedges} that if one cuts a weight-$\frk w$ quantum wedge by an independent chordal $\SLE_\kappa(\frk w^--2 ; \frk w^+-2)$ curve with force points immediately to the left and right of the starting point (or a concatenation of such curves in the thin wedge case) then one obtains a weight-$\frk w^-$ quantum wedge and an independent-$\frk w^+$ quantum wedge which can be glued together according to quantum boundary length to recover the original weight-$\frk w$ quantum wedge. Similarly, by~\cite[Theorem~1.5]{wedges}, if one cuts a weight-$\frk w$ quantum cone by an independent whole-plane $\SLE_\kappa(\frk w-2)$ curve, then one obtains a weight-$\frk w$ quantum wedge whose left and right boundaries can be glued together according to quantum length to recover the original weight-$\frk w$ quantum cone.

More generally, if we are given two quantum surfaces $\mcl S_1$ and $\mcl S_2$ with boundaries and an identification between their boundaries, one can attempt to conformally weld $\mcl S_1$ and $\mcl S_2$ along their boundaries according to this identification. In other words, one can ask if there exists a quantum surface $\mcl S$ such that $\mcl S_1$ and $\mcl S_2$ are sub-surfaces of $\mcl S$ whose union is all of $\mcl S$ and the identification between the boundaries of $\mcl S_1$ and $\mcl S_2$ induced by the inclusion maps agrees with the given identification. 

If such a quantum surface $\mcl S$ exists (which will always be the case in settings we consider), one can further ask if it is unique. The surface $\mcl S$ is endowed with a distinguished subset $A$ corresponding to the gluing interface between $\mcl S_1$ and $\mcl S_2$. Uniqueness of $\mcl S$ is equivalent to the condition that this set $A$ is \emph{conformally removable} in $\mcl S$, i.e.\ every homeomorphism from $\mcl S$ (viewed as a topological space) to itself which is conformal on $\mcl S\setminus A$ is in fact conformal on all of $\mcl S$. 

An important paper on the topic of conformal removability is~\cite{jones-smirnov-removability}, where it is shown in particular that boundaries of H\"older domains are conformally removable in $\BB C$ or in any simply connected sub-domain of~$\BB C$. $\SLE_{\kappa}$ curves for $\kappa \in (0,4)$ are boundaries of H\"older domains~\cite{schramm-sle} so are conformally removable. It is further explained in~\cite[Proposition~3.16]{wedges} that countable, locally finite unions of $\SLE_\kappa$-type curves which do accumulate only at a discrete set of points are conformally removable in $\BB C$ or in any simply connected sub-domain of $\BB C$.

It is a seemingly quite difficult open problem to determine whether an $\SLE_\kappa$ curve is conformally removable for $\kappa \in [4,8)$. 
If we knew that $\SLE_{\kappa'}$-type curves were conformally removable for $\kappa' \in (4,8)$, the proofs in the present paper could be greatly simplified.

\subsubsection{The $\sqrt{8/3}$-LQG metric} 
\label{sec-lqg-metric}

It is shown in~\cite{lqg-tbm1,lqg-tbm2,lqg-tbm3} that in the special case when $\gamma=\sqrt{8/3}$, a $\sqrt{8/3}$-LQG surface admits a metric in addition to its quantum area measure and quantum length measure. In particular, if $D\subset \BB C$ and $h$ is a GFF-type distribution on $D$, then $h$ induces a metric $\frk d_h$ on $D$ called the \emph{$\sqrt{8/3}$-LQG metric}. 
This metric is constructed using a growth process called QLE$(8/3,0)$ which is obtained by, roughly speaking, randomly re-shuffling small increments of an $\SLE_6$ curve in a manner which depends on $h$ and taking a limit as the size of the increments tends to 0 (see~\cite{qle} for a different construction of QLE which is expected, but not yet proven, to be equivalent to the one used to define $\frk d_h$ in the case when $\kappa=6$).  
 
The metric $\frk d_h$ is a well-defined functional of the quantum surface $(D , h)$ in the sense that if $(D,h)$ and $(\wt D , \wt h)$ are related by a conformal map $f : D \rta \wt D$ as in~\eqref{eqn-lqg-coord}, then a.s.\ $\frk d_h(z,w) = \frk d_{\wt h}(f(z) , f(w))$ for each $z,w\in D$. In fact, it can be deduced from the results of~\cite{shef-wang-lqg-coord} and the construction of $\frk d_h$ in~\cite{lqg-tbm1,lqg-tbm2,lqg-tbm3} that this a.s.\ holds for all choices of conformal map $f$ simultaneously. 

The metric $\frk d_h$ is a.s.\ determined by the field $h$, and conversely 
it is shown in~\cite[Theorem~1.4]{lqg-tbm3} that the metric measure space structure of $(D,\frk d_h , \mu_h)$ a.s.\ determines the quantum surface $(D,h)$. Hence in the special case when $\gamma=\sqrt{8/3}$, a $\sqrt{8/3}$-LQG surface can equivalently be viewed as a random metric measure space. 

Several particular types of quantum surfaces are equivalent (i.e., they differ by a measure-preserving isometry) to certain \emph{Brownian surfaces}, random metric measure spaces which arise as the scaling limits of various types of uniform random planar maps; see Section~\ref{sec-bead-mchar-intro}. 
 
\subsubsection{Curves on quantum surfaces} 
\label{sec-surface-curve}

Our main interest in this paper is in quantum surfaces decorated by various types of curves, typically $\SLE_{\kappa}$- or $\SLE_{\kappa'}$-type curves for $\kappa =\gamma^2$ or $\kappa' = 16/\gamma^2$ (as in~\eqref{eqn-gamma-kappa}). It will be important for us to distinguish between curves in subsets of $\BB C$ and \emph{curve-decorated quantum surfaces}.
By the latter, we mean an equivalence class of $(k+3)$-tuples $(D ,h , x_1,\dots,x_k , \eta)$ for some $k\in\BB N$, where $(D,h,x_1,\dots,x_k)$ is an equivalence class representative for a quantum surface with $k$ marked points and $\eta$ is a parameterized curve in $D$, with two such $(k+3)$-tuples $(D,h,x_1,\dots,x_k , \eta)$ and $(\wt D, \wt h, \wt x_1,\dots, \wt x_k , \wt\eta)$ defined two be equivalent if there is a conformal map $f : D\rta \wt D$ which satisfies the conditions of~\eqref{eqn-lqg-coord} and also satisfies $f\circ \eta = \wt\eta$. 

We allow ``curves" $\eta$ which are defined on a general closed subset of $\BB R$, rather than just a single interval, which arise naturally if we want to restrict a curve to the pre-image of some set whose pre-image is not connected.

Since most of our curves will be originally defined on subsets of $\BB C$, we introduce the following notation.

\begin{defn} \label{def-surface-curve}
Let $ (D,h,x_1,\dots,x_k)$ be an embedding of a quantum surface $\mcl S$ and let $\eta$ be a curve in $\BB C$. Let $I$ be the closure of the interior of $ \eta^{-1}(\ol D) $ and define $\eta_{\mcl S}$ to be the curve $\eta|_I$, viewed as a curve on $\mcl S$, so that $(\mcl S , \eta_{\mcl S})$ is a curve-decorated quantum surface represented by the equivalence class of $ (D,h,x_1,\dots,x_k , \eta)$ modulo conformal maps.
\end{defn}

Note that $\eta_{\mcl S}$ does not encode times when $\eta$ bounces off the boundary of $\bdy D$ without spending a positive interval of time in $\ol D$ (this is because we restrict to $I$ instead of $\eta^{-1}(\ol D)$).

\subsubsection{Convergence of quantum surfaces} 
\label{sec-surface-topology}

In this subsection, we will describe how to define a topology on certain families of quantum surfaces (possibly decorated by curves) using the perspective that a quantum surface is the same as an equivalence class of measure spaces modulo conformal maps. We will only write out the definitions explicitly in the case of curve-decorated quantum surfaces. A topology on quantum surfaces without curves is simply obtained by removing all references to the curve $\eta$. 

We first define a metric on certain curve-decorated measure spaces viewed modulo conformal maps, which leads to a notion of convergence for simply connected curve-decorated quantum surfaces.
For $k\in\BB N_0$, let $\BB M_k^{\op{CPU}}$ be the set of all equivalence classes $\mcl K $ of $(k+3)$-tuples $(D , \mu , \eta , x_1,\dots , x_k)$ with $D\subset \BB C$ a simply connected domain, $D\not=\BB C$, $\mu$ a Borel measure on $D$ which is finite on compact subsets of $ D$, $\eta : \BB R\rta \ol D$ a curve which extends continuously to the extended real line $\BB R\cup \{-\infty ,\infty\}$, and $x_1, x_2 ,  \dots , x_k \in D\cup \bdy D$ (with $\bdy D$ viewed as a collection of prime ends). 
Two such $(k+3)$-tuples $(D,\mu, \eta , x_1,\dots , x_k)$ and $(\wt D,\wt \mu , \wt\eta, \wt x_1 , \dots , \wt x_k)$ are declared to be equivalent if there exists a conformal map $f : D\rta \wt D$ such that 
\eqb \label{eqn-cp-equiv}
f_* \mu = \wt \mu ,\quad f \circ \eta = \wt\eta,  \quad \op{and} \quad f(x_j) = \wt x_j,\: \forall j\in [1,k]_{\BB Z} .
\eqe  

The discussion in Section~\ref{sec-lqg-intro} implies that a finite-area curve-decorated $\gamma$-quantum surface $(D,h, \eta , x_1,\dots,x_k)$ with $k$ marked points, for a simply connected domain $D$, $D\not=\BB C$, can be viewed as an element of $\BB M_k^{\op{CPU}}$, with $\mu = \mu_h$ the $\gamma$-quantum area measure.
We note that this element of $\BB M_k^{\op{CPU}}$ a.s.\ determines the quantum surface since the $\gamma$-LQG measure a.s.\ determines the field~\cite{bss-lqg-gff}. Curve-decorated quantum surfaces will be our main examples of elements of $\BB M_k^{\op{CPU}}$. 

To define a metric on $\BB M_k^{\op{CPU}}$, we note that the Riemann mapping theorem implies that each $\mcl K\in \BB M_k^{\op{CPU}}$ admits an embedding of the form $(\BB D , \mu , \eta ,  0, x_2,\dots,x_k)$ (i.e., the domain is $\BB D$ and the first marked point is~$0$).  For a domain $D\subset \BB C$, let $\BB d_D^{\op{P}}$ be the Prokhorov metric on finite Borel measures on $D$ and let $\BB d_D^{\op{U}}$ be the uniform metric on curves in $D$.  We define the \emph{conformal Prokhorov-uniform distance} between elements $\mcl K ,\wt{\mcl K} \in \BB M_k^{\op{CPU}}$ by the formula 
\eqb \label{eqn-cp-dist}
\BB d_k^{\op{CPU}}\left(\mcl K , \wt{\mcl K} \right) =  \inf_{\substack{ (\BB D, \mu , \eta  , x_1 , x_2 , \dots , x_k) \in \mcl K ,\\ (\BB D , \wt\mu  , \wt\eta,  \wt x_1 , \wt x_2, \dots , \wt x_k) \in \wt{\mcl K} } } 
\left\{     \BB d_{\BB D}^{\op{P}}(\mu ,\wt \mu) + \BB d_{\BB D}^{\op{U}}(\eta,\wt\eta) + \sum_{j=1}^k | x_j - \wt x_j|         \right\} .
\eqe  
It is easily verified that $\BB d_k^{\op{CPU}}$ is a metric on $\BB M_k^{\op{CPU}}$, whereby $\mcl K$ and $\wt{\mcl K}$ are $\BB d_{k}^{\op{CPU}}$-close if they can be embedded into $\BB D$ in such a way that their measures are close in the Prokhorov distance, their curves are close in the uniform distance, and their corresponding marked points are close in the Euclidean distance.

Hence the above construction gives us a topology on simply connected curve-decorated quantum surfaces. We will also have occasion to consider convergence of beaded quantum surfaces, i.e.\ those which can be represented as a countable ordered collection of finite quantum surfaces, with each such surface attached to its neighbors at a pair of points, such that quantum area of the beads is locally finite, i.e., the total quantum area of the beads between any two give beads is finite. Examples of beaded quantum surfaces include thin quantum wedges. 

Suppose $\mcl K = (\mcl S , \eta_{\mcl S})$ is a curve-decorated beaded quantum surface with the property that $\eta_{\mcl S}$ enters the beads of~$\mcl S$ in chronological order and does not re-enter any bead after entering a subsequent bead.  Let $T\in (0,\infty]$ be the total mass of the beads in $\mcl S$ and suppose that each bead in $\mcl S$ has $k$ marked points.  We can view $\mcl K$ as a function $[0,T] \rta \BB M_k^{\op{CPU}}$ which is defined at Lebesgue a.e.\ point of $[0,T]$ as follows. For $t\in [0,T]$, we define $\mcl K_t$ to be the curve-decorated quantum surface consisting of the first bead of $\mcl S$ with the property that the sum of the quantum masses of the previous beads (not including the bead itself) is at least $t$, equipped with the segment of $\eta_{\mcl S}$ which is contained in this bead.  We extend this function to all of $[0,\infty)$ by declaring that $\mcl K_t$ is the trivial element of $\BB M_k^{\op{CPU}}$ (i.e., the one whose total mass is 0, whose marked points all coincide, and whose curve is constant at the marked point) for each $t > T$.

We define $\BB M_k^{\op{bead}}$ to be the set of all Borel measurable functions $\mcl K : [0,\infty) \rta \BB M_k^{\op{CPU}}$ which are defined a.e., so that a beaded curve-decorated quantum surface is an element of $\BB M_k^{\op{bead}}$ in the manner described above. We then define a metric on $\BB M_k^{\op{bead}}$ by
\eqb \label{eqn-bead-dist}
\BB d_k^{\op{bead}}\left(  \mcl K , \wt{\mcl K} \right) = \int_0^\infty e^{-t} \left( 1 \wedge \BB d_k^{\op{CPU}}\left( \mcl K_t , \wt{\mcl K}_t \right) \right)  \, dt .
\eqe

\subsection{Space-filling $\SLE_{\kappa'}$ and the peanosphere construction}
\label{sec-peano-prelim}

We will now review the definition of space-filling $\SLE_{\kappa'}$ and its relationship to various $\gamma$-quantum surfaces. We continue to assume that $\gamma$, $\kappa$, and $\kappa'$ are related as in~\eqref{eqn-gamma-kappa}. 

\subsubsection{Imaginary geometry} 
\label{sec-ig-prelim}

Let 
\eqb 
 \chi = \frac{2}{\sqrt\kappa} - \frac{\sqrt\kappa}{2} . \label{eqn-ig-param} 
\eqe
Suppose $D$ is a simply connected domain in $\BB C$ and $h^{\op{IG}}$ is a GFF on $D$ with boundary data chosen in such a way that if $\phi : \BB H\rta D$ i a conformal map, then the boundary data of $h^{\op{IG}}\circ \phi - \chi \op{arg}\phi'$ on $\BB R$ is piecewise constant and changes only finitely many times (the superscript IG stands for ``imaginary geometry" and is used to distinguish this GFF from the GFF-type distributions used elsewhere to define LQG surfaces). The work~\cite{ig1} studies various couplings of $h^{\op{IG}}$ with certain chordal $\SLE_\kappa(\ul\rho)$ and $\SLE_{\kappa'}(\ul\rho)$ curves between points of $\bdy D$ which are called \emph{flow lines} and \emph{counterflow lines} of $h^{\op{IG}}$, respectively. These curves are a.s.\ determined by $h^{\op{IG}}$ and are local sets for $h^{\op{IG}}$ in the sense of~\cite{ss-contour}.  One can consider flow lines with any angle $\theta$ in a certain range depending on the boundary data of $h^{\op{IG}}$, which gives rise to couplings of multiple $\SLE_\kappa(\ul\rho)$ curves started from the same point with the same GFF. Flow lines with the same angle started at different points a.s.\ merge upon intersecting, whereas flow lines with different angles can cross at most once (depending on the angles). 

It is shown in~\cite{ig4} that one can also make sense of flow lines of $h^{\op{IG}}$ started from points in the interior of~$D$. These curves are not exactly $\SLE_\kappa(\ul\rho)$ curves but locally look like such curves. Furthermore, suppose $h^{\op{IG}}$ is a whole-plane GFF, viewed modulo a global additive multiple of $2\pi \chi$ (with $\chi$ as in~\eqref{eqn-ig-param}). Then one can make sense of flow lines of $h^{\op{IG}}$ started from any given point in $\BB C$ as well as counterflow lines of $h^{\op{IG}}$ started from $\infty$ and targeted at any given point in $\BB C$. The flow lines of $h^{\op{IG}}$ for any angle in $[-\pi , \pi)$ are whole-plane $\SLE_\kappa(2-\kappa)$ curves~\cite[Theorem~1.1]{ig4} and the counterflow lines are whole-plane $\SLE_{\kappa'}(\kappa'-6)$ curves~\cite[Theorem~1.6]{ig4}. 
By~\cite[Theorem~1.15]{ig4}, for $z\in\BB C$ it is a.s.\ the case that the flow lines of $h^{\op{IG}}$ started from $z$ with angles $\pm\pi/2$ are the left and right outer boundaries of the counterflow line of $h^{\op{IG}}$ from~$\infty$ to~$z$ if we lift this counterflow line to a path in the universal cover of $\BB C$.

We note that these two flow lines will a.s.\ intersect each other infinitely often if $\kappa \in (2,4)$, but not if $\kappa \leq 2$~\cite[Theorem~1.11]{ig4}. This is related to the fact that SLE$_{\kappa'}(\kappa'-6)$ has cut points for $\kappa' \in (4,8)$, but not for $\kappa'\geq 8$.

\subsubsection{Space-filling $\SLE_{\kappa'}$} 
\label{sec-wpsf}

In this subsection we review the construction of space-filling $\SLE_{\kappa'}$ from~\cite[Sections~1.2.3 and~4.3]{ig4} and~\cite[Section~1.4.1]{wedges}. The basic idea of the construction is to define the outer boundary of $\eta'$ stopped at the first time it hits each specified $z\in\BB Q^2$, which will be a pair of SLE$_\kappa$-type curves, then follow these outer boundaries in order to get a space-filling curve. The outer boundary curves will be defined using imaginary geometry. 

Suppose that $h^{\op{IG}}$ is a whole-plane GFF viewed modulo a global additive multiple of $2\pi\chi$, with $\chi$ as in~\eqref{eqn-ig-param}.  For $z\in\BB Q^2$, let $\eta_z^-$ and $\eta_z^+$ be the flow lines of $h^{\op{IG}}$ started from $z$ with angles $\pi/2$ and $-\pi/2$, respectively (recall Section~\ref{sec-ig-prelim}). 

We define a total order on $\BB Q^2$ by declaring that $z$ comes before $w$ if and only if $z$ lies in a connected component of $\BB C\setminus (\eta_w^-\cup\eta_w^+)$ whose boundary is traced by the left side of $\eta_w^-$ and the right side of $\eta_w^+$.  It follows from the argument of~\cite[Section~4.3]{ig4} that there is a space-filling curve $\eta'$ from $\infty$ to $\infty$ in $\BB C$ which hits points in $\BB Q^2$ in order and is continuous when parameterized by Lebesgue measure. Furthermore, the law of this curve does not depend on the choice of countable dense subset $\BB Q^2$. This curve is called \emph{whole-plane space-filling $\SLE_{\kappa'}$ from $\infty$ to $\infty$}. 

For $\kappa' \geq 8$, whole-plane space-filling $\SLE_{\kappa'}$ is just a two-sided variant of ordinary $\SLE_{\kappa'}$.  For $\kappa'\in (4,8)$, the space-filling $\SLE_{\kappa'}$ curve $\eta'$ evolves in a similar manner to an $\SLE_{\kappa'}$-type curve, but whenever it hits itself and forms a bubble, it enters the bubble and fills it in with a space-filling loop rather than just continuing outside the bubble. In particular, whole-plane space-filling $\SLE_{\kappa'}$ from $\infty$ to $\infty$ in this case is not described by a Loewner evolution, even locally.  The path targeted at a given point $z \in \BB C$ (i.e., parameterized by capacity as seen from $z$) has the law of an $\SLE_{\kappa'}(\kappa'-6)$ process and is in fact the counterflow line of~$h^{\op{IG}}$ from~$\infty$ to~$z$.  That is, this counterflow line can be recovered from~$\eta'$ by skipping all of the bubbles filled in by~$\eta'$ before it hits~$z$. 

One can perform a similar construction to the above starting from a GFF on a proper simply connected sub-domain of $\BB C$ with appropriate boundary data, rather than a GFF on $\BB C$ (this construction is described explicitly in~\cite[Sections~1.2.3 and~4.3]{ig4}).  This gives rise to chordal space-filling $\SLE_{\kappa'}(\rho^L ; \rho^R)$ processes for $\rho^L ,\rho^R \in (-2,\kappa'/2-2)$.  For $\kappa' \geq 8$, chordal space-filling $\SLE_{\kappa'}(\rho^L;\rho^R)$ with $\rho^L,\rho^R \in (-2,\kappa'/2-4]$ is identical to ordinary chordal $\SLE_{\kappa'}(\rho^L;\rho^R)$, and if either $\rho^L \in (\kappa'/2-4,\kappa'/2-2)$ or $\rho^R \in (\kappa'/2-4,\kappa'/2-2)$ or $\kappa' \in (4,8)$ it is obtained from chordal $\SLE_{\kappa'}(\rho^L;\rho^R)$ by iteratively filling in the bubbles it disconnects from its target point. 

As explained in~\cite[Footnote 4]{wedges}, whole-plane space-filling $\SLE_{\kappa'}$ from $\infty$ to $\infty$ can equivalently be constructed from chordal $\SLE_{\kappa'}$, as follows. Let $h^{\op{IG}}$ be a whole-plane GFF viewed modulo a global additive multiple of $2\pi\chi$, as above, and let $\eta_0^-$ and $\eta_0^+$ be the flow lines of $h$ started from~$0$ with angles $\pi/2$ and $-\pi/2$, respectively.  Conditional on $\eta_0^-$ and $\eta_0^+$, sample an independent chordal space-filling $\SLE_{\kappa'}$ in each connected component $U$ of $\BB C\setminus (\eta_0^-\cup \eta_0^+)$, between the two points of $\bdy U$ where $\eta_0^-$ and $\eta_0^+$ intersect (or between 0 and $\infty$, if $\kappa' \geq 8$ in which case $\eta_0^\pm$ do not intersect). 
Then concatenate these chordal space-filling $\SLE_{\kappa'}$ curves.

\subsubsection{Peanosphere construction} 
\label{sec-peanosphere}

It is particularly natural to consider a $\gamma$-quantum cone (recall Section~\ref{sec-wedge}) decorated by an independent whole-plane space-filling $\SLE_{\kappa'}$ curve from $\infty$ to $\infty$. The reason for this is the so-called \emph{peanosphere} or \emph{mating-of-trees} construction, which we now describe. 

Let $\mcl C = (\BB C , h , 0, \infty)$ be a $\gamma$-quantum cone and let $\eta'$ be a whole-plane space-filling $\SLE_{\kappa'}$ curve from $\infty$ to $\infty$ sampled independently from $h$ and parameterized in such a way that $\eta'(0) =0$ and the $\gamma$-quantum area measure satisfies $\mu_h(\eta'([s,t])) = t-s$ whenever $s,t\in\BB R$ with $s<t$. 
For $t\geq 0$, let $L_t$ be equal to the $\gamma$-quantum length of the segment of the left boundary of $\eta'([t,\infty))$ which is shared with $\eta'([0,t])$ minus the $\gamma$-quantum length of the segment of the left boundary of $\eta'([0,t])$ which is not shared with $\eta'([t,\infty))$; and for $t < 0$, let $L_t$ be the $\gamma$-quantum length of the segment of the left boundary of $\eta'([t,0])$ which is not shared with $\eta'((-\infty,t])$ minus the $\gamma$-quantum length of the segment of the left boundary of $\eta'([t,0])$ which is shared with $\eta'((-\infty,t])$. 
Define $R_t$ similarly with ``right" in place of ``left". 
Also let $Z := (L,R) : \BB R\rta \BB R^2$. 

It is shown in~\cite[Theorem~1.9]{wedges} (see also~\cite{kappa8-cov} for the case $\kappa' \geq 8$) that there is a deterministic constant $\alpha = \alpha(\gamma) > 0$ such that $Z$ evolves as a pair of correlated two-dimensional Brownian motions with variances and covariances given by~\eqref{eqn-bm-cov}.  

The Brownian motion $Z$ is referred to as the \emph{peanosphere Brownian motion}. The reason for the name is that $Z$ can be used to construct a random curve-decorated topological space called an \emph{infinite-volume peanosphere}, which a.s.\ differs from $(\BB C ,\eta')$ by a curve-preserving homeomorphism. See Figure~\ref{fig-peano} for an illustration. We remark that there is also a finite-volume analog of the peanosphere construction, with a quantum sphere in place of a $\gamma$-quantum cone and a pair of correlated Brownian excursions in place of a pair of correlated Brownian motions. See~\cite{sphere-constructions} for more details.

By~\cite[Theorem~1.11]{wedges}, $Z$ a.s.\ determines the curve-decorated quantum surface $(\mcl C , \eta'_{\mcl C})$. This determination is local, in the following sense.
For $a,b\in\BB R\cup\{-\infty,\infty\}$ with $a<b$, let $\mcl S_{a,b}$ be the beaded quantum surface parameterized by the interior of $\eta'([a,b])$. 
Then the curve-decorated quantum surface $(\mcl S_{a,b} , \eta'_{\mcl S_{a,b}})$, but not its particular embedding into $\BB C$, is a.s.\ determined by $(Z-Z_a)|_{[a,b]}$ (see, e.g.,~\cite[Lemma~3.12]{ghs-bipolar} for a careful justification of this point). 
Furthermore, by~\cite[Theorem~1.9]{wedges}, for $t\in\BB R$, the quantum surfaces $\mcl S_{-\infty,t}$ and $\mcl S_{t,\infty}$ are $\frac{3\gamma^2}{2}$-quantum wedges and the curve-decorated quantum surfaces $(\mcl S_{-\infty,t} , \eta'_{\mcl S_{-\infty,t}})   $ and $(\mcl S_{t,\infty} , \eta'_{\mcl S_{t,\infty}})$ are independent (here we use Definition~\ref{def-surface-curve}).
In particular, the hypotheses of Theorem~\ref{thm-wpsf-char-intro} are satisfied in the case when $(\wt h , \wt\eta') = (h,\eta')$ is an embedding into $(\BB C , 0, \infty)$ of a $\gamma$-quantum cone together with an independent whole-plane space-filling $\SLE_{\kappa'}$.

\begin{figure}[ht!]
	\begin{center}
		\includegraphics[scale=1]{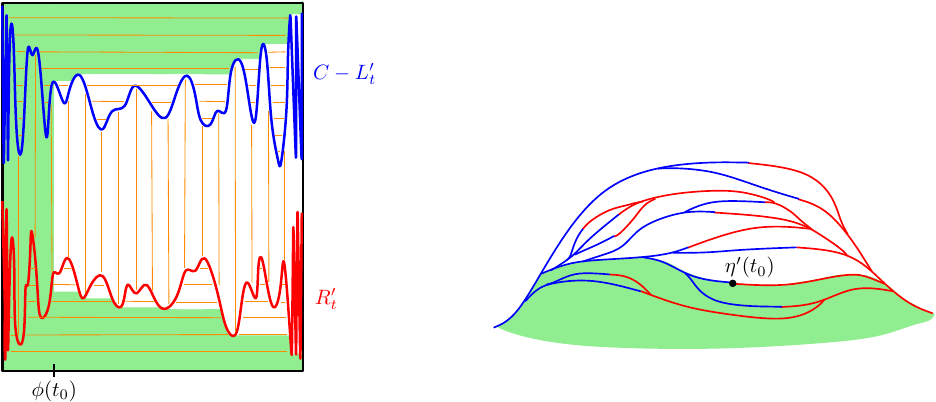}
	\end{center}
	\caption[Peanosphere construction]{\label{fig-peano} 
		An illustration of the definition of the peanosphere as a curve-decorated topological space, which first appeared in~\cite{ghs-bipolar}. Let $Z=(L_t,R_t)_{t\in \BB R}$ be a correlated two-dimensional Brownian motion as in~\eqref{eqn-bm-cov}. Let $\phi: \BB R \rta (0,1)$ be an increasing, continuous and bijective function, and for $t\in(0,1)$ define $ L_t'  :=\phi(L_{\phi^{-1}(t)})$ and $ R_t' :=\phi(R_{\phi^{-1}(t)})$. The left figure shows $ R'$ and $C- L'$, where $C$ is a constant chosen so large that the two graphs do not intersect.
		We draw horizontal lines above the graph of $C- L'$ and below the graph of $ R'$, in addition to vertical lines between the two graphs, and then we identify points which lie on the same horizontal or vertical line segment. By Moore's theorem~\cite{moore} the resulting quotient space $S$ is a topological sphere. This sphere is decorated with a space-filling path $\eta'$ 
		where $\eta'(t)$ for $t \in \BB R$ is the equivalence class of $\phi(t)\in(0,1)$. 
		We call the pair $(S,\eta')$ an \emph{infinite-volume peanosphere}. 
		It is shown in~\cite{wedges} that a $\gamma$-quantum cone decorated by an independent whole-plane space-filling $\SLE_{\kappa'}$ parameterized by $\gamma$-quantum mass is a.s.\ topologically equivalent to the infinite-volume peanosphere constructed from its peanosphere Brownian motion~$Z$. 
	}
\end{figure}

\subsubsection{Describing events in terms of the peanosphere Brownian motion} 
\label{sec-cone-time}

Many objects associated with the pair $(h,\eta')$ can be described explicitly in terms of the peanosphere Brownian motion $Z$. 
Here we list some such objects which will be particularly important for this paper. To describe them, we will need the following definition.

\begin{defn}\label{def-cone-time}
A time $t \in [a,b]$ is called a \emph{$\pi/2$-cone time} for a function $Z = (L,R) : \BB R \rta \BB R^2$ if there exists $t'< t$ such that $L_s \geq L_t$ and $R_s \geq R_t$ for each $s\in [t'   , t ]$. Equivalently, $Z([t'   , t ])$ is contained in the cone $Z(t) + \{z\in \BB C : \op{arg} z \in [0,\pi/2]\}$. We write $  v_Z(t)$ for the infimum of the times $t'$ for which this condition is satisfied, i.e.\ $  v_Z(t)$ is the entrance time of the cone. The \emph{$\pi/2$-cone interval} corresponding to the time $t$ is $[v_Z(t) , t]$ and the corresponding \emph{$\pi/2$-cone excursion} is $(Z-Z_{v_Z(t)})|_{[v_Z(t) , t]}$. 
\end{defn}

\begin{itemize}
\item The curve $\eta'$ hits the left (resp.\ right) outer boundary of $\eta'((-\infty,t])$ at a time $s \geq t$ without forming a bubble if and only if $L$ (resp.\ $R$) attains a running infimum at time $s$ relative to time $t$. 
\item The points where the left and right outer boundaries of $\eta'([t,\infty))$ intersect are precisely the times $s\geq t$ at which $L$ and $R$ attain a simultaneous running infimum. 
\item The set of \emph{bubbles} filled in by $\eta'$ --- i.e., the closures of sets which are disconnected from $\infty$ by $\eta'$ at times when it hits the left or right outer boundary of its past --- are precisely the sets of the form $\eta'([v_Z(t) , t])$ for $t$ a $\pi/2$-cone time for~$Z$. Furthermore, the quantum area (resp.\ quantum boundary length) of the bubble $\eta'([v_Z(t) ,t])$ is equal to $t-v_Z(t)$ (resp.\ $|Z_t-Z_{v_Z(t)}|$) and the boundary of $\eta'([v_Z(t) ,t])$ is traversed by the left (resp.\ right) side of $\eta'$ if and only if $R_t = R_{v_Z(t)}$ (resp.\ $L_t = L_{v_Z(t)}$). 
\end{itemize}

We note that the second and third sets mentioned above are empty for $\kappa'\geq 8$. 

See Figure~\ref{fig-cone-time} for an illustration of Definition~\ref{def-cone-time}. 
A positively correlated Brownian motion (which corresponds to $\kappa' \in (4,8)$) a.s.\ has an uncountable fractal set of $\pi/2$-cone times, whereas an uncorrelated or negatively correlated Brownian motion (corresponding to $\kappa' \geq 8$) a.s.\ has no $\pi/2$-cone times~\cite{shimura-cone,evans-cone}. 
We note that the right endpoint of a $\pi/2$-cone interval containing $t$ is the same as a simultaneous running infimum for $L$ and $R$ relative to time $t$.

\begin{figure}[ht!]
	\begin{center}
		\includegraphics[scale=1]{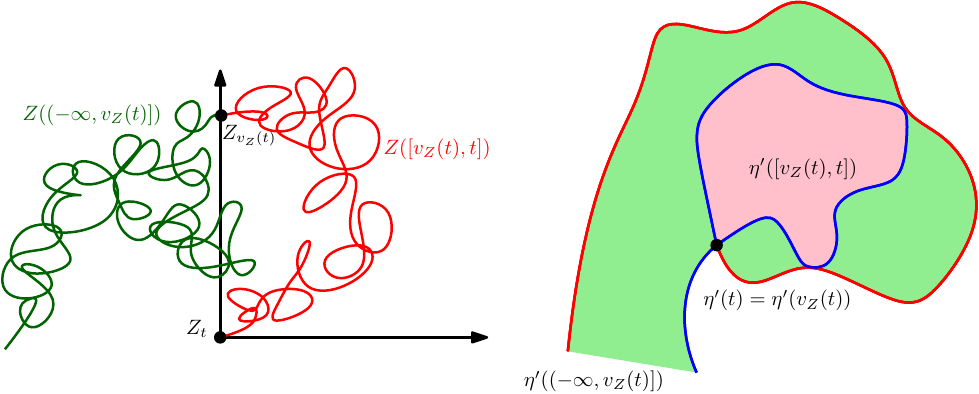}
	\end{center}
	\caption[A $\pi/2$-cone time for Brownian motion]{\label{fig-cone-time} 
\textbf{Left:} A $\pi/2$-cone time $t$ for $Z$ along with the corresponding cone entrance time $v_Z(t)$. \textbf{Right:} The corresponding behavior of the space-filling $\SLE_{\kappa'}$ curve $\eta'$. At time $v_Z(t)$, $\eta'$ closes off a bubble which it fills in during the time interval $[v_Z(t) ,t]$. 
	}
\end{figure}

There are certain special $\pi/2$-cone times (corresponding to special bubbles filled in by $\eta'$) which will be especially important in this paper. 

\begin{defn}\label{def-maximal}
A $\pi/2$-cone time for $Z$ is called a \emph{maximal $\pi/2$-cone time} in an (open or closed) interval $I\subset \BB R$ if $[ v_Z(t) , t] \subset I$ and there is no $\pi/2$-cone time $t'$ for $Z$ such that $[v_Z(t') , t']\subset I$ and $[ v_Z(t) , t] \subset (v_Z(t') , t')$. In this case the interval $[v_Z(t) , t]$ is called a \emph{maximal $\pi/2$-cone interval} for $Z$ in $I$. 
\end{defn}

For example, if $\eta_{-\infty,t}$ denotes the SLE$_{\kappa'}(\kappa'-6)$ curve obtained by parameterizing $\eta'|_{(-\infty,t]}$ by capacity as seen from $\eta'(t)$, then the set of maximal $\pi/2$-cone intervals for $Z$ in $(-\infty,t)$ is in one-to-one correspondence with the set of closures of connected components of $\BB C\setminus \eta_{-\infty,t}$ whose boundaries are entirely traced by either the left side or the right side of $\eta_{-\infty,t}$ via $[v_Z(t) , t] \leftrightarrow \eta'([v_Z(t),t])$. Indeed, this follows from~\cite[Lemma~10.4]{wedges} applied to the time reversal of $\eta'$ (which has the same law as $\eta'$).

\subsection{Stronger characterization theorem for whole-plane space-filling $\SLE_{\kappa'}$}
\label{sec-wpsf-char}

Here we state a precise version of the characterization result Theorem~\ref{thm-wpsf-char-intro}.
This is the version of the theorem which we will actually prove, and the version which will be used to deduce our other characterization theorems in Sections~\ref{sec-chordal} and~\ref{sec-metric-char}. 

Before stating the theorem, we introduce some notation which will be used to formulate the Markov property condition. Let $(\wt h , \wt \eta' , Z)$ be a coupling of an embedding into $\BB C$ of a $\gamma$-quantum cone, a space-filling curve in $\BB R\rta \BB C$, and a correlated two-dimensional Brownian motion with variances and covariance as in~\eqref{eqn-bm-cov}.  For $a,b\in\BB R \cup \{-\infty,\infty\}$ with $a < b$, let $\wt{\mcl F}_{a,b}$ be the $\sigma$-algebra generated by $(Z-Z_a)|_{[a,b]}$ (or $(Z-Z_b)|_{[a,b]}$ if $a = - \infty$ or $Z$ if $a= -\infty$ and $b=\infty$) and the singly marked quantum surfaces $( \wt\eta'([v_Z(s) , s]) , \wt h|_{\wt\eta'([v_Z(s) , s])} , \wt\eta'(s))$, where $s$ ranges over all $\pi/2$-cone times for $Z$ which are maximal in some interval contained in $(a,b)$ with rational endpoints (see Section~\ref{sec-cone-time} for the definition of a $\pi/2$-cone time).  We note that the quantum surfaces parameterized by $\pi/2$-cone times used in the $\sigma$-algebras $\wt{\mcl F}_{a,b}$ correspond to the ``bubbles" filled in by $\wt\eta'$ during the time interval $[a,b]$; see Section~\ref{sec-cone-time}. 

\begin{thm}[Whole-plane space-filling SLE characterization] \label{thm-wpsf-char}
Let $\kappa' \in (4,8)$ and $\gamma = 4/\sqrt{\kappa'} \in (\sqrt 2 , 2)$. 
Suppose that $(\wt h, \wt \eta'  ,Z )$ is a coupling where $\wt h$ is an embedding into $(\BB C , 0 , \infty)$ of a $\gamma$-quantum cone, $\wt\eta':\BB R\rta \BB C$ is a random continuous curve parameterized by $\gamma$-quantum mass with respect to $\wt h$ with $\wt\eta'(0)=0$, and $Z$ is a correlated two-dimensional Brownian motion with variances and covariance as in~\eqref{eqn-bm-cov}. 
Assume that the following conditions are satisfied.
\begin{enumerate}
\item \label{item-wpsf-char-wedge} (Markov property) For each $t\in\BB R$, the $\sigma$-algebras $\wt{\mcl F}_{-\infty,t}$ and $\wt{\mcl F}_{t,\infty}$ defined just above are independent and the doubly marked beaded quantum surface $  (\wt \eta'([t,\infty)) , \wt h|_{\wt\eta'([t,\infty))}  , \wt\eta'(t) , \infty)$ has the law of a $\frac{3\gamma}{2}$-quantum wedge and is independent from $\wt{\mcl F}_{-\infty,t}$. 
\item \label{item-wpsf-char-homeo}  (Topology and consistency) The curve-decorated topological space $(\BB C ,  \wt\eta')$ is equivalent to the infinite-volume peanosphere generated by $Z$. Equivalently, if $((\BB C , h , 0, \infty) ,\eta')$ is the pair consisting of a $\gamma$-quantum cone and an independent space-filling $\SLE_{\kappa'}$ from $\infty$ to $\infty$ parameterized by $\gamma$-quantum mass with respect to $h$ which is determined by $Z$ via~\cite[Theorem~1.11]{wedges}, then there is a homeomorphism $\Phi : \BB C\rta\BB C$ with $\Phi\circ \eta' = \wt\eta'$. Moreover, $\Phi$ a.s.\ pushes forward the $\gamma$-quantum length measure on $\bdy\eta'([t,\infty))$ with respect to $h$ to the $\gamma$-quantum length measure on $\bdy\wt\eta'([t,\infty))$ with respect to~$\wt h$ for each $t\in\BB Q$.
\end{enumerate}
Then $(\wt h, \wt \eta')$ is an embedding into $(\BB C , 0 , \infty)$ of a $\gamma$-quantum cone together with an independent whole-plane space-filling $\SLE_{\kappa'}$ from $\infty$ to $\infty$ parameterized by $\gamma$-quantum mass with respect to $\wt h$. In fact, the map $\Phi$ of condition~\ref{item-wpsf-char-homeo} is a.s.\ given by multiplication by a complex number.
\end{thm}

As explained after the statement of Theorem~\ref{thm-wpsf-char-intro}, in the setting of Theorem~\ref{thm-wpsf-char} the $\gamma$-quantum length measure on $\bdy\wt\eta'([t,\infty))$ with respect to~$\wt h$ is well-defined for all times $t\in\BB R$ simultaneously since we know that the future unexplored quantum surface is a $\frac{3\gamma}{2}$-quantum wedge.
 
Theorem~\ref{thm-wpsf-char} is slightly stronger than Theorem~\ref{thm-wpsf-char-intro}.  The topology and consistency hypothesis, namely condition~\ref{item-wpsf-char-homeo} of Theorem~\ref{thm-wpsf-char}, is just a more precise version of condition~\ref{item-wpsf-char-homeo-intro} from Theorem~\ref{thm-wpsf-char-intro}.  Both theorems also still make the assumption that the future quantum surface has the law of a $\frac{3\gamma}{2}$-quantum wedge.  However, the quantum domain Markov property assumed in Theorem~\ref{thm-wpsf-char} is weaker than the analogous Markov property assumed in Theorem~\ref{thm-wpsf-char-intro} in two respects.
\begin{itemize}
\item The information about the ``past" which we consider in Theorem~\ref{thm-wpsf-char} is encoded by the $\sigma$-algebra $\wt{\mcl F}_{-\infty,t}$, which is contained in the $\sigma$-algebra generated by the curve-decorated quantum surface\\
 $(  \wt\eta'((-\infty, t] ) , \wt h|_{\wt\eta'((-\infty, t]  ) }  ,  \wt\eta'|_{(-\infty, t] }   )$; whereas in Theorem~\ref{thm-wpsf-char-intro} we consider the entire past curve-decorated quantum surface $(  \wt\eta'((-\infty, t] ) , \wt h|_{\wt\eta'((-\infty, t]  ) }  ,  \wt\eta'|_{(-\infty, t] }   )$.
\item The Markov property in Theorem~\ref{thm-wpsf-char} is split into two parts: instead of assuming independence of the entire past and future curve-decorated quantum surfaces, we assume only that each of $\wt{\mcl F}_{t,\infty}$ and $  (\wt\eta'([t,\infty)) , \wt h|_{\wt\eta'([t,\infty))}  , \wt\eta'(t) , \infty)$ (which are each determined by the future curve-decorated quantum surface) is independent from $\wt{\mcl F}_{-\infty,t}$. Note that we do \emph{not} assume that these two objects are jointly independent from $\wt{\mcl F}_{-\infty,t}$.  
\end{itemize}
We remark that for a space-filling $\SLE_{\kappa'}$ on an independent $\gamma$-quantum cone, for $a<b$ the peanosphere Brownian motion increment $(Z-Z_a)|_{[a,b]}$ a.s.\ determines the corresponding curve-decorated quantum surface $(\eta'([a,b]) , h|_{\eta'([a,b])} , \eta'(a) , \eta'(b))$ but the analogous statement for the given triple $(\wt h , \wt\eta' , Z)$ in Theorem~\ref{thm-wpsf-char} is not known a priori. In particular, it is a priori possible that $\wt{\mcl F}_{-\infty,t}$ is a strict subset of $\sigma(  \wt\eta'((-\infty, t] ) , \wt h|_{\wt\eta'((-\infty, t]  ) }  ,  \wt\eta'|_{(-\infty, t] }   )$.

Sections~\ref{sec-disk-law} through~\ref{sec-swapping} will be devoted to the proof of Theorem~\ref{thm-wpsf-char}.

\section{Laws of surfaces and curves}
\label{sec-disk-law}

Throughout this section we assume we are in the setting of Theorem~\ref{thm-wpsf-char}, so in particular $(\wt h  ,\wt\eta')$ is our given field-curve pair; $(h,\eta')$ is an embedding into $(\BB C , 0, \infty)$ of a $\gamma$-quantum cone decorated by an independent space-filling $\SLE_{\kappa'}$; and $Z = (L,R)$ is the peanosphere Brownian motion for the pair $(h,\eta')$.  

In Section~\ref{sec-surface-def}, we will define several objects associated with the pairs $(h,\eta')$ and $(\wt h , \wt\eta')$ which will be used throughout the remainder of the paper. These include the surface~$\mcl S_{a,b}$ parameterized by $\eta'([a,b])$, the chordal $\SLE_{\kappa'}$-type curve~$\eta_{a,b}$ contained in $\eta'([a,b])$, the surface~$\mcl S_{a,b}^0$ parameterized by the bubbles cut out by~$\eta_{a,b}$, the function~$P_{t,\infty}$ which encodes the quantum areas and left/right quantum boundary lengths of the beads of~$\mcl S_{t,\infty}$, and the analogs of these objects with $(\wt h , \wt\eta')$ in place of $(h,\eta')$. 

In the remaining subsections, we study the laws of the above objects. In Section~\ref{sec-sle-msrble}, we will prove some measurability statements for the objects introduced in Section~\ref{sec-surface-def}, which in particular imply that the surface $\mcl S_{a,b}$ is a.s.\ determined by the structure of the bubbles cut out by the curve $\eta_{a,b}$, in analogy with the measurability statements in~\cite[Theorems~1.16 and~1.17]{wedges}.  In Section~\ref{sec-sle-disk}, we show that the bubbles of the surface $\mcl S_{a,b}^0$ cut out by $\eta_{a,b}$ are independent quantum disks conditional on area and boundary length.  In Section~\ref{sec-general-disk}, we show that the same is true for the analogs of these surfaces with $(\wt h , \wt\eta')$ in place of $(h,\eta')$.  In Section~\ref{sec-partial-surface}, we prove equality of the joint laws of certain collections of quantum surfaces defined in terms of $(h,\eta')$ and their analogs defined in terms of $(\wt h , \wt\eta')$ which will be important in Section~\ref{sec-swapping}.

Most of the arguments of this section are somewhat routine in nature. The reader might therefore find it useful to skip Sections~\ref{sec-sle-msrble} through~\ref{sec-partial-surface} on a first read, and refer back to the various lemmas as they are used.

We will frequently use the notation for curves on quantum surfaces from Definition~\ref{def-surface-curve}: if $h$ is an embedding of a quantum surface $\mcl S$ into a domain $D\subset  \BB C$ and $\eta'$ is a curve in $\BB C$, we write $\eta'_{\mcl S}$ for the curve $\eta'|_{(\eta')^{-1}(\ol D)}$, viewed as a curve on $\mcl S$.

\subsection{Definitions of surfaces and curves}
\label{sec-surface-def}

In this subsection we will define several objects associated with the pairs $(h,\eta')$ and $(\wt h , \wt\eta')$ which we will use throughout the remainder of the paper. 
See Figure~\ref{fig-surface-def} for an illustration.

\begin{figure}[ht!]
	\begin{center}
		\includegraphics[scale=1]{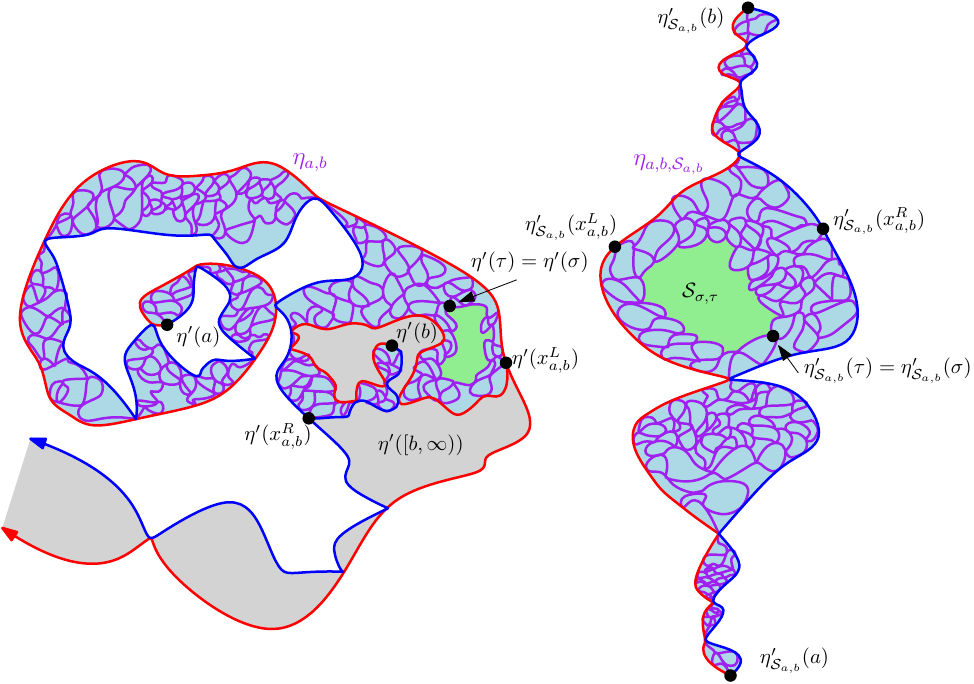}
	\end{center}
	\caption[Ordinary SLE$_\kappa$ obtained by skipping the bubbles filled in by space-filling SLE$_\kappa$]{
\textbf{Left:} The curve segments $\eta'([a,b])$ (light blue) and the curve segment $\eta'([b,\infty))$ (grey) together with the chordal $\SLE_{\kappa'}$-type curve $\eta_{a,b} : [a,b] \rta \eta'([a,b])$ (purple) obtained by skipping the bubbles filled in by $\eta'$ during the time interval $[a,b]$. One of these bubbles is shown in light green. This bubble is equal to the set $\eta'([\sigma , \tau])$ where $[\sigma,\tau] = [\sigma_{a,b}(t) , \tau_{a,b}(t)]$ is the maximal $\tfrac{\pi}{2}$-cone interval for $Z$ in $[a,b]$ which contains the time $t\in [a,b]$. The marked points $\eta'(x_{a,b}^L)$ and $\eta'(x_{a,b}^R)$, respectively, are the points where the left and right outer boundaries, respectively, of $\eta'((-\infty,a])$ and $\eta'([b,\infty))$ meet. 
\textbf{Right:} The beaded quantum surface $\mcl S_{a,b}$ parameterized by $\eta'([a,b])$. Note that $\mcl S_{a,b}$ does not encode the exterior self-intersections of $\eta'([a,b])$; one can think of this surface as being obtained by ``unwinding" $\eta'([a,b])$. The curve $\eta_{a,b,\mcl S_{a,b}}$ on this surface is shown in purple. The surface $\mcl S_{a,b}^0$ is parameterized by the set of bubbles cut out by $\eta_{a,b,\mcl S_{a,b}}$, i.e. the region $\mcl S_{a,b} \setminus \eta_{a,b,\mcl S_{a,b}}$. 
For $\tau=\tau_{a,b}(t)$ and $\sigma=\sigma_{a,b}(t)$, the bubble $\eta'([\sigma,\tau])$ parameterizes the quantum surface $\mcl S_{\sigma,\tau}$ (which we will eventually show is a quantum disk), which is the sub-surface of $\mcl S_{a,b}^0$ corresponding to the bubble cut out by $\eta_{a,b}$ which contains $\eta'(t)$. 
 }\label{fig-surface-def} 
\end{figure}

Define the $\gamma$-quantum cones
\eqb \label{eqn-cone-def}
\mcl C:= \left(\BB C , h , 0,\infty \right) \quad \op{and} \quad \wt{\mcl C} := \left(\BB C , \wt h , 0,\infty \right)
\eqe
For $a,b\in \BB R \cup \{-\infty,\infty\}$ with $a \leq b$, define the quadruply marked and possibly beaded quantum surfaces
\begin{align} \label{eqn-increment-surface}
\mcl S_{a,b}  &:= \left( \eta'([a,b]) , h|_{\eta'([a,b])} , \eta'(a) ,\eta'(b)  , \eta'(x_{a,b}^L ) ,  \eta'(x_{a,b}^R) \right) \quad \op{and} \notag \\
\wt{\mcl S}_{a,b}  &:= \left( \wt\eta'([a,b]) , \wt h|_{\wt\eta'([a,b])} , \wt\eta'(a) , \wt \eta'(b) , \wt \eta'( x_{a,b}^L ), \wt \eta'( x_{a,b}^R )   \right)  
\end{align} 
where here $x_{a,b}^L$ (resp.\ $x_{a,b}^R$) is the first time in $[a,b]$ at which the Brownian motion coordinate $L$ (resp.\ $R$) attains its minimum value on $[a,b]$ or $x_{a,b}^L = x_{a,b}^R = \infty$ if $[a,b]$ is unbounded. By the peanosphere construction~$\eta'(x_{a,b}^L)$ (resp.\ $\eta'(x_{a,b}^R)$) is the point in $\bdy \eta'([a,b])$ where the left (resp.\ right) outer boundaries of $\eta'((-\infty,a])$ and $\eta'([b,\infty))$ meet, and by condition~\ref{item-wpsf-char-homeo} of Theorem~\ref{thm-wpsf-char}, the same is true with $\wt\eta'$ in  place of $\eta'$. Note that in some degenerate cases (e.g., if either~$a$ or~$b$ is equal to~$\infty$) some of the marked points of $\mcl S_{a,b}$ are equal, so there are actually fewer than~$4$ marked points. 

By condition~\ref{item-wpsf-char-wedge} in Theorem~\ref{thm-wpsf-char}, each $\wt{\mcl S}_{t,\infty}$ for $t\in\BB R$ is a $\frac{3\gamma}{2}$-quantum wedge.  By the last statement of~\cite[Theorem~1.9]{wedges}, $\left( \mcl S_{-\infty,t} ,  \mcl S_{t,\infty}\right)$ for $t\in\BB R$ is a pair of independent $\frac{3\gamma}{2}$-quantum wedges which are conformally welded according to quantum length along their boundaries to form $\mcl C$. 
  
For $a,b\in\BB R \cup \{-\infty , \infty\}$ with $a < b$, we define a curve $\eta_{a,b} : [a,b] \rta \eta'([a,b])$ by skipping all of the bubbles filled in by $\eta'$ during the time interval $[a,b]$. More precisely, if $t\in [a,b]$ is such that $\eta'(t)$ is contained in a bubble filled in by $\eta'$ during the time interval $[a,b]$, we let $[\sigma_{a,b}(t),\tau_{a,b}(t)]$ be the time interval during which~$\eta'$ fills in the largest such bubble.  Otherwise, we let $\sigma_{a,b}(t) = \tau_{a,b}(t) = t$. 
Then $\eta'(\sigma_{a,b}(t)) = \eta'(\tau_{a,b}(t))$, and we define $\eta_{a,b}(t) $ to be the common value of these two quantities.
We note that the times $\sigma_{a,b}(t)$ and $\tau_{a,b}(t)$ can be recovered from $\eta_{a,b}$ by the formulas
\begin{align} \label{eqn-tau-def}
  \sigma_{a,b}(t) := \sup\left\{s \leq t \,:\,  \eta_{a,b}(s) \not=  \eta_{a,b}(t) \right\}  \quad \op{and} \quad 
   \tau_{a,b}(t) := \inf\left\{s \geq t \,:\, \eta_{a,b}(s) \not=  \eta_{a,b}(t) \right\}   .
\end{align}

In the special case when $a = -\infty$, the curve $\eta_{-\infty,b}$ is obtained from~$\eta'$ by cutting out the bubbles which~$\eta'$ disconnects from $\eta'(b)$. By translation invariance~\cite[Theorem~1.9]{wedges}, $\eta_{-\infty,b}$ has the same law as $\eta_{-\infty,0}$, which is the chordal $\SLE_{\kappa'}(\kappa'-6)$ counterflow line from $ \infty$ to~$0$ associated with $ \eta'$. In general, $\eta_{a,b}$ is an $\SLE_{\kappa'}$-type curve contained in $\eta'([a,b])$. 
 
Let $\mcl S_{a,b}^0$ be the sub-surface of $\mcl S_{a,b}$ parameterized by the set of bubbles cut out by $\eta_{a,b,\mcl S_{a,b}}$, each marked by the point where they are cut off by $\eta_{a,b,\mcl S_{a,b}}$.  Equivalently, $\mcl S_{a,b}^0$ is parameterized by the set of bubbles filled in by $\eta'_{\mcl S_{a,b}}$ during the time interval $[a,b]$, each marked by the point where $\eta'$ starts (equivalently finishes) filling it in.  We view~$\mcl S_{a,b}^0$ as a quantum surface with (at most) 4 marked points, namely the ones it inherits from~$\mcl S_{a,b}$ (recall~\eqref{eqn-increment-surface}).  By Lemma~\ref{lem-chordal-sle} just below, the bubbles of $\mcl S_{a,b}^0$ are the same as the set of singly-marked quantum surfaces $\mcl S_{v_Z(s) , s} = (\eta'([v_Z(s) , s]) , h|_{\eta'([v_Z(s) , s])} , \eta'(s))$ where $[v_Z(s) ,s]$ ranges over all maximal $\tfrac{\pi}{2}$-cone excursions for $Z$ in $[a,b]$ (recall Section~\ref{sec-cone-time}).  

The surface $\mcl S_{a,b}^0$ is the same as the equivalence class of $\mcl S_{a,b}$ modulo homeomorphisms with are conformal on $\mcl S_{a,b} \setminus \eta_{a,b,\mcl S_{a,b}}$. Since we do not know that the $\SLE_{\kappa'}$-type curve $\eta_{a,b}$ is conformally removable, such a homeomorphism is not necessarily conformal on all of $\mcl S_{a,b}$.  
 
Define the curve $\wt\eta_{a,b}$ and the surface $\wt{\mcl S}_{a,b}^0$ in the same manner as above but with $(\wt h , \wt\eta')$ in place of $(h,\eta')$.  By condition~\ref{item-wpsf-char-homeo} in Theorem~\ref{thm-wpsf-char}, a.s.\ $\Phi\circ \wt\eta_{a,b} = \eta_{a,b}$. Furthermore, the definitions of the times~$\sigma_{a,b}(t)$ and~$\tau_{a,b}(t)$ from~\eqref{eqn-tau-def} are unchanged if we replace $(h,\eta')$ by $(\wt h,\wt\eta')$.  

In the following lemma, we describe the times $\sigma_{a,b}(t)$ and $\tau_{a,b}(t)$ in terms of the $\tfrac{\pi}{2}$-cone times~$s$ for~$Z$ and the associated cone entrance times $v_Z(s)$ (Definition~\ref{def-cone-time}). Recall also Definition~\ref{def-maximal} of a maximal $\tfrac{\pi}{2}$-cone interval. 

\begin{lem} \label{lem-chordal-sle}
Suppose $a,b\in\BB R \cup \{-\infty , \infty\}$ with $a < b$.  
For $t\in [a,b]$, $[\sigma_{a,b}(t) , \tau_{a,b}(t)]$ is the same as the maximal $\tfrac{\pi}{2}$-cone interval for $Z$ in $[a,b]$ containing $t$ (Definition~\ref{def-maximal}), if it exists, or the singleton $\{t\}$ otherwise. In particular, if $\tau_{a,b}(t) \not =t$ then $\tau_{a,b}(t)$ is a $\tfrac{\pi}{2}$-cone time for $Z$ and $v_Z(\tau_{a,b}(t)) = \sigma_{a,b}(t)$. Furthermore, the quantum area (resp.\ quantum boundary length) of the surface $\mcl S_{\sigma_{a,b}(t)  , \tau_{a,b}(t)}$ is equal to $\tau_{a,b}(t) - \sigma_{a,b}(t)$ (resp.\ $|Z_{\tau_{a,b}(t)} - Z_{\sigma_{a,b}(t)}|$). 
\end{lem}
\begin{proof}
This follows from the correspondence between bubbles filled in by $\eta'$ and $\tfrac{\pi}{2}$-cone excursions for $Z$; see Section~\ref{sec-cone-time}. 
\end{proof}

 We note that Lemma~\ref{lem-chordal-sle} and condition~\ref{item-wpsf-char-homeo} in Theorem~\ref{thm-wpsf-char} imply that
\eqb \label{eqn-bubble-surface-msrble}
( \wt{\mcl S}_{a,b}^0 , \wt\eta_{a,b,\wt{\mcl S}_{a,b}^0} ) \in \wt{\mcl F}_{a,b} 
\eqe
where $\wt{\mcl F}_{a,b} $ is the $\sigma$-algebra defined just above Theorem~\ref{thm-wpsf-char}.

The last object we introduce in this subsection is a function which encodes the areas and left/right quantum boundary lengths of the beads of $\mcl S_{t,\infty}$ and $\wt{\mcl S}_{t,\infty}$. 
For $t\in \BB R$ and $s\geq t$, let 
\eqb \label{eqn-bead-function}
P_{t,\infty}(s) := \left( \ol T_{t}(s) - \ul T_t(s) , L_{\ul T_t(s)} - L_{\ol T_t(s)} , R_{\ul T_t(s)} - R_{\ol T_t(s)} \right) 
\eqe
where $\ol T_t(s)$ (resp.\ $\ul T_t(s)$) is the first time after (resp.\ the last time before) $s$ at which the Brownian motion coordinates $L$ and $R$ attain a simultaneous running infimum relative to time $t$. 
We also set $P_{t,\infty}(s) = (0,0,0)$ for $s\leq t$. 

We remark that $P_{t,\infty}$ a.s.\ determines $\left(  \ol T_t(s) , L_{\ol T_t(s)}   , R_{\ol T_t(s)}  \right)$ (by summing the distinct values taken by $P_{t,\infty}(s')$ for $s' \leq s$) since the set of times where $L$ and $R$ attain a simultaneous running infimum a.s.\ has Minkowski dimension $1 - \kappa'/8 < 1/2$ (c.f.\ Lemma~\ref{lem-inf-subordinator} below) and $Z$ is a.s.\ H\"older continuous of any exponent $<1/2$, so the total variation of $Z$ over this set is a.s.\ equal to 0. 

The significance of the function $P_{t,\infty}$ is contained in the following lemma (c.f.\ Section~\ref{sec-cone-time}). 

\begin{lem} \label{lem-bead-inf}
Let $t\in\BB R$ and $s\geq t$. Almost surely, the coordinates of $P_{t,\infty}(s)$ are a.s.\ equal to the quantum area, left quantum boundary length, and right quantum boundary length of the bead of $\mcl S_{t,\infty}$ containing $\eta_{\mcl S_{t,\infty}}'(s)$, respectively. The same is true with $\wt{\mcl S}_{t,\infty}$ and $\wt\eta'$ in place of $\mcl S_{t,\infty}$ and $\eta'$.  
\end{lem}
\begin{proof}
The beads of $\mcl S_{t,\infty}$ are parameterized by the connected components of the interior of $\eta'([t,\infty))$, which are filled in order by $\eta'|_{[t,\infty)}$. The intervals of time during which $\eta'|_{[t,\infty)}$ is filling in one of these components are the same as the maximal time intervals in $[t,\infty)$ during which $\eta'$ does not hit the left and right boundaries of $\eta'((-\infty,t])$ simultaneously. Since a time when $\eta'|_{[t,\infty)}$ hits its left (resp.\ right) boundary without forming a bubble is the same as a time when $L$ (resp.\ $R$) attains a running infimum relative to time $0$ and $\eta'$ never forms a bubble at a time when it hits it left and right boundary simultaneously, we obtain the first statement of the lemma.  The second statement follows from the first statement together with condition~\ref{item-wpsf-char-homeo} in Theorem~\ref{thm-wpsf-char}. 
\end{proof}

\subsection{Measurability results for $(h,\eta')$} 
\label{sec-sle-msrble} 

In this subsection and the next, we will consider the objects in Section~\ref{sec-surface-def} obtained from the $\gamma$-quantum cone/space-filling $\SLE_{\kappa'}$ pair $(h,\eta')$; we will generalize some of the results about these objects to the analogous objects defined in terms of $(\wt h , \wt\eta')$ in the later subsections.
The focus of the present section is on measurability results for these objects, i.e., statements that some object a.s.\ determines another. 

We will start by proving some measurability results for objects defined in terms of the correlated two-dimensional Brownian motion $Z$, then generalize to statements about $(h,\eta')$ using~\cite[Theorem~1.11]{wedges}.
Our first statement tells us that the restriction of $Z$ to an interval $[a,b]$ is a.s.\ determined by the maximal $\tfrac{\pi}{2}$-cone excursions of $Z$ in $[a,b]$ (Definitions~\ref{def-cone-time} and~\ref{def-maximal}) plus a small amount of additional information. 
For the statement of the lemma, we recall the characterization of the times $\sigma_{a,b}(t)$ and $\tau_{a,b}(t)$ from Lemma~\ref{lem-chordal-sle}. 
\begin{lem} \label{lem-bm-determined}
Let $a,b\in \BB R \cup \{\infty\}$ with $a < b$. Let $ \mcl M_{a,b}$ be the set of times in $[a,b]$ at which the coordinates of $Z$ attain a simultaneous running infimum relative to time $a$ and let
\eqbn
\mcl H_{a,b} := \sigma\left( \sigma_{a,b}(t) , \tau_{a,b}(t),  (Z - Z_{\tau_{a,b}(t)} )|_{[\sigma_{a,b}(t) ,\tau_{a,b}(t)]} \,:\, t\in [a,b]  \right) \vee \sigma\left(  \mcl M_{a,b} \right)  .
\eqen
Then $(Z-Z_a)|_{[a,b]}$ is $\mcl H_{a,b}$-measurable.
\end{lem}

The set 
\eqb \label{eqn-cone-complement}
[a,b]\setminus \bigcup_{t\in [a,b]} (\sigma_{a,b}(t) , \tau_{a,b}(t))
\eqe
a.s.\ has zero Lebesgue measure. This fact does not imply Lemma~\ref{lem-bm-determined} since there can be ``local time" fluctuations of $Z$ on this small set, so it is not obvious that there is a unique way to concatenate the excursions $(Z - Z_{\tau_{a,b}(t)} )|_{[\sigma_{a,b}(t) ,\tau_{a,b}(t)]} $ to recover $(Z-Z_a)|_{[a,b]}$. 

We will deduce Lemma~\ref{lem-bm-determined} from the following stronger statement for the case when $a = -\infty$.  The reason why we can get a stronger statement in this case is that the set~\eqref{eqn-cone-complement} is regenerative when $a = -\infty$ (but not for any other values of $a$).

\begin{lem} \label{lem-bm-determined-infty}
Let $b\in\BB R$ and let for $r \leq b$, let
\eqbn
\mcl H_{-\infty,b}'(r) := \sigma\left(   (Z - Z_{\tau_{-\infty,b}(t)})|_{[r,b]}   \right). 
\eqen
Then $(Z-Z_b)|_{[r,b]}$ is $\mcl H_{-\infty,b}'(r)$-measurable.
\end{lem}  
\begin{proof}
By translation invariance we can assume without loss of generality that $b = 0$, so that $Z_b = 0$.  Condition on $\mcl H_{-\infty,0}'(r)$ and sample $Z^1$ and $Z^2$ from the regular conditional law of $Z$ given $\mcl H_{-\infty,0}'(r)$ in such a way that they are conditionally independent given $\mcl H_{-\infty,0}'(r)$.  Then $Z^1 \eqD Z^2 \eqD Z$; for each $t\in [r,0]$, $\tau_{-\infty, 0}(t)$ is the right endpoint of the maximal $\tfrac{\pi}{2}$-cone interval in $(-\infty, 0]$ containing $t$ for each of $Z^1$ and $Z^2$; and for each $t\in [r,0]$, $Z^1_t - Z^1_{\tau_{-\infty,0}(t)}  = Z_t^2 - Z_{\tau_{-\infty,0}(t)}^2$.  We must show that $Z^1|_{[r,0]} = Z^2|_{[r,0]}$ a.s. 

To see this, we consider the discontinuous process
\eqbn
W_t :=  Z_t- Z_{\tau_{-\infty,0}(t)}, \quad \forall t \leq 0 . 
\eqen
Since $\tau_{-\infty,0}(t)$ is the smallest $s\geq t$ for which $W_s = 0$, it follows that $\sigma(W|_{[t,0]}) = \mcl H_{-\infty,0}'(t)$. Since each $\tau_{-\infty,0}(t)$ is determined by $Z|_{[t,0]}$,  
\eqbn
W|_{[t,0]} \in \sigma\left( Z^1|_{[t,0]} \right) \cap \sigma\left(Z^2|_{[t,0]} \right) . 
\eqen
For $t \leq 0$, the regular conditional law of $W|_{(-\infty,t]}$ given $Z^1|_{[t,0]}$ can be described as follows.
The process $W|_{(-\infty,t]}$ evolves as a Brownian motion with variances and covariances as in~\eqref{eqn-bm-cov} started from $W_t$ at time $t$ and run backward until the first time it exits the cone $W_{\tau_{-\infty,0}(t)} + [0,\infty)^2$ (this time is $\sigma_{-\infty,0}(t)$). Conditional on $W|_{[\sigma_{-\infty,0}(t) , 0]}$, the rest of the process, $W|_{(-\infty,\sigma_{-\infty,0}(t)]}$, has the same law as $W$.  This conditional law depends only on $W|_{[t,0]}$, so $W|_{(-\infty,t]}$ is conditionally independent from $Z^1|_{[t,0]}$ given $W|_{[t,0]}$. 
 
By our choice of coupling, for each $t\in [r,0]$ the regular conditional law of $Z^2$ given $Z^1|_{[t,0]}$ and $W$ depends only on $W$.  Since the conditional law of $W|_{(-\infty,t]}$ given $Z^1|_{[t,0]}$ depends only on $W|_{[t,0]}$, we infer that the conditional law of $Z^2$ given $Z^1|_{[t,0]}$ depends only on $W|_{[t,0]}$.  Since $W|_{[t,0]}$ is determined by $Z^2|_{[t,0]}$, the regular conditional law of $(Z^2 - Z^2_t)|_{(-\infty, t]}$ given $Z^1|_{[t,0]} $ and $ Z^2|_{[t,0]}$ is the same as its regular conditional law given only $Z^2|_{[t,0]}$, which is the same as the law of $Z$, run backward starting from $0$. 

In particular, if we set
\eqbn
\mcl G_t := \sigma\left( Z^1|_{[t,0]} ,\, Z^2|_{[t,0]} \right) 
\eqen
then for $s \leq t$ it holds that $\BB E\left[ Z_s^2 \,|\, \mcl G_t \right] = Z_t^2$.
By symmetry, $\BB E\left[ Z_s^1 \,|\, \mcl G_t \right] = Z_t^1$. 
Therefore, $Z^1 - Z^2$ is a backward continuous $\{\mcl G_t\}_{t\leq 0}$-martingale. 
By our choice of coupling, $Z^1-Z^2$ is constant on the intersection with $[r,0]$ of every maximal $\tfrac{\pi}{2}$-cone interval for $Z^1$ (equivalently, for $Z^2$) in $(-\infty,0]$. 

The set $\mcl R := (-\infty,0] \setminus \bigcup_{t\leq 0} (\sigma_{-\infty,0}(t) , \tau_{-\infty,0}(t))$ is the set of so-called ancestor free times for $Z$, run backward from time $0$ (see~\cite[Section~10.2]{wedges}). The set $\mcl R$  is easily seen to be regenerative and scale-invariant, so has the law of the range of a stable subordinator run backward from time $0$. The Hausdorff dimension of $\mcl R$ is a.s.\ equal to $\kappa'/8$ (see, e.g.,~\cite[Example 2.3]{ghm-kpz}), so the index of this subordinator is $\kappa'/8$.
In particular, the Minkowski dimension of the set $\mcl R\cap [r,0]$ is a.s.\ equal to $\kappa'/8$. Since $Z^1-Z^2$ is constant on $[r,0]\setminus\mcl R$ and is a.s.\ H\"older continuous of any exponent less than $1/2$, and $\kappa'/8 < 1$ we infer that the quadratic variation of $Z^1-Z^2$ on $[r,0]$ is a.s.\ equal to $0$. Therefore, $(Z^1-Z^2)|_{[r,0]} = 0$ a.s.
\end{proof}

Before we can deduce Lemma~\ref{lem-bm-determined} from Lemma~\ref{lem-bm-determined-infty}, we first need the following fact about the simultaneous running infima of $L$ and $R$ which will tell us that the set $\mcl M_{a,b}$ of Lemma~\ref{lem-bm-determined} is in some sense negligible. 

\begin{lem} \label{lem-inf-subordinator}
Let $ \mcl M_{0,\infty}$ be the set of times $s\geq 0$ at which $L$ and $R$ attain a simultaneous running infimum relative to time $0$. Then $\mcl M_{0,\infty}$ has the law of the range of a $1-\kappa'/8$-stable subordinator. 
\end{lem}
\begin{proof}
It is clear from the Markov property of Brownian motion that $\mcl M_{0,\infty}$ is a scale-invariant regenerative set, so is either empty or has the law of the range of a stable subordinator.  A time at which $L$ and $R$ attain a simultaneous running infimum relative to time $0$ is the same as a $\tfrac{\pi}{2}$-cone time for $Z$ whose corresponding $\tfrac{\pi}{2}$-cone interval $[v_Z(s) ,s]$ contains $0$. Hence it follows from~\cite[Theorem~1]{evans-cone} (see also the proofs of~\cite[Lemma~8.5]{wedges}) that the Hausdorff dimension of $\mcl M_{0,\infty}$ is a.s.\ equal to $1-\kappa'/8$ if $\kappa'\in (4,8)$. 
The statement of the lemma follows.
\end{proof}

\begin{proof}[Proof of Lemma~\ref{lem-bm-determined}]
For $r_1,r_2  \in [a,b]$ with $r_1< r_2$, let 
\eqbn
E(r_1,r_2) :=  \left\{ [r_1,r_2] \cap \mcl M_{a,b} = \emptyset\right\} ,
\eqen
and note that $E(r_1,r_2) \in \mcl H_{a,b}$ by definition. 

If $t \in [r_1,r_2]$ and $\sigma_{-\infty,r_2}(t) < a$, then $  \tau_{-\infty , r_2 }(t)  \in \mcl M_{a,b} \cap [t,r_2]$.  Therefore, on the event $E(r_1,r_2)$ we have $\sigma_{-\infty, r_2}(t) \geq a$ for each $t\in [r_1,r_2]$, whence the $\tfrac{\pi}{2}$-cone interval $[\sigma_{-\infty , r_2 }(t) , \tau_{-\infty , r_2 }(t)]$ is entirely contained in the maximal $\tfrac{\pi}{2}$-cone interval $[\sigma_{a,b}(t) , \tau_{a,b}( t)]$ in $[a,b]$.  The times $\sigma_{-\infty , r_2 }(t) $ and $ \tau_{-\infty , r_2 }(t)$ are determined by $(Z-Z_{\tau_{a,b}(t)})|_{[\sigma_{a,b}(t) ,\tau_{a,b}(t)]}$ on this event since $[\sigma_{-\infty , r_2 }(t) , \tau_{-\infty , r_2 }(t)]$ must be the maximal $\tfrac{\pi}{2}$-cone interval for $Z$ in $[\sigma_{a,b}(t) , r_2]$ which contains $t$. 

The preceding paragraph implies that on $E(r_1,r_2)$, each of the Brownian cone excursions $(Z-Z_{\tau_{-\infty,r_2}(t)} )$ in the interval of time $[\sigma_{-\infty , r_2 }(t) , \tau_{-\infty , r_2 }(t)]$ for $t\in [r_1,r_2]$ is a.s.\ determined by $\mcl H_{a,b}$. By Lemma~\ref{lem-bm-determined-infty}, on $E(r_1,r_2)$ the $\sigma$-algebra $\mcl H_{a,b}$ a.s.\ determines $(Z-Z_{r_2})|_{[r_1,r_2]}$.  Exhausting over all rational values of $r_1,r_2\in [a,b]$, we find that $\mcl H_{a,b}$ a.s.\ determines $Z_t - Z_{\ul T_{a}(t)}$ for each $t\in [a,b]$, where $\ul T_{a}(t)$ is the largest $s\leq t$ such that $s\in \mcl M_{a,b}$ (as in~\eqref{eqn-bead-function}). 

The set $\mcl M_{a,b}$ has the law of range of a $1-\kappa'/8$-stable subordinator, so a.s.\ has Minkowski dimension $1-\kappa'/8 < 1/2$. Since $Z$ is a.s.\ H\"older continuous for any exponent smaller than $1/2$, it follows that the total variation of $Z$ over $\mcl M_{a,b}$ is a.s.\ equal to zero. Therefore, $(Z-Z_a)|_{[a,b]}$ is a.s.\ determined by $\{ Z_t - Z_{\ul T_{a,\infty}(t)} : t\in [a,b]\}$, and hence by $\mcl H_{a,b}$. 
\end{proof}

Now we turn our attention to measurability statements for objects defined directly in terms of the pair $(h, \eta')$. 
For this purpose we first need the following statement about the law of $\eta'|_{[a,b]}$. 

\begin{lem}
\label{lem-wpsf-law}
Let $a < b$. 
The conditional law of the curve $\eta'|_{[a,b]}$ given $\eta'|_{\BB R\setminus [a,b]}$ and $h$ is that of a concatenation of independent chordal space-filling $\SLE_{\kappa'}$ curves in the interior connected components of $\eta'([a,b])$, whose laws are described as follows.
\begin{itemize}
\item For each interior connected component $U$ of $\eta'([a,b])$ whose boundary is part of $\bdy \eta'((-\infty,a])$, the conditional law of the segment of $\eta'$ contained in $\ol U$ is that of a chordal space-filling $\SLE_{\kappa'}$ in $U$ between the points where $\eta'$ starts and finishes filling in $U$. 
\item For the interior connected component $U_*$ of $\eta'([a,b])$ whose boundary contains non-trivial arcs of each of $\bdy \eta'((-\infty,a])$ and $\bdy \eta'([b,\infty))$, the conditional law of the segment of $\eta'$ contained in $\ol U_*$ is that of a chordal space-filling $\SLE_{\kappa'}(\kappa'/2-4 ; \kappa'/2-4)$ in $U_*$ between the points where $\eta'$ starts and finishes filling in $U_*$, with force points located at the two points of $\bdy U_*$ where $\bdy \eta'((-\infty,a])$ and $\bdy \eta'([b,\infty))$ meet. 
\item For each interior connected component $U $ of $\eta'([a,b])$ whose boundary is part of $\bdy \eta'([b,\infty))$, the conditional law of the segment of $\eta'$ contained in $\ol U $ is that of a chordal space-filling $\SLE_{\kappa'}(\kappa'/2-4 ; \kappa'/2-4)$ in $U $ between the points where $\eta'$ starts and finishes filling in $U$, with force points located immediately to the left and right of the starting point. 
\end{itemize} 
\end{lem}
\begin{proof}
Let $h^{\op{IG}}$ be the whole-plane GFF, viewed modulo a global additive multiple of $2\pi \chi$, used to construct~$\eta'$ as in Section~\ref{sec-wpsf}, where here~$\chi$ is as in~\eqref{eqn-ig-param} ($\op{IG}$ stands for ``imaginary geometry").  Recall that the outer boundary of~$\eta'$ at the first time it hits a rational $z\in\BB C$ is the union of the flow lines of~$h^{\op{IG}}$ started from~$z$ with angles $\pm\tfrac{\pi}{2}$. 

Since $\eta'(\cdot+a)  - \eta'(a) \eqD \eta'$~\cite[Theorem~1.9]{wedges} and by the description of whole-plane space-filling SLE$_{\kappa'}$ in~\cite[Footnote 4]{wedges}, the conditional law of~$\eta'|_{[a,\infty)}$ given $h$ and $\eta'|_{(-\infty,a]}$ is that of a concatenation of chordal space-filling $\SLE_{\kappa'}$ curves in the interior connected components of $\eta'([a,\infty))$.  The time $b$ is a reverse stopping time for the conditional law of $\eta'|_{[a,\infty)}$ given $h$ and $\eta'|_{(-\infty,a]}$ ($b$ is the largest $t \geq a$ such that the set $\BB C\setminus (\eta'((-\infty,a]) \cup \eta'([t,\infty)) )$ has quantum mass $b-a$).  Hence the conditional law of the field $h^{\op{IG}}|_{\eta'([a,b])}$ given $\eta'|_{\BB R\setminus [a,b]}$ and $h$ is that of an independent GFF with Dirichlet boundary conditions in each of the interior connected components of~$\eta'([a,b])$ specified by the boundary data for~$h$ along the outer boundaries of~$\eta'((-\infty,a])$ and~$\eta'([b,\infty))$. In particular, using the description of flow line boundary data in~\cite[Theorem~1.1]{ig4}, we find that with $\lambda' := \pi/\sqrt{\kappa'}$ and~$\chi$ as above, 
\begin{itemize}
\item The boundary data for~$h^{\op{IG}}|_{\eta'([a,b])}$ along the left (resp.\ right) outer boundary of $\eta'((-\infty,a])$ is given by $\lambda'- \chi \op{arg} \phi'$ (resp.\ $-\lambda' - \chi \op{arg}\phi'$), where~$\phi$ denotes a conformal map which takes a given interior connected component of~$\eta'([a,b])$ to~$\BB H$ in such a way that its intersection with the left (resp.\ right) outer boundary of~$\eta'((-\infty,a])$ is sent to $(-\infty,0]$ (resp.\ $[0,\infty)$). 
\item The boundary data for $h^{\op{IG}}|_{\eta'([a,b])}$ along the left (resp.\ right) outer boundary of $\eta'([b,\infty))$ is given by $ \lambda'(\kappa'/2-3) - \chi\op{arg}\phi'$ (resp.\ $-\lambda'(\kappa'/2-3) - \chi\op{arg}\phi'$), with $\phi$ as above.  
\end{itemize}

From the construction of space-filling $\SLE_{\kappa'}$ in~\cite[Section~1.2.3]{ig4}, we see that the segment of $\eta'$ contained in the closure of each interior connected component of $\eta'([a,b])$ is equal to the space-filling $\SLE_{\kappa'}$ counterflow line of $h^{\op{IG}}|_U$. 
The above description of the boundary data for $h^{\op{IG}}|_{U}$ implies that the conditional law of $\eta'$ must be as in the statement of the lemma. 
\end{proof}

We next prove a lemma which implies in particular that the quantum surface decorated by an $\SLE_{\kappa'}$-type curve $(\mcl S_{a,b} , \eta_{a,b,\mcl S_{a,b}} )$ is a.s.\ determined by the curve-decorated quantum surface $(\mcl S_{a,\tau}^0 , \eta_{a,b,\mcl S_{a,\tau}^0})$ from Section~\ref{sec-surface-def}, which we recall only encodes the conformal structure and the topology of the bubbles cut out by $\eta_{a,b,\mcl S_{a,b}}$ (rather than the full conformal structure of $\mcl S_{a,b}$).  One should think of this result as an analog of the measurability statements in~\cite[Theorems~1.16 and~1.17]{wedges} for the pair $(\mcl S_{a,b} , \eta_{a,b,\mcl S_{a,b}})$.

\begin{lem} \label{lem-partial-surface-msrble}
Let $a,t,b\in \BB R$ with $a \leq t \leq b$ and let $\tau=\tau_{a,b}(t)$ be as in~\eqref{eqn-tau-def}.  The following measurability statements hold. 
\begin{enumerate}
\item $(\mcl S_{a,\tau} , \eta'_{\mcl S_{a,\tau}} )$ is a.s.\ determined by $(\mcl S_{a,\tau}^0   , \eta'_{\mcl S_{a,\tau}})$.  \label{item-msrble-wpsf}
\item $(\mcl S_{a,\tau} , \eta_{a,b,\mcl S_{a,\tau}})$ is a.s.\ determined by $(\mcl S_{a,\tau}^0 , \eta_{a,b,\mcl S_{a,\tau}^0})$.   \label{item-msrble-chordal}
\end{enumerate}
\end{lem}
\begin{proof}  
We first prove assertion~\ref{item-msrble-wpsf}. 
It is clear that $\mcl S_{a,\tau}^0$ a.s.\ determines $\tau$, since $\tau - a$ is the total quantum mass of $\mcl S_{a,\tau}^0$. 
Recalling Lemma~\ref{lem-chordal-sle}, we find that $(\mcl S_{a,\tau}^0   , \eta'_{\mcl S_{a,\tau}})$ a.s.\ determines the Brownian cone excursions $(Z-Z_t)|_{[v_Z(t) , t]}$ for each maximal $\tfrac{\pi}{2}$-cone time $t$ for $Z$ in $[a,b]$. 
Furthermore, $\mcl S_{a,\tau}^0$ a.s.\ determines the quantum mass of each bead of $\mcl S_{a,\tau}$, which together with the last two marked points of $\mcl S_{a,b}^0$ from~\eqref{eqn-increment-surface} determines the times in $[a,b]$ at which the two coordinates of $Z$ attain a simultaneous running infimum (Lemma~\ref{lem-bead-inf}).  
Hence Lemma~\ref{lem-bm-determined} implies that for each $r\in [a,b]$, $(\mcl S_{a,\tau}^0   , \eta'_{\mcl S_{a,\tau}})$ a.s.\ determines $(Z-Z_a)|_{[a,r]}$ on the event $\{r \leq \tau\}$. Exhausting over all rational $r\in [a,b]$ shows that $(\mcl S_{a,\tau}^0   , \eta'_{\mcl S_{a,\tau}})$ a.s.\ determines $(Z-Z_a)|_{[a,\tau]}$.
It follows from the results of~\cite{wedges} (see, e.g.,~\cite[Lemma~3.12]{ghs-bipolar} for a careful explanation) that for each $r\in [a,b]$, $(Z-Z_a)|_{[a,\tau]}$ a.s.\ determines $(\mcl S_{a,r} , \eta'_{\mcl S_{a,r}} )$ on the event $\{r \leq \tau\}$. Again exhausting over all rational $r\in [a,b]$, we obtain assertion~\ref{item-msrble-wpsf}. 
  
Now we deduce assertion~\ref{item-msrble-chordal} from assertion~\ref{item-msrble-wpsf}. 
Although $\tau$ is not a stopping time for $Z$, it is a stopping time for the conditional law of the curve $\eta_{a,b,\mcl S_{a,b}}$ given $\mcl S_{a,b}$. Hence we can use Lemma~\ref{lem-wpsf-law} and the Markov property of chordal space-filling $\SLE_{\kappa'}(\rho^L;\rho^R)$ to obtain that if we condition on $(\mcl S_{a,\tau} , \eta_{a,b,\mcl S_{a,\tau}})$ then the regular conditional law of $\eta'_{\mcl S_{a,\tau}^0 }  $ is that of an independent space-filling $\SLE_{\kappa'}(\kappa'-6)$ loop in each of the bubbles of $\mcl S_{a,\tau}^0$, based at the marked point of the bubble.  
This conditional law depends only on $(\mcl S_{a,\tau}^0 , \eta_{a,b,\mcl S_{a,\tau}^0})$, so we get the same conditional law if we condition only on $(\mcl S_{a,\tau}^0 , \eta_{a,b,\mcl S_{a,\tau}^0})$. Hence $(\mcl S_{a,\tau}^0 , \eta_{ \mcl S_{a,\tau^0}}' )$ is conditionally independent from $(\mcl S_{a,\tau} , \eta_{a,b,\mcl S_{a,\tau}})$ given $(\mcl S_{a,\tau}^0 , \eta_{a,b,\mcl S_{a,\tau}^0})$. By assertion~\ref{item-msrble-wpsf}, $(\mcl S_{a,\tau}^0 , \eta_{ \mcl S_{a,\tau}^0}' )$ a.s.\ determines $(\mcl S_{a,\tau} , \eta_{a,b,\mcl S_{a,\tau}})$, so in fact $(\mcl S_{a,\tau}^0 , \eta_{a,b,\mcl S_{a,\tau}^0})$ a.s.\ determines $(\mcl S_{a,\tau} , \eta_{a,b,\mcl S_{a,\tau}})$, which is assertion~\ref{item-msrble-chordal}.  
\end{proof}

From Lemma~\ref{lem-partial-surface-msrble}, we easily deduce a measurability statement for the whole $\gamma$-quantum cone $\mcl C$ from~\eqref{eqn-cone-def}.

\begin{lem} \label{lem-full-surface-msrble}
Let $a,t,b\in \BB R$ with $a \leq t \leq b$ and let $\tau=\tau_{a,b}(t)$ be as in~\eqref{eqn-tau-def}. 
The $\gamma$-quantum cone~$\mcl C$, the future $\frac{3\gamma}{2}$-quantum wedge $\mcl S_{a,\infty}$, and the curve $\eta_{a,b,\mcl C}|_{[a,t]}$ are a.s.\ determined by the $4$-tuple $\left(\mcl S_{-\infty,a} ,  \mcl S_{a,\tau}^0 ,  \eta_{a,b , \mcl S_{a,\tau}^0} ,    \mcl S_{\tau,\infty}    \right)$. 
\end{lem}
\begin{proof}
By Lemma~\ref{lem-partial-surface-msrble}, the curve-decorated quantum surface $(\mcl S_{a,\tau}^0 ,  \eta_{a,b , \mcl S_{a,\tau}^0})$ a.s.\ determines $(\mcl S_{a,\tau} ,  \eta_{a,b , \mcl S_{a,\tau}})$. By the peanosphere construction, the surface $\mcl S_{a,\infty}$ is obtained by conformally welding together $\mcl S_{a,\tau}$ and $\mcl S_{\tau,\infty}$ along their boundaries and the surface $\mcl C$ is obtained by conformally welding together $\mcl S_{-\infty,a}$ and $\mcl S_{a,\infty}$ together along their boundaries. The boundary of each of the space-filling $\SLE_{\kappa'}$ segments $\eta'((-\infty,a])$, $\eta'([a,\tau])$, and $\eta'([\tau,\infty))$ consists of a finite union of segments of non-crossing $\SLE_{\kappa}$-type curves for $\kappa = 16/\kappa' \in (0,4)$, so is a.s.\ conformally removable by~\cite[Proposition~3.16]{wedges}.  Hence there is a.s.\ only one way to perform the above conformal welding operations, and we obtain the statement of the lemma.
\end{proof}

\subsection{Bubbles cut out by $\eta_{a,b}$ are quantum disks}
\label{sec-sle-disk}

Suppose we are in the setting of Section~\ref{sec-surface-def}, and recall in particular the chordal $\SLE_{\kappa'}$-type curves $\eta_{a,b} : [a,b] \rta \eta'([a,b])$, the associated maximal $\tfrac{\pi}{2}$-cone intervals $[\sigma_{a,b}(t) , \tau_{a,b}(t)]$ from~\eqref{eqn-tau-def}, and the surface $\mcl S_{a,b}^0$ parameterized by the bubbles cut out by $\eta_{a,b}$. The goal of this subsection is to prove that the bubbles of $\mcl S_{a,b}^0$ are conditionally independent quantum disks if we condition on the values of $Z$ outside of the time intervals corresponding to these bubbles. In particular, we will prove the following proposition.  

\begin{prop} \label{prop-sle-disk}
Let $a,b\in \BB R \cup \{-\infty , \infty\}$ with $a < b$.  
Almost surely, the conditional law of the bubbles of the quantum surface $\mcl S_{a,b}^0$ given $Z|_{(-\infty,a]}$, $Z|_{[b,\infty)}$, and $\{ Z_{\tau_{a,b}(t)} : t\in [a,b]\}$ is that of a collection of independent singly marked quantum disks with given areas and boundary lengths. 
\end{prop}

Recall that the bubbles of $\mcl S_{a,b}^0$ are the same as the quantum surfaces $\mcl S_{\sigma_{a,b}(t) ,\tau_{a,b}(t)}$ for $t\in (a,b)$ with $\tau_{a,b}(t) \not= t$. 
The main difficulty in the proof of Proposition~\ref{prop-sle-disk} lies in showing that each of these bubbles is a quantum disk conditional on its quantum area and boundary length. 

Throughout most of this subsection, we fix $a,b\in \BB R \cup \{-\infty , \infty\}$ with $a < b$ and $t \in (a,b)$ and to lighten notation we write
\eqb \label{eqn-sle-disk-abbrv}
 \sigma = \sigma_{a,b}(t) \quad \op{and} \quad  \tau  =\tau_{a,b}(t)    .
\eqe 
We will only allow $t$ to vary at the very end, when we argue that the bubbles are conditionally independent. 
The idea of the proof is to relate the bubble $\mcl S_{\sigma,\tau}$ to the bubbles cut out by the $\SLE_{\kappa'}(\kappa'-6)$ curves $\eta_{-\infty,r}$ for $r\in \BB R$, which we know are quantum disks by~\cite[Theorem~1.17]{wedges} (see Lemma~\ref{lem-bubble-cond-law}). 
 
Throughout the proof, we use the following notation. For $r \geq t$, let $\sigma^r = \sigma_{-\infty,r}(t)$ and $\tau^r = \tau_{-\infty ,r}(t)$ be as in~\eqref{eqn-tau-def} with $a=-\infty$ and $b=r$. 
Equivalently, by Lemma~\ref{lem-chordal-sle}, $[\sigma^r,\tau^r]$ is the maximal $\tfrac{\pi}{2}$-cone interval for $Z$ in $(-\infty,r]$ which contains $t$. 
Define the $\sigma$-algebra
\eqb \label{eqn-disk-filtration}
\mcl F^r := \sigma\left( \sigma^r , \tau^r ,  Z|_{\BB R\setminus [\sigma^r , \tau^r]} \right).
\eqe 
The main reason for our interest in the bubbles $\mcl S_{\sigma^r,\tau^r}$ is that these bubbles are in fact quantum disks conditional on $\mcl F^r$.  

\begin{lem} \label{lem-bubble-cond-law}
If $r\geq t$ and we condition on the $\sigma$-algebra $\mcl F^r$ of~\eqref{eqn-disk-filtration}, then the conditional law of the quantum surface $\mcl S_{\sigma^r,\tau^r}$ parameterized by $\eta'([\sigma^r,\tau^r])$ is that of a singly marked quantum disk with area $ \tau^r - \sigma^r$ and boundary length $|Z_{\tau^r } - Z_{\sigma^r}|$.
\end{lem}
\begin{proof}
We will deduce the lemma from the results of~\cite{wedges}. 
Recall that the curve $\eta_{-\infty,r}$ is a whole-plane $\SLE_{\kappa'}(\kappa'-6)$ from $\infty$ to $\eta'(r)$. 
Furthermore, by the translation invariance of the law of the pair $(h,\eta')$~\cite[Theorem~1.9]{wedges}, $(\BB C , h , \eta'(r) , \infty)$ is a $\gamma$-quantum cone independent from $\eta_{-\infty,r}$ (viewed as a curve modulo monotone re-parameterization). 

Let $\mcl H^r_1$ be the $\sigma$-algebra generated by the curve-decorated quantum surface $(\mcl S_{r,\infty} , \eta'_{\mcl S_{r,\infty}})$ and the quantum areas and the quantum boundary lengths of all of the bubbles of the quantum surfaces $\mcl S_{-\infty,r}^0$ in the order in which they are cut out by $\eta_{-\infty,r}$.  
Note that $\mcl S_{-\infty,r}^0$, and $\eta_{-\infty,r,\mcl S_{-\infty,r}^0}$ are determined by $(\mcl S_{-\infty,r} , \eta'_{\mcl S_{-\infty,r}})$ so are independent from $(\mcl S_{r,\infty} , \eta'_{\mcl S_{r,\infty}})$. 
 
By~\cite[Theorem~1.17]{wedges} and the Markov property of whole-plane space-filling $\SLE_{\kappa'}$, if we condition on $\mcl H_1^r$ then the conditional law of the curve-decorated quantum surface $( \mcl S_{-\infty,r}^0 , \eta'_{\mcl S_{-\infty,r}^0})$ is that of a collection of independent singly marked quantum disks with given areas and boundary lengths, each decorated by a space-filling $\SLE_{\kappa'}(\kappa'-6)$ loop based at the marked point. 

We can determine from $\mcl H^r_1$ which of the bubbles of $\mcl S_{-\infty,r}^0$ is $\mcl S_{\sigma^r, \tau^r}$ as follows. If we read off the quantum areas of the bubbles of $\mcl S_{-\infty,r}^0$ in reverse chronological order, then $\mcl S_{\sigma^r,\tau^r}$ is the first bubble which we encounter with the property that the sum of the quantum areas of the previous bubbles is at least $t$. 

Consequently, the conditional law of $\mcl S_{\sigma^r,\tau^r}$ given $\mcl H_1^r$ is that of a singly marked quantum disk with given area and boundary length. Furthermore, by the above independence statement, we get the same conditional law for $\mcl S_{\sigma^r,\tau^r}$ if we further condition on the $\sigma$-algebra $\mcl H_2^r$ generated by $\mcl H_1^r$ and the curve-decorated quantum surfaces $(\mcl D , \eta'_{\mcl D})$ for each bubble $\mcl D$ of $\mcl S_{-\infty,r}^0$ other than $\mcl S_{\sigma^r,\tau^r}$. 

By the relationship between $\tfrac{\pi}{2}$-cone excursions for $Z$ and bubbles filled in by $\eta'$, we have $\mu_h(\mcl S_{\sigma^r,\tau^r}) =  \tau^r - \sigma^r$ and $\nu_h(\bdy \mcl S_{\sigma^r,\tau^r}) = |Z_{\tau^r } - Z_{\sigma^r}|$ (note that one of the two coordinates of $Z_{\tau^r } - Z_{\sigma^r}$ is zero and the other is negative).
Hence it remains to show that a.s.\ $\mcl F^r\subset\mcl H_2^r$. The $\sigma$-algebra $\mcl H_2^r$ a.s.\ determines $(Z-Z_r)|_{[r,\infty)}$, $ \tau^r - \sigma^r$, $Z_{\tau^r } - Z_{\sigma^r}$, and $(Z-Z_{v_Z(s)})|_{[v_Z(s) ,s]}$ for each maximal $\tfrac{\pi}{2}$-cone time $s$ for $Z$ in $(-\infty , r]$ other than $[\sigma^r,\tau^r]$. 

As explained in~\cite[Proposition~10.3]{wedges}, the complement of the union of all of the maximal $\tfrac{\pi}{2}$-cone excursions for $Z$ in $(-\infty,r]$ (the so-called ancestor free times for $Z$, run backward from time $r$) has the law of the range of $\kappa'/8$-stable subordinator, and if we pre-compose $Z$ with this stable subordinator we obtain a pair of independent $\kappa'/4$-stable processes. These $\kappa'/4$-stable processes are a.s.\ determined by their jumps, which in turn are a.s.\ determined by the quantities $Z_s - Z_{v_Z(s)}$ as $s$ ranges over all maximal $\tfrac{\pi}{2}$-cone times for $Z$ in $(-\infty,r]$. From this, we deduce that the maximal $\tfrac{\pi}{2}$-cone excursions $(Z-Z_{v_Z(s)})|_{[v_Z(s) ,s]}$ for $Z$ in $(-\infty , r]$ other than $[\sigma^r,\tau^r]$ a.s.\ determine the values of $Z - Z_r$ on $(-\infty, r] \setminus [\sigma^r,\tau^r]$. 
Hence $\mcl F^r \subset \mcl H_2^r$. 
\end{proof}
 
Our next lemma will allow us to compare the laws of the bubble $\mcl S_{\sigma,\tau}$ of~\eqref{eqn-sle-disk-abbrv} and the quantum disk $\mcl S_{\sigma^r,\tau^r}$. 

\begin{lem} \label{lem-bubble-time}
Suppose we are in the setting described just above.
For $r' \geq r \geq t$, we have (in the notation~\eqref{eqn-disk-filtration}) $\mcl F^{r'} \subset \mcl F^r$, so $\{\mcl F^r\}_{r\geq t}$ is a decreasing filtration. 
Furthermore, if we let $\sigma$ and $\tau$ be as in~\eqref{eqn-sle-disk-abbrv}, then 
\eqb \label{eqn-bubble-time}
\rho = \sup\left\{ r \geq t : [\sigma^r,\tau^r] =  [\sigma,\tau]  \right\}
\eqe
is a reverse stopping time for $\{\mcl F^r\}_{r\geq t}$ and a.s.\ $\rho > t$. Furthermore, there a.s.\ exists $\delta > 0$ such that $[\sigma^r ,\tau^r] = [\sigma , \tau]$ for each $r \in [\rho - \delta ,\rho]$. 
\end{lem}
\begin{proof}
The $\sigma$-algebra $\mcl F^r$ determines $\sigma^r$, $\tau^r$, and $Z_s$ for each $s\in \BB R\setminus [\sigma^r,\tau^r]$. This information determines each $\tfrac{\pi}{2}$-cone interval for $Z$ which is not contained in $[\sigma^r,\tau^r]$, so in particular determines the endpoints of the maximal $\tfrac{\pi}{2}$-cone interval for $Z$ in $I$ which contains $[\sigma^r,\tau^r]$ for each open (possibly unbounded) interval in $\BB R$ which contains $[\sigma^r,\tau^r]$. 

If $r' \geq r$, then $[\sigma^r ,\tau^r]$ is a $\tfrac{\pi}{2}$-cone interval in $(-\infty , r']$ so is contained in $[\sigma^{r'} , \tau^{r'}]$. 
Hence the preceding paragraph implies that $\sigma^{r'} , \tau^{r'} \in \mcl F^r$. It is clear that $\mcl F^{r'}$ is determined by $[\sigma^{r'} , \tau^{r'}]$ and $Z_s$ for $s\in \BB R\setminus [\sigma^r,\tau^r]$, whence $\mcl F^{r'} \subset \mcl F^r$. 

We next argue that $\rho > t$ a.s., i.e.\ there a.s.\ exists some $r >t$ for which $[\sigma^r ,\tau^r] = [\sigma,\tau]$. 
By symmetry, we can assume without loss of generality that $\tau$ is a left $\tfrac{\pi}{2}$-cone time for $Z$, i.e.\ $R_\sigma = R_\tau$.
Almost surely, $\sigma$ is not a local minimum for $R$ and a.s.\ $\sigma > a$, so there a.s.\ exists $\ep > 0$ and $s_* \in (a,\sigma)$ such that $R_{s_*} \leq R_\tau - \ep$. 
By continuity, there a.s.\ exists $\delta > 0$ such that $R_s \geq R_\tau - \ep$ for each $s\in [\tau , \tau+\delta]$.
If $s$ is a $\tfrac{\pi}{2}$-cone time for $Z$ belonging to $(\tau,\tau+\delta]$, then the cone entrance time satisfies $v_Z(s) \geq s_* \geq a$ so by maximality of $[\sigma,\tau]$, we must have $v_Z(s) \geq \tau$. 

If $r \in (\tau , \tau+\delta]$, then $[\sigma,\tau]$ is a $\tfrac{\pi}{2}$-cone interval for $Z$ in $(-\infty, r]$ so by maximality of $[\sigma^r,\tau^r]$, we have $\tau^r \in [\tau,\tau+\delta]$ and $\sigma^r \leq \sigma$. By the preceding paragraph, $[\sigma^r,\tau^r] =  [\sigma,\tau] $ for each such $r$. Hence $\rho \geq \tau+\delta$ and $[\sigma^r ,\tau^r] = [\sigma , \tau]$ for each $r \in [\rho - \delta ,\rho]$. 
  
Since $\rho >t$ a.s.\ and the intervals $[\sigma^r,\tau^r]$ are decreasing in $r$, if $r\geq t$ then a.s.\ $\rho \geq r$ if and only if $[\sigma^r,\tau^r]\subset [a,b]$. This latter event is $\mcl F^r$-measurable, so $\{\rho\geq r\} \in \mcl F^r$. 
\end{proof}

\begin{proof}[Proof of Proposition~\ref{prop-sle-disk}]
We first consider the single bubble $\mcl S_{\sigma,\tau}$. 
Let $\rho$ be as in~\eqref{eqn-bubble-time} and for $n\in\BB N$, let $\rho_n = 2^{-n} \lfloor 2^n\rho \rfloor$ be the largest integer multiple of $2^{-n}$ which is smaller than $\rho$. By Lemma~\ref{lem-bubble-time}, each $\rho_n$ is a reverse stopping time for the reverse filtration $\{\mcl F^r\}_{r\geq t}$ and a.s.\ $\rho_n$ increases to $\rho$.  By Lemma~\ref{lem-bubble-cond-law}, for each $n\in\BB N$ the conditional law of $\mcl S_{\sigma^{\rho_n} ,\tau^{\rho_n}}$ given $\mcl F^{\rho_n}$ is that of a singly marked quantum disk with area $\tau^{\rho_n} - \sigma^{\rho_n}$ and boundary length $|Z_{\tau^{\rho_n}} - Z_{\sigma^{\rho_n}}|$.

By the last statement of Lemma~\ref{lem-bubble-time}, a.s.\ $[\sigma^{\rho_n} , \tau^{\rho_n}] = [\sigma,\tau]$ and hence $\mcl S_{\sigma^{\rho_n} ,\tau^{\rho_n}} = \mcl S_{\sigma , \tau}$ for all large enough $n\in\BB N$. Since $\mcl F^\rho = \bigcap_{n\in\BB N} \mcl F^{\rho_n}$, the backward martingale convergence theorem implies that the conditional law of $\mcl S_{\sigma,\tau}$ given $\mcl F^{\rho}$ is that of a singly marked quantum disk with area $\tau-\sigma$ and boundary length $|Z_\tau-Z_\sigma|$.  We get the same conditional law if we condition only on $\sigma$, $\tau$, and $ Z_\tau - Z_\sigma$.  By~\cite[Theorem~2.1]{sphere-constructions}, the curve-decorated quantum surface $(\mcl S_{\sigma,\tau} , \eta'_{\mcl S_{\sigma,\tau}})$ is a.s.\ determined by its peanosphere Brownian motion $(Z-Z_\sigma)|_{[\sigma,\tau]}$, which will be important just below. 

We now allow $t$ to vary and prove that the bubbles of $\mcl S_{a,b}^0$ are conditionally independent given $Z|_{(-\infty,a]}$, $Z|_{[b,\infty)}$, and $\{ Z_{\tau_{a,b}(t)} : t\in [a,b]\}$.  Indeed, under this conditioning the Brownian motion increments $(Z-Z_{v_Z(s)})|_{[v_Z(s) , s]}$ as $s$ ranges over all maximal $\tfrac{\pi}{2}$-cone times for $Z$ in $[a,b]$ (Definition~\ref{def-maximal}) are conditionally independent. By the above measurability statement, each of these Brownian motion increments a.s.\ determines the corresponding bubble of $\mcl S_{a,b}^0$. The proposition statement follows.
\end{proof}

\subsection{Bubbles cut out by $\wt\eta_{a,b}$ are quantum disks}
\label{sec-general-disk}

Recall the future curves $\wt\eta_{a,b} : [a,b] \rta \wt\eta'([a,b])$ associated with $(\wt h , \wt\eta')$ and the surface $\wt{\mcl S}_{a,b}^0$ parameterized by the bubbles it cuts out.  
We recall also the times $\sigma_{a,b}(t)$ and $\tau_{a,b}(t)$ from~\eqref{eqn-tau-def}, and note that the definitions of these times do not change if we replace $\eta_{a,b}$ with $\wt\eta_{a,b}$. 

In this subsection we will prove the exact analog of Proposition~\ref{prop-sle-disk} for the pair $(\wt h , \wt\eta')$.  

\begin{prop} \label{prop-general-disk}
Let $a,b\in \BB R \cup \{-\infty , \infty\}$ with $a < b$.  
If we condition on $Z|_{(-\infty , a]}$, $Z|_{[b,\infty)}$, and $\{ Z_{\tau_{a,b}(t)} : t\in [a,b] \} $ then the conditional law of the bubbles of the quantum surface $\wt{\mcl S}_{a,b}^0$ is that of a collection of independent quantum disks with given areas and boundary lengths.
\end{prop}

As in Section~\ref{sec-sle-disk}, throughout most of this subsection we will fix $a,b\in \BB R \cup \{-\infty , \infty\}$ with $a < b$ and $t \in (a,b)$ and write $\tau = \tau_{a,b}(t)$ and $\sigma = \sigma_{a,b}(t)$, as in~\eqref{eqn-sle-disk-abbrv}.

The main difficulty in the proof of Proposition~\ref{prop-general-disk} is showing that the single bubble $\wt{\mcl S}_{\sigma,\tau}$ has the law of a quantum disk under certain conditioning. This will be accomplished in Lemma~\ref{lem-general-disk0}. 
Then, in the proof of Proposition~\ref{prop-general-disk} itself, we will allow $t$ to vary. 
 
The idea of the proof of Proposition~\ref{prop-general-disk} is to use the fact that the surfaces~$\wt{\mcl S}_{r,\infty}$ and~$\mcl S_{r,\infty}$ agree in law for each $r\in \BB R$ (condition~\ref{item-wpsf-char-wedge} in Theorem~\ref{thm-wpsf-char}) and a limiting argument as $r$ decreases to~$\sigma$ to relate the laws of~$\wt{\mcl S}_{\sigma,\tau}$ and~$\mcl S_{\sigma,\tau}$; we know that the law of the latter is that of a quantum disk by Proposition~\ref{prop-sle-disk}. This is carried out in Lemma~\ref{lem-general-disk0}. However, we will need that $\wt{\mcl S}_{\sigma,\tau}$ is still a quantum disk if we condition on certain information, so before stating and proving Lemma~\ref{lem-general-disk0} we will need to establish some probabilistic properties of the times $\sigma$ and $\tau$. In particular, we will define a filtration for which these times are stopping times (Lemmas~\ref{lem-partial-filtration} and~\ref{lem-partial-surface-determined}) and establish a certain strong Markov property for stopping times with respect to this filtration (Lemma~\ref{lem-bm-future-ind}). 
 
The times $ \sigma$ and $\tau$ are \emph{not} stopping times for $Z$ or even for the filtration $\{\wt{\mcl F}_{-\infty,r}\}_{r\in \BB R}$ of Theorem~\ref{thm-wpsf-char} since, using the characterization from Lemma~\ref{lem-chordal-sle}, we cannot see if a given $\tfrac{\pi}{2}$-cone time for $Z$ is maximal in $[a,b]$ without knowing some information about what happens after this time. However, as we will see just below, these times are stopping times for a larger filtration which we introduce in the following lemma.

\begin{lem} \label{lem-partial-filtration}
For $r\in\BB R$, let $\wt{\mcl F}_{-\infty,r}$ be the $\sigma$-algebra of Theorem~\ref{thm-wpsf-char} and let $P_{r,\infty}$ be the function in~\eqref{eqn-bead-function} which encodes the ordered sequence of areas and boundary lengths of the beads of $\mcl S_{r,\infty}$ (equivalently $\wt{\mcl S}_{r,\infty}$). 
Define $\sigma$-algebras
\eqb \label{eqn-partial-filtration-def}
\mcl G_r := \sigma\left( (Z-Z_r)|_{(-\infty,r]} \right) \vee \sigma( P_{r,\infty} ) \quad \op{and} \quad
\wt{\mcl G}_r := \wt{\mcl F}_{-\infty,r} \vee \sigma( P_{r,\infty} ) .
\eqe 
For each $r\in\BB R$, we have $\mcl G_r \subset \wt{\mcl G}_r$. Furthermore, the $\sigma$-algebras $\{ \mcl G_r \}_{r\in\BB R}$ and $\{ \wt{\mcl G}_r \}_{r\in\BB R}$ are increasing in $r$.
\end{lem}
\begin{proof}
By definition, for each $r\in\BB R$ it holds that $\sigma\left( (Z-Z_r)|_{(-\infty,r]} \right) \subset \wt{\mcl F}_{-\infty,r}$.  Therefore, $\mcl G_r \subset \wt{\mcl G}_r$. 

It is clear that $ \sigma\left( (Z-Z_r)|_{(-\infty,r]} \right) \subset \sigma\left( (Z-Z_{r'})|_{(-\infty,r']} \right)$ and $\wt{\mcl F}_{-\infty, r} \subset \wt{\mcl F}_{-\infty, r'}$ for $r \leq r'$. Hence to show that our $\sigma$-algebras are increasing in $r$, it suffices to show that $P_{r,\infty} $ is a.s.\ determined by $  (Z-Z_{r'})|_{(-\infty,r']} $ and $ P_{r' ,\infty}  $ for $r \leq r'$. 

By definition, $P_{r,\infty}$ determines and is determined by $\{(s-r , L_s - L_r , R_s - R_r) \}_{s \in \Sigma_r}$, where $\Sigma_r$ is the set of times $s\geq r$ at which $L$ and $R$ attain a simultaneous running infimum.  Each time in $s \in \Sigma_r \cap [r',\infty)$ also belongs to $\Sigma_{r'}$, and the corresponding triple $(s-r , L_s - L_r , R_s - R_r)$ is determined by $(s-r' , L_s - L_{r'} , R_s - R_{r'})$ and $(Z-Z_{r'})|_{(-\infty , r']}$. Furthermore, if $s'\in \Sigma_{r'}$ then we can determine whether $s' \in \Sigma_r$ from the triple $(s-r' , L_s - L_{r'} , R_s - R_{r'})$ together with $(Z-Z_{r'})|_{(-\infty , r']}$. Hence $P_{r,\infty} $ is a.s.\ determined by $  (Z-Z_{r'})|_{(-\infty,r']} $ and $ P_{r' ,\infty}  $ for $r \leq r'$ and we obtain the statement of the lemma. 
\end{proof}

The next lemma implies in particular that $\sigma$ and $\tau$ are stopping times for both of the filtrations introduced in Lemma~\ref{lem-partial-filtration}.

\begin{lem} \label{lem-partial-surface-determined}
Define the $\sigma$-algebras $\{\wt{\mcl G}_r\}_{r\in\BB R}$ as in~\eqref{eqn-partial-filtration-def}.  The times $\sigma$ and $\tau $ in Proposition~\ref{prop-partial-surface-law} are stopping times for $\{ \mcl G_r\}_{r\in\BB R} $ (and hence also for $\{\wt{\mcl G}_r\}_{r\in\BB R}$).  Furthermore, the curve-decorated quantum surface $( \mcl S_{a,\tau }^0 , \eta_{a,b,\mcl S_{a,\tau}^0})$ is $ \mcl G_{\tau }$-measurable and the curve-decorated quantum surface $(\wt{\mcl S}_{a,\tau }^0 , \wt\eta_{a,b,\wt{\mcl S}_{a,\tau}^0})$ is $\wt{\mcl G}_{\tau }$-measurable.   
\end{lem} 
\begin{proof}
We first check that $\sigma$ and $\tau$ are stopping times for $\{\mcl G_r\}_{r\in\BB R}$.  Recall from Lemma~\ref{lem-chordal-sle} that $[\sigma , \tau]$ is the maximal $\tfrac{\pi}{2}$-cone interval for $Z$ in $[a,b]$ which contains $t$. 

If $r \notin [a,t]$, then $\{\sigma \leq r\}$ is either the empty event or the whole probability space and if $r\in [a,t]$, then $\{\sigma \leq r\}$ is a.s.\ the same as the event that there is a $\tfrac{\pi}{2}$-cone time $s$ for $Z$ in $[t,b]$ such that $v_Z(s) \in [a,r]$.  A $\tfrac{\pi}{2}$-cone time $s$ for $Z$ in $[t,b]$ with $v_Z(s) \leq r$ is the same as a time $s\in [t,b]$ at which the two coordinates of $(Z-Z_r)|_{[r,\infty)}$ attain a simultaneous running infimum, and for such a $\tfrac{\pi}{2}$-cone time $s$ we have $v_Z(s) \in [a,r]$ if and only if there is a $v \in [a,r]$ such that either
\eqbn
L_{v} - L_r \leq L_s - L_r \quad \op{or} \quad R_{v} - R_r \leq R_s - R_r  .
\eqen
It is clear that $(Z-Z_r)|_{(-\infty,r]}$ and $P_{r,\infty}$ together a.s.\ determine whether this is the case. Thus $\{\sigma \leq r\} \in \mcl G_r$, so $\sigma$ is a $\{\mcl G_r\}_{r\in\BB R}$-stopping time.

Now we consider the time $\tau$. The event $\{\tau \leq r\}$ is proper and non-empty only if $r \in [t,b]$.  For such an $r$, we a.s.\ have $\tau \leq r$ if and only if the maximal $\tfrac{\pi}{2}$-cone interval $[v_Z(s) , s]$ for $Z$ in $[a,r]$ which contains $t$ is also maximal in $[a,b]$.  The $\tfrac{\pi}{2}$-cone interval $[v_Z(s) , s]$ is maximal in $[a,b]$ if and only if there is no $\tfrac{\pi}{2}$-cone time $s'$ for $Z$ with $s' > r$ and $v_Z(s') \in [a , v_Z(s)]$. Such a $\tfrac{\pi}{2}$-cone time $s'$ is the same as a time $s' \in [r,b]$ such that the two coordinates of $(Z-Z_r)|_{[r,\infty)}$ attain a simultaneous running infimum at time $s'$; there is no $v \in [s , r]$ such that either $L_{v} - L_r \leq L_{s'} - L_r $ or $R_{v} - R_r \leq R_{s'} - R_r $; and there is a $v \in [a,v_Z(s)]$ such that either $L_{v} - L_r \leq L_{s'} - L_r $ or $R_{v} - R_r \leq R_{s'} - R_r $.  We can tell whether such a time $s'$ exists from $(Z-Z_r)|_{(-\infty,r]}$ and $P_{r,\infty}$. Hence $\{\tau\leq r\} \in \mcl G_r$ and $\tau$ is a $\{\mcl G_r\}_{r\in\BB R}$-stopping time.

By~\cite[Theorem~2.1]{sphere-constructions} and Proposition~\ref{prop-sle-disk}, each of the bubbles $\mcl S_{v_Z(s) , s}$ of $ \mcl S_{a,\tau }^0$ is a.s.\ determined by the corresponding $\tfrac{\pi}{2}$-cone excursion $(Z-Z_{v_Z(s)})|_{[v_Z(s) ,s]}$ for $Z$, so is a.s.\ determined by $\mcl G_r$ on the event $\{\tau\leq r\}$. By the peanosphere construction the equivalence class of the curve-decorated quantum surface $( \mcl S_{a,\tau }^0 , \eta_{a,b,  \mcl S_{a,\tau}^0})$ modulo curve-preserving homeomorphisms is a.s.\ determined by $(Z-Z_\tau)|_{(-\infty,\tau]}$. Hence $(\mcl S_{a,\tau }^0 , \eta_{a,b,\mcl S_{a,\tau}^0}) \in \mcl G_{\tau }$

The bubbles of the quantum surface $\wt{\mcl S}_{a, \tau }^0$ is parameterized by the ordered sequence of bubbles cut out by the curve $\wt\eta_{a,b,\wt{\mcl S}_{a,r}^0}$ before time $\tau $, so are a.s.\ determined by $\wt{\mcl F}_{-\infty,r}$ and $\tau $ on the event $\{\tau  \leq r\}$.  By condition~\ref{item-wpsf-char-homeo} of Theorem~\ref{thm-wpsf-char}, $(\wt{\mcl S}_{a,\tau }^0 , \wt\eta_{a,b,\wt{\mcl S}_{a,\tau}^0})$ a.s.\ differs from $(\mcl S_{a,\tau }^0 , \eta_{a,b,\mcl S_{a,\tau}^0})$ by a curve-preserving homeomorphism, so by the above the equivalence class of $( \mcl S_{a,\tau }^0 , \eta_{a,b,  \mcl S_{a,\tau}^0})$ modulo curve-preserving homeomorphisms is a.s.\ determined by $(Z-Z_\tau)|_{(-\infty,\tau]} \in \wt{\mcl G}_\tau$.  Therefore $(\wt{\mcl S}_{a,\tau }^0 , \wt\eta_{a,b,\wt{\mcl S}_{a,\tau}^0}) \in \wt{\mcl G}_\tau$. 
\end{proof}

It is easy to see from the conditions of Theorem~\ref{thm-wpsf-char} that for $r\in\BB R$, each of $\wt{\mcl F}_{r,\infty}$ and $\wt{\mcl S}_{r,\infty}$ is conditionally independent from the $\sigma$-algebra $\wt{\mcl G}_r$ of~\eqref{eqn-partial-filtration-def} given the function $P_{r,\infty}$ of~\eqref{eqn-bead-function}. The next lemma extends this property to stopping times for this filtration (which, by Lemma~\ref{lem-partial-surface-determined}, includes the times $\sigma$ and $\tau$). 
  
\begin{lem} \label{lem-bm-future-ind}
Let $T$ be an a.s.\ finite stopping time for the filtration $\{\wt{\mcl G}_r\}_{r\in\BB R}$ of Lemma~\ref{lem-partial-filtration}.  Each of $\wt{\mcl F}_{T,\infty}$ and $\wt{\mcl S}_{T,\infty}$ is conditionally independent from $\wt{\mcl G}_T$ given $P_{T,\infty}$. 
\end{lem}  

We note that Lemma~\ref{lem-bm-future-ind} does \emph{not} imply that the pair $( \wt{\mcl F}_{T,\infty} , \wt{\mcl S}_{T,\infty} ) $ is jointly conditionally independent from $\wt{\mcl G}_T$ given $P_{T,\infty}$.

\begin{proof}[Proof of Lemma~\ref{lem-bm-future-ind}]
First we consider the case when $T=r$ is deterministic.  By condition~\ref{item-wpsf-char-wedge} in Theorem~\ref{thm-wpsf-char}, each of $\wt{\mcl F}_{r,\infty}$ and $\wt{\mcl S}_{r,\infty}$ is independent from $\wt{\mcl F}_{-\infty,r}$. It is clear from the definition~\eqref{eqn-bead-function} that $P_{r,\infty}$ is a.s.\ determined by each of $\wt{\mcl F}_{r,\infty}$ and $\wt{\mcl S}_{r,\infty}$. Hence the statement of the lemma is true if $T$ is deterministic.
 
Now transfer to a general stopping time $T$ as in the statement of the lemma via a limiting argument. For $n\in\BB N$, let $T^n =  2^{-n} \lceil 2^n T \rceil $, so that each $T^n$ is a stopping time for $\{\wt{\mcl G}_r\}_{r\in\BB R}$ and $T^n$ a.s.\ decreases to $T$.  For $n\in\BB N$, let
\eqbn
\mcl H_{T^n}:= \sigma\left( (Z-Z_T)|_{[T,T^n]}  , P_{T^n , \infty} \right)  .
\eqen
Then $\sigma(P_{T,\infty}) \subset \mcl H_{T^n} \subset \wt{\mcl G}_{T^n}$ and $\bigcap_{n=1}^\infty \mcl H_{T^n} = \sigma(P_{T,\infty})$.  By the above statement for deterministic times and since $T^n$ takes on only countably many possible values, it follows that the statement of the lemma is true for each $T^n$, i.e.\ for each $n\in\BB N$, each of $\wt{\mcl F}_{T^n ,\infty}$ and $\wt{\mcl S}_{T^n ,\infty}$ is conditionally independent from $\wt{\mcl G}_{T^n}$ given $P_{T^n,\infty}$.  Since $\wt{\mcl G}_{T } \subset \wt{\mcl G}_{T^n}$ and $P_{T^n,\infty} \subset \mcl H_{T^n}$, also each of $\wt{\mcl F}_{T^n ,\infty}$ and $\wt{\mcl S}_{T^n ,\infty}$ is conditionally independent from $\wt{\mcl G}_{T }$ given $\mcl H_{T^n}$. 
 
For $m\geq n$, the time $T^n$ and the $\sigma$-algebra $\mcl H_{T^n}$ are a.s.\ determined by $(Z-Z_{T^m})|_{[T^m , \infty)} \in \wt{\mcl F}_{T^m,\infty}$, hence $\mcl H_{T^n}$ is conditionally independent from $\wt{\mcl G}_T$ given $\mcl H_{T^m}$.  Therefore each of $\wt{\mcl F}_{T^n ,\infty}$ and $\wt{\mcl S}_{T^n ,\infty}$ is conditionally independent from $\wt{\mcl G}_{T }$ given $\mcl H_{T^m}$.
Taking a limit as $m\rta \infty$ ($n$ fixed) and applying the backward martingale convergence theorem to $\BB P\left[ A\cap B \,|\, \mcl H_{T^m} \right]  $ for events $A  \in \wt{\mcl F}_{T^n,\infty}$ or $A\in \sigma(\wt{\mcl S}_{T^n,\infty})$ and $B \in \wt{\mcl G}_T$ shows that each of $\wt{\mcl F}_{T^n ,\infty}$ and $\wt{\mcl S}_{T^n ,\infty}$ is conditionally independent from $\wt{\mcl G}_T$ given $P_{T,\infty}$. 

By the continuity of $Z$, the $\wt{\mcl F}_{T^n,\infty}$-measurable functions which equal $Z-Z_{T^n}$ on $[T^n,\infty)$ and $0$ on $(-\infty,T^n]$ converge uniformly to $(Z-Z_T)|_{[T,\infty)}$ on $[T,\infty)$. 
If $[v_Z(s) , s]$ is a $\pi/2$-cone interval contained in $(T,\infty)$, then a.s.\ $[v_Z(s) ,s]$ is contained in $[T^n,\infty)$ for large enough $n$. 
On the other hand, if $a < t < b$ are rational times such that $\sigma_{a,b}(t) = T$ then since the maximal $\pi/2$-cone interval $[\sigma_{a,b}(t) , \tau_{a,b}(t)]$ can be approximated arbitrarily closely by smaller $\pi/2$-cone intervals contained in $[\sigma_{a,b}(t) , \tau_{a,b}(t)]$
and since $\wt\eta'$ is continuous, the $\wt{\mcl F}_{T^n,\infty}$-measurable quantum surfaces $\wt{\mcl S}_{\sigma_{a\vee T^n ,b }(t)  , \tau_{a \vee T^n  ,b  }(t)  }$ converge a.s.\ in the topology of Section~\ref{sec-surface-topology} to $\wt{\mcl S}_{\sigma_{a,b}(t)  , \tau_{a,b}(t) }$. 
From these convergence statements and conditional independence of $\wt{\mcl F}_{T^n,\infty}$ and $\wt{\mcl G}_T$ given $P_{T,\infty}$, we infer that $\wt{\mcl F}_{T,\infty}$ is conditionally independent from $\wt{\mcl G}_T$ given $P_{T,\infty}$.

Furthermore, by continuity of $\wt\eta'$ a.s.\ $ \wt{\mcl S}_{T^n ,\infty} \rta \wt{\mcl S}_{T,\infty}$ with respect to the topology induced by the metric~\eqref{eqn-bead-dist} of Section~\ref{sec-surface-topology}.  Therefore $\wt{\mcl S}_{T  ,\infty}  $ is conditionally independent from $\wt{\mcl G}_T$ given $P_{T,\infty}$ and we conclude.
\end{proof}

The following lemma tells us the conditional law of the single bubble $\wt{\mcl S}_{\sigma,\tau}$ of $\wt{\mcl S}_{a,b}^0$ given the previous bubbles as well as the segments of the Brownian motion $Z$ corresponding to the complement of the bubble. This lemma is the key ingredient in the proof of Proposition~\ref{prop-general-disk}. 
  
\begin{lem} \label{lem-general-disk0} 
Almost surely, the conditional law of the quantum surface $\wt{\mcl S}_{\sigma,\tau}$ parameterized by the bubble $\wt\eta'([ \sigma, \tau])$ cut out by $\wt\eta_{a,b}$ given the $\sigma$-algebra $\wt{\mcl G}_\sigma$ of~\eqref{eqn-partial-filtration-def} and $(Z-Z_\tau)|_{[\tau,\infty)}$ is that of a singly marked quantum disk with area $\tau-\sigma$ and boundary length $|Z_\tau - Z_\sigma|$.  
\end{lem} 
\begin{proof}
The basic idea of the proof is that the quantum surfaces $\mcl S_{r,\infty}$ and $\wt{\mcl S}_{r,\infty}$ agree in law for each $r\in\BB R$; and as $r$ decreases to $\sigma$, the bead of $\mcl S_{r,\infty}$ (resp.\ $\wt{\mcl S}_{r,\infty}$) which contains $\eta_{\mcl S_{r,\infty}}'(t)$ (resp.\ $\wt\eta_{\wt{\mcl S}_{r,\infty}}'(t)$) converges in the topology of quantum surfaces (Section~\ref{sec-surface-topology}) to $\mcl S_{\sigma,\tau}$ (resp.\ $\wt{\mcl S}_{\sigma,\tau}$). Since we know that $\mcl S_{\sigma,\tau}$ is a quantum disk (Proposition~\ref{prop-sle-disk}), this will show that also $\wt{\mcl S}_{\sigma,\tau}$ is a quantum disk. One then has to check that one gets the same conditional law when conditioning on $\wt{\mcl G}_\sigma$ and $(Z-Z_\tau)|_{[\tau,\infty)}$, which follows from the previous lemmas and some abstract nonsense. 
\medskip

\noindent\textit{Step 1: setup.}
For $n\in\BB N$, let $\sigma^n =   2^{-n} \lceil 2^n \sigma \rceil  $. 
By Lemma~\ref{lem-partial-surface-determined}, $\sigma $ is a stopping time for the filtration $\{\wt{\mcl G}_r \}_{r \in\BB R}$, hence the same is the case for each $\sigma^n$. 
By condition~\ref{item-wpsf-char-wedge} in Theorem~\ref{thm-wpsf-char}, for each $n\in\BB N$ the conditional law of $\wt{\mcl S}_{\sigma^n,\infty}$ given $ \wt{\mcl G}_{\sigma^n} $ is that of a collection of independent beads of a $\frac{3\gamma}{2}$-quantum wedge with given areas and left/right quantum boundary lengths. 

If $t > \sigma^n  $, let $\wt{\mcl D}^n$ be the bead of $\wt{\mcl S}_{\sigma^n,\infty}$ which contains $\wt\eta'_{\wt{\mcl S}_{\sigma^n,\infty}}(t)$, equivalently the first bead of $\wt{\mcl S}_{\sigma^n,\infty}$ with the property that the total quantum mass of the previous beads is at least $t-\sigma^n$ (including $\wt{\mcl D}^n$ itself). If $t \leq \sigma^n$ let $\wt{\mcl D^n}$ be the trivial one-point quantum surface. 
Also let $\tau^n$ be the time at which $\wt\eta'_{\wt{\mcl S}_{\sigma^n,\infty}}$ finishes filling in this bead, equivalently $\tau^n -\sigma^n$ is the quantum mass of $\wt{\mcl D}^n$.    
 
The conditional law of $\wt{\mcl D}^n$ given  $ \wt{\mcl G}_{\sigma^n} $ is that of a single bead of a $\frac{3\gamma}{2}$-quantum wedge with given area and left/right boundary lengths. This conditional law depends only on the realization of $P_{\sigma^n ,\infty}$, so in particular $\wt{\mcl D}^n$ is conditionally independent from $\wt{\mcl G}_{\sigma^n}$ given $P_{\sigma^n,\infty}$. 
\medskip

\noindent\textit{Step 2: bubbles agree in law.}
We will now compare $\wt{\mcl D}^n$ to the analogous object defined in terms of the pair $(h , \eta')$. 
For $n\in\BB N$, let $\mcl D^n$ be the bead of $\mcl S_{\sigma^n,\infty}$ which contains $ \eta'_{ \mcl S_{\sigma^n,\infty}}(t)$ (or a point if $t \leq \sigma^n$). 
By the conclusion of the preceding paragraph applied in the case when $(\wt h , \wt\eta') = (h,\eta')$, we find that the conditional law of $\mcl D^n$ given $\sigma^n$ and $P_{\sigma^n,\infty}$ is that of a single bead of a $\frac{3\gamma}{2}$-quantum wedge with given area and left/right boundary lengths (here we note that the definitions of $\sigma^n$ and $P_{\sigma^n,\infty}$ depend only on $(h,\eta')$), i.e.\ the conditional laws of $\wt{\mcl D}^n$ and $\mcl D^n$ given $P_{\sigma^n,\infty}$ are the same. Therefore, 
\eqb \label{eqn-disk-same-law}
\left( \wt{\mcl D}^n ,   P_{\sigma^n,\infty} \right)  \eqD \left( \mcl D^n ,   P_{\sigma^n,\infty} \right)  , \quad \forall n\in\BB N .
\eqe  
 
By continuity of $\wt\eta'$, as $n\rta \infty$ the sub-domain of $\BB C$ with two marked boundary points which parameterizes $\wt{\mcl D}^n$ converges a.s.\ in the Caratheodory sense to $(\wt\eta'([\sigma , \tau]) ,  \wt\eta'(\tau))$. Hence $\wt{\mcl D}^n$ converges a.s.\ in the topology of Section~\ref{sec-surface-topology} to $\wt{\mcl S}_{\sigma,\tau}$. 
Furthermore, since $\sigma^n\rta\sigma$ a.s., by continuity of $Z$ a.s.\ $P_{\sigma^n,\infty} \rta P_{\sigma,\infty}$ uniformly. 
Therefore, a.s.\ 
\eqb \label{eqn-general-disk-as}
\left( \wt{\mcl D}^n ,  P_{\sigma^n,\infty} \right) \rta \left( \wt{\mcl S}_{\sigma,\tau} ,   P_{\sigma ,\infty} \right) .
\eqe 
Similarly, a.s.\ 
\eqb \label{eqn-sle-disk-as}
\left(  \mcl D^n ,   P_{\sigma^n,\infty} \right) \rta \left(  \mcl S_{\sigma,\tau} ,   P_{\sigma ,\infty} \right) .
\eqe 
The convergence~\eqref{eqn-sle-disk-as} also occurs in law, so by~\eqref{eqn-disk-same-law} and~\eqref{eqn-general-disk-as}, 
\eqb \label{eqn-disk-same-law'} 
 \left( \wt{\mcl S}_{\sigma,\tau} ,   P_{\sigma ,\infty} \right)  \eqD    \left(  \mcl S_{\sigma,\tau} ,   P_{\sigma ,\infty} \right) .
\eqe 
In particular, Proposition~\ref{prop-sle-disk} implies that the conditional law of $\wt{\mcl S}_{\sigma,\tau}$ given its area and boundary length (which are necessarily equal to $\tau-\sigma$ and $|Z_\tau - Z_\sigma|$, respectively) is that of a singly marked quantum disk with given area and boundary length. 
Furthermore, since $P_{\sigma,\infty}$ is a.s.\ determined by $ \sigma$, $Z_\tau-Z_\sigma$, and $(Z-Z_\tau)|_{[\tau,\infty)}$, Proposition~\ref{prop-sle-disk} implies that $\wt{\mcl S}_{\sigma,\tau}$ is conditionally independent from $P_{\sigma,\infty}$ given $\tau-\sigma$ and $|Z_\tau - Z_\sigma|$. Note that $\tau-\sigma$ and $|Z_\tau - Z_\sigma|$ are a.s.\ determined by $P_{\sigma,\infty}$ since these quantities give the area and left/right boundary lengths of the first bead of $\wt{\mcl S}_{\sigma,\infty}$. 
\medskip

\noindent\textit{Step 3: adding extra conditioning.}
We will now argue that we get the same conditional law if we further condition on $\wt{\mcl G}_\sigma$ and $(Z-Z_\tau)|_{[\tau,\infty)}$. 
We showed above that $\wt{\mcl D}^n$ is conditionally independent from $\wt{\mcl G}_{\sigma^n}$ given $P_{\sigma^n,\infty}$ for each $n\in\BB N$. 
Since $\wt{\mcl G}_\sigma \subset \wt{\mcl G}_{\sigma^n}$ for each $n\in\BB N$ (Lemma~\ref{lem-partial-filtration}), $\wt{\mcl D}^n$ is conditionally independent from $\wt{\mcl G}_{\sigma }$ given $P_{\sigma^n,\infty}$. The function $P_{\sigma^n,\infty}$ is determined by $(Z-Z_\sigma)|_{[\sigma,\infty)}$ so by Lemma~\ref{lem-bm-future-ind} is conditionally independent from $\wt{\mcl G}_{\sigma}$ given $P_{\sigma,\infty}$. Hence the conditional law of $\wt{\mcl D}^n$ given $\wt{\mcl G}_{\sigma }$ can be sampled from as follows.
\begin{enumerate}
\item Sample $P_{\sigma^n,\infty}$ from its conditional law given $\wt{\mcl G}_{\sigma }$ (which depends only on $P_{\sigma,\infty}$).
\item Sample $\wt{\mcl D}^n$ from its conditional law given $P_{\sigma^n,\infty}$. 
\end{enumerate}
In particular, this conditional law depends only on $P_{\sigma,\infty}$ so $\wt{\mcl D}^n$ is conditionally independent from $\wt{\mcl G}_\sigma$ given $P_{\sigma,\infty}$. 
Since this holds for each $n\in\BB N$, it follows from~\eqref{eqn-general-disk-as} that $\wt{\mcl S}_{\sigma,\tau}$ is conditionally independent from $\wt{\mcl G}_{\sigma }$ given $P_{\sigma,\infty}$. 

By Lemma~\ref{lem-bm-future-ind} and since $\tau$ is a $\{\wt{\mcl G}_r\}_{r\in\BB R}$-stopping time (Lemma~\ref{lem-partial-surface-determined}), the Brownian motion segment $(Z-Z_\tau)|_{[\tau,\infty)}$ is conditionally independent from $\wt{\mcl G}_\tau$ given $P_{\tau,\infty}$. Since the two coordinates of $(Z-Z_\sigma)|_{[\sigma,\infty)}$ attain a simultaneous running infimum at time $\tau$, it follows from~\eqref{eqn-bead-function} that $P_{\tau,\infty} = P_{\sigma,\infty}|_{[\tau,\infty)}$. The time $\tau$ is determined by $P_{\sigma,\infty}$ since $\tau-\sigma$ is the length of the first interval of times on which $P_{\tau,\infty}$ is constant. Therefore, $P_{\tau,\infty}$ is a.s.\ determined by $P_{\sigma,\infty}$ so $(Z-Z_\tau)|_{[\tau,\infty)}$ is conditionally independent from $\wt{\mcl G}_\tau$ given $P_{\sigma,\infty}$. 
It is clear that $\wt{\mcl G}_\sigma$ and $\sigma(\wt{\mcl S}_{\tau,\sigma})$ are contained in $\wt{\mcl G}_\tau$. 
By combining this with the preceding paragraph, we see that $\wt{\mcl S}_{\sigma,\tau}$ is conditionally independent from $\wt{\mcl G}_\sigma$ and $(Z-Z_\tau)|_{[\tau,\infty)}$ given $P_{\sigma,\infty}$. 

Hence the conditional law of $\wt{\mcl S}_{\sigma,\tau}$ given $\wt{\mcl G}_\sigma$ and $(Z-Z_\tau)|_{[\tau,\infty)}$ is the same as its conditional law given only $P_{\sigma,\infty}$. The discussion just after~\eqref{eqn-disk-same-law'} implies that this conditional law is that of a singly marked quantum disk with area $\tau-\sigma$ and boundary length $Z_\sigma - Z_\tau$.  
\end{proof}

\begin{proof}[Proof of Proposition~\ref{prop-general-disk}]
Fix times $a < t_1 < \dots <  t_n < b$. By Lemma~\ref{lem-general-disk0}, for each $k\in [1,n]_{\BB Z}$ the conditional law of the bubble $\wt{\mcl S}_{\sigma_{a,b}(t_k) , \tau_{a,b}(t_k)}$ given $\wt{\mcl G}_{  \sigma_{a,b}(t_k)} $ and $(Z-Z_{\tau_{a,b}(t_k)})|_{[\tau_{a,b}(t_k), \infty)}$ is that of a singly marked quantum disk with area $\tau_{a,b}(t_k) - \sigma_{a,b}(t_k)$ and boundary length $|Z_{\tau_{a,b}(t_k)} - Z_{\sigma_{a,b}(t_k)}|$. 

By definition, $\wt{\mcl G}_{  \sigma_{a,b}(t_k)} $ and $(Z-Z_{\tau_{a,b}(t_k)})|_{[\tau_{a,b}(t_k), \infty)}$ together a.s.\ determine $Z_s$ for $s\in \BB R\setminus [\sigma_{a,b}(t_k) , \tau_{a,b}(t_k)]$ (here we note that $\wt{\mcl S}_{\sigma_{a,b}(t_k) , \tau_{a,b}(t_k)}$ is the same as the first bead of $\wt{\mcl S}_{\sigma_{a,b}(t_k) , \infty}$, so its area and boundary length are $\wt{\mcl G}_{  \sigma_{a,b}(t_k)} $-measurable). In particular, $\wt{\mcl G}_{  \sigma_{a,b}(t_k)} $ and $(Z-Z_{\tau_{a,b}(t_k)})|_{[\tau_{a,b}(t_k), \infty)}$ together a.s.\ determine $Z|_{(-\infty , a]}$, $Z|_{[b,\infty)}$, and $\{ Z_{\tau_{a,b}(t)} : t\in [a,b] \} $. 
Furthermore, the $\sigma$-algebra $\wt{\mcl G}_{  \sigma_{a,b}(t_k)} $ a.s.\ determines the previous bubbles $\wt{\mcl S}_{\sigma_{a,b}(t_j) , \tau_{a,b}(t_j)}$ for $j\in [1,k-1]_{\BB Z}$ which are not equal to $\wt{\mcl S}_{\sigma_{a,b}(t_k) , \tau_{a,b}(t_k)}$. 

The preceding two paragraphs together imply that the conditional law of $\wt{\mcl S}_{\sigma_{a,b}(t_k) , \tau_{a,b}(t_k)}$ given $Z|_{(-\infty , a]}$, $Z|_{[b,\infty)}$, $\{ Z_{\tau_{a,b}(t)} : t\in [a,b] \} $, and 
$\wt{\mcl S}_{\sigma_{a,b}(t_j) , \tau_{a,b}(t_j)}$ for $j\in [1,k-1]_{\BB Z}$ such that this bubble is not equal to $\wt{\mcl S}_{\sigma_{a,b}(t_k) , \tau_{a,b}(t_k)}$
is that of a singly marked quantum disk with area $\tau_{a,b}(t_k) - \sigma_{a,b}(t_k)$ and boundary length $|Z_{\tau_{a,b}(t_k)} - Z_{\sigma_{a,b}(t_k)}|$. 
This holds for each $k \in [1,n]_{\BB Z}$, so we infer that the conditional law of the distinct bubbles in $\{ \wt{\mcl S}_{\sigma_{a,b}(t_k) , \tau_{a,b}(t_k)} \}_{k\in [1,n]_{\BB Z}}$ given $Z|_{(-\infty , a]}$, $Z|_{[b,\infty)}$, and $\{ Z_{\tau_{a,b}(t)} : t\in [a,b] \} $ is that of a collection of independent singly marked quantum disks with given areas and boundary lengths. Since the number of times $n$ we are considering can be made arbitrarily large, we conclude.
\end{proof}

\subsection{Joint laws of intermediate surfaces}
\label{sec-partial-surface}

In this subsection we will prove an equality in law between certain quantum surfaces defined in terms of $(\wt h , \wt\eta')$ and their counterparts defined in terms of $(h,\eta')$ which builds on the results of the previous two subsections. This result will be a key input in the ``curve-swapping" argument used in Section~\ref{sec-swapping} to prove Theorem~\ref{thm-wpsf-char}. 
Throughout, we fix $a,t,b\in\BB R $ with $a < t < b$ and we let $\sigma = \sigma_{a,b}(t)$ and $\tau = \tau_{a,b}(t)$ be as in~\eqref{eqn-tau-def}.  
We recall the future surface $\mcl S_{a,\infty}$ and the surface $\mcl S_{a,\tau}^0$ parameterized by the bubbles cut out by $\eta_{a,\tau} = \eta_{a,b}|_{[0,\tau]}$ from Section~\ref{sec-surface-def} and their counterparts with $(\wt h ,\wt\eta')$ in place of $(h,\eta')$. 
 
The main goal of this subsection is to prove the following. 

\begin{prop} \label{prop-partial-surface-law}
In the setting described just above,
\begin{align} \label{eqn-partial-surface-law}
 \left(\wt{\mcl S}_{a,\tau}^0 , \wt \eta_{a,b , \wt{\mcl S}_{a,\tau}^0} ,       \wt{\mcl S}_{\tau,\infty}    \right) 
 \eqD \left( \mcl S_{a,\tau}^0 ,  \eta_{a,b , \mcl S_{a,\tau}^0} ,    \mcl S_{\tau,\infty}    \right)  . 
\end{align}  
\end{prop}

We emphasize that $\wt{\mcl S}_{a,b}^0$ for $a < b$ encodes only the quantum surfaces parameterized by the bubbles cut out by $\wt\eta'$ in the time interval $[a,b]$ and the topology of how these quantum surfaces are glued together but not the whole quantum surface structure of $\wt{\mcl S}_{a,b}$. 
 
The proof of Proposition~\ref{prop-partial-surface-law} proceeds by way of two lemmas, which are in turn consequences of the results of Sections~\ref{sec-sle-disk} and~\ref{sec-general-disk}. The first lemma is essentially an extension of Proposition~\ref{prop-general-disk}.

\begin{lem} \label{lem-partial-surface-initial}
We have the equalities in law
\begin{align} \label{eqn-full-surface-initial}
&\left(\wt{\mcl S}_{a,b}^0 ,\wt \eta_{a,b , \wt{\mcl S}_{a,b}^0} ,   Z|_{\BB R\setminus [a,b]} , \{ Z_{\tau_{a,b}(s)}   : s\in [a,b]\}  \right) \notag \\ 
&\qquad \eqD \left( \mcl S_{a,b}^0 , \eta_{a,b , \mcl S_{a,b}^0} ,  Z|_{\BB R\setminus [a,b]} , \{ Z_{\tau_{a,b}(s)}   : s\in [a,b]\}  \right)  
\end{align}
and 
\eqb \label{eqn-partial-surface-initial}
\left( \wt{\mcl S}_{a,\tau}^0 ,\wt \eta_{a,b , \wt{\mcl S}_{a,\tau}^0}      , P_{\tau,\infty}   \right) \eqD \left( \mcl S_{a,\tau}^0 , \eta_{a,b , \mcl S_{a,\tau}^0}   , P_{\tau,\infty}  \right)  .
\eqe 
\end{lem} 
\begin{proof}
By Propositions~\ref{prop-sle-disk} and~\ref{prop-general-disk}, the conditional laws of the bubbles of $\mcl S_{a,b}^0$ and those of $\wt{\mcl S}_{a,b}^0$ given $ Z|_{\BB R\setminus [a,b]}$ and $\{ Z_{\tau_{a,b}(s)}   : s\in [a,b]\}$ a.s.\ agree.
By condition~\ref{item-wpsf-char-homeo} in Theorem~\ref{thm-wpsf-char}, the equivalence classes of the curve-decorated quantum surfaces $(\wt{\mcl S}_{a,b}^0 ,\wt \eta_{a,b , \wt{\mcl S}_{a,b}^0} )$ and $(\mcl S_{a,\tau}^0 , \eta_{a,b , \mcl S_{a,b}^0})$ are a.s.\ given by the same deterministic functional of $\{ Z_{\tau_{a,b}(s)}   : s\in [a,b]\}$ (c.f.~\cite[Figure~1.15, Line~3]{wedges}).  Hence~\eqref{eqn-full-surface-initial} holds. 

We now deduce~\eqref{eqn-partial-surface-initial} from~\eqref{eqn-full-surface-initial}. 
By~\eqref{eqn-tau-def} (resp.\ its analog with $\wt\eta_{a,b}$ in place of $\eta_{a,b}$), the time $\tau = \tau_{a,b}(t)$ is the infimum of the times $s \geq t$ such that $\eta_{a,b}(s) \not= \eta_{a,b}(t)$ (resp.\ $\wt\eta_{a,b}(s) \not= \wt\eta_{a,b}(t)$), so can be obtained by applying the same deterministic functional to either the left or the right side of~\eqref{eqn-full-surface-initial}. 
The quantum surface $\wt{\mcl S}_{a,\tau}^0$ is parameterized by the bubbles of $\wt{\mcl S}_{a,b}^0$ filled in by $\wt\eta'$ before time $\tau = \tau_{a,b}(t)$, so is a.s.\ determined by the left side of~\eqref{eqn-full-surface-initial}. 
Similarly, $\mcl S_{a,b}^0$ is a.s.\ determined by the right side of~\eqref{eqn-full-surface-initial}, in the same deterministic manner.  
None of the times $\ul T_{\tau }$ or $\ol T_{\tau }$ appearing in the definition~\eqref{eqn-bead-function} of $P_{\tau,\infty}$ is contained in one of the intervals $(\sigma_{a,b}(s) , \tau_{a,b}(s))$ for $s\in [a,b]$. 
Therefore, $P_{\tau,\infty}$ is a.s.\ determined by $  Z|_{[b,\infty)}  $ and $ \{ Z_{\tau_{a,b}(s)} : s \in [a,b]\}$. 
We therefore obtain~\eqref{eqn-partial-surface-initial} by applying the same deterministic functional to both sides of~\eqref{eqn-full-surface-initial}.  
\end{proof}

The second lemma we need for the proof of Proposition~\ref{prop-partial-surface-law} is an extension of Lemma~\ref{lem-bm-future-ind}.

\begin{lem} \label{lem-partial-surface-future}
Let $T$ be a stopping time for the filtration $\{ \mcl G_r\}_{r\in\BB R}$ of Lemma~\ref{lem-partial-filtration}.
The following conditional laws a.s.\ coincide.
\begin{enumerate}
\item The conditional law of $\wt{\mcl S}_{T,\infty}$ given $\wt{\mcl G}_T$. 
\item The conditional law of $ \mcl S_{T,\infty}$ given $\mcl G_T$.
\item The conditional law of either $\mcl S_{T,\infty}$ or $\wt{\mcl S}_{T,\infty}$ given $P_{T,\infty}$. 
\end{enumerate}  
\end{lem}
\begin{proof}
The time $T$ is also a stopping time for the larger filtration $\{\wt{\mcl G}_r\}_{r\in\BB R}$
so by Lemma~\ref{lem-bm-future-ind}, the conditional law of $\wt{\mcl S}_{T,\infty}$ given $\wt{\mcl G}_T$ is the same as its conditional law given only $P_{T,\infty}$. 
By the same lemma, applied in the special case when $(\wt h , \wt\eta') = (h,\eta')$, we find that the conditional law of  $\mcl S_{T,\infty}$ given $ \mcl G_T$ is the same as its conditional law given only $P_{T,\infty}$.  
Hence it suffices to show that 
\eqb \label{eqn-future-bdy-eqd}
(\wt{\mcl S}_{T,\infty} , P_{T,\infty}) \eqD (\mcl S_{T,\infty} , P_{T,\infty}) .
\eqe

In the case when $T =r$ is deterministic, the conditions in Theorem~\ref{thm-wpsf-char} imply that both sides of~\eqref{eqn-future-bdy-eqd} have the same law as a $\frac{3\gamma}{2}$-quantum wedge together with the function which gives the areas and left/right boundary lengths of its beads. Hence~\eqref{eqn-future-bdy-eqd} holds in this case. 
In the case when $T$ takes on only countably many possible values, we infer from the deterministic case that the conditional law of $\wt{\mcl S}_{T,\infty}$ given $P_{T,\infty}$ is that of a sequence of beads of a $\frac{3\gamma}{2}$-quantum wedge with areas and left/right boundary lengths specified by $P_{T,\infty}$; and the same holds for $\mcl S_{T,\infty}$. Hence~\eqref{eqn-future-bdy-eqd} holds in the case when $T$ takes on only countably many values. 

In general, we let $T^n  =2^{-n} \lceil 2^{n} T \rceil $, so that each $T^n$ is a stopping time for $\{\mcl G_r\}_{r\in\BB N}$ and $T^n\rta T$ a.s. By taking the limits of both sides of~\eqref{eqn-future-bdy-eqd} with $T^n$ in place of $T$, we obtain~\eqref{eqn-future-bdy-eqd} in general. 
\end{proof}

\begin{proof}[Proof of Proposition~\ref{prop-partial-surface-law}]   
By Lemma~\ref{lem-partial-surface-determined}, $\tau$ is a stopping time for the filtration $\{ \mcl G_r\}_{r\in\BB R}$ of Lemma~\ref{lem-partial-filtration}. Furthermore, the triples in~\eqref{eqn-partial-surface-initial} from Lemma~\ref{lem-partial-surface-initial} satisfy
\eqbn
\left(\wt{\mcl S}_{a,\tau}^0 ,\wt \eta_{a,b , \wt{\mcl S}_{a,\tau}^0} ,  P_{\tau,\infty}   \right) \in \wt{\mcl G}_\tau 
\quad \op{and} \quad
\left(  \mcl S_{a,\tau}^0 , \eta_{a,b , \mcl S_{a,\tau}^0}    ,    P_{\tau,\infty}   \right) \in \mcl G_\tau . 
\eqen 
Hence Lemma~\ref{lem-partial-surface-future} implies that the conditional law of $\wt{\mcl S}_{\tau,\infty}$ given the left side of~\eqref{eqn-partial-surface-initial} a.s.\ coincides with the conditional law of $\mcl S_{\tau,\infty}$ given the right side of~\eqref{eqn-partial-surface-initial} (and this conditional law depends only on $P_{\tau,\infty}$). 
By combining this with~\eqref{eqn-partial-surface-initial}, we obtain~\eqref{eqn-partial-surface-law}.  
\end{proof}

\section{Diameter estimate for space-filling SLE segments}
\label{sec-diam-sum}

In this section we will prove an estimate (Proposition~\ref{prop-diam-sum} just below) which will be used to bound the diameters of the hulls (and thereby, via~\cite[Lemma~9.6]{wedges} the distortion of the conformal maps) in the curve-swapping argument used in the proof of Theorem~\ref{thm-wpsf-char}. In Section~\ref{sec-stability}, we will also introduce a regularity event which is similar to the one used in~\cite[Section~9.4.2]{wedges} which will enable us to apply our estimate. We emphasize that Proposition~\ref{prop-diam-sum} and the regularity event of Section~\ref{sec-stability} are the only results from this section used later in the paper, but some of the results involved in the proof of Proposition~\ref{prop-diam-sum} are of independent interest.
 
Before stating our estimate, we first define a particular embedding of a $\gamma$-quantum cone which is also used in~\cite[Section~9.4.2]{wedges}.  Fix once and for all a smooth non-negative bump function $\phi$ supported on $(-1/100, 0]$ with total integral one.  Suppose $h$ is an embedding of a $\gamma$-quantum cone in $(\BB C , 0, \infty)$.  For $r \geq 0$, let $h_r(0)$ be the circle average of $h$ over $\bdy B_r(0)$. 
For $t\in\BB R$, define
\eqb \label{eqn-smooth-bm}
V_t := \frac{1}{\gamma} \log \BB E\left[ \mu_h\left(  B_{e^{-t}}(0)  \right)  \,|\, h_{e^{-t}}(0)  \right]  .
\eqe   
For $C \in \BB R$, we say that the embedding $h$ is a \emph{$C$-smooth canonical description} of $(\BB C , h , 0, \infty)$ if 
\eqb \label{eqn-smooth-def}
\inf\left\{ s \in \BB R : \int_{\BB R}  V_t \phi(s-t) \, dt = C \right\} =0 .
\eqe 
Note that if $h$ is a $C$-smooth canonical description and $a > 0$, then $h + a$ is a $ C + a$-smooth canonical description. We will need to choose the parameter $C$ appropriately in Section~\ref{sec-stability} in order to make the probability of a certain event close to 1.
A smooth canonical description is not unique: applying a rotation gives another smooth canonical description.  Similarly to the circle average embedding, if $h$ is a $C$-smooth canonical description then $h|_{\BB D}$ agrees in law with the restriction to $\BB D$ of $h^0 - \gamma \log |\cdot| + \xi$, where $h^0$ is a whole-plane GFF normalized so that its circle average over $\bdy \BB D$ is zero and $\xi$ is a random variable whose absolute value is stochastically dominated by a constant (depending only on $C$) plus the modulus of centered Gaussian random variable with variance bounded above by a universal constant.

Our reason for considering the smooth canonical description is that this is the embedding used in~\cite[Section~9]{wedges} to prove that a certain regularity condition for the $\gamma$-quantum cone holds whenever we ``swap out" part of the surface and replace it with another surface (see Section~\ref{sec-stability}). This estimate will be important in Section~\ref{sec-swapping}. The advantage of the smooth canonical description over the circle average embedding is that it behaves more nicely than the circle average embedding under conformal maps between subsets of $\BB C$ which are not affine.

The main result of this section is the following proposition. 

\begin{prop} \label{prop-diam-sum}
Let $\kappa' \in (4,8)$ and $\gamma = 4/\sqrt{\kappa'}$. 
Let $(\BB C ,h , 0, \infty)$ be a $\gamma$-quantum cone and suppose that $h$ is either a circle-average embedding or a $C$-smooth canonical description for some $C\in\BB R$. 
Also let $\eta'$ be a whole-plane space-filling $\SLE_{\kappa'}$ from $\infty$ to $\infty$ independent from $h$ and parameterized by $\gamma$-quantum mass with respect to $h$. 
Let $a,b\in\BB R$ with $a<b$ and let $S$ be a deterministic closed subset of $\BB D\setminus \{0\}$. 

Recall the chordal $\SLE_{\kappa'}$-type curve $\eta_{a,b} : [a,b] \rta \eta'([a,b])$ from Section~\ref{sec-disk-law}.  
For $n\in\BB N$ and $k\in[1,n]_{\BB Z}$, let $t_{n,k} := a + \frac{k}{n}(b-a)$ and let 
\eqbn 
G_{n,k} := \left\{\eta'([t_{n,k-1} , t_{n,k}]) \cap \eta_{a,b} \not=\emptyset\right\}  \cap \left\{ \eta_{a,b}([t_{n,k-1} , t_{n,k}]) \subset S  \right\} .
\eqen
There is an $\alpha = \alpha(\kappa') >0$ such that for $n\in\BB N$,
\eqb \label{eqn-diam-sum}
 \sum_{k=1}^n \BB E\left[ \op{diam}(\eta'([t_{n,k-1} , t_{n,k}]))^2 \BB 1_{G_{n,k} } \right] \leq n^{-\alpha + o_n(1)} 
\eqe 
where the rate of convergence of the $o_n(1)$ depends on $C,a,b,S$, and $\kappa'$. 
\end{prop}
 
The important point in Proposition~\ref{prop-diam-sum} is that the right side of~\eqref{eqn-diam-sum} decays like a negative power of $n$. For our purposes the particular value of $\alpha$ does not matter. Out proof can give an explicit bound for $\alpha$, but we do not attempt to optimize it so we will not record the bound.

Throughout this section, we will frequently abuse notation by using the same symbol for a curve and its image, when there is no danger of ambiguity.
To prove Proposition~\ref{prop-diam-sum}, we will first prove an upper bound for $\BB E\left[ \op{Area}\left( B_\ep(\eta_{a,b}\cap S) \right) \right]$, where here $\op{Area}$ denotes two-dimensional Lebesgue measure and $B_\ep(\cdot)$ denotes the Euclidean $\ep$-neighborhood (recall Section~\ref{sec-basic}). To do this we will use a version of the KPZ formula (Proposition~\ref{prop-dimM-upper}) which relates the expected Lebesgue measure of the $\ep$-neighborhood of a subset of $\BB C$ which is independent from $h$ (but not from $\eta'$) to the expected number of $\ep$-length intervals needed to cover its pre-image under $\eta'$. This statement is similar to the main result of~\cite{ghm-kpz}, but we prove only a one-sided bound and our estimate concerns expected areas rather than a.s.\ Hausdorff dimensions. The proof of our KPZ formula is given in Section~\ref{sec-dimM-upper}. 

In Section~\ref{sec-sle-area}, we will use our KPZ formula to reduce the problem of estimating $\BB E\left[ \op{Area}\left( B_\ep(\eta_{a,b}\cap S) \right) \right]$ to the problem of estimating the expected number of $\ep$-length intervals needed to cover a certain subset of~$\BB R$ described in terms of the peanosphere Brownian motion $Z$. We will then estimate this quantity via a Brownian motion calculation. Actually, the curve $\eta_{a,b}$ is not independent from $h$ (since $\eta'$ is parameterized by $\mu_h$-mass) so we will instead apply the KPZ formula to bound $\BB E\left[\sup_{a,b\in\BB R ,\: a < b} \op{Area}\left( B_\ep(\eta_{a,b}\cap S) \right) \right]$, noting that the supremum inside the expectation depends only on $\eta'$ viewed modulo time parameterization.
 
In Section~\ref{sec-diam-sum-proof}, we will deduce the diameter estimate~\eqref{eqn-diam-sum} from the aforementioned estimate for $\BB E\left[ \op{Area}\left( B_\ep(\eta_{a,b}\cap S) \right) \right]$. This is accomplished by means of an estimate (stated as Lemma~\ref{lem-sle-ball} below) from~\cite{ghm-kpz} which tells us that the diameter of a space-filling SLE segment is very unlikely to be larger than the square of its area, up to an $o(1)$ error in the exponent. 

In Section~\ref{sec-stability}, we will introduce a regularity event (defined in terms of the $C$-smooth canonical description) which will allow us to apply the estimate of Proposition~\ref{prop-diam-sum} in Section~\ref{sec-swapping}.

\subsection{KPZ formula for expected areas}
\label{sec-dimM-upper}

Let $\kappa' > 4$ and $\gamma = 4/\sqrt{\kappa'}$.  
Let $(\BB C ,h , 0, \infty)$ be a $\gamma$-quantum cone and let $\eta'$ be an independent whole-plane space-filling $\SLE_{\kappa'}$ from $\infty$ to $\infty$ parameterized by $\gamma$-quantum mass with respect to $h$. 
In this subsection, we will prove a KPZ-type formula which gives an upper bound for the expected Lebesgue measure of the $\ep$-neighborhood of a set $X\subset \BB C$ which is independent from $h$ (but not necessarily from $\eta'$) in terms of the number of $\ep$-length intervals needed to cover a set $\wh X\subset \BB R$ with $\eta'(\wh X) = X$. 
We emphasize that, unlike most of the results of this paper, the results of this subsection are valid for any $\kappa' > 4$ (and hence any $\gamma \in (0,2)$) rather than only for $\kappa' \in (4,8)$ ($\gamma \in (\sqrt 2, 2)$). 

\begin{prop} \label{prop-dimM-upper}
Let $\kappa' > 4$ and $\gamma = 4/\sqrt{\kappa'}$. 
Let $(\BB C ,h , 0, \infty)$ be a $\gamma$-quantum cone and suppose that $h$ is either a circle-average embedding or a $C$-smooth canonical description for some $C\in\BB R$. 
Also let $\eta'$ be a whole-plane space-filling $\SLE_{\kappa'}$ from $\infty$ to $\infty$ independent from $h$ and parameterized by $\gamma$-quantum mass with respect to $h$. 
Let $S$ be a deterministic closed subset of $\BB D $ and let $X$ be a random subset of $S$ such that the pair consisting of the set $X$ and the curve $\eta'$, viewed modulo monotone re-parameterization, is independent from $h$ (e.g., $X$ could be determined by $\eta'$, viewed modulo monotone-reparameterization).

Let $\wh X\subset \BB R$ be a random set such that $X\subset \eta'(\wh X)$ a.s.\ (e.g., $\wh X = (\eta')^{-1}(X)$). 
For $\ep > 0$, let $N_\ep$ be the number of intervals of length $\ep$ needed to cover $\wh X$. 
Suppose that for some $\beta\in [0,1]$ and $A > 0$, we have
\eqb \label{eqn-ball-count-limsup}
  \BB E[N_\ep]  \leq A \ep^{- \beta } , \quad\forall \ep \in (0,1) .
\eqe 
Then for any
\eqb
\Delta < 2 + \frac{\gamma^2}{2} \beta^2  -  \left(2 + \frac{\gamma^2}{2} \right) \beta  ,
\eqe
it holds that
\eqb \label{eqn-dimM-upper}
 \BB E\left[ \op{Area}(B_\ep(X)) \right] \preceq  \ep^\Delta 
\eqe
with the implicit constant depending only on $\Delta ,A, C,S$ (not on the particular choice of $X$). 

\end{prop}

Proposition~\ref{prop-dimM-upper} is a one-sided version of the KPZ formula~\cite{kpz-scaling,shef-kpz}. The proposition is closely related to~\cite[Theorem~1.1]{ghm-kpz} (which gives an analogous relation for Hausdorff dimension, with an equality in place of an inequality) and~\cite[Proposition~3.4]{ghs-dist-exponent} (which gives an upper bound for the Minkowski content of $(\eta')^{-1}(X)$ in terms of the Minkowski content of $X$). See also~\cite{aru-kpz,grv-kpz,bjrv-gmt-duality,benjamini-schramm-cascades,wedges,shef-renormalization,shef-kpz,rhodes-vargas-log-kpz,gp-kpz} for additional rigorous versions of the KPZ formula. 

Proposition~\ref{prop-dimM-upper} and the aforementioned related results are especially useful for computing dimensions and exponents for sets defined in terms of SLE. The reason for this is that for many SLE sets of interest, the set $(\eta')^{-1}(X)$ admits a simple description in terms of the peanosphere Brownian motion $Z$. This reduces an SLE computation to a Brownian motion computation which is often much easier. 
We will apply Proposition~\ref{prop-dimM-upper} in Section~\ref{sec-sle-area} below to prove an upper bound for the expected area of the $\ep$-neighborhood of $\eta_{a,b}\cap S$, where $\eta_{a,b} $ is the curve appearing in Proposition~\ref{prop-diam-sum}. 
This bound, in turn, will be the key input in the proof of Proposition~\ref{prop-diam-sum}. 

The proof of the KPZ relation Proposition~\ref{prop-dimM-upper} is similar in spirit to the proof of the upper bound for the Hausdorff dimension of $X$ in~\cite[Theorem~1.1]{ghm-kpz}. 
We first state a basic preliminary estimate for space-filling $\SLE_{\kappa'}$ (Lemma~\ref{lem-sle-ball}) and a basic estimate for the $\gamma$-LQG measure (Lemma~\ref{lem-mu_h-lower}) which follow from estimates in~\cite{ghm-kpz}. We then consider for $\alpha>0$ the set $X_\ep^\alpha \subset X$ which is, roughly speaking, the set of points in $X$ which are contained in a segment of $\eta'$ with diameter of order $\ep$ and $\mu_h$-mass of order $\ep^{2+\frac{\gamma^2}{2}-\alpha}$. The set $X_\ep^\alpha$ is typically much smaller than $X$ since the $\mu_h$-mass of a Lebesgue-typical segment of $\eta'$ with diameter of order $\ep$ is of order $\ep^{2+\frac{\gamma^2}{2}}$. 
However, since $X$ is independent from $h$ one can prove a lower bound for the area of $B_\ep(X_\ep^\alpha)$ in terms of the area of $B_\ep(X )$ (Lemma~\ref{lem-area-X_ep}). 
We will then prove an upper bound for the area of $B_\ep(X_\ep^\alpha)$ in terms of the quantities $N_{\ep}$ from Proposition~\ref{prop-dimM-upper}, deduce from this and Lemma~\ref{lem-area-X_ep} an upper bound for $\op{Area}(B_\ep(X))$ in terms of $N_\ep$ and $\alpha$, and optimize over $\alpha$ to conclude.
 
The following SLE estimate from~\cite{ghm-kpz} is one of the key inputs in the proof of Proposition~\ref{prop-dimM-upper}, and will also be used in the proof of Proposition~\ref{prop-diam-sum} below.

\begin{lem} \label{lem-sle-ball}
Fix $u  \in (0,1)$ and for $\ep > 0$, let $\mcl E_\ep = \mcl E_\ep(u)$ be the event that the following is true. For each $\delta \in (0,\ep]$ and each $a,b\in\BB R$ with $a<b$ such that $\eta'([a,b]) \subset \BB D$ and $\op{diam}(\eta'([a,b])) \geq \delta^{1-u}$, the set $\eta'([a,b])$ contains a Euclidean ball of radius at least $\delta$. The $\BB P\left[ \mcl E_\ep^c \right] = o_\ep^\infty(\ep)$. 
\end{lem}
\begin{proof}
This is a re-statement of~\cite[Proposition~3.4]{ghm-kpz} but with a sub-polynomial bound for the rate at which $\BB P[\mcl E_\ep^c] \rta 0$. As explained in~\cite[Remark~3.9]{ghm-kpz}, this bound follows from results in~\cite{hs-euclidean}. 
\end{proof}

Next we state a basic preliminary estimate for the $\gamma$-quantum area measure $\mu_h$. 

\begin{lem} \label{lem-mu_h-lower}
Suppose that $h$ is either a circle-average embedding or a $C$-smooth canonical description of a $\gamma$-quantum cone for some $C\in\BB R$.  For each closed set $S\subset \BB D  $ and each $\alpha >0$, 
\eqb \label{eqn-mu_h-lower}
\BB P\left[ \mu_h(B_\ep(z)) \geq \ep^{2 + \frac{\gamma^2}{2} - \alpha} \right] \geq \ep^{\frac{\alpha^2}{2\gamma^2} +o_\ep(1) } , \quad \forall z\in S ,\quad \forall \ep \in (0,1) 
\eqe 
where the rate of convergence of the $o_\ep(1)$ depends only on $C$ and $S$. 
\end{lem}
\begin{proof}
Recall that $h|_{\BB D}$ agrees in law with the restriction to $\BB D$ of a whole-plane GFF normalized so that its circle average over $\bdy \BB D$ is a constant depending only on $C$ plus $-\gamma \log |\cdot|$ plus a random variable $\xi$ whose absolute value has at most a Gaussian tail (we set $\xi = 0$ in the case of a circle-average embedding).  For $\ep  > 0$ and $z\in \BB C$, let $h_\ep(z)$ be its circle average over $\bdy B_\ep(z)$.  For each $z\in S$ and each $\ep \in (0,  \frac12 \op{dist}(z , \bdy\BB D  ) ]$, the circle average $h_\ep(z) -\xi$ is Gaussian with mean at least $-O(1)$ and variance $\log\ep^{-1} + O(1)$, where here $O(1)$ denotes a finite quantity which is bounded above in absolute value by constants depending only on $C$ and $S$ (see, e.g., the calculations in~\cite[Section~3.1]{shef-kpz}). Note that here we use that $-\gamma\log|\cdot|$ is non-negative on $\BB D$.

Since $\BB P[|\xi| > (\log \ep^{-1})^{2/3}] = o_\ep^\infty(\ep)$, for $z\in S$ and $\ep \in (0,1)$, 
\eqbn
\BB P\left[ h_\ep(z) \geq \frac{\alpha}{\gamma} \log \ep^{-1} \right] 
\geq \ep^{\frac{\alpha^2}{2\gamma^2} + o_\ep(1)} .
\eqen
On the other hand, by standard tail estimates for the $\gamma$-LQG measure (see, e.g.,~\cite[Lemma~3.12]{ghm-kpz}), for each $u \in (0,1)$ it holds that
\eqbn
\BB P\left[ \mu_h(B_\ep(z)) < \ep^{2+\frac{\gamma^2}{2}  + u }  e^{\gamma h_\ep(z)}  \right]  = o_\ep^\infty(\ep) .
\eqen
Combining the above two estimates and sending $u\rta0 $ yields~\eqref{eqn-mu_h-lower}.
\end{proof}

In the remainder of this subsection, we assume we are in the setting of Proposition~\ref{prop-dimM-upper}.  
Let $\wh\eta'$ be the curve $\eta'$, viewed modulo monotone parameterization, so that $\wh\eta'$ is independent from $h$.  
Fix a small parameter $u \in (0,1/100)$ and let
\eqbn
q := 1-u \in (0,1) .
\eqen
Let $K$ be a deterministic upper bound for the number of times $\eta'$ can hit any given point (such a bound exists; see, e.g.,~\cite[Corollary~6.5]{ghm-kpz}). For $z\in \BB C$ and $i\in \{1,\dots,K\}$, let $\tau^i(z)$ be the $i$th time at which $\eta'$ hits $z$ (if it exists) and for $\ep \in (0,1)$, define
\eqb \label{eqn-multihit-time}
\sigma_\ep^i(z) := \inf\left\{t\geq \tau^i(z) :   \eta'(t) \notin B_{\ep^{q}}(z) \right\} .
\eqe 
We note that the times $\tau^i(z)$ and $\sigma_\ep^i(z)$ depend on the parameterization of $\eta'$, hence on $h$, but the sets $\eta'([\tau^i(z) , \sigma_\ep^i(z)])$ depend only on $\wh\eta'$, so are independent from $h$. 

For $\alpha>0$ and $\ep \in (0,1)$, define 
\eqb  \label{eqn-X_ep-def}
X_\ep^\alpha := \left\{z\in X  :  \exists i\in\{1,,\dots,K\} \: \text{with} \: \mu_h\left(\eta'([\tau^i(z) , \sigma_\ep^i(z) ]) \right) \geq  \ep^{ 2+\frac{\gamma^2}{2} -\alpha  }   \right\} .
\eqe 

\begin{lem} \label{lem-area-X_ep} 
For $\ep \in (0,1)$, define the event $\mcl E_{\ep^u} = \mcl E_{\ep^u}(u)$ as in Lemma~\ref{lem-sle-ball} (with $\ep^u$ in place of $\ep$) and let $X_\ep^\alpha$ be as in~\eqref{eqn-X_ep-def}. 
For $\alpha > 0$ and $\ep \in (0,1)$, a.s.\ 
\eqb \label{eqn-area-X_ep}
\BB E\left[\op{Area} \left( B_{\ep^{q} } (X_\ep^\alpha ) \right)  \,|\, \wh\eta' , X \right] \BB 1_{\mcl E_{\ep^{u}}} \geq \ep^{ \frac{\alpha^2}{2\gamma^2}  + o_\ep(1) }  \op{Area}( B_{\ep^{q}}(X) )  \BB 1_{\mcl E_{\ep^{u}}} .
\eqe  
where the rate of the $o_\ep(1)$ is deterministic and depends only on $S$, $C$, $u$, $\alpha$. 
\end{lem}
\begin{proof}
We have, a.s.,
\eqb \label{eqn-X_ep-integral}
\BB E\left[  \op{Area}  ( B_{\ep^{q} } (X_\ep^\alpha ) ) \,|\, \wh\eta' , X \right] \BB 1_{\mcl E_{\ep^{u}}} =     \BB 1_{\mcl E_{\ep^{u}}} \int_{B_{\ep^{q}}(X)}  \BB P\left[ \op{dist}(w,X_\ep^\alpha) \leq \ep^{q} \,|\, \wh\eta' , X \right]     \, dw  .
\eqe
By the definition of $\mcl E_{\ep^{u}}$ and~\eqref{eqn-multihit-time}, on $\mcl E_{\ep^u}$, for the set $\wh\eta'([\tau^i(z), \sigma^i_\ep(z) ])$ contains a ball of radius at least $\ep $ for each $z\in S$ and $i\in\{1,\dots,K\}$. Since $X\subset S$ a.s.\ and $h$ is independent from $(\wh\eta' , X)$, we can apply Lemma~\ref{lem-mu_h-lower} to get that a.s.\ 
\eqbn \label{eqn-G_ep-prob}
\BB P\left[ z \in X_\ep^\alpha \,|\, \wh\eta' , X \right] \BB 1_{\mcl E_{\ep^{u}}} \geq \ep^{\frac{\alpha^2}{2\gamma^2}  + o_\ep(1) } \BB 1_{\mcl E_{\ep^{u}}} , 
\quad \forall z\in X ,
\eqen  
with the $o_\ep(1)$ satisfying the conditions in the lemma statement. 
For each $w\in B_{\ep^{q}}(X)$, there is a $z\in X$ with $|z-w| \leq \ep^{q}$. For this choice of $z$, a.s.\ 
\eqb \label{eqn-dist-prob}
\BB P\left[ \op{dist}(w,X_\ep^\alpha) \leq \ep^{q} \,|\, \wh\eta' , X \right] \BB 1_{\mcl E_{\ep^{u}}} 
\geq \BB P\left[ z \in X_\ep^\alpha \,|\, \wh\eta' , X \right] \BB 1_{\mcl E_{\ep^{u}}} 
\geq \ep^{\frac{\alpha^2}{2\gamma^2} + o_\ep(1) } \BB 1_{\mcl E_{\ep^{u}}} .
\eqe  
By combining~\eqref{eqn-X_ep-integral} and~\eqref{eqn-dist-prob}, we get~\eqref{eqn-area-X_ep}.
\end{proof}

\begin{proof}[Proof of Proposition~\ref{prop-dimM-upper}] 
Let $u\in (0,1/100)$ and let $q = 1-u$ be as above. Also fix $\alpha > 0$. Define the event $\mcl E_{\ep^u} = \mcl E_{\ep^u}(u)$ as in Lemma~\ref{lem-sle-ball} and the set $X_\ep^\alpha$ as in~\eqref{eqn-X_ep-def}. Also let 
\eqbn
\wh X_\ep^\alpha := \wh X \cap (\eta')^{-1}(X_\ep^\alpha) 
\eqen
so that $X_\ep^\alpha \subset \eta'(\wh X_\ep^\alpha)$. 
Write $\op{len}$ for one-dimensional Lebesgue measure.
We will prove an upper bound for $\BB E\left[ \op{Area} B_{\ep^{q} } (X_\ep^\alpha )   \BB 1_{\mcl E_{\ep^{u}}} \right]$ in terms of $\BB E\left[ \op{len}(\wh X_\ep^\alpha) \right]$, transfer this to an upper bound for $\BB E\left[ \op{Area}( B_{\ep^{q}}(X) )  \right]$ using Lemma~\ref{lem-area-X_ep}, then optimize over all possible values of $\alpha$. 

Throughout the proof, for $\ep \in (0,1)$ we define
\eqbn
\wt \ep  := \ep^{  2+\frac{\gamma^2}{2}  -\alpha } 
\eqen
and we recall that $N_{\wt\ep/2}$ is the minimal number of intervals of length $\wt\ep/2$ needed to cover $\wh X$. 
\medskip

\noindent\textit{Step 1: covers of $\wh X$ and $X_\ep^\alpha$.} Let $\mcl I_\ep$ be a set of at most $N_{\wt\ep/2}$ closed intervals of length $\wt\ep/2$ whose union contains $\wh X$. 
Let $\mcl I_\ep^\alpha$ be the set of those intervals $I\in\mcl I_\ep$ such that the following is true. 
There is a $z\in S$ and an $i\in \{1,\dots,K\}$ such that $\mu_h(\eta'( [\tau^i(z) , \sigma_\ep^i(z)])) \geq \wt\ep$ and $\tau^i(z) \in I$, where here we define $\tau^i(z)$ and $\sigma_\ep^i(z)$ as in~\eqref{eqn-multihit-time} and the discussion just above with respect to the quantum mass parameterization (i.e., the parameterization of $\eta'$).  

For $I\in\mcl I_\ep^\alpha$, let $t_I$ be the infimum of the times $\tau^i(z) \in I $ for $z\in S$ and $i\in \{1,\dots,K\}$ for which the above condition is satisfied and let $\wt I := I\cap [t_I,\infty)$ be the sub-interval of $I$ lying to the right of $t_I$. 
Also let
\eqbn
\wt{\mcl I}_\ep^\alpha := \left\{\wt I \,:\, I\in \mcl I_\ep^\alpha \right\} .
\eqen 
By~\eqref{eqn-X_ep-def}, for each $z\in X_\ep^\alpha$ there exists $i\in \{1,\dots,K\}$ such that $\mu_h(\eta'( [\tau^i(z) , \sigma_\ep^i(z)])) \geq \wt\ep$. Since the union of the intervals in $\mcl I_\ep$ covers $\wh X$ and $X_\ep^\alpha \subset X\subset \eta'(\wh X)$, we infer that
\eqbn
X_\ep^\alpha \subset \bigcup_{\wt I \in \wt{\mcl I}_\ep^\alpha} \eta'(\wt I) .
\eqen
\medskip

\noindent\textit{Step 2: diameter bound for $\eta'(\wt I)$.}
We now claim that on $\mcl E_{ \ep^{u} }$,
\eqb \label{eqn-spatial-diam-upper}
\op{diam} \eta (\wt I) \leq  2\ep^{q} \qquad \forall \wt I \in \wt{\mcl I}_\ep^\alpha .
\eqe
Indeed, suppose to the contrary that $\mcl E_{ \ep^{u} }$ but $\op{diam} \eta (\wt I) > 2\ep^{q}$ for some $\wt I\in\wt{\mcl I}_\ep^\alpha$. 
If we let $t_I$ be the left endpoint of $\wt I$ and $z := \eta'(t_I)$, then $t_I = \tau^i(z)$ for some $i\in \{1,\dots,K\}$ and $\mu_h\left(\eta([\tau^i(z) , \sigma_\ep^i(z)]) \right) \geq  \ep^{ 2+\frac{\gamma^2}{2} -\alpha  }$ for this $i$. 
Since we are assuming that $\op{diam} \eta (\wt I) > 2\ep^{q}$, by the definition~\eqref{eqn-multihit-time} of $\sigma_\ep^i(z)$ we must also have $\sigma_\ep^i(z) \in \wt I$. 
Therefore
\eqbn
\op{len}(\wt I) = \mu_h\left(\eta(\wt I)\right) \geq \mu_h\left(\eta([\tau^i(z) , \sigma^i_\ep(z) ]) \right) \geq  \wt \ep . 
\eqen
But, $\op{len}(\wt I) \leq  \op{len}(I) = \wt\ep/2$ by the definition of $\mcl I_\ep$, so this is a contradiction.  
\medskip

\noindent\textit{Step 3: conclusion.}
By~\eqref{eqn-spatial-diam-upper}, on $\mcl E_{\ep^{u}} $ the collection $\{\eta'(\wt I) : I \in \wt{\mcl I}_\ep^\alpha\}$ is a covering of $X_\ep^\alpha$ by at most $N_{\wt\ep/2}$ sets of diameter at most $2\ep^{q}$. The $\ep^{q}$-neighborhoods of these sets cover $B_{\ep^{q}}(X_\ep^\alpha)$ and the Euclidean area of each such neighborhood is at most a universal constant times $\ep^{2q}$. Therefore, 
\eqbn
\op{Area}\left(B_{\ep^{q}}(X_\ep^\alpha) \right)\BB 1_{\mcl E_{\ep^{u}}}  \preceq \ep^{2q} N_{\wt\ep/2}  \BB 1_{\mcl E_{\ep^{u}}}  
\eqen  
with universal implicit constant. 
Since $\mcl E_{\ep^{u}} $ depends only on the unparameterized curve $\wh\eta'$, we can take expectations of both sides conditional on $\wh\eta' , X$ to get that a.s.\ 
\eqbn
\BB E\left[ \op{Area}\left(B_{\ep^{q}}(X_\ep^\alpha) \right)    \,|\, \wh\eta' , X  \right] \BB 1_{\mcl E_{\ep^u}} \leq \ep^{2q} \BB E\left[   N_{\wt\ep/2} \,|\, \wh\eta' , X \right] \BB 1_{\mcl E_{\ep^u}}   .
\eqen
By combining this with Lemma~\ref{lem-area-X_ep}, we get that a.s.\
\eqb \label{eqn-area-compare}
   \op{Area}( B_{\ep^{q}}(X) )   \BB 1_{\mcl E_{\ep^u}} \preceq \ep^{ 2q - \frac{\alpha^2}{2 \gamma^2}   + o_\ep(1) } \BB E\left[   N_{\wt\ep/2} \,|\, \wh\eta' , X \right] \BB 1_{\mcl E_{\ep^u}}   ,
\eqe
where the rate of convergence of the $o_\ep(1)$ depends only on $S , C , u , \alpha$.  

If we assume that~\eqref{eqn-ball-count-limsup} holds, then we can take expectations of both sides of~\eqref{eqn-area-compare} to get
\eqbn
\BB E\left[ \op{Area}( B_{\ep^{q}}(X) )   \BB 1_{\mcl E_{\ep^u}} \right] \leq \ep^{ 2q - \frac{\alpha^2}{2 \gamma^2} - \left( 2 + \frac{\gamma^2}{2} -\alpha   \right) \beta  +  o_\ep(1) }  ,
\eqen
where the rate of convergence of the $o_\ep(1)$ depends only on $A$, $C$, $S$, $u$, $\alpha$ (recall that $A$ is the constant from~\eqref{eqn-ball-count-limsup}). 
By Lemma~\ref{lem-sle-ball}, $\BB P[\mcl E_{\ep^u}] = 1- o_\ep^\infty(\ep)$ so since $X\subset S\subset \BB D$, 
\eqb \label{eqn-area-compare'} 
\BB E\left[ \op{Area}( B_{\ep^{q}}(X) )  \right]   \leq \ep^{ 2q - \frac{\alpha^2}{2 \gamma^2} - \left( 2 + \frac{\gamma^2}{2} -\alpha   \right) \beta   +  o_\ep(1) }   + o_\ep^\infty(\ep) .
\eqe 
The exponent $ 2  - \frac{\alpha^2}{2 \gamma^2} - \left( 2 + \frac{\gamma^2}{2} -\alpha   \right) \beta$ is minimized over all possibly choices of $\alpha$ by taking $\alpha = \gamma^2 \beta$. Choosing this $\alpha$ and making $u$ sufficiently close to zero (and hence $q $ sufficiently close to 1), depending on $\Delta$, yields~\eqref{eqn-dimM-upper}.  The reason why the implicit constant in~\eqref{eqn-dimM-upper} does not depend on the particular choice of $X$ is that the $o_\ep(1)$ and the $o_\ep^\infty(\ep)$ in~\eqref{eqn-area-compare'} do not depend on the particular choice of $X$. 
\end{proof}

\subsection{Expected area of the $\epsilon$-neighborhood of the SLE curve}
\label{sec-sle-area}

As above, $(\BB C ,h , 0, \infty)$ denotes a $\gamma$-quantum cone and $h$ is either either a circle-average embedding or a $C$-smooth canonical description for some $C\in\BB R$. Also, $\eta'$ is an independent whole-plane space-filling SLE$_{\kappa'}$ from $\infty$ to $\infty$ parameterized by $\mu_h$-mass. 
The diameter estimate in Proposition~\ref{prop-diam-sum} will turn out to be a straightforward consequence of the following area estimate, which we will prove using Proposition~\ref{prop-dimM-upper}. 

\begin{prop} \label{prop-sle-nbd-area}
Let $S \subset \BB D\setminus \{0\}$ be a closed set. 
There exists $\Delta = \Delta(\kappa')  > 0$ such that 
\eqb \label{eqn-eucl-dim-S}
\BB E\left[\sup_{\substack{a,b\in \BB R  \\ a < b}} \op{Area}\left( B_\ep(\eta_{a,b} \cap S) \right) \right] \preceq \ep^\Delta ,
\eqe
with the implicit constant depending only on $C$ and $S$. 
\end{prop}

In fact, the statement of Proposition~\ref{prop-sle-nbd-area} is somewhat stronger than what we need: really we only need~\eqref{eqn-eucl-dim-S} for a fixed choice of $a$ and $b$. However, we have to prove a bound for the supremum over $a$ and $b$ since for a fixed choice of $a$ and $b$ the curve $\eta_{a,b}$ depends on $h$ (since $\eta'$ is parameterized by $\mu_h$-mass) so the KPZ formula of Proposition~\ref{prop-dimM-upper} does not apply to $\eta_{a,b}\cap S$. 

Throughout the rest of this subsection, for a set $\wh X \subset \BB R$ we write $N_\ep(\wh X)$ for the minimal number of intervals of length $\ep$ needed to cover $\wh X$ (c.f.\ Proposition~\ref{prop-dimM-upper}). 
Proposition~\ref{prop-sle-nbd-area} is a straightforward consequence of Proposition~\ref{prop-dimM-upper} combined with the following estimate.

\begin{prop} \label{prop-lqg-dim-S}
Let $S \subset \BB D\setminus \{0\}$ be a closed set. 
There exists $\beta = \beta(\kappa')   \in (0,1) $ such that 
\eqb \label{eqn-lqg-dim-S}
\BB E\left[\sup_{\substack{a,b\in \BB R  \\ a < b}}   N_\ep\left([a,b] \cap (\eta')^{-1}(\eta_{a,b} \cap S) \right)   \right] \preceq \ep^{- \beta} ,
\eqe
with the implicit constant depending only on $C$ and $S$. 
\end{prop}

\begin{proof}[Proof of Proposition~\ref{prop-sle-nbd-area}, assuming Proposition~\ref{prop-lqg-dim-S}]
Observe that the random variable
\[ \sup_{a,b\in \BB R , \: a < b} \op{Area}\left( B_\ep(\eta_{a,b} \cap S) \right)\]
does not depend on the time parameterization of $\eta'$, so is independent from $h$.  For $\ep > 0$, let $a_\ep  < b_\ep$ be chosen so that
\eqb
 \op{Area}\left( B_\ep(\eta_{a_\ep , b_\ep} \cap S) \right)   \geq \frac12  \sup_{\substack{a,b\in \BB R  \\ a < b}} \op{Area}\left( B_\ep(\eta_{a,b} \cap S) \right)  .
\eqe
Then $a_\ep ,b_\ep$ depend on $h$ (since $\eta'$ is paramterized by $\mu_h$-mass), but we can choose $a_\ep , b_\ep$ in such a way that the curve segment $\eta_{a_\ep , b_\ep}$ depends only on $\eta'$ viewed modulo time parameterization, so is independent from $h$. We may therefore apply Proposition~\ref{prop-dimM-upper} with $X = \eta_{a_\ep , b_\ep}$. By Proposition~\ref{prop-lqg-dim-S}, there exists $\beta = \beta(\kappa') \in (0,1)$ such that for every $\delta \in (0,1)$, 
\eqbn
\BB E\left[ N_\delta\left([a_\ep ,b_\ep ] \cap (\eta')^{-1}(\eta_{a_\ep ,b_\ep } \cap S)\right) \right]
\leq \BB E\left[\sup_{\substack{a,b\in \BB R  \\ a < b}}   N_\delta\left( [a,b] \cap (\eta')^{-1}(\eta_{a,b} \cap S) \right)   \right]
 \preceq \delta^\beta ,
\eqen
with the implicit constant depending only on $C$ and $S$. 
Since $\eta'$ hits each point of $\eta_{a_\ep ,b_\ep}$ during the time interval $[a_\ep ,b_\ep]$, it follows that $\eta([a_\ep ,b_\ep ] \cap \eta^{-1}(\eta_{a_\ep ,b_\ep } \cap S)) = \eta_{a_\ep,b_\ep} \cap S$. 
 Therefore, Proposition~\ref{prop-dimM-upper} gives $\BB E\left[\op{Area}\left( B_\ep(\eta_{a_\ep , b_\ep} \cap S) \right) \right] \preceq \ep^\Delta$ for any $\Delta < 2 + \gamma^2\beta^2/2  -  (2+\gamma^2/2)\beta  $. We can choose $\Delta >0$ since $\beta < 1$. By our choice of $a_\ep,b_\ep$, this implies~\eqref{eqn-eucl-dim-S}. 
\end{proof}

The rest of this subsection is devoted to the proof of Proposition~\ref{prop-lqg-dim-S}. The key observation for the proof is as follows. 
Let $Z$ be the peanosphere Brownian motion for $(h,\eta')$, so that $Z$ is a two-dimensional correlated Brownian motion with variances and covariances as in~\eqref{eqn-bm-cov}. 
Then for $a,b\in\BB R$ with $a < b$, the set $[a,b] \cap (\eta')^{-1}(\eta_{a,b})$ is equal to the set of times $t\in [a,b]$ such that $t$ is not contained in any $\pi/2$-cone interval for $Z$ (Definition~\ref{def-cone-time}) which is itself contained in $[a,b]$. This allows us to deduce Proposition~\ref{prop-lqg-dim-S} from Brownian motion estimates. The main Brownian motion estimate needed for the proof is the following lemma.

\begin{lem} \label{lem-no-cone} 
For $t\in\BB R$ and $\ep > 0$, let $\ol\tau_\ep(t)$ be the smallest $\pi/2$-cone time $s$ for $Z$ with $[t , t+\ep] \subset [v_Z(s) , s]$. 
For $T \in (2\ep, \ep^{-100})$, 
\eqb \label{eqn-no-cone}
\BB P\left[ \max\left\{ \ol\tau_\ep(t) - t , t -   v_Z(\ol\tau_\ep(t)) \right\}    > T \right] \preceq  \ep^{1-\kappa'/8} (\log\ep^{-1} \max\{ T^{-(1-\kappa'/8)} , T^{2-8/\kappa' -\kappa'/8} \}      
\eqe 
with the implicit constant depending only on $\kappa'$.  
\end{lem}
\begin{proof}
The time $\ol\tau_\ep(t)$ can equivalently be described as the smallest time after $t+\ep$ at which the coordinates $L$ and $R$ of $Z$ attain a simultaneous running infimum relative to time $t$. We also let $\ul\tau_\ep(t) \in (t,t+\ep)$ be the largest time before $t+\ep$ at which $L$ and $R$ attain a simultaneous running infimum relative to time $t$.  
\medskip

\noindent\textit{Upper bound for $ \ol\tau_\ep(t) $.}
By Lemma~\ref{lem-inf-subordinator}, the law of $\ol\tau_\ep(t) - t$ is that of the first passage time after $\ep$ of a $1-\kappa'/8$-stable subordinator. By the arcsine law for the first passage time (see, e.g.,~\cite[Section 3.1.1]{bertoin-sub}), the law of $\ol\tau_\ep(t) - t$ is given by
\eqbn
\frac{\sin((1-\kappa'/8)\pi)}{\pi} \ep^{1-\kappa'/8} x^{-1}(x-\ep)^{-(1-\kappa'/8)} \,dx ,\quad \forall x > \ep .
\eqen
Integrating gives
\eqb \label{eqn-overshoot-prob}
\BB P\left[ \ol\tau_\ep(t) - t > T \right] 
\asymp \ep^{1-\kappa'/8} \int_T^\infty x^{-1}(x-\ep)^{-(1-\kappa'/8)} \,dx
\preceq \ep^{1-\kappa'/8} \int_T^\infty  x^{-(2-\kappa'/8)} \,dx
\asymp \ep^{1-\kappa'/8} T^{-(1-\kappa'/8)} .
\eqe 
\medskip

\noindent\textit{Lower bound for $v_Z(\ol\tau_\ep(t)) $.}
If $v_Z(\ol\tau_\ep(t)) < t-T$, then
\eqb \label{eqn-cone-long}
\inf_{s\in [t-T ,t]} (L_s - L_t) \geq (L_{\ol\tau_\ep(t)} - L_t) \quad \op{and} \quad \inf_{s\in [t-T,t]} (R_s - R_t) \geq (R_{\ol\tau_\ep(t)} - R_t) .
\eqe 
Since $\ol\tau_\ep(t)$ depends only on $(Z-Z_t)|_{[t,\infty)}$, the quantity $Z_{\ol\tau_\ep(t)} - Z_t$ is independent from $(Z-Z_t)|_{(-\infty, t]}$. 
It follows from the estimate~\cite[Equation~(4.3)]{shimura-cone} for the probability that a Brownian motion stays in a cone for an interval of time together with a linear change of coordinates (c.f.~\cite[Lemma~2.8]{ghs-dist-exponent}) that the conditional probability given $(Z-Z_t)|_{[t,\infty)}$ that~\eqref{eqn-cone-long} holds is at most a $\kappa'$-dependent constant times
\eqb \label{eqn-cone-long-prob0}
  \left[ (L_t - L_{\ol\tau_\ep(t)}  )  \wedge (R_t - R_{\ol\tau_\ep(t)}  )   \right]      \left[ (L_t - L_{\ol\tau_\ep(t)}  )  \vee (R_t - R_{\ol\tau_\ep(t)}  )   \right]^{\kappa'/4-1}    T^{-\kappa'/8}    .
\eqe 
By~\eqref{eqn-cone-long-prob0} and the elementary inequality $(a\wedge b)  (a\vee b)^{\kappa'/4-1}  \leq (a\vee b)^{\kappa'/4}$, 
\eqb \label{eqn-cone-long-prob}
\BB P\left[ t - v_Z(\ol\tau_\ep(t)) > T \,|\,   (Z-Z_t)|_{[t,\infty)} \right] \preceq    |Z_t - Z_{\ol\tau_\ep(t)}|^{\kappa'/4} T^{-\kappa'/8} . 
\eqe
 
To estimate the expectation of~\eqref{eqn-cone-long-prob}, we introduce a truncation. Let
\eqb 
G_\ep(t) := \left\{ |Z_{t+s} - Z_t| \leq s^{1/2 }  (\log \ep^{-1})^{4/\kappa'}   , \quad \forall s \in [t+\ep , T] \right\} \cap \left\{ \ol\tau_\ep(t) - t \leq T \right\} .
\eqe 
Then $G_\ep(t) \in \sigma\left( (Z-Z_t)|_{[t,\infty)} \right)$ and, by the Gaussian tail bound together with~\eqref{eqn-overshoot-prob} (recall that $T\leq \ep^{-100}$),
\eqb \label{eqn-cone-long-event}
\BB P\left[ G_\ep(t)^c \right] 
\preceq \ep^{1-\kappa'/8} T^{-(1-\kappa'/8)} + o_\ep^\infty(\ep) 
\preceq \ep^{1-\kappa'/8} T^{-(1-\kappa'/8)}  .
\eqe
If $G_\ep(t)$ occurs, then the right side of~\eqref{eqn-cone-long-prob} is at most 
\eqb  \label{eqn-cone-long-trunc}
T^{-\kappa'/8}  ( \log\ep^{-1} ) (\ol\tau_\ep(t) - t)^{ \kappa'/8} \BB 1_{(\ol\tau_\ep(t) - t \leq T     )}  .
\eqe 
We now take the unconditional expectation of~\eqref{eqn-cone-long-trunc} (which is an upper bound for the right side of~\eqref{eqn-cone-long-prob}), then apply~\eqref{eqn-overshoot-prob}. This gives 
\allb \label{eqn-no-cone-left}
\BB P\left[ v_Z(\ol\tau_\ep(t)) - t > T ,\, G_\ep(t) \right] 
&\preceq T^{-\kappa'/8}  ( \log\ep^{-1} ) \BB E\left[ (\ol\tau_\ep(t) - t)^{ \kappa'/8}  \BB 1_{(\ol\tau_\ep(t) - t \leq T     )}  \right] \notag\\
&\leq T^{-\kappa'/8} (\log \ep^{-1} ) \int_\ep^T \BB P\left[ (\ol\tau_\ep(t) - t)^{ \kappa'/8}  > x \right] \,dx \notag\\
&\preceq T^{-\kappa'/8} (\log \ep^{-1} ) \ep^{1-\kappa'/8}    \int_\ep^T  x^{-(8/\kappa' - 1)} \,dx \notag\\
&\preceq T^{2-8/\kappa' -\kappa'/8} (\log \ep^{-1} ) \ep^{1-\kappa'/8} . 
\alle 
We now obtain~\eqref{eqn-no-cone} by combining the estimates~\eqref{eqn-overshoot-prob},~\eqref{eqn-cone-long-event}, and~\eqref{eqn-no-cone-left}.
\end{proof}

Using Lemma~\ref{lem-no-cone}, we get the following variant of Proposition~\ref{prop-lqg-dim-S} where we restrict to times $a,b$ in a large finite interval.

\begin{lem} \label{lem-lqg-dim-sup}
Fix $u \in (0,1-\kappa'/8)$. There exists $\wt\beta = \wt\beta(u,\kappa') \in (0,1)$ such that
\eqb \label{eqn-lqg-dim-sup}
\BB E\left[\sup_{\substack{a,b \in [-\ep^{-u} , \ep^{-u} ] \\ a < b}} N_\ep([a,b] \cap (\eta')^{-1}(\eta_{a,b}) )   \right] \preceq \ep^{-\wt\beta} ,
\eqe
with the implicit constant depending only on $u$.
\end{lem}
\begin{proof}
Let $\alpha \in (0,1)$, to be chosen later. With $\ol\tau_\ep(t)$ as in Lemma~\ref{lem-no-cone}, let
\eqb
\mcl T^\alpha_\ep := \left\{ t\in [-\ep^{-u} , \ep^{-u}] \cap (\ep\BB Z) : \max\left\{ \ol\tau_\ep(t) - t , t -   v_Z(\ol\tau_\ep(t)) \right\} > \ep^\alpha \right\} .
\eqe
To lighten notation, also set $\theta:= \max\{1-\kappa'/8, \kappa'/8 + 8/\kappa'- 2\}$. 
By Lemma~\ref{lem-no-cone} applied with $T =\ep^\alpha$, for each $t \in  [-\ep^{-u} , \ep^{-u}] \cap (\ep\BB Z)$, 
\eqb
\BB P\left[  t \in \mcl T_\ep^\alpha \right] \preceq  \ep^{1-\kappa'/8 - \alpha\theta} \log\ep^{-1} .
\eqe
By a union bound over $O_\ep(\ep^{-1-u})$ times $t \in   [-\ep^{-u} , \ep^{-u}] \cap (\ep\BB Z)$, 
\eqbn \label{eqn-bad-time-count}
\BB E\left[ \# \mcl T_\ep^\alpha \right] \preceq \ep^{-u - \kappa'/8 - \alpha\theta} \log\ep^{-1}  .
\eqen 

Now consider some $a,b$ with $-\ep^{-u} \leq a  <b \leq \ep^{-u}$.
If $s \in [a,b] \cap (\eta')^{-1}(\eta_{a,b})$ then by definition there is no $\pi/2$-cone interval for $Z$ which contains $s$ and is contained in $[a,b]$. By the definition of $\ol\tau_\ep(t)$, it follows that if $t \in  [-\ep^{-u} , \ep^{-u}] \cap (\ep\BB Z)$ is chosen so that $s\in [t,t+\ep]$, then either $\ol\tau_\ep(t)  > b$ or $   v_Z(\ol\tau_\ep(t))  < a$. 
Therefore, $[a,b] \cap (\eta')^{-1}(\eta_{a,b})$ is contained in the union of $[a , a+\ep^\alpha]$, $[b-\ep^\alpha,b]$, and and the intervals $[t,t+\ep]$ for $t \in \mcl T_\ep^\alpha$.
Hence $N_\ep([a,b] \cap (\eta')^{-1}(\eta_{a,b}) ) \leq 2 \ep^{-(1-\alpha)} + \#\mcl T_\ep^\alpha$. 
By~\eqref{eqn-bad-time-count}, 
\eqb
\BB E\left[ \sup_{\substack{a,b \in [-\ep^{-u} , \ep^{-u} ] \\ a < b}} N_\ep([a,b] \cap (\eta')^{-1}(\eta_{a,b}) ) \right] \preceq  \ep^{-(1-\alpha)}  + \ep^{-u-\kappa'/8 - \alpha\theta} \log\ep^{-1} .
\eqe
We now choose $\alpha$ so that $-1 +\alpha = - u - \kappa'/8 - \alpha\theta$, i.e., $\alpha = (1-u-\kappa'/8)/(1+\theta)$. This gives~\eqref{eqn-lqg-dim-sup} for any $\wt\beta > 1 - (1-u-\kappa'/8) / (1+\theta)$, which lies in $(0,1)$ since $u \in (0,1-\kappa'/8)$ and $\theta > 0$. 
\end{proof}

To deduce Proposition~\ref{prop-lqg-dim-S} from Lemma~\ref{lem-lqg-dim-sup}, we will use some basic SLE/LQG estimates to argue that it is very unlikely for $\eta_{a,b}$ to intersect $S$ if $a,b\notin [-\ep^{-u} , \ep^{-u}]$. To do this we need some intermediate lemmas.

\begin{lem} \label{lem-sle-segment-contain}
Let $a \leq b < c \leq d$. Then $\eta_{a,d} \cap \eta'([b,c]) \subset \eta_{b,c}$.
\end{lem}
\begin{proof} 
We will prove that $\eta'([b,c]) \setminus \eta_{b,c} \subset \eta'([b,c]) \setminus \eta_{a,d}  $. Suppose $t \in [b,c]$ with $\eta'(t) \in \eta'([b,c]) \setminus \eta_{b,c}$. Then there is a $\pi/2$-cone time $s$ for $Z$ such that $t\in [v_Z(s) ,s] \subset [b,c]$. Since $[v_Z(s) , s]$ is also contained in $[a,d]$, we have $\eta'(t) \notin \eta_{a,d}$. 
\end{proof}

The following lemma will be used to give an a priori upper bound for $N_\ep(\eta_{a,b}\cap S)$, even when $a$ and $b$ are not necessarily contained in $[-\ep^{-u} , \ep^{-u}]$. 

\begin{lem} \label{lem-S-cover}
Let $S$ be a closed subset of $\BB D\setminus \{0\}$. 
For each $p \in (1,4/\gamma^2)$, 
\eqb \label{eqn-S-cover}
\BB E\left[ N_\ep((\eta')^{-1}(S) )^p \right] \preceq   \ep^{-p } 
\eqe
with the implicit constant depending only on $C, S,p,\gamma$.
\end{lem}
\begin{proof}
Let $U$ be an open set such that $S\subset U$ and $\ol U\subset \BB D\setminus \{0\}$. 
Let $\mcl R_\ep^{\op{in}}$ (resp.\ $\mcl R_\ep^{\op{out}}$) be the set of $t\in \ep\BB Z$ such that $\eta'([t,t+\ep])$ intersects $S$ and $\eta'([t,t+\ep])$ is (resp.\ is not) contained in $U$. Then $N_\ep((\eta')^{-1}(S) ) \leq \#\mcl R_\ep^{\op{in}} + \#\mcl R_\ep^{\op{out}}$. 

Since $\eta'$ is parameterized by $\mu_h$-mass, we have
\eqbn
\#\mcl R_\ep^{\op{in}} \leq \ep^{-1} \mu_h(U) .
\eqen
Recall that $\ol U \subset \BB D\setminus \{0\}$ and $h|_{\BB D}$ agrees in law with a whole-plane GFF plus $-\gamma\log|\cdot|$ plus a Gaussian random variable with constant order mean and variance. We can therefore apply a standard estimate for GMC measures (see, e.g.,~\cite[Theorem 2.11]{rhodes-vargas-review}) to get that $\mu_h(U)$ has finite moments up to order $4/\gamma^2$, so
\eqb \label{eqn-S-cover-in}
\BB E\left[ (\#\mcl R_\ep^{\op{in}})^p \right] \preceq \ep^{-p} .
\eqe

To bound $\#\mcl R_\ep^{\op{out}}$, let $d$ be the Euclidean distance from $S$ to $\bdy U$. By Lemma~\ref{lem-sle-ball}, it holds with probability $1-o_\delta^\infty(\delta)$ that every segment of $\eta$ which is contained in $\BB D$ and has Euclidean diameter at least $d$ contains a Euclidean ball of radius at least $\delta$. If this is the case, then each of the segments $\eta'([t,t+\ep])$ for $t\in \mcl R_\ep^{\op{out}}$ contains a Euclidean ball of radius at least $\delta$, and these Euclidean balls are disjoint. Therefore,  
\eqbn
\BB P\left[\pi \delta^2 \#\mcl R_\ep^{\op{out}} > \op{Area}(U) \right]  = o_\delta^\infty(\delta)  
\eqen
at a rate which depends only on $S,U$. 
Hence the law of $\#\mcl R_\ep^{\op{out}}$ has a superpolynomially small upper tail, so $\BB E\left[ (\#\mcl R_\ep^{\op{out}})^p \right] \preceq 1$. 
Combining this with~\eqref{eqn-S-cover-in} gives~\eqref{eqn-S-cover}.
\end{proof}

\begin{proof}[Proof of Proposition~\ref{prop-lqg-dim-S}]
Fix $u \in (0,1-\kappa'/8)$ be chosen in a manner depending only on $\kappa'$ (e.g., we could take $u = 1/2 - \kappa'/16$) and let 
\eqbn
F_\ep := \left\{S \subset \eta'([-\ep^{-u} , \ep^{-u}]) \right\} . 
\eqen
By a basic SLE / LQG estimate (see, e.g.,~\cite[Lemma A.4]{ghs-dist-exponent}), there exists $\beta_0 = \beta_0(u,\kappa')$ such that
\eqb \label{eqn-lqg-dim-S-event}
\BB P\left[ F_\ep^c \right] \preceq \ep^{\beta_0} .
\eqe
By Lemma~\ref{lem-sle-segment-contain}, if $F_\ep$ occurs, then for $a<b$,
\eqbn
\eta_{a,b} \cap S \subset \eta_{a,b} \cap \eta'([a\vee(-\ep^{-u}) , b \wedge \ep^{-u}]) \subset  \eta_{a\vee(-\ep^{-u}) , b \wedge \ep^{-u}}  .
\eqen
Hence, on $F_\ep$, 
\eqb \label{eqn-lqg-dim-S-compare}
[a,b] \cap (\eta')^{-1}(\eta_{a,b}\cap S) \subset [a',b'] \cap (\eta')^{-1}(\eta_{a',b'}\cap S)  
\eqe
where $a' =a\vee(-\ep^{-u}) $ and $b' =  b \wedge \ep^{-u} $.

Obviously, $N_\ep([a,b] \cap (\eta')^{-1}(\eta_{a,b}\cap S) ) \leq N_\ep((\eta')^{-1}(S) )$. From this and~\eqref{eqn-lqg-dim-S-compare}, we get
\eqb \label{eqn-lqg-dim-S-split}
\sup_{\substack{a,b\in\BB R \\ a<b} }  N_\ep\left([a,b] \cap (\eta')^{-1}(\eta_{a,b}\cap S) \right) 
\leq \sup_{\substack{a,b \in [-\ep^{-u} , \ep^{-u} ] \\ a < b}} N_\ep([a,b] \cap (\eta')^{-1}(\eta_{a,b}\cap S)) \BB 1_{F_\ep} + N_\ep((\eta')^{-1}(S)) \BB 1_{F_\ep^c} . 
\eqe
We now take the expectation of both sides of~\eqref{eqn-lqg-dim-S-split}. To bound the expectation of the first term on the right side, we use Lemma~\ref{lem-lqg-dim-sup} and to bound the second term on the right side we use H\"older's inequality followed by~\eqref{eqn-lqg-dim-S-event} and Lemma~\ref{lem-S-cover}.
This gives that for $p \in (1,4/\gamma^2)$, 
\eqbn
\BB E\left[\sup_{\substack{a,b\in\BB R \\ a<b} }   N_\ep\left([a,b] \cap (\eta')^{-1}(\eta_{a,b}\cap S)\right)   \right]
\preceq \ep^{-\wt\beta}  + \ep^{- 1 + (1-1/p)\beta_0} .
\eqen
Thus~\eqref{eqn-lqg-dim-S} holds with $\beta = \max\{\wt\beta ,  1 - (1-1/p)\beta_0\}$. 
\end{proof}

\subsection{Proof of Proposition~\ref{prop-diam-sum}}
\label{sec-diam-sum-proof}

The idea of the proof is to use Lemma~\ref{lem-sle-ball} to reduce the problem of estimating squared diameters of our space-filling SLE segments to the problem of estimating the area of a small neighborhood of $\eta_{a,b} \cap S$, which we can bound using Proposition~\ref{prop-sle-nbd-area}. We will need the following technical lemma.

\begin{lem} \label{lem-sle-diam}
Suppose we are in the setting of Proposition~\ref{prop-diam-sum}. 
For $n\in\BB N$ and $p > 2\gamma^2$, 
\eqb \label{eqn-sle-diam}
 \BB P\left[ \sup_{k\in [1,n]_{\BB Z}}  \op{diam}\left(\eta'([t_{n,k-1} , t_{n,k}]) ) \right) \BB 1_{G_{n,k}}  > n^{ -(2+\gamma^2/2 + p)^{-1} } \right] \leq n^{- \tfrac{p^2 -2\gamma^2}{\gamma^2(4+\gamma^2 + 2p)}  + o_n(1)} 
\eqe 
where the rate of convergence of the $o_\ep(1)$ depends only on $a,b,p$, and $\gamma$.
\end{lem}
\begin{proof}
Fix $u\in (0,1)$ (which we will eventually send to 0) and let
\eqbn
\ep := ((b-a) n^{-1} )^{(2+\gamma^2/2 + p)^{-1}}   .
\eqen 
Let $\mcl E_{\ep^{\frac{1}{1-u}}}$ be the event of Lemma~\ref{lem-sle-ball} with $\ep^{\frac{1}{1-u}}$ in place of $\ep$. 
 
By standard estimates for the LQG area measure (see, e.g.,~\cite[Lemma~A.1]{ghs-dist-exponent} or the proof of Lemma~\ref{lem-mu_h-lower}) and recalling the relationship between $h$ and a whole-plane GFF, for each $z\in B_{1-r}(0)\setminus B_r(0)$,
\eqbn
\BB P\left[ \mu_h(B_\ep(z)) < \ep^{2+\gamma^2/2 + p}  \right] \leq \ep^{\frac{p^2}{2\gamma^2} + o_\ep(1)} .
\eqen
If $\mcl E_{\ep^{\frac{1}{1-u}}}$ occurs, then each segment of $\eta'$ contained in $S$ with diameter at least $\ep$ contains a Euclidean ball of radius at least $\ep^{\frac{1}{1-u}}$. 
By the union bound, the conditional probability given $\eta'$ that any such segment has quantum mass smaller than $\ep^{2+\gamma^2/2 + p}$ is at most $\ep^{\frac{p^2}{2\gamma^2} -1 + o_u(1) +  o_\ep(1)}$. Recalling the definition of $\ep$, we see that the probability that there is a segment of $\eta'$ contained in $S$ with quantum mass at most $(b-a) n^{-1}$ and diameter at least $n^{-(2+\gamma^2/2 + p)^{-1} }$ is at most 
\eqbn
n^{-(2+\gamma^2/2 + p)^{-1} \left( \frac{p^2}{2\gamma^2} -1 \right) + o_u(1) + o_n(1)} .
\eqen
Sending $u\rta 0$ now yields the statement of the lemma. 
\end{proof}

\begin{proof}[Proof of Proposition~\ref{prop-diam-sum}] 
We start by introducing some regularity events which occur with high probability.
Fix $p > 2\gamma^2$ (to be chosen later) and for $n\in\BB N$, let
\eqbn
\ep_n := n^{ -(2+\gamma^2/2 + p)^{-1} } \quad \op{and} \quad
\mcl G_n := \left\{  \sup_{k\in [1,n]_{\BB Z}} \op{diam}\left(\eta'([t_{n,k-1} , t_{n,k}])) \right)  \BB 1_{G_{n,k}}  \leq \ep_n    \right\} .
\eqen
By Lemma~\ref{lem-sle-diam},
\eqb  \label{eqn-sup-diam-prob}
\BB P[\mcl G_n^c] \leq n^{-\tfrac{p^2 -2\gamma^2}{\gamma^2(4+\gamma^2 + 2p)} + o_n(1) } .
\eqe
 
Let $u \in (0,1/2)$ be a small constant. Let $\mcl E_{\ep_n^u}$ be as in Lemma~\ref{lem-sle-ball} with this choice of $u$ and with $\ep = \ep_n^u$. 
By Lemma~\ref{lem-sle-ball}, $\BB P[\mcl E_{\ep_n^u }^c ] = o_n^\infty(n)$. 

If $  \mcl G_n$ occurs then each of the segments $\eta'([t_{n,k-1} , t_{n,k}]))$ for $k\in [1,n]_{\BB Z}$ such that $G_{n,k}$ occurs has diameter at most $\ep_n  < \ep_n^u$.  
Hence, if also $\mcl E_{\ep_n^u}$ occurs, then each such segment contains a Euclidean ball of radius at least $\op{diam}\left(\eta'([t_{n,k-1} , t_{n,k}])\right)^{\frac{1}{1-u}}$. 
Therefore,
\begin{align} \label{eqn-diam-area}
\op{diam}\left(\eta'([t_{n,k-1} , t_{n,k}])\right)^2 \BB 1_{G_{n,k}}  
\preceq   \op{Area}\left(\eta'([t_{n,k-1} , t_{n,k}])\right)^{ 1-u }  \BB 1_{G_{n,k}}    
\end{align}
with universal implicit constant.  
 
On the event $\mcl G_n$, each of the segments $\eta'([t_{n,k-1} , t_{n,k}]))$ for $k\in[1,n]_{\BB Z}$ such that $  G_{n,k}$ occurs is contained in the $\ep_n$-neighborhood of $\eta_{a,b} \cap S $. 
By combining this with~\eqref{eqn-diam-area}, applying H\"older's inequality, and noting that the space-filling SLE segments $\eta'([t_{n,k-1} , t_{n,k}])$ intersect only along their boundaries, we find that on $  \mcl G_n \cap \mcl E_{\ep_n^u}$, 
\alb
\sum_{k=1}^n \op{diam}\left(\eta'([t_{n,k-1} , t_{n,k}])\right)^2 \BB 1_{G_{n,k}}  
&\preceq  \sum_{k=1}^n \op{Area}\left(\eta'([t_{n,k-1} , t_{n,k}])\right)^{1-u} \BB 1_{G_{n,k}}  \\
&\preceq  n^u \left( \sum_{k=1}^n \op{Area}\left(\eta'([t_{n,k-1} , t_{n,k}])\right) \BB 1_{G_{n,k}} \right)^{1-u}   \\
&\preceq  n^u   \op{Area}\left( B_{\ep_n } \left(\eta_{a,b} \cap S  \right)  \right)^{1-u}   .
\ale
It is clear that the sum on the left is always at most $4 n$, so by~\eqref{eqn-sup-diam-prob}, 
\eqbn
\BB E\left[ \sum_{k=1}^n \op{diam}\left(\eta'([t_{n,k-1} , t_{n,k}])\right)^2 \BB 1_{G_{n,k}} \right] 
\preceq n^u \BB E\left[ \op{Area}\left( B_{ \ep_n } \left(\eta_{a,b} \cap S \right) \right)     \right]^{1-u}  +   n^{1 -\frac{p^2 -2\gamma^2}{\gamma^2(4+\gamma^2 + 2p)} + o_n(1) }     ,
\eqen
where here we have used Jensen's inequality to bring the $1-u$ outside of the expectation.
By Proposition~\ref{prop-sle-nbd-area}, the right side of this inequality is at most
\eqb \label{eqn-diam-sum-conclude}
  n^{   -  (1-u) (2+\gamma^2/2 + p)^{-1} \Delta +   u +   o_n(1)   }     +   n^{1 -\frac{p^2 -2\gamma^2}{\gamma^2(4+\gamma^2 + 2p)} + o_n(1) }    
\eqe 
where $\Delta$ is as in Proposition~\ref{prop-sle-nbd-area}. 
We now choose $p > 2\gamma^2$ to be sufficiently large that $\frac{p^2 -2\gamma^2}{\gamma^2(4+\gamma^2 + 2p)} > 1$, which makes it so that both powers of $n$ on the right side of ~\eqref{eqn-diam-sum-conclude} are positive. Making $u$ sufficiently small now yields the statement of the proposition for an appropriate choice of $\alpha$. 
\end{proof}

\subsection{Stability event}
\label{sec-stability}

Recall the definition of a $C$-smooth canonical description of a $\gamma$-quantum cone from the beginning of this section. 
In this subsection we use this embedding to define a regularity event on which the diameters of space-filling SLE increments satisfy a certain bound. 
Suppose $\mcl C = (\BB C , h , 0, \infty)$ is a $\gamma$-quantum cone and $\eta'$ is a whole-plane space-filling $\SLE_{\kappa'}$ independent from $h$ and parameterized by $\gamma$-quantum mass with respect to $h$ so that $\eta'(0) = 0$. 

Following~\cite[Section~9.4.2]{wedges}, for $C  , t_0 \in \BB R$, a closed simply connected set $S\subset \BB D\setminus \{0\}$ with non-empty interior, $r > 0$, and $a,b \in \BB R$ with $a<b$, we let $\mcl A_{C,t_0}(S,r,a,b)$ be the event that the following is true. Suppose that $K\subset \BB C \setminus \{0\}$ be a hull (i.e., $K$ is compact and the complement of $K$ in the Riemann sphere is simply connected) and suppose there exists a conformal map $f$ from the unbounded connected component of $\BB C \setminus \eta'([a,b])$ to $\BB C\setminus K$ such that the following holds. 
\begin{enumerate}
\item $f(0) =0$, $f(\infty)=\infty$, and $\lim_{z\rta\infty} f(z)/z > 0$. \label{item-stability-map}
\item Suppose $h'$ is any distribution on $\BB C$ which agrees with the translated LQG pushforward field $h( f^{-1}(\cdot) + \eta'(t_0))  + Q\log|(f^{-1})'|$ on $\BB C\setminus K$ and whose circle average $h'_r(0)$ over $\bdy B_r(0)$ is well-defined for $r \geq 1$. For every such distribution $h'$, if we define $V_t'$ for $t\leq 0$ as in~\eqref{eqn-smooth-bm} with $h'$ in place of $h$, then~\eqref{eqn-smooth-def} holds (for our given choice of $C$) with $V'$ in place of $V$.   \label{item-stability-field}
\end{enumerate} 
Then, for every hull $K$ as above, we have $K\subset S$ and $\op{diam}(K) \geq r$. See Figure~\ref{fig-stability} for an illustration. 

\begin{figure}[ht!]
\begin{center}
\includegraphics[scale=.75]{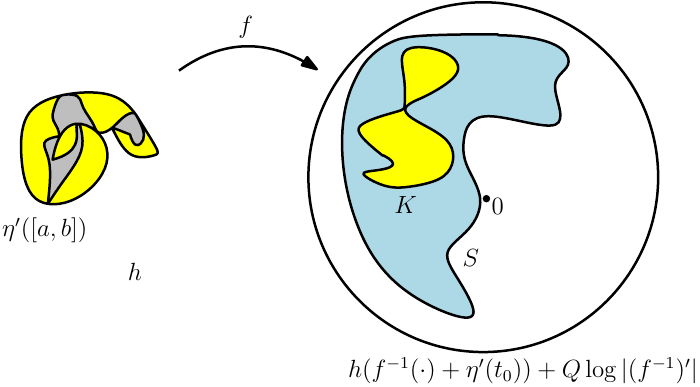}
\end{center}
\caption[Illustration of the stability event]{
Illustration of the definition of the event $\mcl A_{C,t_0}(S,r,a,b)$ that the following is true: if we conformally map the unbounded connected component to $\BB C\setminus \eta'([a,b])$ to $\BB C\setminus K$ via a conformal map $f$ with $f(0) =0$, $f(\infty)=\infty$, and $\lim_{z\rta\infty} f(z)/z > 0$ and which pushes forward $h$ to another field with the same normalization, then $K\subset S$ and $\op{diam}(K)\geq r$. Our interest in this event is primarily due to Lemma~\ref{lem-stable-coeff}.}\label{fig-stability}
\end{figure} 

The event $\mcl A_{0,0}(S,r,a,b)$ is the one considered in~\cite[Section~9.4.2]{wedges}. The event $\mcl A_{C,t_0}(S,r,a,b)$ is defined in the same way but with the $C$-smooth canonical description in place of the $0$-smooth canonical description and the addition of the translation by $\eta'(t_0)$ in the LQG coordinate change formula. 
As explained in~\cite[Remark~9.16]{wedges}, the event $\mcl A_{C,t_0}(S,r,a,b)$ is measurable with respect to the $\sigma$-algebra $\mcl F_{a,b}$ generated by the field $h\circ F_{a,b}^{-1} + Q\log |(F_{a,b}^{-1})'|$, where $F_{a,b} : \BB C\setminus \ol{\BB D} \rta \BB C\setminus \eta'([a,b])$ is the unique conformal transformation with $\lim_{z\rta\infty} z(F_{a,b}(z) - z) = 0$. In particular, $\mcl A_{C,t_0}(S,r,a,b)$ depends only on the curve-decorated quantum surface $(\mcl C ,\eta'_{\mcl C})$, not its particular embedding into $\BB C$. 

The reason why we include the parameters $C$ and $t_0$ is so that for fixed times $ a<b$, we can choose the other parameters in such a way that $\BB P[\mcl A_{C,t_0}(S,r,a,b)]$ is as close to 1 as we like. In order to arrange that $\BB P[\mcl A_{0,0}(S,r,a,b)]$ is close to 1, we would have to adjust $a$ and $b$.

\begin{lem} \label{lem-stable-prob}
For each $\ep \in (0,1)$ and each $a,b\in\BB R$ with $ a<b$, there exists parameters $C$, $t_0$, $S$, and $r$ as above such that $\BB P\left[ \mcl A_{C,t_0}(S,r,a,b) \right] \geq 1-\ep$. 
\end{lem}
\begin{proof}
It is proven at the end of the proof of~\cite[Lemma~9.7]{wedges} (which appears at the end of~\cite[Section 9.4.2]{wedges}) that there exists $ a' < b'$, a closed simply connected set $S \subset \BB D\setminus \{0\}$ with non-empty interior, and $r>0$ such that 
\eqbn
\BB P\left[ \mcl A_{0,0}(S,r,a',b')  \right] \geq 1-\ep .
\eqen
We will deduce the statement of the lemma for fixed $a$ and $b$ by scaling and translation. 
Let $C  , t_0 \in \BB R$ be chosen so that
\eqbn
[a',b'] = \left[ e^{-\gamma C} a  + t_0 , e^{-\gamma C} b   +t_0 \right].
\eqen
Let $\wh h:= h(\cdot+ \eta'(t_0) ) + C$ and $\wh\eta' :=   \eta'( e^{-\gamma C} \cdot + t_0) - \eta'(t_0)$. By the translation invariance of the law of the pair $(h,\eta')$~\cite[Theorem~1.9]{wedges} and the scaling property of the law of the quantum cone~\cite[Proposition~4.13(i)]{wedges}, we infer that the surface-curve pairs $((\BB C , \wh h , 0 , \infty) , \wh\eta')$ and $((\BB C , h , 0,\infty) , \eta')$ agree in law.
Furthermore, $\wh\eta'([a,b]) = \eta'([a' , b'])$. 

To prove the statement of the lemma, it suffices to show that if $\mcl A_{0,0}(S,r,a',b')$ occurs, then $\mcl A_{C,t_0}(S,r,a,b)$ occurs with $\wh h$ in place of $h$.
Suppose $K$, $f$, and $\wh h'$ are a hull, conformal map, and distribution as in the definition of $\mcl A_{C,t_0}(S,r,a,b)$ with $\wh h$ in place of $h$.
Since $\wh h'$ satisfies the condition~\eqref{eqn-smooth-def} in the definition of the $C$-smooth canonical description and by the scaling property of the $\gamma$-LQG measure, we infer that the field $h' := \wh h'   - C $ satisfies the condition in the definition of the $0$-smooth canonical description (i.e., with $C=0$). 
Since the conditions on $K$ and $f$ in the definition of $\mcl A_{C,t_0}(S,r,a,b)$ do not depend on $C$ and $t_0$ and since 
\eqbn
h'|_{\BB C\setminus K} = \wh h(f^{-1}(\cdot) +\eta'(t_0))  + Q\log|(f^{-1})'|  -C  = h \circ f^{-1}  + Q\log|(f^{-1})'|   
\eqen 
we see that the triple $(K,f,h')$ satisfies the conditions in the definition of $\mcl A_{0,0}(S,r,a',b')$. 
Hence if $\mcl A_{0,0}(S,r,a',b')$ occurs, then $K\subset S$ and $\op{diam}(K) \geq r$, so $\mcl A_{C,t_0}(S,r,a,b)$ occurs.   
\end{proof}

The main reason for our interest in the events $\mcl A_{C,t_0}(S,r,a,b)$ comes from the following elementary lemma, which is an easy consequence of the complex analysis facts described in~\cite[Section~9.3]{wedges}.   

\begin{lem} \label{lem-stable-coeff}
Suppose $\mcl A_{C,t_0}(S,r,a,b)$ occurs and let $K$ and $f$ be a hull and a conformal map as in the definition of $\mcl A_{C,t_0}(S,r,a,b)$. 
Expand $f$ as a Laurent series at $\infty$, 
\eqbn
f(z) = \alpha_{-1} z + \alpha_0   + \sum_{j=1}^\infty \alpha_j z^{-j} 
\eqen
for $\{\alpha_j\}_{j\geq -1}$ complex coefficients. Note that in fact $\alpha_{-1}$ is positive and real by our choice of $f$. 
Then $ r/4 \leq \alpha_{-1} \leq 4/r $ and $|\alpha_0 | \leq  4/r + 1$.
\end{lem}

Lemma~\ref{lem-stable-coeff} will allow us to control how much Euclidean diameters of certain sets are distorted when we replace the quantum surface parameterized by $ \eta'([a,b]) $ by another quantum surface with the same area and boundary length and then embed the new quantum surface thus obtained into $\BB C$ via the $C$-smooth canonical description. This is what will enable us to apply Proposition~\ref{prop-diam-sum} to bound the distortions of the conformal maps in the curve-swapping argument of Section~\ref{sec-swapping}. 

\begin{proof}[Proof of Lemma~\ref{lem-stable-coeff}]
Since the statement of the lemma does not depend on the particular choice of embedding $h$, we can assume without loss of generality that $h$ is normalized so that $h(\cdot + \eta'(t_0))$ is a $C$-smooth centering description (so that condition~\ref{item-stability-field} in the definition of $\mcl A_{C,t_0}(S,r,a,b)$ holds with $f$ equal to the identity map and $K$ equal to the hull generated by $\eta'([a,b])$). 

Let $g$ be the unique conformal map from the $\BB C\setminus \ol{\BB D}$ to the unbounded connected component of $\BB C\setminus \eta'([a,b])$ whose Laurent expansion at $\infty$ is given by
\eqbn
g (z) = \beta_{  -1} z + \beta_{ 0}   + \sum_{j=1}^\infty \beta_{ j} z^{-j} 
\eqen
with $\beta_{ -1}$ positive and real. 
Also let $\wt g : \BB C\setminus \ol{\BB D} \rta \BB C\setminus K$ be the unique conformal map from the $\BB C\setminus \ol{\BB D}$ to the unbounded connected component of $\BB C\setminus K$ whose Laurent expansion at $\infty$ is given by
\eqbn
\wt g(z) = \wt \beta_{ -1} z + \wt\beta_{ 0}  + \sum_{j=1}^\infty \wt\beta_{j} z^{-j}  .
\eqen
Then $f = \wt g \circ g^{-1}$. 

As $z\rta\infty$, we have $g (z) = \beta_{ -1} z  + \beta_{ 0} + O_z(1/z)$
and pre-composing with $g^{-1}$ gives $g^{-1}(z) = \beta_{-1}^{-1} z  - \beta_{-1}^{-1}  \beta_{0}  + O_z(1/z)$ (here we use that $g^{-1}(z)/z$ tends to a finite constant as $z\rta\infty$). Similarly, $\wt g(z) = \wt\beta_{ -1} z + \wt\beta_{0} + O_z(1/z)$. Therefore, 
\eqbn
f(z) = \wt \beta_{ -1} \beta_{ -1}^{-1} z -  \wt \beta_{ -1} \beta_{ -1}^{-1}  \beta_{ 0}  +  \wt\beta_{0}  + O_z(1/z)
\eqen
whence 
\eqb \label{eqn-coeff-compare}
\alpha_{-1} = \wt\beta_{ -1} \beta_{-1}^{-1} \quad \op{and}\quad \alpha_0 = -  \wt\beta_{ -1} \beta_{ -1}^{-1}  \beta_{ 0}  +  \wt\beta_{ 0} .
\eqe 

Suppose now that $\mcl A_{C,t_0}(S,r,a,b)$ occurs.
By standard estimates for conformal maps (see in particular~\cite[Proposition~9.11]{wedges}), 
\eqbn
\frac{1}{4} \op{diam}(\eta'([a,b]) ) \leq \beta_{ -1} \leq R, 
\eqen
where $R$ is the radius of the smallest closed ball containing $\eta'([a,b])$. By the definition of $\mcl A_{C,t_0}(S,r,a,b)$ (applied with $f$ equal to the identity), on this event we have $\op{diam}(\eta'([a,b]) ) \geq r$ and $R \leq 1$. Therefore, $\frac14 r \leq \beta_{-1} \leq 1$. Similarly, $\frac14 r \leq \wt\beta_{ -1} \leq 1$. 

As explained in~\cite[Section~9.3]{wedges}, $\beta_{ 0}$ is equal to $\BB E[w]$, where $w$ is sampled according to harmonic measure from $\infty$ on the outer boundary of $\eta'([a,b])$ (normalized to be a probability measure). By our choice of embedding of $h$ and the definition of $\mcl A_{C,t_0}(S,r,a,b)$, we have $\eta'([a,b]) \subset S \subset \BB D$, so $|\beta_{ 0}| \leq 1$. Similarly $|\wt\beta_{ 0}| \leq 1$. 

The statement of the lemma follows by combining the above estimates with~\eqref{eqn-coeff-compare}. 
\end{proof}

\section{The curve-swapping argument}
\label{sec-swapping}

In this section we will prove Theorem~\ref{thm-wpsf-char}.
Throughout, we continue to use the notation introduced in Section~\ref{sec-surface-def}. 

Fix $a,b\in\BB R$ with $a< b$. Recall from Section~\ref{sec-surface-def} the surface $\mcl S_{a,b}$ (resp.\ $\wt{\mcl S}_{a,b}$) obtained by restricting $h$ (resp.\ $\wt h$) to $\eta'([a,b])$ (resp.\ $\wt\eta'([a,b])$) and the non-space-filling curve $\eta_{a,b,\mcl S_{a,b}}$ (resp.\ $\wt\eta_{a,b,\wt{\mcl S}_{a,b}}$) on this surface.  
The main input in the proof of Theorem~\ref{thm-wpsf-char} is the following proposition, which most of this section will be devoted to proving.

\begin{prop} \label{prop-inc-agree}
One has $  (\wt{\mcl S}_{a,b} , \wt\eta_{a,b,\wt{\mcl S}_{a,b}}) \eqD (\mcl S_{a,b} , \eta_{a,b,\mcl S_{a,b} })$. 
\end{prop}

Throughout the proof of Proposition~\ref{prop-inc-agree}, we fix $a < b$ and do not always make dependencies on $a$ and $b$ explicit.
To prove the proposition, we will start in Section~\ref{sec-surface-inc-def} by defining for each $n\in\BB N$ and each $k\in[0,n]_{\BB Z}$ a $\gamma$-quantum cone $\mathring{\mcl C}_{n,k}$ decorated by a non-space-filling curve $\rng\eta_{n,k}$ and a space-filling curve $\mathring{\eta}'_{n,k}$ via a deterministic procedure involving the pairs $(\wt h , \wt\eta')$ and $(h,\eta')$. These curve-decorated quantum surfaces give rise to an interpolation $(\mathring{\mcl C}_{n,k}, \rng\eta_{n,k} , \mathring{\eta}'_{n,k} )$, $k \in [0,n]_{\BB Z}$, between the laws of $ (\wt{\mcl S}_{a,b} , \wt\eta_{a,b,\wt{\mcl S}_{a,b}})$ and $(\mcl S_{a,b} , \eta_{a,b,\mcl S_{a,b} }) $ in the following manner. The sub-surface of $\mathring{\mcl C}_{n,0}$ (resp.\ $\rng{\mcl C}_{n,n}$) parameterized by $\rng\eta'_{n,0}([a,b])$ (resp.\ $\rng{\eta}'_{n,n}([a,b])$), decorated by the curve $\rng\eta_{n,0}$ (resp.\ $\rng\eta_{n,n}$) is exactly $(\wt{\mcl S}_{a,b} , \wt\eta_{a,b,\wt{\mcl S}_{a,b}})$ (resp.\ has the same law as $(\mcl S_{a,b} , \eta_{a,b,\mcl S_{a,b} })$). 
Furthermore, the laws of the first and last triples $( \mathring{\mcl C}_{n,0}, \rng\eta_{n,0} , \mathring{\eta}'_{n,0} )$ and $( \mathring{\mcl C}_{n,n}, \rng\eta_{n,n} , \mathring{\eta}'_{n,n } )$ do not depend on $n$. For $n\in\BB N$ and $k \in [0,n]_{\BB Z}$ let
\eqb \label{eqn-inc-def}
t_{n,k} := a + \frac{k}{n} (b-a). 
\eqe 

Roughly speaking, each surface $\mathring{\mcl C}_{n,k-1}$ is obtained from $\mathring{\mcl C}_{n,k}$ by replacing the sub-surface of $\mathring{\mcl C}_{n,k}$ parameterized by $\mathring \eta_{n,k}([ t_{n,k-1},t_{n,k} ])$ by the sub-surface of our given $\gamma$-quantum cone $\wt{\mcl C}$ parameterized by $\wt\eta'_{\wt{\mcl C}}([ t_{n,k-1},t_{n,k}])$. See Figure~\ref{fig-swapping-maps} for an illustration of the surfaces $\rng{\mcl C}_{n,k}$ (plus some additional notation).

As explained in Section~\ref{sec-surface-inc-embed}, the triples $( \mathring{\mcl C}_{n,k}, \rng\eta_{n,k} , \mathring{\eta}'_{n,k} )$ will be defined in such a way that they all have the same topology. More precisely, if we embed these curve-decorated quantum surfaces into $\BB C$ and identify the curves with their images under these embeddings, then there exists for each $n\in\BB N$ and $k_1,k_2\in [0,n]_{\BB Z}$ a homeomorphism $f_{n,k_1,k_2} : \BB C\rta\BB C$ which satisfies $f_{n,k_1,k_2} \circ \rng\eta_{n,k_1} = \rng\eta_{n,k_2}$ and $f_{n,k_1,k_2} \circ \rng\eta_{n,k_1}' = \rng\eta_{n,k_2}'$. Furthermore, this homeomorphism can be taken to be conformal everywhere except on a small segment of the curve $\rng\eta_{n,k_1}$. 

The curves $\rng\eta_{n,k }$ are not known to be conformally removable (they are non-space filling $\SLE_{\kappa'}$-type curves) so we cannot conclude from the above that the maps $f_{n,k_1,k_2}$ are conformal. However, in Section~\ref{sec-surface-inc-distortion} we will use Proposition~\ref{prop-diam-sum} and an elementary distortion estimate for conformal maps (namely,~\cite[Lemma~9.6]{wedges}), to show that if we choose our embeddings in a certain way, then the maps $f_{n,k , k-1}$ are likely to be close to the identity when $n$ is large, in the sense that the sum over all $k$ of the distortion of these maps is small. Composing these maps and taking a limit as $n\rta\infty$ will show that the map $f_{n,n,0}$ (which does not depend on $n$) is the identity map, which will prove Proposition~\ref{prop-inc-agree}. 

In Section~\ref{sec-wpsf-char-proof}, we will use Proposition~\ref{prop-inc-agree} to conclude the proof of Theorem~\ref{thm-wpsf-char}. The basic idea is that the proposition statement applied with $[a,b] = [\ep(j-1),\ep j]$ for $j\in\BB Z$ together with the conformal removability of the boundaries of the segments $\eta'([\ep(j-1),\ep j])$ implies that the sets of times $\{\eta'(\ep j)\}_{j\in\BB Z}$ and $\{\wt\eta'(\ep j)\}_{j\in\BB Z}$ agree in law for each $\ep > 0$, so we can send $\ep\rta 0$ and use the continuity of the curves $\eta'$ and $\wt\eta'$ to conclude.

\subsection{Defining the intermediate curves and surfaces}
\label{sec-surface-inc-def}

In this subsection we will define the curve-decorated quantum surfaces $\left(  \mathring{\mcl C}_{n,k} , \rng\eta_{n,k} ,  \mathring{\eta}'_{n,k} \right)$ for $n\in\BB N$ and $k\in [0,n]_{\BB Z}$ in the outline just after the statement of Proposition~\ref{prop-inc-agree}. 
This will be accomplished using the objects introduced in Section~\ref{sec-surface-def} together with the results of Section~\ref{sec-partial-surface}.

For $t\in [a,b]$ define the time $\tau (t) = \tau_{a,b}(t) \in [a,b]$ as in~\eqref{eqn-tau-def}.
Also recall the surfaces $\mcl S_{a,b}^0 \subset \mcl S_{a,b}$ and $\wt{\mcl S}_{a,b}^0 \subset \wt{\mcl S}_{a,b}$ parameterized by the bubbles cut out by $\eta_{a,b}$ and $\wt\eta_{a,b}$, respectively. 
\medskip

\noindent\textit{Defining $\gamma$-quantum cones $\rng{\mcl C}_{n,k}$.}
By Lemma~\ref{lem-full-surface-msrble}, for each $k \in [0,n]_{\BB Z}$ there is a deterministic functional $F_{n,k}$ which takes in the $4$-tuple
\eqbn
\left( \mcl S_{-\infty,a} , \mcl S_{a,\tau(t_{n,k}) }^0 ,  \eta_{a,b , \mcl S_{a,\tau(t_{n,k})}^0} ,    \mcl S_{\tau(t_{n,k}) ,\infty}    \right) 
\eqen
and a.s.\ outputs the triple $( \mcl C , \mcl S_{a,\infty}, \eta_{a,b,\mcl C}|_{[a , t_{n,k} ]} )$. 
In particular, the functional $F_{n,k}$ imposes a conformal structure on all of $\mcl S_{a,\tau(t_{n,k})}^0$ (not just on its bubbles) to get $\mcl S_{a,\tau(t_{n,k})}$, conformally welds $\mcl S_{a,\tau(t_{n,k})}^0$ and $\mcl S_{\tau(t_{n,k}) ,\infty}$ together according to quantum length along their boundaries to get $\mcl S_{a,\infty}$, then conformally welds $\mcl S_{a,\infty}$ and $\mcl S_{-\infty,a}$ together along their boundaries to get $\mcl C$. The curve $\eta_{a,b,\mcl C}|_{[a , t_{n,k} ]} $ is the image of $ \eta_{a,b , \mcl S_{a,\tau(t_{n,k})}^0}|_{[a,t_{n,k}]}$ under the conformal welding map. 

Let $\rng{\mcl S}_{-\infty,a}$ be $\frac{3\gamma}{2}$-quantum wedge, independent from everything else. 
By Proposition~\ref{prop-partial-surface-law} and the independence of $\mcl S_{-\infty,a}$ and $(\mcl S_{a,\infty} , \eta'_{\mcl S_{a,\infty}})$, we have for each $k\in [1,n]_{\BB Z}$ the equality of joint laws
\eqb \label{eqn-use-partial-surface-law}
 \left( \rng{\mcl S}_{-\infty,a} , \wt{\mcl S}_{a,\tau(t_{n,k}) }^0 , \wt \eta_{a,b , \wt{\mcl S}_{a,\tau(t_{n,k})}^0} ,       \wt{\mcl S}_{\tau(t_{n,k}) ,\infty}   \right) 
 \eqD \left(  \mcl S_{-\infty,a} ,  \mcl S_{a,\tau(t_{n,k}) }^0 ,  \eta_{a,b , \mcl S_{a,\tau(t_{n,k}) }^0} ,    \mcl S_{\tau(t_{n,k}) ,\infty}    \right)  .
\eqe 
Hence we can apply $F_{n,k}$ to the left $4$-tuple in~\eqref{eqn-use-partial-surface-law} to define quantum surfaces each with a distinguished sub-surface and a distinguished curve
\eqb \label{eqn-surface-inc-function}
\left(\mathring{\mcl C}_{n,k} , \mathring{\mcl S}_{n,k} ,  \mathring \eta_{n,k}|_{[a , t_{n,k}]}  \right) := F_{n,k} \left( \rng{\mcl S}_{-\infty,a} , \wt{\mcl S}_{a,\tau(t_{n,k})}^0 , \wt \eta_{a,b , \wt{\mcl S}_{a,\tau(t_{n,k})}^0} ,       \wt{\mcl S}_{\tau(t_{n,k}) ,\infty}    \right) 
\eqe 
for $k\in [0,n]_{\BB Z}$ such that
\eqb \label{eqn-surface-inc-law}
\left(\mathring{\mcl C}_{n,k} , \mathring{\mcl S}_{n,k} , \mathring \eta_{n,k}|_{[a , t_{n,k}]} \right) \eqD \left( \mcl C , \mcl S_{a,\infty}, \eta_{a,b,\mcl C}|_{[a ,t_{n,k} ]} \right).
\eqe   
See Figure~\ref{fig-swapping-def} for an illustration of the above construction. 
\medskip
 
\begin{figure}[ht!]
\begin{center}
\includegraphics[scale=.8]{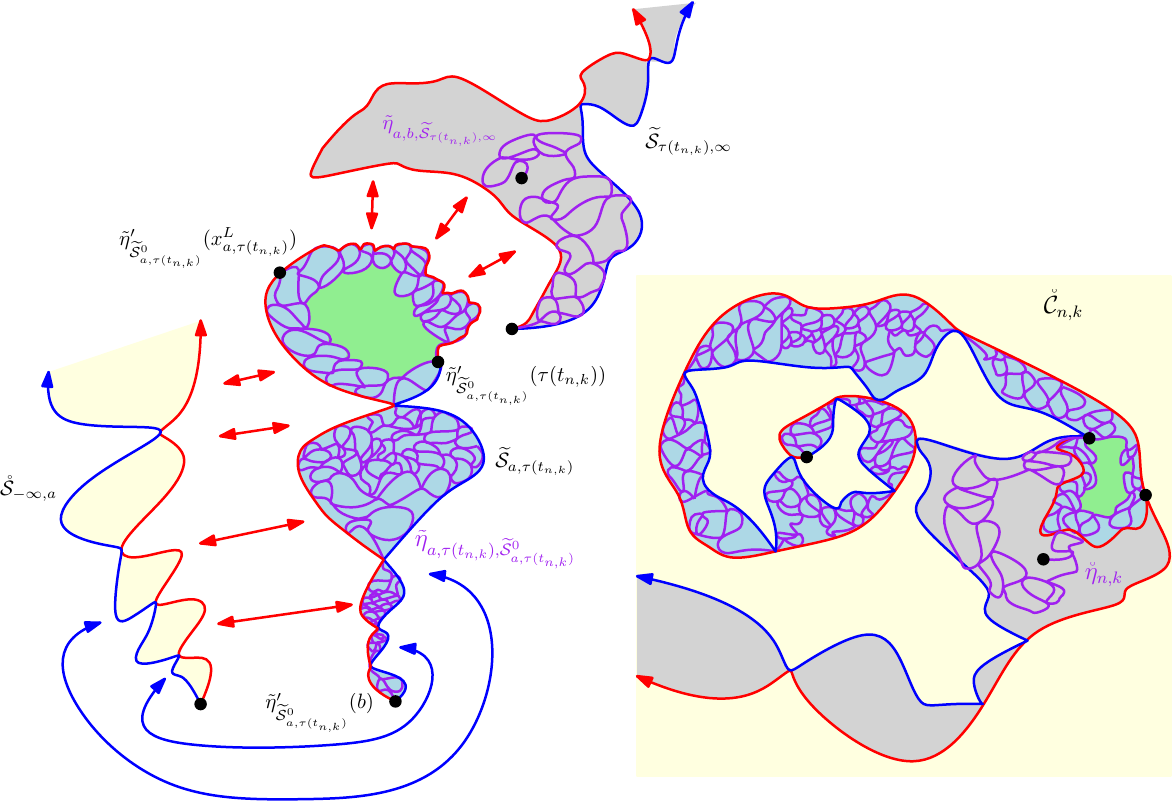}
\end{center}
\caption[Intermediate SLE-decorated LQG surfaces used in the proof of Theorem~\ref{thm-wpsf-char}]{
The $\gamma$-quantum cone $\rng{\mcl C}_{n,k}$ is determined by the surfaces $\rng{\mcl S}_{-\infty,a} $ (light yellow), $\wt{\mcl S}_{a,\tau(t_{n,k}) }^0$ (light blue), and $    \wt{\mcl S}_{\tau(t_{n,k}),\infty}$ (grey) in the same manner that the original $\gamma$-quantum cone $\mcl C $ is determined by $\mcl S_{-\infty,a} $, $\mcl S_{a,\tau(t_{n,k}) }^0$, and $    \mcl S_{\tau(t_{n,k}),\infty}$. In particular, to get $\rng{\mcl C}_{n,k}$ we impose a conformal structure on all of $\wt{\mcl S}_{a,\tau(t_{n,k})}^0$ (not just on its bubbles) then conformally weld the resulting surface together with $\rng{\mcl S}_{-\infty,a} $ and  $ \wt{\mcl S}_{\tau(t_{n,k}),\infty}$ according to quantum length along their boundaries.
The curve $\rng\eta_{n,k} : [a,b] \rta \rng{\mcl C}_{n,k}$ (purple) is the concatenation of the images of the curves $\wt{\mcl S}_{a,\tau(t_{n,k}) }^0 , \wt \eta_{a,b , \wt{\mcl S}_{a,\tau(t_{n,k})}^0}|_{[a,t_{n,k}]}$ and $\wt\eta_{a,b,\wt{\mcl S}_{\tau(t_{n,k}),\infty} }|_{[t_{n,k} , \infty)}$.
The space-filling curve $\rng\eta'_{n,k}$ (not shown) is obtained by filling in the bubbles cut out by $\rng\eta_{n,k}$ with conformal images of segments of $\wt\eta'$ then concatenating the resulting curve with the images of the space-filling curve segments $\rng\eta'|_{\rng{\mcl S}_{-\infty,a}}$ and $\wt\eta'|_{[t_{n,k},\infty)}$ under the conformal welding map. 
 }\label{fig-swapping-def} 
\end{figure}

\noindent\textit{Defining non-space-filling and space-filling curves $\rng\eta_{n,k}$ and $\rng\eta_{n,k}'$.}
We will now extend the $\SLE_{\kappa'}$-type curves $ \mathring \eta_{n,k}|_{[a,t_{n,k}]}$ to space-filling $\SLE_{\kappa'}$-type curves defined on all of $[0,\infty)$. 
We first extend $\mathring \eta_{n,k}$ to all of $[a,b]$ by concatenating it with the curve on $\rng{\mcl C}_{n,k}$ which corresponds to the image under the conformal welding map of the curve $\wt\eta_{a,b,\wt{\mcl S}_{\tau(t_{n,k}),\infty} }|_{[t_{n,k} , \infty)}$ on $\wt{\mcl S}_{\tau(t_{n,k}),\infty}$. 

To obtain a space-filling curve, we first fill in the bubbles cut off by $\rng \eta_{n,k}$ (thus extended) with conformal images of segments of $\wt\eta'$. 
More precisely, let $\mcl I$ be the set of intervals in $[a,b]$ of the form $[\sigma(t) , \tau(t)]$ for $t \in [a,b]$ with $\tau(t) \not=t$, equivalently the set of maximal $\pi/2$-cone intervals for $Z$ in $[a,b]$, or the set of intervals in $[a,b]$ on which $\wt\eta_{n,k}$ is constant. 
For $I = [\sigma(t) ,\tau(t)] \in {\mcl I}$, define the singly marked quantum surface $\wt{\mcl S}_I := \wt{\mcl S}_{\sigma(t) , \tau(t)}$ parameterized by $I$. 
Then the surfaces $\wt{\mcl S}_I$ for $I\in\mcl I$ are the same as the bubbles of $\wt{\mcl S}_{a,b}^0$.
Furthermore, each of the curves $\wt\eta'_{\wt{\mcl S}_I}$ is a space-filling loop which starts and ends at the marked point of $\wt{\mcl S}_{\sigma(t) , \tau(t)}$. 

By the definitions of $\mathring{\mcl C}_{n,k}$ and of $\mathring\eta_{n,k}$, each of the surfaces $\wt{\mcl S}_I$ for $I\in\mcl I$ is a sub-surface of $\rng{\mcl C}_{n,k}$ (in fact, a sub-surface of the future wedge $\rng{\mcl S}_{n,k}$) cut out by $\rng\eta_{n,k}$, the curve $\rng\eta_{n,k}$ is constant on $I$, and the marked point of $\wt{\mcl S}_I$ corresponds to the single point $\rng\eta_{n,k}(I)$. We let $\rng\eta'_{n,k}|_{[a,b]}$ be the curve on $\rng{\mcl C}_{n,k}$ which agrees with the image under the conformal welding map of $\wt\eta'_{\wt{\mcl S}_I}$ on each interval $I \in \mcl I$ and is equal to $\rng \eta_{n,k}(t)$ at each time $t$ in the nowhere dense set $[a,b]\setminus \bigcup_{I\in\mcl I} I$.  
Then $\mathring\eta'_{n,k}|_{[a,b]}$ is a continuous curve on $\mathring{\mcl C}_{n,k}$.
We extend $\mathring\eta'_{n,k}|_{[a,b]}$ to a continuous curve $\BB R \rta \rng{\mcl C}_{n,k}$ which fills all of $\rng{\mcl C}_{n,k}$ as follows. For $s >b$, we define $\mathring\eta'_{n,k}(s)$ to be the image of $\eta'_{\mcl S_{-\infty,a}}(s)$ under the conformal welding map. We also let $\rng\eta'_{\rng{\mcl S}_{-\infty,a}}$ be a curve on $\rng{\mcl S}_{-\infty,a}$ whose conditional law given everything else is the same as the conditional law of the space-filling SLE $\eta'_{\mcl S_{-\infty,a}}$ given $\mcl S_{-\infty,a}$ and declare the $\rng\eta'_{n,k}(s) $ is equal to the image of $ \rng\eta'_{\rng{\mcl S}_{-\infty,a}}(s)$ under the conformal welding map for $s < a$.
\medskip

\noindent\textit{Some observations.}
We record some observations about the above objects which are immediate from the construction.
\begin{itemize}
\item The $\frac{3\gamma}{2}$-quantum wedge $\rng{\mcl S}_{n,k} $ is the sub-surface of $\rng{\mcl C}_{n,k}$ parameterized by $\rng\eta'_{n,k}([a,\infty))$. 
\item The curve $\rng\eta_{n,k}$ can be recovered from $\rng\eta'_{n,k}$ in the same manner that $\eta_{a,b}$ is recovered from $\eta'$ (i.e., by cutting out the bubbles filled in by $\rng\eta'_{n,k}|_{[a,b]}$). 
\item Since each connected component of the interior of $\wt\eta'([t_{n,k} , \infty))$ is contained in either $\wt{\mcl S}_{\tau(t_{n,k}),\infty}$ or in one of the bubbles $\wt{\mcl S}_I$, we see that the sub-surface of $\rng{\mcl C}_{n,k}$ parameterized by $\rng\eta_{n,k}'([t_{n,k},\infty))$ is the same as $\wt{\mcl S}_{t_{n,k}, \infty}$. Hence $\wt{\mcl S}_{t_{n,k}, \infty}$ (not just $\wt{\mcl S}_{\tau(t_{n,k}) ,\infty} $) is a sub-surface of $\rng{\mcl C}_{n,k}$.
\item Neither $\left(\mathring{\mcl C}_{n,n} ,  \rng{\mcl S}_{n,n},  \mathring \eta_{n,n}' \right) $ nor $\left(\mathring{\mcl C}_{n,0} , \rng{\mcl S}_{n,0} , \mathring \eta_{n,0}' \right) $ depends on $n$. 
\end{itemize}

The triple $\left(\mathring{\mcl C}_{n,n} , \rng{\mcl S}_{n,k}, \mathring \eta_{n,n} \right) $ has the same law as $( \mcl C , \mcl S_{a,\infty} ,  \eta_{a,b,\mcl C}  )$ (i.e., a $\gamma$-quantum cone decorated by an independent chordal $\SLE_{\kappa'}$ segment). The surface $\mcl S_{a,b}$ is a.s.\ equal to the union of the curve $\eta_{a,b,\mcl C}$ and the set of points which it disconnects from its target point in $\mcl S_{a,\infty}$. Hence $(\mcl S_{a,b} , \eta_{a,b})$ has the same law as the sub-surface of $\rng{\mcl S}_{n,n}$ parameterized by the union of $\rng\eta_{n,n}$ and the set of bubbles cut out by this curve in $\rng{\mcl S}_{n,k}$, equivalently the sub-surface of $\rng{\mcl S}_{n,n}$ parameterized by $\rng\eta'_{n,n}([a,b])$, decorated by the curve $\rng\eta_{n,n}$. 

The surface $ \mathring{\mcl C}_{n,0} $ is obtained by conformally welding together the quantum surfaces~$\rng{\mcl S}_{-\infty,a}$ and~$\wt{\mcl S}_{a,\infty}$ and~$\mathring\eta'_{n,0}$ is the concatenation of the images of~$\rng\eta'_{\rng{\mcl S}_{-\infty,a}}$ and~$\wt\eta'_{\wt{\mcl S}_{a,\infty}}$ under the conformal welding map. In particular, the sub-surface of~$\rng{\mcl S}_{n,0}$ parameterized by~$\rng\eta'_{n,0}([a,b])$, decorated by the curve~$\rng\eta_{n,0}$, is the same as the surface~$\wt{\mcl S}_{a,b}$, decorated by the curve~$\wt \eta'_{\wt{\mcl S}_{a,b}}$.

Therefore, in order to prove Proposition~\ref{prop-inc-agree}, it suffices to show that the sub-surface of~$\rng{\mcl S}_{n,n}$ parameterized by~$\rng\eta'_{n,n}([a,b])$ and the sub-surface of~$\rng{\mcl S}_{n,0}$ parameterized by~$\rng\eta'_{n,0}([a,b])$ a.s.\ agree (as quantum surfaces). This statement will follow immediately from the following proposition, which we will prove in the next subsections.  
  
\begin{prop} \label{prop-embedding-equiv}
Almost surely, $(\rng{\mcl C}_{n,n} , \rng{\eta}'_{n,n}) = (\rng{\mcl C}_{n,0} , \rng{\eta}'_{n,0})$ as curve-decorated quantum surfaces.
\end{prop}

\subsection{Embedding the intermediate surfaces}
\label{sec-surface-inc-embed}

Suppose we are in the setting of Section~\ref{sec-surface-inc-def}. 
For $k \in [0,n]_{\BB Z}$, let~$\mathring h_{n,k}$ be an embedding of the $\gamma$-quantum cone~$\mathring{\mcl C}_{n,k}$ into~$\BB C $. The field~$\mathring h_{n,k}$ is only defined up to a complex affine transformation, but for now we work with an arbitrary choice of embedding (we will specify a particular choice later). By a slight abuse of notation, we identify the curves~$\eta_{n,k}$ and~$\mathring\eta'_{n,k}$ with their images under this embedding, so that $\mathring \eta_{n,k} : [a, b]\rta \BB C$ and $\mathring\eta'_{n,k} : \BB R \rta\BB C$.

\begin{lem} \label{lem-inc-homeo}
Fix any choice of embeddings $\mathring h_{n,k}$ of the quantum surfaces~$\mathring{\mcl C}_{n,k}$ as above.  For each $k_1,k_2 \in [0,n]_{\BB Z}$, there is a homeomorphism $f = f_{n,k_1  \srta    k_2} : \BB C\rta \BB C$ such that $f \circ \mathring\eta_{n,k_1}  =   \mathring\eta_{n,k_2}$ and $f \circ \mathring\eta'_{n,k_1} = \mathring\eta'_{n,k_2}$.  Furthermore, $f$ is conformal on $\BB C\setminus \mathring\eta_{n,k_1}([t_{n,k_1\wedge k_2} , t_{n,k_1\vee k_2} ])$ and satisfies
\eqb \label{eqn-inc-homeo-lqg}
\mathring h_{n,k_2}   =    \mathring h_{n,k_1} \circ f^{-1} + Q\log |(f^{-1})'|  
\quad \text{on} \quad \BB C\setminus \mathring\eta_{n,k_2}([t_{n,k_1\wedge k_2} , t_{n,k_1\vee k_2} ]) .
\eqe 
\end{lem}
\begin{proof}
By definition (recall the discussion just before~\eqref{eqn-use-partial-surface-law}), the surface $\rng{\mcl C}_{n,k}$ is obtained by extending the conformal structure on the bubbles of $\wt{\mcl S}_{a , \tau(t_{n,k}) }^0$ to a conformal structure on the closure of their union, then identifying the surfaces $\rng{\mcl S}_{-\infty,a}$, $\wt{\mcl S}_{a,\tau(t_{n,k})}^0$, and $\wt{\mcl S}_{\tau(t_{n,k}) , \infty}$ according to quantum length along their boundaries. 
By condition~\ref{item-wpsf-char-homeo} in Theorem~\ref{thm-wpsf-char}, the equivalence relation on the boundaries of these three surfaces used to produce $\rng{\mcl C}_{n,k}$ is the same as the one used to produce the original surface $\wt{\mcl C} = (\BB C , \wt h , 0, \infty)$. 

The curve $\rng\eta_{n,k}|_{[a,t_{n,k}]}$ cuts out the bubbles of the image of $\wt{\mcl S}_{a,\tau(t_{n,k})}^0$ under the conformal welding map in the same order that $\wt\eta_{a,b}|_{[a,t_{n,k}]}$ cuts out the bubbles of $\wt{\mcl S}_{a,\tau(t_{n,k})}^0$ and is parameterized by the quantum mass of the bubbles it cuts out. By definition, the curve $\rng\eta_{n,k}|_{[t_{n,k},\infty)}$ is the image of $\wt\eta_{a,b,\wt{\mcl S}_{\tau(t_{n,k}) , \infty)}}|_{[t_{n,k} , \infty)}$ under the conformal welding map. 
  
This holds for each $k\in [0,n]_{\BB Z}$, so for $k_1,k_2\in [0,n]_{\BB Z}$ there exists a homeomorphism $f : \BB C\rta \BB C$ which is conformal on $\BB C\setminus (\bdy \mathring\eta'_{n,k_1}([a,\infty)) \cup \mathring\eta_{n,k_1}([a,t_{n,k_1\vee k_2} ]))$ such that 
\eqb  \label{eqn-inc-homeo-lqg0}
\mathring h_{n,k_2}   =    \mathring h_{n,k_1} \circ f^{-1} + Q\log |(f^{-1})'|  \quad \text{on} \quad  \BB C\setminus    \left(\bdy \mathring\eta'_{n,k_2}([a,\infty)) \cup \mathring\eta_{n,k_2}([a, t_{n,k_1\vee k_2} ])    \right)   , 
\eqe  
$f(\mathring\eta'_{n,k_1}([a,\infty)) )   =  \mathring\eta'_{n,k_2}([a,\infty))$, and $f\circ \mathring\eta_{n,k_1} = \mathring\eta_{n,k_2}$. 
Since the curves $\rng\eta'_{n,k }$ were defined by interpolating the curves $\rng\eta_{n,k}$ with conformal images of segments of the same fixed curve, it follows that also $f\circ \rng\eta'_{n,k_1} = \rng\eta'_{n,k_2}$. 
 
Now we will show that $f$ is in fact conformal on $\BB C\setminus \mathring\eta_{n,k_1}([t_{n,k_1\wedge k_2} , t_{n,k_1\vee k_2} ])$ and check~\eqref{eqn-inc-homeo-lqg}.  
We first argue that $f$ extends conformally across $\bdy \mathring\eta'_{n,k_1}([a,\infty))$ using a conformal removability argument (recall Section~\ref{sec-welding}).
By~\eqref{eqn-surface-inc-law}, each of the pairs $(\mathring{\mcl C}_{n,k_1} , \mathring{\mcl S}_{n,k_1})$ agrees in law with $(\mcl C, \mcl S_{a,\infty})$. 
Since $\mathring{\mcl S}_{n,k_1}$ is the sub-surface of $\mathring{\mcl C}_{n,k_1}$ parameterized by $\rng\eta'_{n,k_1}([a,\infty))$, we find that $\bdy\eta'_{n,k_1}([a,\infty)) \eqD \bdy\eta'([a,\infty))$. By~\cite[Footnote 4]{wedges} and translation invariance of the law of $\eta'$~\cite[Theorem~1.9]{wedges}, the common law of these two boundaries is that of a pair of non-crossing whole-plane $\SLE_{\kappa}(2-\kappa)$ curves, for $\kappa =16/\kappa'$.
By~\cite[Proposition~3.16]{wedges}, $\bdy\eta'_{n,k_1}([a,\infty))$ is conformally removable. Hence $f$ is a.s.\ conformal on $\BB C\setminus \eta_{n,k_1}([a , t_{n,k_1 \vee k_2}])$.  

It remains to check that $f$ extends conformally across $\eta_{n,k_1}([a, t_{n,k_1\wedge k_2}])$. 
By~\eqref{eqn-surface-inc-function}, the sub-surfaces of $\rng{\mcl C}_{n,k_1}$ and $\rng{\mcl C}_{n,k_2}$, respectively, parameterized by the bubbles cut out by $\rng\eta_{n,k_1}|_{[a,\tau(t_{n,k_1})]}$ and $\rng\eta_{n,k_2}|_{[a,\tau(t_{n,k_2})]}$
(i.e., the analogs of $\mcl S_{a,\tau(t_{n,k_1\wedge k_2}) }^0$ for the pairs $(\rng{\mcl C}_{n,k_1} , \rng\eta_{n,k_1})$ and $(\rng{\mcl C}_{n,k_2} , \rng\eta_{n,k_2})$
equipped with the curves $\rng\eta_{n,k_1}|_{[a,\tau(t_{n,k_1})]}$ and $\rng\eta_{n,k_2}|_{[a,\tau(t_{n,k_2})]}$) are equivalent as curve-decorated quantum surfaces.
By~\eqref{eqn-surface-inc-law}, the quantum surfaces 
$(\mathring\eta_{n,k_1}'([a ,  t_{n,k_1\wedge k_2 } ]) , \mathring h_{n,k_1} |_{\mathring\eta_{n,k_1}'([a ,  t_{n,k_1\wedge k_2 } ])})$ and 
$(\mathring\eta_{n,k_2}'([a ,  t_{n,k_1\wedge k_2 } ]) , \mathring h_{n,k_2} |_{\mathring\eta_{n,k_2}'([a ,  t_{n,k_1\wedge k_2 } ])})$ 
each have the same law as $ \mcl S_{a,\tau(t_{n,k_1\wedge k_2})}$. By Lemma~\ref{lem-partial-surface-msrble}, these last two quantum surfaces are equivalent. 

Hence there exists a homeomorphism $\mathring\eta_{n,k_1}'([ a,  t_{n,k_1\wedge k_2} ]) \rta  \mathring\eta_{n,k_2}'([ a,  t_{n,k_1\wedge k_2} ]) $ which is conformal on the interior of $\mathring\eta_{n,k_1}'([ a,  t_{n,k_1\wedge k_2} ])$ and pushes forward $\mathring h_{n,k_1} |_{\mathring\eta_{n,k_1}'([ a,  t_{n,k_1\wedge k_2} ])}$ to $\mathring h_{n,k_2} |_{\mathring\eta_{n,k_2}'([ a,  t_{n,k_1\wedge k_2} ])}$ via the LQG coordinate change formula. By~\eqref{eqn-inc-homeo-lqg0}, this homeomorphism agrees with $f|_{\mathring\eta_{n,k_1}'([ a,  t_{n,k_1\wedge k_2} ]) }$. In particular $f$ is conformal on the interior of $\mathring\eta_{n,k_1}'([ a,  t_{n,k_1\wedge k_2} ])$, which contains $\rng\eta_{n,k_1}([ a,  t_{n,k_1\wedge k_2} ]) \setminus \bdy  \mathring\eta_{n,k_1}'([ a,  t_{n,k_1\wedge k_2} ])$.

By~\eqref{eqn-surface-inc-law}, the boundary of $\mathring\eta_{n,k_1}'([ a,  t_{n,k_1\wedge k_2} ])$ is a locally finite union of segments of $\SLE_\kappa$-type curves for $\kappa = 16/\kappa'$, so is conformally removable. 
Thus $f$ is conformal on $\BB C\setminus \mathring\eta_{n,k_1}'([   t_{n,k_1\wedge k_2} , t_{n,k_1\vee k_2} ])$ and~\eqref{eqn-inc-homeo-lqg} holds.
\end{proof}

Since the curves $\rng\eta'_{n,k_1}$ and $\rng\eta'_{n,k_2}$ are space-filling, the condition that $f_{n,k_1\srta k_2} \circ \mathring\eta_{n,k_1}'  =   \mathring\eta_{n,k_2}'$ uniquely determines the map $f_{n,k_1\srta k_2}$ from Lemma~\ref{lem-inc-homeo}. This same condition also allows us to deduce that the maps $f_{n,k_1\srta k_2}$ are compatible in the sense that 
\eqb \label{eqn-inc-homeo-compat}
 f_{n,k_1\srta k_3} = f_{n,k_2\srta k_3} \circ f_{n,k_2\srta k_1}^{-1} ,\quad \forall k_1,k_2 ,k_3 \in [0,n]_{\BB Z}  .
\eqe  

Recall that the embeddings $\mathring h_{n,k}$ are only defined up to a complex affine transformation. We now specify this transformation. 
For $k\in [0,n]_{\BB Z}$, let $\BB K_{n,k}$ be the set of points in $\BB C$ which are disconnected from $\infty$ by the image of the space-filling curve segment $\mathring\eta'_{n,k}([a,b])$ under the embedding $\mathring h_{n,k}$. 
 
Each of the maps $f_{n,k_1\srta k_2}$ in Lemma~\ref{lem-inc-homeo} restricts to a conformal map $\BB C\setminus \BB K_{n,k_1} \rta  \BB C \setminus  \BB K_{n,k_2}$
By Laurent expansion at $\infty$, we can write $f_{n,k_1\srta k_2} = \alpha_{n,k_1\srta k_2} z + \beta_{n,k_1\srta k_2} + O_z(1/z) $ as $z\rta\infty$ for complex coefficients $\alpha_{n,k_1\srta k_2} $ and $ \beta_{n,k_1\srta k_2}$. 
By possibly applying a complex affine transformation, we can assume without loss of generality that each of the embeddings $\mathring h_{n,k}$ are such that $\alpha_{n,n\srta k} = 1$ and $\beta_{n,n\srta k} =0$ for each $k\in [0,n]_{\BB Z}$. 
By~\eqref{eqn-inc-homeo-compat} this implies that
\eqb \label{eqn-inc-homeo-asymp}
f_{n,k_1\srta k_2} = z + O_z(1/z)  \quad \text{as $z\rta\infty$} \quad \forall k_1,k_2\in [0,n]_{\BB Z}. 
\eqe 

There are still two complex degrees of freedom which come from pre-composing each of the fields $\mathring h_{n,k}$ with the same complex affine transformation. We fix all but one of these degrees of freedom (corresponding to a simultaneous rotation of all of our embeddings) by requiring that there is a conformal map $\BB C\setminus \ol{\BB D} \rta \BB  C\setminus \BB K_{n,n}$ which is of order $z + O_z(1/z)$ as $z\rta\infty$. 
Since we have already assumed that $f_{n,k_1\srta k_2} = z + O_z(1/z)$, each of the $\BB K_{n,k}$'s has the same capacity and harmonic center, so the condition of the preceding sentence is equivalent to requiring that each $\BB K_{n,k}$ has unit capacity and zero harmonic center, in the terminology of~\cite[Section~9]{wedges}. We henceforth assume that our embeddings have been selected in this manner.
     
Our choices of embeddings $\mathring h_{n,n}$ and $\mathring h_{n,0}$ depend only on the curve-decorated quantum surfaces $(\mathring{\mcl C}_{n,n} , \mathring\eta'_{n,n})$ and $(\mathring{\mcl C}_{n,0} , \mathring\eta'_{n,0})$, respectively, which we recall do not depend on $n$. In particular, the hulls $\BB K_{n,n}$ and $\BB K_{n,0}$ and the map $  f_{n,n\srta 0}$ do not depend on $n$. We emphasize this point by writing
\eqb \label{eqn-end-hull}
\BB K = \BB K_{n,n} , \quad \wt{\BB K} := \BB K_{n,0} ,\quad \op{and} \quad \BB f =  f_{n,n\srta 0} : \BB C \rta \BB C.
\eqe 
We note that $\BB f$ maps $\BB C\setminus \BB K$ to $ \BB C\setminus \wt{\BB K}$ conformally.  In the next subsection we will prove the following statement which immediately implies Proposition~\ref{prop-embedding-equiv}.

\begin{prop} \label{prop-hull-agree}
The map $\BB f $ defined in~\eqref{eqn-end-hull} is a.s.\ equal to the identity.
\end{prop}

\subsection{Distortion bound}
\label{sec-surface-inc-distortion}

\begin{figure}[ht!]
\begin{center}
\includegraphics[scale=.8]{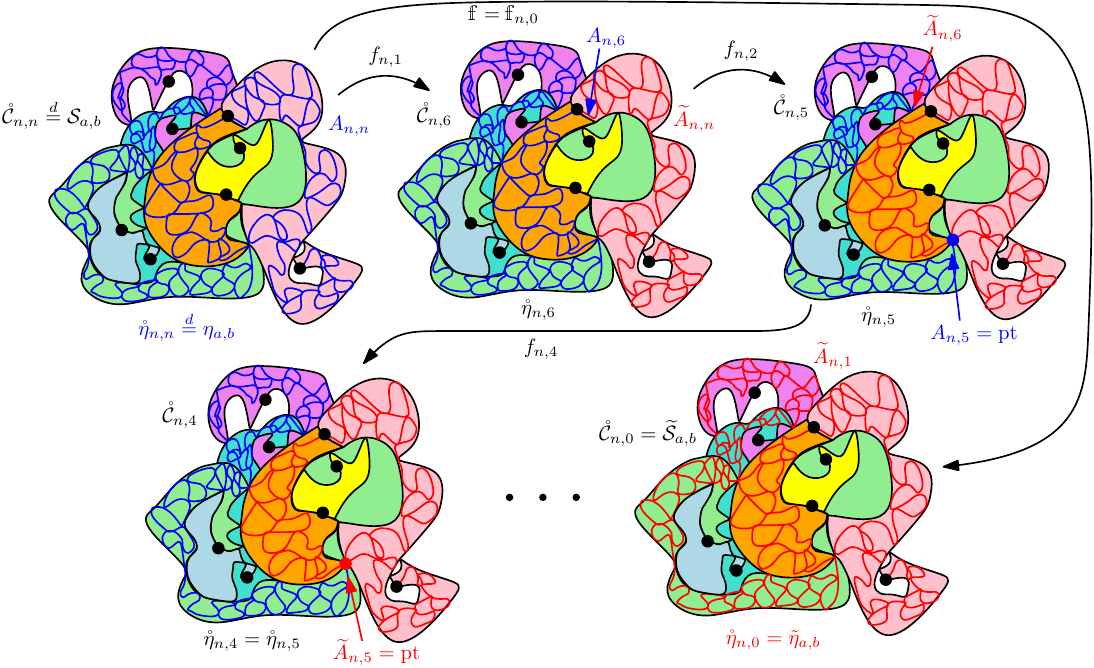}
\end{center}
\caption[Maps between intermediate curve-decorated surfaces]{
Illustration of the intermediate quantum surfaces $\rng{\mcl C}_{n,k}$, the non-space-filling curves $\rng\eta_{n,k}$ on them, and the maps $f_{n,k} = f_{n,k\srta  k-1}$ between them for $n=7$. At each step, we replace the quantum surface parameterized by the space-filling SLE segment $ \rng\eta_{n,k}'([t_{n,k-1} , t_{n,k}])$ (which coincides with the quantum surface parameterized by $\rng\eta_{n,n}'([t_{n,k-1} , t_{n,k}])$) by the quantum surface parameterized by $\rng\eta_{n,k-1}'([t_{n,k-1} ,t_{n,k}])$ (which coincides with the quantum surface parameterized by $\rng\eta_{n,0}'([t_{n,k-1} , t_{n,k}])$). This is indicated by changing the color of the corresponding segment of $\rng\eta_{n,k}$ from blue to red. Each of the maps $f_{n,k}$ is conformal on the complement of $A_{n,k} = \rng\eta_{n,k}([t_{n,k-1} , t_{n,k}])$ and takes $\BB C\setminus K_{n,k}$ to $\BB C\setminus \wt K_{n,k}$. If $E_{n,k}$ does not occur---i.e., $\rng\eta_{n,n}'([t_{n,k-1} , t_{n,k}]) \cap  \rng\eta_{n,n}   $ is a single point---which here is the case for $k=5$ (yellow) and $k=3$ (light blue, map not shown), then $f_{n,k}$ is the identity map. We show in Section~\ref{sec-surface-inc-distortion} that $\BB f$, the composition of all of the $f_{n,k}$'s (which does not depend on $n$), is a.s.\ the identity map by showing that the sum of the distortions of the $f_{n,k}$'s is small when $n$ is large. 
 }\label{fig-swapping-maps} 
\end{figure}

Suppose we are in the setting of Section~\ref{sec-surface-inc-embed}. In particular, define the embeddings $\mathring h_{n,k}$ of $\mathring{\mcl C}_{n,k}$, the curves $\mathring\eta_{n,k}$ (identified with their image under this embedding), and the hulls $\BB K_{n,k}$ as in that subsection. Here we will prove Proposition~\ref{prop-hull-agree}, which will complete the proof of Proposition~\ref{prop-inc-agree}. 

To lighten notation, for $k\in [1,n]_{\BB Z}$ we make the following definitions (see Figure~\ref{fig-swapping-maps} for an illustration).  
\begin{itemize} 
\item Let $ A_{n,k} := \mathring \eta_{n,k}([t_{n,k-1} , t_{n,k}])$ and $\wt A_{n,k} := \mathring \eta_{n,k-1}([t_{n,k-1} , t_{n,k}])$.  
\item Let $K_{n,k}$ (resp.\ $\wt K_{n,k}$) be the set of points in $\BB C$ which are disconnected from $\infty$ by $A_{n,k} $ (resp.\ $\wt A_{n,k}$). 
\item Let $f_{n,k} := f_{n,k\srta  k-1}$ and $\BB f_{n,k} := f_{n,n\srta  k}$, so that $\BB f_{n,n} = \op{Id}$ and $\BB f_{n,0} = \BB f$.  
\end{itemize} 
Note that each of the hulls $K_{n,k}$ for $k\in [1,n]_{\BB Z}$ and $\wt K_{n,k+1}$ for $k\in [0,n-1]_{\BB Z}$ is contained in the hull $ \BB K_{n,k}$ defined at the end of Section~\ref{sec-surface-inc-embed}. 
Furthermore, Lemma~\ref{lem-inc-homeo} implies that $f_{n,k}$ restricts to a conformal map from $\BB C\setminus A_{n,k}$ to $\BB C\setminus \wt A_{n,k}$. 

For $k\in [1,n]_{\BB Z}$, let $E_{n,k}$ be the event that $\mathring\eta'_{n,k} ([t_{n,k-1} , t_{n,k}])$ intersects $\mathring\eta_{n,k} $ for some (equivalently every, by Lemma~\ref{lem-inc-homeo}) $k\in [0,n]_{\BB Z}$. On the event $E_{n,k}^c$, the curve $\mathring\eta_{n,k}$ is constant on $[t_{n,k-1} , t_{n,k}]$, so $A_{n,k}$ and $\wt A_{n,k}$ (hence also $K_{n,k}$ and $\wt K_{n,k}$) are singletons. 
By the Riemann removable singularities theorem, in this case $f_{n,k}$ is conformal on all of $\BB C$ so by~\eqref{eqn-inc-homeo-asymp} $f_{n,k}$ is the identity map. 
  
By~\eqref{eqn-inc-homeo-compat},  
\eqb \label{eqn-map-decomp}
\BB f_{n,k} =  f_{n,k+1} \circ \dots \circ f_{n,n} \quad\op{and}\quad \BB f = f_{n,1} \circ\dots\circ f_{n,n} .
\eqe 
We will now bound the deviation of $\BB f$ from the identity by bounding the deviation of each $f_{n,k}$ from the identity and summing over $k\in [1,n]_{\BB Z}$. 
By the discussion just above, we already know that $f_{n,k}$ is equal to the identity map on $E_{n,k}^c$ so we only need to consider the event $E_{n,k}$. 

By~\cite[Lemma~9.6]{wedges} and~\eqref{eqn-inc-homeo-asymp}, there are universal constants $c_1,c_2 > 0$ such that the following is true. For each $k\in [1,n]_{\BB Z}$ and each $z\in\BB C$ with $\op{dist}(z, K_{n,k}) \geq c_1 \op{diam}(K_{n,k}) $, it holds that
\eqb \label{eqn-use-distortion0}
|f_{n,k}(z) -z| \leq c_2 \op{diam}(K_{n,k})^2 |z - \op{hc}(K_{n,k}) |^{-1}  \BB 1_{E_{n,k}} , 
\eqe 
where here $\op{hc}(K_{n,k})$ is the harmonic center of $K_{n,k}$, i.e.\ the constant coefficient in the Laurent expansion of $f_{n,k}$ at $\infty$ of a conformal map $\BB C\setminus \ol{\BB D} \rta \BB C\setminus K_{n,k}$ which takes $\infty$ to $\infty$ with positive real derivative at $\infty$, and $E_{n,k}$ is the event defined just above (recall that $f_{n,k}$ is the identity on $E_{n,k}^c$, which is why we do not need to restrict left hand side to $E_{n,k}$). The terminology ``harmonic center" is taken from~\cite[Section 9]{wedges}.

In light of~\eqref{eqn-use-distortion0}, we are led to estimate $\op{diam}(K_{n,k})^2$ on the event $E_{n,k}$. It turns out that this is much easier to do if the embeddings $\mathring h_{n,k}$ are replaced by the $C$-smooth canonical descriptions of the surfaces $\rng{\mcl C}_{n,k}$ (defined in Section~\ref{sec-diam-sum}). 
In order to compare our given embeddings to the $C$-smooth canonical descriptions, we will work on the regularity event of Section~\ref{sec-stability}.
We will eventually apply Proposition~\ref{prop-diam-sum} to bound the expected deviation of $\BB f$ from the identity on this event. 

Since $(\rng{\mcl C}_{n,n} , \rng\eta'_{n,n}([a,b]) ) \eqD (\mcl C , \eta' ([a,b]) )$ has the law of a $\gamma$-quantum cone decorated by an independent space-filling $\SLE_{\kappa'}$ increment, we can sample a space-filling $\SLE_{\kappa'}$ curve $\rng\eta''_{n,n}$ from $\infty$ to $\infty$, independent from $\rng h_{n,n}$ and parameterized by $\gamma$-quantum mass with respect to $\rng{h}_{n,n}$, which a.s.\ satisfies $\rng \eta''_{n,n}([a,b]) = \rng\eta'_{n,n}([a,b])$ and which is conditionally independent from everything else given $(\rng{\mcl C}_{n,n} , \rng\eta'_{n,n}([a,b]) )$.  
 
Let $C , t_0 \in \BB R$, let $S$ be a closed simply connected subset of $\BB D\setminus \{0\}$ with non-empty interior, and define the event $\mcl A_{C,t_0}(S,r,a,b)$ as in Section~\ref{sec-stability} with the curve-decorated quantum surface $((\rng{\mcl C}_{n,n} - \rng\eta''_{n,n}(t_0) ) , \rng\eta''_{n,n} ) \eqD (\mcl C , \eta'_{\mcl C}  )$ in place of $(\mcl C , \eta'_{\mcl C}  )$. The reason why we translate by $ \rng\eta''_{n,n}(t_0)$ is to cancel out the translation in condition~\ref{item-stability-field} in the definition of $\mcl A_{C,t_0}(S,r,a,b)$. 

For $k \in [0,n]_{\BB Z}$, let $\rho_{n,k}   > 0$ and $ \beta_{n,k}  \in \BB C$ be chosen so that the field
\eqb \label{eqn-smooth-switch}
\mathring h_{n,k}^C   := \mathring h_{n,k}(\rho_{n,k}^{-1}( \cdot - \beta_{n,k}) ) + Q\log \rho_{n,k}^{-1}
\eqe 
is a $C$-smooth canonical description of $\mathring{\mcl C}_{n,k}$. Let 
\eqbn
K_{n,k}^C  := \rho_{n,k}  K_{n,k} + \beta_{n,k}
\eqen
be the embedding of $K_{n,k}$ into $\BB C$ corresponding to $\mathring h_{n,k}^C $.

On $\mcl A_{C,t_0}(S,r,a,b)$, we can arrange that the upper bound in~\eqref{eqn-use-distortion0} depends on $\op{diam}(K_{n,k}^C)$ (which can be bounded using Proposition~\ref{prop-diam-sum}) instead of $\op{diam}(K_{n,k})$. 

\begin{lem} \label{lem-use-distortion}
On $\mcl A_{C,t_0}(S,r,a,b)$, it holds for each $k\in [1,n]_{\BB Z}$ and each $z\in\BB C$ with $\op{dist}(z, K_{n,k}) \geq \frac{4 c_1}{r} \op{diam}(K_{n,k}^C)  $ that
\eqb \label{eqn-use-distortion}
|f_{n,k}(z) -z| \preceq \op{diam}(K_{n,k}^C)^2 \op{dist}\left(z , \op{CH}(K_{n,k}) \right)^{-1} \BB 1_{E_{n,k}} \BB 1_{(K_{n,k}^C\subset S)} 
\eqe 
with a deterministic implicit constant depending only on $r$, 
where here $\op{CH}(\cdot)$ denotes the convex hull.  
\end{lem}
\begin{proof}
We will bound the right side of~\eqref{eqn-use-distortion0} in terms of $\op{diam}(K_{n,k}^C)$ and $\op{dist}\left(z , \op{CH}(K_{n,k})\right)$. 
Recall the definition of the maps $\BB f_{n,k} =f_{n,n\srta  k}$ from the discussion just after Proposition~\ref{prop-hull-agree}. By Lemma~\ref{lem-inc-homeo}, $\BB f_{n,k}$ restricts to a conformal map $\BB C\setminus \BB K \rta \BB C\setminus \BB K_{n,k}$. 
The restriction of the field $\mathring h_{n,k}^C$ to $\BB C\setminus  (\rho_{n,k} \BB K_{n,k} + \beta_{n,k})$ is the LQG pushforward of the field $\mathring h_{n,n}  |_{\BB C\setminus \BB K }$ under the conformal map $\rho_{n,k} \BB f_{n,k} + \beta_{n,k}$, which looks like $z\mapsto \rho_{n,k} z + \beta_{n,k}$ at $\infty$.   
By the definition of $\mcl A_{C,t_0}(S,r,a,b)$ (applied with $K = \rho_{n,k} \BB K_{n,k} + \beta_{n,k}$ and $f=  \rho_{n,k} \BB f_{n,k} + \beta_{n,k}$), on this event  
\eqb \label{eqn-hull-contain}
K_{n,k}^C \subset \rho_{n,k} \BB K_{n,k} + \beta_{n,k}  \subset S   .
\eqe 
Furthermore, by~\eqref{eqn-inc-homeo-asymp}, $\rho_{n,k} \BB f_{n,k}(z) + \beta_{n,k} = \rho_{n,k} z + O_z(1)$ as $z\rta\infty$, so Lemma~\ref{lem-stable-coeff} shows that 
\eqb \label{eqn-use-stable-coeff}
\frac{r}{4} \leq \rho_{n,k} \leq \frac{4}{r} .
\eqe
Together,~\eqref{eqn-hull-contain} and~\eqref{eqn-use-stable-coeff} imply that
\eqb \label{eqn-hull-diam}
\op{diam}(K_{n,k}) =   \frac{ \op{diam}(K_{n,k}^C) }{\rho_{n,k}}   \leq \frac{4}{r} \op{diam}(K_{n,k}^C)   .
\eqe
The harmonic center $\op{hc}(K_{n,k})$ can be expressed as the expected value of a random variable sampled from harmonic measure from $\infty$ on $\bdy K_{n,k}$. Consequently, $\op{hc}(K_{n,k}) \subset \op{CH}(K_{n,k})$, so 
\eqb \label{eqn-hull-dist}
|z - \op{hc}(K_{n,k})| \geq \op{dist}\left(z , \op{CH}(K_{n,k}) \right)  .
\eqe
By plugging~\eqref{eqn-hull-contain},~\eqref{eqn-hull-diam}, and~\eqref{eqn-hull-dist} into~\eqref{eqn-use-distortion0} we find that~\eqref{eqn-use-distortion} holds.
\end{proof}

The estimate of Lemma~\ref{lem-use-distortion} is our key tool for showing that the maps $f_{n,k}$ are close to the identity (recall~\eqref{eqn-map-decomp}), but in order to apply it we need to introduce a regularity event to make sure that the inverse distance factor is not too big. 
To this end, fix $u>0$ to be chosen later in a manner depending only on $\gamma$ and for $n\in\BB N$ and $k\in [1,n]_{\BB Z}$, define events
\eqb \label{eqn-min-dist}
 M_{n,k}(z ) := \left\{ \op{dist}\left( \BB f_{n,j}(z)  ,  K_{n,j} \right)   \geq n^{-u} , \: \forall j \in [k,n]_{\BB Z}  \right\}   \quad \op{and} \quad
 M_n(z) := M_{n,1}(z) .
\eqe

\begin{lem} \label{lem-end-distortion}
Let $\alpha =\alpha(\kappa')>0$ be the constant from Proposition~\ref{prop-diam-sum} and suppose that $u\in (0,\alpha/4)$. 
With $M_{n,k}(z)$ as in~\eqref{eqn-min-dist} and $\BB f_{n,k}$ as in the beginning of this subsection, it is a.s.\ the case that on $\mcl A_{C,t_0}(S,r,a,b)$,
\eqb \label{eqn-end-distortion}
\max_{k\in [1,n]_{\BB Z}} \sup\left\{ |\BB f_{n,k-1}(z) - z| \,:\, z\in\BB C ,\, M_{n,k}(z) \, \op{occurs} \right\}   = O_n(n^{-(\alpha/2-u)} ) 
\eqe 
as $n\rta\infty$ along powers of 2, at a possibly random rate. In particular, on $\mcl A_{C,t_0}(S,r,a,b)$, a.s.\
\eqb \label{eqn-end-distortion0}
  \sup\left\{ |\BB f (z) - z| \,:\, z\in\BB C ,\, M_n(z) \, \op{occurs} \right\}   = O_n(n^{-(\alpha/2-u)} )   
\eqe 
as $n\rta\infty$ along powers of 2.
\end{lem} 
\begin{proof} 
For $n\in\BB N$ and $k\in [1,n]_{\BB Z}$, define the following slightly modified version of $M_{n,k}(z)$: 
\eqb \label{eqn-min-dist'}
\wh M_{n,k}(z ) := \left\{ \op{dist}\left( \BB f_{n,j}(z) , \op{CH}(K_{n,j}) \right)   \geq \max\left\{ \frac{4 c_1}{r} \op{diam}(K_{n,j}^C) ,\, \frac12 n^{-u} \right\} , \: \forall j \in [k,n]_{\BB Z}  \right\}   ,
\eqe  
where here $c_1$ is the constant introduced just above~\eqref{eqn-use-distortion}. 
We will use $\wh M_{n,k}(z)$ in place of $M_{n,k}(z)$ throughout most of the proof, and at the very end argue that $M_{n,k}(z)\setminus \wh M_{n,k}(z)$ is unlikely when $n$ is large. 

By~\eqref{eqn-map-decomp} and Lemma~\ref{lem-use-distortion} (applied with $\BB f_{n,j}(z)$ for $j\in [k,n]_{\BB Z}$ in place of $z$), on $\wh M_{n,k}(z) \cap \mcl A_{C,t_0}(S,r,a,b)$,
\eqb \label{eqn-use-min-dist}
|\BB f_{n,j-1}(z) - \BB f_{n,j}(z) | \preceq n^u \op{diam}(K_{n,k}^C)^2  \BB 1_{E_{n,k}} \BB 1_{(K_{n,k}^C\subset S)} ,\quad \forall j \in [k,n]_{\BB Z}
\eqe 
with a deterministic implicit constant depending only on $r$. 
By the triangle inequality followed by~\eqref{eqn-use-min-dist}, for each $k \in [1,n]_{\BB Z}$,
\begin{align} \label{eqn-use-distortion'}
|\BB f_{n,k-1} (z) - z|  \BB 1_{\wh M_{n,k}(z)} \BB 1_{\mcl A_{C,t_0}(S,r,a,b)} 
&\leq \sum_{j=k }^n | \BB f_{n,j-1}(z) - \BB f_{n,j}(z)| \BB 1_{\wh M_{n,k}(z)} \BB 1_{\mcl A_{C,t_0}(S,r,a,b)}  \notag\\ 
&\preceq n^u \sum_{j=1}^n \op{diam}(K_{n,j}^C)^2  \BB 1_{E_{n,j}} \BB 1_{(K_{n,j}^C\subset S)}  
\end{align}
with deterministic implicit constant depending only on $r$. 

By~\eqref{eqn-surface-inc-law}, each $ K_{n,j}^C  $ has the same law as the set of points disconnected from $\infty$ by $ \eta_{a,b}([t_{n,j-1} , t_{n,j}])$ if the field $h$ is the $C$-smooth canonical description. 
We always have $\eta_{a,b}([t_{n,j-1} , t_{n,j}]) \subset \eta'([t_{n,j-1} , t_{n,j}])$ and if $\eta'([t_{n,j-1} , t_{n,j}])$ does not intersect $\eta_{a,b}$ then $\eta_{a,b}([t_{n,j-1} , t_{n,j}])$ is a single point. Hence $\op{diam}(K_{n,j}^C) \BB 1_{E_{n,j}} \BB 1_{(K_{n,j}^C\subset S)}  $ is stochastically dominated by 
$\op{diam}(\eta'([t_{n,j-1} , t_{n,j}])) \BB 1_{G_{n,j} } $,
where 
\eqbn 
G_{n,j} := \left\{\eta'([t_{n,j-1} , t_{n,j}]) \cap \eta_{a,b} \not=\emptyset\right\}  \cap \left\{ \eta_{a,b}([t_{n,j-1} , t_{n,j}]) \subset S  \right\} .
\eqen
Hence the expectation of the right side of~\eqref{eqn-use-distortion'} satisfies
\eqb \label{eqn-diam-sum-compare}
\BB E\left[ n^u \sum_{j=1}^n \op{diam}(K_{n,j}^C)^2  \BB 1_{E_{n,j}} \BB 1_{(K_{n,j}^C\subset S)}   \right]  
\leq n^u \sum_{j=1}^n \BB E\left[ \op{diam}(\eta'([t_{n,j-1} , t_{n,j}]))^2 \BB 1_{G_{n,j} } \right] .
\eqe 
By Proposition~\ref{prop-diam-sum}, the right side of this inequality is $O_n(n^{-(\alpha-u)})$, at a rate depending only on $C$, $a$, $b$, $S$, and $\kappa'$.  
By Chebyshev's inequality and the Borel-Cantelli lemma, it is a.s.\ the case that 
\eqb \label{eqn-diam-max-compare}
 \max_{j \in[1,n]_{\BB Z}} \left( \op{diam}(K_{n,j}^C)^2 \BB 1_{E_{n,j}} \BB 1_{(K_{n,j}^C\subset S)} \right)   
 \leq  n^u \sum_{j=1}^n \op{diam}(K_{n,j}^C)^2  \BB 1_{E_{n,j}} \BB 1_{(K_{n,j}^C\subset S)} 
 = O_n(n^{-(\alpha/2 - u) })
\eqe 
as $n\rta\infty$ along powers of $2$. 

By~\eqref{eqn-use-distortion'} and the second bound of~\eqref{eqn-diam-max-compare}, if $\mcl A_{C,t_0}(S,r,a,b)$ occurs then it is a.s.\ the case that~\eqref{eqn-end-distortion} holds with $\wh M_{n,k}(z)$ in place of $M_{n,k}(z)$.  

By~\eqref{eqn-diam-max-compare}, it is a.s.\ the case that 
\eqbn
\max_{j\in [1,n]_{\BB Z}} \frac{4 c_1}{r} \op{diam}(K_{n,j}^C)  = O_n(n^{-(\alpha/2-u)}) = o_n(n^{-u})
\eqen
as $n\rta\infty$ along powers of 2, uniformly over all $j\in [1,n]_{\BB Z}$. 
In particular, for large enough dyadic values of $n$ the convex hull $\op{CH}(K_{n,j}^C)$ is contained in the $\frac12 n^{-u}$-neighborhood of $K_{n,j}^C$. 
Recalling the definition~\eqref{eqn-min-dist} of $M_{n,k}(z)$, we see that there a.s.\ exists a random $n_0 \in \BB N$ such that for each dyadic $n\geq n_0$, $z\in \BB C$, and $k\in [1,n]_{\BB Z}$, the event $M_{n,k}(z) \setminus \wh M_{n,k}(z)$ does not occur. Hence~\eqref{eqn-end-distortion} follows from the analogous statement with $\wh M_{n,k}(z)$ in place of $M_{n,k}(z)$. 

The relation~\eqref{eqn-end-distortion0} is just~\eqref{eqn-end-distortion} with $k=1$. 
\end{proof}

To apply Lemma~\ref{lem-end-distortion}, we need to show that $M_n(z)$ is likely to occur for most points of $\BB C$ when $n$ is large. This is the purpose of the next lemma. 
 
\begin{lem} \label{lem-distortion-conclude}
If $\mcl A_{C,t_0}(S,r,a,b)$ occurs, then for each $z\in \BB C\setminus \rng\eta_{n,n}([a,b])$ 
it holds that $M_n(z)$ occurs for sufficiently large dyadic values of $n\in\BB N$ (here we recall from Section~\ref{sec-surface-inc-def} that $\rng\eta_{n,n}([a,b])$ does not depend on $n$). 
Hence on $\mcl A_{C,t_0}(S,r,a,b)$, we have $\BB f(z) = z$ for each $z\in \BB C\setminus \rng\eta_{n,n}([a,b])$. 
\end{lem}
\begin{proof}
Throughout the proof we assume that $\mcl A_{C,t_0}(S,r,a,b)$ occurs.
For $z\in\BB C \setminus \rng\eta_{n,n}([a,b])$ and $n\in\BB N$, let $k_z^n$ be the largest $k \in [1,n]_{\BB Z}$ for which $\op{dist}\left(\BB f_{n,k }(z) , \rng\eta_{n,k}([a,b]) \right) \leq  \frac18 n^{-u/2}$, or $k_z^n=0$ if no such $k$ exists. We claim that it is a.s.\ the case that for each $z\in\BB C$ with $\op{dist}(z,\rng\eta_{n,n}([a,b]) )\geq n^{-u/2}$, we have $k_z^n = 0$ for large enough dyadic values of $n\in\BB N$.
 
By the definition~\eqref{eqn-min-dist} of $M_{n,k}(z)$, if $n$ is chosen sufficiently large that $\op{dist}(z,\rng\eta_{n,n}([a,b]) )\geq n^{-u/2}$, then the event $M_{n,k_z^n+1}(w)$ occurs for each $w\in B_{n^{-u}}(z)$. By~\eqref{eqn-end-distortion}, it is a.s.\ the case that  
\eqb \label{eqn-min-dist-deriv}
\sup\left\{  |\BB f_{n,k_z^n} (w) - w| \,:\,  w\in B_{n^{-u}}(z) \right\}        = O_n(n^{-(\alpha/2-u)} )  ,
\eqe 
as $n\rta\infty$ along powers of 2.
We will use~\eqref{eqn-min-dist-deriv} and some complex analysis estimates to show that $|\BB f_{n,k_z^n }'(z)|$ is close to 1 when $n$ is large, then use then the Koebe quarter theorem to conclude that $\op{dist}\left(\BB f_{n,k_z^n }(z) , \rng\eta_{n,k}([a,b]) \right)  >  \frac18 n^{-u/2}$ for large enough $n$, which will contradict the definition of $k_z^n$ unless $k_z^n = 0$. 

Throughout the rest of the proof we assume that $n$ is chosen sufficiently large that $\op{dist}(z,\rng\eta_{n,n}([a,b]) )\geq n^{-u/2}$, so that $\BB f_{n,k_z^n}$ is conformal on $B_{n^{-u/2}}(z)$. 
Choose $w\in   B_{n^{-u}}(z)$. By Taylor's theorem with remainder,
\eqbn
\BB f_{n,k_z^n }(w) = \BB f_{n,k_z^n }(z) + \BB f_{n,k_z^n }'(z)(w-z) + O_n\left( |z-w|^2 \sup_{x \in B_{n^{-u}}(z)} |\BB f_{n,k_z^n }''(x)| \right) .
\eqen
Applying~\eqref{eqn-min-dist-deriv} and re-arranging gives
\allb \label{eqn-min-dist-taylor}
&    \BB f_{n,k_z^n }'(z)(z-w) = z-w + O_n\left( |z-w|^2 \sup_{x \in B_{n^{-u}}(z)} |\BB f_{n,k_z^n}''(x)| \right)  + O_n(n^{-(\alpha/2-u)} ) \notag \\
&\qquad \Rightarrow  \BB f_{n,k_z^n}'(z) =  1 +  O_n\left( n^{-u} \sup_{x \in B_{n^{-u}}(z)}| \BB f_{n,k_z^n}''(x)| \right)  + O_n(n^{-(\alpha/2- 2u)} ) .
\alle

We need to bound the second derivative term in~\eqref{eqn-min-dist-taylor}. 
By re-scaling and applying the (easy) second coefficient case of the Bieberbach-de Branges theorem, we get $|\BB f_{n,k_z^n}''(x)| \leq 4 n^{u/2} | \BB f_{n,k_z^n}'(x)|$ for $x\in B_{n^{-u}}(z)$. By the Koebe quarter theorem and since $\BB f_{n,k_z^n}$ maps $\rng\eta_{n,n}([a,b])$ to $\rng\eta_{n,k_z^n}([a,b])$,  
\eqb \label{eqn-min-dist-koebe}
| \BB f_{n,k_z^n}'(x)| \leq 4  \frac{  \op{dist}\left( \BB f_{n,k_z^n}(x) , \rng\eta_{n,k_z^n}([a,b])  \right)    }{  \op{dist}\left( x, \rng\eta_{n, n}([a,b])  \right)} ,\quad \forall x \in B_{n^{-u}}(z) . 
\eqe  
By definition of $k_z^n$, the numerator on the right is at most $\frac18 n^{-u/2} +2 n^{-u}$ and since $\op{dist}(z,\rng\eta_{n,n}([a,b]) )\geq n^{-u/2}$ the denominator is at least $n^{-u/2} - 2 n^{-u}$. 
Therefore, $| \BB f_{n,k_z^n}'(x)| = O_n(1)$ so by~\eqref{eqn-min-dist-taylor} and our above bound for $|\BB f_{n,k_z^n}''(x)| $, we have $\BB f_{n,k_z^n}'(z) =  1  + o_n(1) $. In particular,~\eqref{eqn-min-dist-koebe} applied with $x=z$ shows that
\eqbn
\op{dist}\left( \BB f_{n,k_z^n}(z ) , \rng\eta_{n,k_z^n}([a,b])  \right) \geq \frac14 (1+o_n(1)) \op{dist}\left( x, \rng\eta_{n,k_z^n}([a,b])  \right) \geq \frac14 n^{-u/2}(1+o_n(1)) .
\eqen
This contradicts the definition of $k_z^n$ for large enough $n$ unless $k_z^n=0$, so we conclude that $k_z^n$ must be equal to 0 for large enough dyadic values of $n$. By~\eqref{eqn-min-dist} this implies that $M_n(z)$ occurs for large enough dyadic values of $n$.  
The last statement of the lemma follows from the first statement and~\eqref{eqn-end-distortion0}.
\end{proof}

\begin{proof}[Proof of Propositions~\ref{prop-inc-agree}, \ref{prop-embedding-equiv}, and~\ref{prop-hull-agree}]
Lemma~\ref{lem-distortion-conclude} implies that on $\mcl A_{C,t_0}(S,r,a,b)$, we have $\BB f(z) = z$ for each $z\in \BB C\setminus \rng\eta_{n,n}([a,b])$. Since $ \rng\eta_{n,n}([a,b])$ has empty interior and $\BB f$ is continuous, a.s.\ $\BB f$ is the identity map on $\mcl A_{C,t_0}(S,r,a,b)$. By Lemma~\ref{lem-stable-prob}, for given $a<b$ we can choose $C,t_0,S$, and $r$ in such a way that $\BB P[\mcl A_{C,t_0}(S,r,a,b)]$ is as close to 1 as we like. We thus obtain Proposition~\ref{prop-hull-agree}.  Propositions~\ref{prop-inc-agree} and~\ref{prop-embedding-equiv} immediately follow.
\end{proof}

\subsection{Proof of Theorem~\ref{thm-wpsf-char}}
\label{sec-wpsf-char-proof}

Proposition~\ref{prop-inc-agree} implies that for each $a,b\in\BB Q$ with $a<  b$, the curve-decorated quantum surfaces $(\mcl S_{a,b} , \eta_{a,b}) $ and $ (\wt{\mcl S}_{a,b } , \wt\eta_{a,b })$ agree in law. By Lemma~\ref{lem-partial-surface-determined} (applied with $t =b$), $ \mcl S_{a,b}$ is a.s.\ determined by $(\mcl S_{a,b}^0 , \eta_{a,b,\mcl S_{a,b}^0})$. Hence $\wt{\mcl S}_{a,b}$ is a.s.\ determined by $(\wt{\mcl S}_{a,b}^0 , \wt\eta_{a,b,\wt{\mcl S}_{a,b}^0})$.  In particular, by~\eqref{eqn-bubble-surface-msrble}, $\wt{\mcl S}_{a,b} $ is measurable with respect to the $\sigma$-algebra $\wt{\mcl F}_{a,b}$ of Theorem~\ref{thm-wpsf-char}. 

Now fix $\ep > 0$. By the preceding paragraph, $\wt{\mcl S}_{(j-1) \ep , j\ep} \in  \wt{\mcl F}_{(j-1)\ep , j \ep} \subset \wt{\mcl F}_{(j-1)\ep , \infty} \cap \wt{\mcl F}_{-\infty, j\ep} $ for each $j\in\BB Z$. By condition~\ref{item-wpsf-char-wedge} in the theorem statement, the quantum surfaces $\{\wt{\mcl S}_{(j-1) \ep , j\ep}\}_{j\in\BB Z}$ are independent, so
\eqbn
\{\wt{\mcl S}_{(j-1) \ep , j\ep}\}_{j\in\BB Z} \eqD \{ \mcl S_{(j-1) \ep , j\ep}\}_{j\in\BB Z}.
\eqen

Recall the $\gamma$-quantum cones $\mcl C = (\BB C ,h , 0 ,\infty)$ and $\wt{\mcl C} = (\BB C , \wt h , 0, \infty)$ from~\eqref{eqn-cone-def}.  By condition~\ref{item-wpsf-char-homeo} in the theorem statement, $\wt{\mcl C}$ is obtained by conformally welding the quantum surfaces $\{\wt{\mcl S}_{(j-1) \ep , j\ep}\}_{j\in\BB Z}$ in chronological order according to $\gamma$-quantum length along their boundaries, and $\wt\eta'_{\wt{\mcl C}}$ fills in these surfaces in chronological order. Similar statements hold for the pair $(\mcl C , \eta'_{\mcl C})$. 

Each of the surfaces $\mcl S_{(j-1) \ep , j\ep}$ is parameterized by the space-filling SLE segment $\eta'_{\mcl C}([(j-1)\ep , j\ep])$.  The boundary of each of these segments is a finite union of $\SLE_{\kappa}$-type curves for $\kappa = 16/\kappa'$.  Therefore,~\cite[Proposition~3.16]{wedges} implies that a.s.\ the union of the boundaries of these segments is conformally removable. Hence there is only one way to conformally weld together the quantum surfaces $\mcl S_{(j-1) \ep , j\ep}$ according to quantum length along their boundaries. We infer from the above descriptions of $(\wt{\mcl C} , \wt\eta'_{\wt{\mcl C}})$ and $(\mcl C , \eta'_{\mcl C})$ that
\eqbn
\left(\wt{\mcl C} , \{\wt \eta'_{\wt{\mcl C}}(j\ep)\}_{j\in\BB Z} \right) \eqD \left(\mcl C , \{\eta'_{ \mcl C }(j\ep)\}_{j\in\BB Z} \right).
\eqen
Since $\ep >0$ can be made arbitrarily small and the curves $ \wt \eta'_{\wt{\mcl C}}$ and $\eta'_{\mcl C}$ are continuous, we infer that in fact $(\wt{\mcl C} , \wt\eta'_{\wt{\mcl C}}) \eqD (\mcl C , \eta'_{\mcl C})$. 

By condition~\ref{item-wpsf-char-homeo} in the theorem statement, $Z$ is the peanosphere Brownian motion for each of the SLE-decorated quantum cones $(\wt{\mcl C} , \wt\eta'_{\wt{\mcl C}}) $ and $ (\mcl C , \eta'_{\mcl C})$, so by~\cite[Theorem~1.11]{wedges}, a.s.\ $(\wt{\mcl C} , \wt\eta'_{\wt{\mcl C}}) = (\mcl C , \eta'_{\mcl C})$ as curve-decorated quantum surfaces. Hence there a.s.\ exists a conformal automorphism of $\BB C$ which fixes 0 (i.e., multiplication by a complex number) which takes $ \eta'$ to $\wt\eta'$. Since $\Phi\circ \eta' = \wt\eta'$ and $\eta'$ is space-filling, it follows that this conformal automorphism coincides with $\Phi$.  \qed

\section{Characterization of chordal SLE on a bead of a thin wedge}
\label{sec-chordal}

In this section we will state and prove an analog of Theorem~\ref{thm-wpsf-char} for a chordal $\SLE_{\kappa'}$ on a single bead of a $\frac{3\gamma}{2}$-quantum wedge, which is a precise version of Theorem~\ref{thm-bead-char-intro}.

\subsection{Statement and outline}
\label{sec-chordal-statement}
 
Suppose $h^\bead$ is an embedding into $(\BB H , 0 , \infty)$ of a single bead of a $\frac{3\gamma}{2}$-quantum wedge with specified area $\frk a$ and left/right boundary lengths $\frk l^L$ and $\frk l^R$. Let $\eta^\bead : [0,\infty) \rta \ol{\BB H}$ be an independent chordal $\SLE_{\kappa'}$ from $0$ to $\infty$ in $\BB H$, with some choice of parameterization.  For $t  \geq 0$, let $L_t^\bead$ (resp.\ $R_t^\bead$) be the $\nu_{h^\bead}$-length of the boundary arc of the unbounded connected component of $\BB H\setminus \eta^\bead([0,t])$ which lies to the left (resp.\ right) of $\eta^\bead(t)$ and define the \emph{boundary length process} $Z_t^\bead := (L_t^\bead, R_t^\bead)$. 
Note that $Z_0^\bead = (\frk l^L , \frk l^R)$.

\begin{figure}[ht!]
\begin{center}
\includegraphics[scale=1]{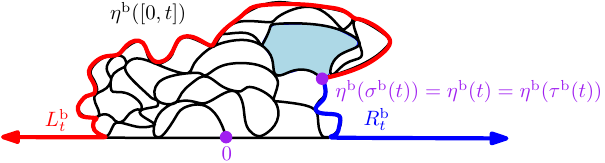} 
\caption[SLE boundary length process]{\label{fig-bdy-process-sle} An SLE$_{\kappa'}$ curve $\eta^\bead$ on a bead of a $\frac{3\gamma}{2}$-quantum wedge (parameterized by $\BB H$). If we parameterize $\eta^\bead$ by the $\mu_h$-mass it disconnects from $\infty$, then for a typical time $t$ the curve $\eta^\bead$ is constant on the time interval $[\sigma^\bead(t) , \tau^\bead(t)]$, and $\tau^\bead(t)  -\sigma^\bead(t)$ gives the $\mu_h$-mass of the bubble disconnected from $\infty$ by $\eta^\bead$ at time $\sigma^\bead(t)$ (light blue). The boundary length process $L_t^\bead$ (resp.\ $R_t^\bead$) gives the $\nu_h$-length of the red (resp.\ blue) arc. }
\end{center}
\end{figure}

The process $L_t^\bead$ (resp.\ $R_t^\bead$) is a right continuous process with left hand limits with no upward jumps, but with a downward jump whenever $\eta^\bead$ disconnects a bubble from $\infty$ on its left (resp.\ right) side. The magnitude of the downward jump corresponds to the boundary length of the bubble.  In contrast to the peanosphere Brownian motion $Z$ from Theorem~\ref{thm-wpsf-char}, the process $Z^\bead$ does not a.s.\ determine the pair $(h^\bead , \eta^\bead)$ modulo conformal maps for any choice of parameterization because it does not encode the behavior of the field $h^\bead$ away from the curve $\eta^\bead$.  It does, however, encode the equivalence class of $(\BB D , \eta)$ modulo curve-preserving homeomorphisms.
  
For our chordal SLE characterization theorem, we will parameterize $\eta^\bead$ in such a way that the boundary length process $Z^\bead$ \emph{does} encode the quantum mass of the bubbles cut off from $\infty$ by $\eta^\bead$. In particular, we will parameterize $\eta^\bead$ by the $\mu_{h^\bead}$-mass it disconnects from $\infty$, as defined in Definition~\ref{def-chordal-parameterization}. With this choice of parameterization, the process $Z^\bead$ encodes the quantum areas of the connected components of $\BB H\setminus\eta^\bead([0,t])$ in the following manner. For $t\in [0,\frk a]$, let $[\sigma^\bead(t) , \tau^\bead(t)]$ be the largest interval containing $t$ on which $Z^\bead$ is constant (note that this interval may be a single point).  Also let $D_t$ be the last bubble disconnected from $\infty$ by $\eta^\bead$ before time $t$ (or $D_t = \eta^\bead(t)$ if no such last bubble exists). As noted above, the boundary length of $D_t$ is encoded by the magnitude of the downward jump of $L^\bead$ or $R^\bead$ at time $t$, and this together with our choice of parameterization gives
\eqb \label{eqn-bead-bubble-process}
\mu_{h^\bead}(D_t) = \tau^\bead(t) - \sigma^\bead(t)   \quad \op{and} \quad \nu_{h^\bead}(\bdy D_t) = | Z_{\sigma^\bead(t)}  - \lim_{s\rta \sigma^\bead(t)^-} Z^\bead_s|  .
\eqe 
See Figure~\ref{fig-bdy-process-sle} for an illustration of the above definitions.

We now state our chordal $\SLE_{\kappa'}$ characterization theorem.

\begin{thm}[Chordal SLE characterization on a bead of a thin wedge] \label{thm-bead-char}
Let $\kappa' \in (4,8)$ and $\gamma = 4/\sqrt{\kappa'} \in (\sqrt 2 , 2)$. 
Let $(\frk a , \frk l^L, \frk l^R )\in (0,\infty)^3$ and suppose we are given a random triple $(\wt h^\bead , \wt \eta^\bead , Z^\bead )$ where $\wt h^\bead$ is an embedding into $(\BB H , 0 , \infty)$ of a single bead of a $\frac{3\gamma}{2}$-quantum wedge with area $\frk a$ and left/right boundary lengths $\frk l^L$ and $\frk l^R$, $\wt\eta^\bead : [0,\frk a] \rta \ol{\BB H}$ is a random continuous curve parameterized by the $\mu_{\wt h^\bead}$-mass it disconnects from $\infty$ (Definition~\ref{def-chordal-parameterization}), and $Z^\bead = (L^\bead, R^\bead)$ has the law of the boundary length process of a chordal $\SLE_{\kappa'}$ on a doubly marked quantum disk with area $\frk a$ and left/right boundary lengths $\frk l^L$ and $\frk l^R$, parameterized by the quantum mass it disconnects from $\infty$.
Assume that the following conditions are satisfied.
\begin{enumerate}
\item (Law of complementary connected components) Suppose $t\in [0,\frk a]$ and we condition on $Z^\bead|_{[0,t]}$ and the time $\tau^\bead(t)$ from~\eqref{eqn-bead-bubble-process}. The conditional law of the collection of singly marked quantum surfaces obtained by restricting $\wt h^\bead$ to the bubbles disconnected from $\infty$ by $\wt\eta^\bead$ before time $t$, each marked by the point where $\wt\eta^\bead$ finishes tracing its boundary, is that of a collection of independent singly marked quantum disks parameterized by the intervals of time in $[0,\tau^\bead(t)]$ on which $Z^\bead$ is constant, with areas and boundary lengths determined by $Z^\bead$ as in~\eqref{eqn-bead-bubble-process}. The conditional law of the doubly marked quantum surface obtained by restricting $\wt h^\bead$ to the unbounded connected component of $\BB H\setminus \wt\eta^\bead([0,t])$, with marked points $\wt\eta^\bead(t)$ and $\infty$, is that of a single bead of a $\frac{3\gamma}{2}$-quantum wedge with area $\frk a -\tau^\bead(t)$ and left/right boundary lengths $L_t^\bead$ and $R_t^\bead$, and this quantum surface is independent from the quantum surfaces in the preceding sentence.
\label{item-bead-char-wedge}  
\item (Topology and consistency) 
There is a pair $((\BB H , h^\bead , 0, \infty) ,\eta^\bead)$ consisting of a bead of a $\frac{3\gamma}{2}$-quantum wedge with area $\frk a$ and left/right boundary lengths $\frk l^L$ and $\frk l^R$ and an independent chordal $\SLE_{\kappa'}$ from $0$ to $\infty$ in $\BB H$ parameterized by the $\mu_{h^\bead}$-mass it disconnects from $\infty$ and a homeomorphism $\Phi^\bead : \BB H\rta\BB H$ with $\Phi^\bead\circ \eta^\bead  = \wt\eta^\bead $. Moreover, $\Phi^\bead$ pushes forward the $\gamma$-quantum length measure with respect to $h^\bead$ on the boundary of the unbounded connected component of $\BB H\setminus \wt\eta^\bead([0,t])$ to the  $\gamma$-quantum length measure with respect to $\wt h^\bead$ on the boundary of the unbounded connected component of $\BB H\setminus \wt\eta^\bead([0,t])$ for each $t\in [0,\frk a] \cap \BB Q$. \label{item-bead-char-homeo} 
\end{enumerate}
Then $(\wt h^\bead,\wt\eta^\bead)$ is an embedding into $(\BB H , 0 , \infty)$ of a single bead of a $\frac{3\gamma}{2}$-quantum wedge together with an independent chordal $\SLE_{\kappa'}$ from $0$ to $\infty$ in $\BB H$ parameterized by the $\mu_{\wt h^\bead}$-mass it disconnects from $\infty$.  
\end{thm}

We will prove in Lemma~\ref{lem-bead-sle-as} below that the hypotheses of Theorem~\ref{thm-bead-char} are satisfied when $(\wt h^\bead ,\wt\eta^\bead) = (h^\bead , \eta^\bead)$, at least for almost every triple $(\frk a , \frk l^L , \frk l^R)$. This statement is not needed for the proof of Theorem~\ref{thm-bead-char} but the theorem statement is vacuous without it.

By the same explanation given in Section~\ref{sec-wpsf-char-intro}, the conditions of Theorem~\ref{thm-bead-char} enable us to define the $\gamma$-quantum length measure with respect to $\wt h^\bead$ on the boundary of each connected component of $\BB H\setminus \wt\eta^\bead([0,t])$ for all $t\in [0,\frk a]$ simultaneously. 
In fact, by considering a time immediately before $\wt\eta^\bead$ disconnects a given bubble from $\infty$, we see that the homeomorphism $\Phi^\bead$ of condition~\ref{item-bead-char-homeo} is boundary-length preserving on each connected component of $\BB H\setminus \eta^\bead([0,t])$ for each $t\in [0,\frk a]$. 
 
In the next subsection we will deduce Theorem~\ref{thm-bead-char} from Theorem~\ref{thm-wpsf-char}. Here we give an outline of the proof.
We will first prove a variant of Theorem~\ref{thm-bead-char} where the area/boundary length triple $(\frk a , \frk l^L, \frk l^R)$ is random. 
The reason for this is that, as we will see below, this enables us to couple $(\wt h^\bead , \wt\eta^\bead)$ with a pair $(\wt h , \wt \eta' )$ consisting of an embedding into $(\BB C , 0, \infty)$ of a $\gamma$-quantum cone and a space-filling curve in $\BB C$ which can be shown to satisfy the hypotheses of Theorem~\ref{thm-wpsf-char}.  The quantum surface $(\BB H ,\wt h^\bead  , 0 ,\infty)$ will be one of the beads of the future quantum surface $\wt{\mcl S}_{0,\infty}$ and the curve $\wt\eta'$ restricted to this bead will be obtained from $\wt\eta^\bead$ by filling in each of the complementary connected components of $\wt\eta^\bead$ by a space-filling $\SLE_{\kappa'}$ loop.
Such a coupling does not work for a deterministic triple $(\frk a , \frk l^L , \frk l^R)$ since a $\frac{3\gamma}{2}$-quantum wedge does \emph{not} a.s.\ have a bead with area $\frk a$ and left/right boundary lengths $\frk l^L$ and $\frk l^R$. 

Once we have set up our coupling, we will first check that the hypotheses of Theorem~\ref{thm-bead-char} are satisfied in the case when $\wt\eta^\bead$ is actually a SLE$_{\kappa'}$ independent from $\wt h^\bead$ (Lemmas~\ref{lem-bead-sle} and~\ref{lem-bead-sle-as}). 
We will then argue that the pair $(\wt h ,\wt\eta')$ satisfies the hypotheses of Theorem~\ref{thm-wpsf-char} (Lemma~\ref{lem-wpsf-to-chordal}). The proof that this is the case is straightforward, but there are a few technicalities in verifying the various independence statements.

\subsection{Proof of Theorem~\ref{thm-bead-char}}
\label{sec-chordal-proof}

We now proceed with the details of the proof. Let $\BB M$ be the infinite measure on beads of a $\frac{3\gamma}{2}$-quantum wedge (recall~\cite[Definition~4.15]{wedges}) and let $\BB m $ be the infinite measure on $(0,\infty)^3$ which is the pushforward of $\BB M$ under the function which assigns to each bead its vector of area and left/right quantum boundary lengths. 
Let $\frk A \subset (0,\infty)^3$ be a Borel set such that $\BB m(\frk A)$ is finite and positive. 

Throughout most of this section, we assume that we are in the setting of Theorem~\ref{thm-bead-char} except that the triple $(\frk a , \frk l^L , \frk l^R)$ is sampled from the probability measure $\BB m(\frk A)^{-1} \BB m|_{\frk A}$, instead of being deterministic. We assume that the conditions in the theorem statement all hold for the conditional law given $(\frk a , \frk l^L , \frk l^R)$. 
We will only return to the case of deterministic $(\frk a , \frk l^L , \frk l^R)$ at the very end of the proof. 

Let $\mcl C = (\BB C , h  , 0 , \infty)$ be a $\gamma$-quantum cone and let $\eta'$ be a whole-plane space-filling $\SLE_{\kappa'}$ independent from $h$ and parameterized by $\gamma$-quantum mass with respect to $h$ in such a way that $\eta'(0) = 0$. Let $Z$ be the corresponding peanosphere Brownian motion. Define the quantum surfaces $\mcl S_{a,b}$ parameterized by $\eta'([a,b])$ for $a < b$ as in~\eqref{eqn-increment-surface}. Recall in particular that each of the future beaded surfaces $\mcl S_{t,\infty}$ for $t\in\BB R$ has the law of a $\frac{3\gamma}{2}$-quantum wedge. 

Let $\mcl B$ be the first bead of the $\frac{3\gamma}{2}$-quantum wedge $\mcl S_{0,\infty}$ whose vector of area and left/right quantum boundary lengths belongs to $\frk A$. Let $\ul T$ and $\ol T$ be the times at which $\eta'$ starts and finishes filling in this bead, so that
\eqb \label{eqn-bead-embed}
\mcl B = \mcl S_{\ul T , \ol T} = \left( \eta'([\ul T , \ol T]) , h|_{\eta'([\ul T , \ol T])  }  ,  \eta'(\ul T) , \eta'(\ol T) \right)  
\eqe 
and $\ol T - \ul T$ is the quantum mass of $\mcl B$. 
By our choice of $\frk A$ and since the beads of $\mcl S_{0,\infty}$ are a Poisson point process sampled from $\BB M$, we find that $\mcl B$ is well-defined a.s.\ and that the law of $\mcl B$ is that of a single bead of a $\frac{3\gamma}{2}$-quantum wedge conditioned to have area and left/right quantum boundary lengths in $\frk A$, i.e.\ $\mcl B \eqD (\BB H , h^\bead , 0, \infty)$. 

Let $\eta_{0,\infty}$ be the curve obtained from $\eta'$ by skipping the bubbles filled in by $\eta'$ in the time interval $[0,\infty)$, as in Section~\ref{sec-surface-def}, and define the times $\sigma(t)  = \sigma_{0,\infty}(t)$ and $\tau = \tau_{0,\infty}(t)$ for $t\geq 0$ as in~\eqref{eqn-tau-def}, so that $\eta_{0,\infty}$ is constant on each time interval $[\sigma(t) , \tau(t)]$. 

By~\cite[Footnote 4]{wedges} the conditional law of $\eta_{0,\infty}(\cdot - \ul T) |_{[\ul T , \ol T]}$ given $\wt h$ and $\eta'(\ul T , \ol T)$ is that of a chordal $\SLE_{\kappa'}$ from $\eta'(\ul T)$ to $\eta'(\ol T)$ in $\eta'([\ul T , \ol T])$, parameterized by the $\mu_h$-mass of the region it disconnects from $\eta'(\ol T)$. Therefore the curve-decorated quantum surface $(\mcl B , \eta_{0,\infty,\mcl B}(\cdot - \ul T) )$ agrees in law with $((\BB H , h^\bead , 0, \infty) , \eta^\bead)$ (the latter viewed as a curve-decorated quantum surface). 

Thus, we can couple $(\wt h^\bead , \wt\eta^\bead)$ and $(h^\bead,\eta^\bead)$ with $(h ,\eta')$ in such a way that a.s.\ $(\mcl B , \eta_{0,\infty,\mcl B}(\cdot - \ul T) )$ and $((\BB H , h^\bead , 0, \infty) , \eta^\bead)$ agree as curve-decorated quantum surfaces and $(h,\eta')$ is conditionally independent from $(\wt h^\bead , \wt\eta^\bead)$ and $(h^\bead,\eta^\bead)$ given $(\mcl B , \eta_{0,\infty,\mcl B}(\cdot - \ul T) )$. Henceforth fix such a coupling. 

For our choice of coupling, $\frk a = \ol T - \ul T$ and corresponding boundary length process appearing in Theorem~\ref{thm-bead-char} is given by
\eqb \label{eqn-bead-process-peano}
Z^\bead_t = (L_t^\bead, R_t^\bead) =   (Z_{\tau(t) - \ul T}    -  Z_{\ul T}) + (Z_{\ol T}  - Z_{\ul T}) ,\quad \forall t\in [0,\frk a]  .
\eqe 
Note that the term $Z_{\ol T} - Z_{\ul T}$ comes from the fact that $Z^\bead_0 = (\frk l^L , \frk l^R)$. 
Furthermore, the time intervals on which $Z^\bead$ is constant defined just above Theorem~\ref{thm-bead-char} satisfy $\sigma^\bead(t) = \sigma(t  ) - \ul T$ and $\tau^\bead(t) = \tau(t ) - \ul T$. 

The above coupling together with the results of Section~\ref{sec-disk-law} allows us to deduce the following lemma, which tells us that for random $(\frk a , \frk l^L , \frk l^R)$, the conditions of Theorem~\ref{thm-bead-char} are satisfied in the special case when $(\wt h^\bead , \wt\eta^\bead) = (h^\bead,\eta^\bead)$.

\begin{lem} \label{lem-bead-sle}
For any choice of set $\frk A \subset(0,\infty)^3$ as above, a slightly stronger version of condition~\ref{item-bead-char-wedge} in Theorem~\ref{thm-bead-char} holds with $(h^\bead , \eta^\bead)$ in place of $(\wt h^\bead , \wt\eta^\bead )$. 

More precisely, for $t\in [0,\frk a]$ let $\mcl W^\bead_t$ be the doubly marked quantum surface obtained by restricting $h^\bead$ to the unbounded connected component of $\BB H\setminus \eta^\bead([0,t])$, with marked points $\eta^\bead(t)$ and $\infty$. 
If we condition on $(\frk a , \frk l^L , \frk l^R)$, $Z^\bead|_{[0,t]}$, and the time $\tau^\bead(t)$ from~\eqref{eqn-bead-bubble-process}, then the conditional law of $\mcl W^\bead_t$ is that of a single bead of a $\frac{3\gamma}{2}$-quantum wedge with area $\frk a -\tau^\bead(t)$ and left/right boundary lengths $L_t^\bead$ and $R_t^\bead$. The conditional law of the collection of singly marked quantum surfaces obtained by restricting $ h^\bead$ to the bubbles disconnected from $\infty$ by $ \eta^\bead$ before time $t$, each marked by the point where $ \eta^\bead$ finishes tracing its boundary, is that of a collection of independent singly marked quantum disks parameterized by the intervals of time in $[0,\tau^\bead(t)]$ on which $Z^\bead$ is constant, with areas and boundary lengths determined by $Z^\bead$ as in~\eqref{eqn-bead-bubble-process}.

Furthermore, the curve-decorated quantum surfaces $( \mcl W^\bead_t ,  \eta^\bead_{ \mcl W^\bead_t})$ and $(\mcl B\setminus \mcl W^\bead_t , \eta^\bead_{\mcl B\setminus \mcl W^\bead_t})$ are conditionally independent given $(\frk a ,\frk l^L , \frk l^R)$, $Z_t^\bead$, and $\tau^\bead(t)$. 
\end{lem}

We note that the last statement about curve-decorated quantum surfaces is not part of the hypotheses of Theorem~\ref{thm-bead-char}, and is stronger than the independence assertion in condition~\ref{item-bead-char-wedge} of Theorem~\ref{thm-bead-char}.

\begin{proof}[Proof of Lemma~\ref{lem-bead-sle}]
Recall the $\gamma$-quantum cone/space-filling $\SLE_{\kappa'}$ pair $(h,\eta')$ defined above. 
For $t\geq 0$, let $P_{t,\infty}$ be the function as in~\eqref{eqn-bead-function} which encodes the areas and left/right boundary lengths of the beads of the surface $\mcl S_{t,\infty}$. Then $[\ul T , \ol T]$ is the leftmost interval of times on which $P_{0,\infty}(s) \in \frk A$ and $(\frk a , \frk l^L ,\frk l^R)$ is the value of $P_{0,\infty}$ on this interval. 
Furthermore, if we define the time intervals $s\in [\sigma(s) , \tau(s)] \subset [0,\infty)$ on which $\eta_{0,\infty}$ is constant as above, then for each $t\in [0,\frk a] = [0, \ol T - \ul T]$, we have by our choice of coupling that
\eqb \label{eqn-cone-bead-surface}
\mcl W^\bead_t = \mcl S_{\tau (\ul T+t) , \ol T}  .
\eqe  

By~\eqref{eqn-bead-process-peano}, $Z^\bead$ is determined by $\{Z_{\tau (s)} : s \geq 0\}$. It therefore follows from Proposition~\ref{prop-sle-disk} that the conditional law given $Z^\bead$ of the collection of singly marked quantum surfaces obtained by restricting $ h^\bead$ to the bubbles disconnected from $\infty$ by $ \eta^\bead$ before time $t$ is as in the statement of the lemma.

To study the conditional law of the future curve-decorated quantum surface $( \mcl W^\bead_t ,  \eta^\bead_{ \mcl W^\bead_t})$, 
we first observe that $\mcl W_t^\bead$ is either empty or is a bead of $\mcl S_{\ul T + t,\infty}$ (since it corresponds to a connected component of $\BB C\setminus \eta'((-\infty,\ul T+t])$). 

Let $\mcl G_r = \sigma\left( (Z-Z_r)|_{(-\infty,r]} \right) \vee \sigma( P_{r,\infty} ) $ for $r\in\BB R$ be the $\sigma$-algebra of Lemma~\ref{lem-partial-filtration}. 
Our above description of the times $\ul T$ and $\ol T$ in terms of $P_{0,\infty}$ shows that $\ul T$, $\ol T$, and the area and left/right quantum boundary lengths $(\frk a , \frk l^L , \frk l^R)$ of $\mcl B$ are $\mcl G_0$-measurable. 
By this and~\eqref{eqn-cone-bead-surface}, the conditional law of $\mcl W^\bead_t$ given $\mcl G_{\ul T +t}$ is that of a single bead of a $\frac{3\gamma}{2}$-quantum wedge with given area and left/right boundary lengths. It is clear from the domain Markov property that the conditional law of $\eta^\bead|_{[t, \frk a ]}$ given $\mcl G_{\ul T +t}$ and $\mcl W^\bead_t$ is that of a chordal $\SLE_{\kappa'}$ between the two marked points of $\mcl W^\bead_t$, parameterized by the quantum mass it disconnects from its target point. 
In particular, the conditional law of $( \mcl W^\bead_t ,  \eta^\bead_{ \mcl W^\bead_t})$ depends only on the area and left/right quantum boundary lengths of $\mcl W^\bead_t$, which are determined by $(\frk a ,\frk l^L , \frk l^R)$, $Z_t^\bead$, and $\tau^\bead(t)$.  

The restricted Brownian motion $(Z-Z_{\ul T+t})|_{(-\infty,\ul T + t]}$, the vector $(\frk a , \frk l^L,\frk l^R)$, and the time $\tau^\bead(t) = \tau(\ul T + t) - \ul T$ are all $ \mcl G_{\ul T+t}$-measurable. Since each curve-decorated quantum surface $(\mcl S_{a,b}, \eta'_{\mcl S_{a,b}})$ is a.s.\ determined by $(Z-Z_a)|_{[a,b]}$ and $\eta^\bead_{\mcl S_{-\infty,\ul T + t}}$ is a.s.\ determined by $(\mcl S_{-\infty,\ul T+t} , \eta'_{\mcl S_{-\infty,\ul T + t}})$ and $\tau (t)$, we infer that $(\mcl B\setminus \mcl W^\bead_t , \eta^\bead_{\mcl B\setminus \mcl W^\bead_t}) \in \mcl G_{\ul T+ t}$. In particular, the collection of singly marked quantum surfaces obtained by restricting $ h^\bead$ to the bubbles disconnected from $\infty$ by $ \eta^\bead$ before time $t$ is $\mcl G_{\ul T+t}$-measurable.
The statement of the lemma follows by combining this with our above description of the conditional law of this collection of singly marked quantum surfaces given $Z^\bead $ and the conditional law of $\mcl W^\bead_t$ given $\mcl G_{\ul T+t}$. 
\end{proof}

We record the following consequence of Lemma~\ref{lem-bead-sle}, which says that the hypotheses of Theorem~\ref{thm-bead-char} in the case when $(\wt h^\bead ,\wt\eta^\bead) = (h^\bead ,\eta^\bead)$ are satisfied for almost every deterministic choice of area and left/right boundary length vector $(\frk a , \frk l^L , \frk l^R)$. 
 
\begin{lem} \label{lem-bead-sle-as}
For Lebesgue-a.e.\ triple $(\frk a_0 , \frk l_0^L , \frk l_0^R) \in (0,\infty)^3$, the conclusion of Lemma~\ref{lem-bead-sle} remains true if we fix this value of $(\frk a , \frk l^L , \frk l^R) = (\frk a_0 , \frk l_0^L , \frk l_0^R)$ instead of sampling $(\frk a , \frk l^L , \frk l^R)$ from $\frk A$.
\end{lem}
\begin{proof}
Lemma~\ref{lem-bead-sle} is a statement about the conditional law of $(h^\bead , \eta^\bead)$ given $(\frk a , \frk l^L , \frk l^R)$ which holds $\BB m(\frk A)^{-1} \BB m|_{\frk A}$-a.s. Letting $\frk A$ vary shows that the statement of the lemma holds for $\BB m$-a.e.\ triple $(\frk a_0 , \frk l_0^L , \frk l_0^R) \in (0,\infty)^3$. Since the areas and left/right boundary lengths of the beads of $\mcl S_{0,\infty}$ are a Poisson point process with intensity measure $\BB m$ and by Lemma~\ref{lem-bead-inf}, we infer that $\BB m$ is mutually absolutely continuous with respect to Lebesgue measure on $(0,\infty)^3$. The statement of the lemma follows.
\end{proof}

Using Lemma~\ref{lem-bead-sle}, we see that the conditions in Theorem~\ref{thm-bead-char} actually imply a slightly stronger set of conditions after possibly re-choosing $(h^\bead , \eta^\bead , \Phi^\bead)$. 
 
\begin{lem} \label{lem-bead-char-conformal}
In the statement of Theorem~\ref{thm-bead-char}, we can choose the pair $(h^\bead ,\eta^\bead)$ and the homeomorphism $\Phi^\bead$ in condition~\ref{item-bead-char-homeo} in such a way that $\Phi^\bead$ is conformal on each connected component of $\BB H\setminus \eta^\bead([0,t])$ and pushes forward $h^\bead|_{ \BB H\setminus \eta^\bead([0,t])}$ to $\wt h^\bead|_{\BB H\setminus \wt\eta^\bead([0,t])}$ via the $\gamma$-LQG coordinate change formula. 
\end{lem}
\begin{proof}
By condition~\ref{item-bead-char-wedge} in Theorem~\ref{thm-bead-char} and Lemma~\ref{lem-bead-sle}, each applied with $t = \frk a$, the collection of singly-marked quantum surfaces obtained by restricting $\wt h^\bead$ to the bubbles cut out by $\wt\eta^\bead$ and the collection of singly-marked quantum surfaces obtained by restricting~$h^\bead$ to the bubbles cut out by $ \eta^\bead$ have the same conditional law given~$Z^\bead$. 
Hence we can find a coupling of $(\wt h^\bead , \wt\eta^\bead)$ and $(h^\bead , \eta^\bead)$ with another pair $(\rng h^\bead , \rng\eta^\bead) \eqD (h^\bead , \eta^\bead)$ such that~$Z^\bead$ is the left/right boundary length process of $(\rng h^\bead , \rng\eta^\bead)$ and the quantum surface obtained by restricting~$\rng h^\bead$ to each of the bubbles disconnected from $\infty$ by $\rng\eta^\bead$ agrees with the quantum surface obtained by restricting $\wt h^\bead$ to the corresponding bubble disconnected from $\infty$ by $\wt\eta^\bead$. We will prove that the conditions of Theorem~\ref{thm-bead-char} plus the additional condition in the statement of the lemma are satisfied with $(\rng h^\bead , \rng\eta^\bead)$ in place of $(h^\bead , \eta^\bead)$. 
 
The discussion just above Lemma~\ref{lem-bead-sle} together with the peanosphere construction (see~\cite[Figure~1.15, Line~3]{wedges}) imply that the curve-decorated topological space $(\BB H , \eta^\bead)$ is a.s.\ determined by $Z^\bead$ modulo a curve-preserving homeomorphism which also preserves the $\gamma$-quantum boundary length measure on each connected component of $\BB H\setminus \eta^\bead$. 
Hence condition~\ref{item-bead-char-homeo} in Theorem~\ref{thm-bead-char} implies that a.s.\ $(\BB H , \wt\eta^\bead)$, $(\BB H , \eta^\bead)$, and $(\BB H , \rng\eta^\bead)$ differ by curve-preserving homeomorphisms which also preserve the $\gamma$-quantum boundary length measure on each complementary connected component of the curve. Let $\rng\Phi^\bead : \BB H \rta \BB H$ be such a curve- and boundary measure-preserving homeomorphism from $(\BB H , \rng\eta^\bead)$ to $(\BB H , \wt\eta^\bead)$. 
 
By definition of $\rng\eta^\bead$, for each connected component $\rng D$ of $\BB H\setminus \rng\eta^\bead$, there exists a conformal map $f_{\rng D} : \rng D \rta \rng\Phi^\bead(\rng D)$ which pushes forward $\rng h^\bead|_{\rng D}$ to $\wt h^\bead|_{\rng\Phi^\bead(\rng D)}$ via the $\gamma$-LQG coordinate change formula. 
The map $f_{\rng D}$ preserves the $\gamma$-quantum boundary length measure on $\bdy \rng D$, so $\rng\Phi^\bead|_{\bdy \rng D} = f_{\rng D}|_{\bdy \rng D}$. 
Therefore, if we re-define~$\rng\Phi^\bead$ to be equal to $f_{\rng D}$ on each such connected component~$\rng D$, then~$\Phi^\bead$ remains a curve- and boundary measure-preserving homeomorphism and pushes forward each $\rng h^\bead|_{\rng D}$ to $\wt h^\bead|_{\rng\Phi^\bead(\rng D)}$ via the $\gamma$-LQG coordinate change formula.  

Since condition~\ref{item-bead-char-wedge} depends on the pair $(h^\bead,\eta^\bead)$ only via the function $Z^\bead$, it remains only to check the boundary length consistency statement in condition~\ref{item-bead-char-homeo} with $(\rng h^\bead , \rng\eta^\bead , \rng \Phi^\bead)$ in place of $(h^\bead , \eta^\bead , \Phi^\bead)$. To see this, we observe that condition~\ref{item-bead-char-homeo} in Theorem~\ref{thm-bead-char} (together with the analogous property for $(h^\bead,\eta^\bead)$) imply that a.s.\ $\nu_{\wt h^\bead}(\wt\eta^\bead([0,t]) \cap \wt\eta^\bead([t,\frk a]) ) =0$ for each $t\in [0,\frk a]$. Hence $\nu_{\wt h^\bead}$-a.e.\ point on the boundary of the unbounded connected component of $\wt\eta^\bead([0,t])$ lies on the boundary of a connected component of $\BB H\setminus \wt\eta^\bead$. Since $\rng\Phi^\bead$ preserves the $\gamma$-quantum boundary length measure on each such component, we conclude. 
\end{proof}

Henceforth assume that $(h^\bead,\eta^\bead)$ and $\Phi^\bead$ have been chosen as in Lemma~\ref{lem-bead-char-conformal}. Recall from the discussion above Lemma~\ref{lem-bead-sle} that we have coupled $(\wt h^\bead , \wt\eta^\bead)$ and $(h^\bead,\eta^\bead)$ with $(h ,\eta')$ in such a way that a.s.\ $(\mcl B , \eta_{0,\infty,\mcl B}(\cdot - \ul T) )$ and $((\BB H , h^\bead , 0, \infty) , \eta^\bead)$ agree as curve-decorated quantum surfaces and $(h,\eta')$ is conditionally independent from $(\wt h^\bead , \wt\eta^\bead)$ and $(h^\bead,\eta^\bead)$ given $(\mcl B , \eta_{0,\infty,\mcl B}(\cdot - \ul T) )$. 
We now use this coupling to construct a pair $(\wt h , \wt\eta')$ which we will eventually show satisfies the conditions of Theorem~\ref{thm-wpsf-char}. 

Let $(\eta')^\bead : [0,\frk a] \rta \BB H$ be the space-filling curve which is the image of $\eta'_{\mcl B}(\cdot - \ul T)$ under the embedding~$h^\bead$ of~$\mcl B$ into $(\BB H , 0, \infty)$. Let $(\wt\eta')^\bead := \Phi^\bead \circ ( \eta')^\bead : [0,\frk a] \rta \BB H$. 
By the condition on $\Phi^\bead$ in Lemma~\ref{lem-bead-char-conformal} and our choice of coupling, the conditional law of $(\wt\eta')^\bead$ given $(h^\bead,\eta^\bead)$ is obtained by replacing each of the constant increments $\wt\eta^\bead |_{[\sigma^\bead(t) , \tau^\bead(t)]}$ by an independent space-filling $\SLE_{\kappa'}$ loop based at $\wt\eta^\bead(t)$ in the bubble disconnected from $\infty$ by $\wt\eta^\bead$ at time $\sigma^\bead(t)$ based at $\wt\eta^\bead(t)$, parameterized by $\gamma$-quantum mass with respect to $\wt h^\bead$. 

Let $\wt{\mcl C} = (\BB C , \wt h , 0, \infty)$ be the quantum surface obtained from the $\gamma$-quantum cone $\mcl C$ by replacing the bead $\mcl B$ with the quantum surface $\wt{\mcl B} := (\BB H , \wt h^\bead , 0,\infty) \eqD \mcl B$, equivalently the quantum surface obtained by conformally welding together $(\BB H , \wt h^\bead , 0,\infty)$ and the complementary sub-surface $\mcl C\setminus\mcl B$ according to $\gamma$-quantum lengths along their boundaries. 
The boundary of the region $\eta'([\ul T , \ol T])$ in $\BB C$ which parameterizes $\mcl B$ is a union of two non-crossing $\SLE_\kappa$ segments, for $\kappa = 16/\kappa' \in (0,4)$, so is a.s.\ conformally removable, so there is a unique way to perform this conformal welding and $(\BB C , \wt h , 0, \infty)$ is well-defined. By our choice of coupling, $(h,\eta')$ is conditionally independent from $(\wt h^\bead ,\wt\eta^\bead)$ given $\mcl B$, so $(\BB C, \wt h , 0, \infty)$ has the law of a $\gamma$-quantum cone. 

Let $\wt\eta' : \BB R\rta \BB C$ be the curve which is the image under the conformal welding map of the concatenation of $\eta_{\mcl C\setminus \mcl B}'|_{(-\infty,\ul T]}$, $\wt\eta'_{\wt{\mcl B}}(\cdot + \ul T)$, and $\eta_{\mcl C\setminus \mcl B}'|_{[\ol T , \infty)}$. Then $\wt\eta'$ is a space-filling curve from $\infty$ to $\infty$ parameterized by $\gamma$-quantum mass with respect to $\wt h$. 
The main input in the proof of Theorem~\ref{thm-bead-char} is the following lemma.

\begin{lem} \label{lem-wpsf-to-chordal}
The conditions of Theorem~\ref{thm-wpsf-char} are satisfied for the pairs $(h,\eta')$ and $(\wt h ,\wt\eta')$ defined above.
Hence $(\wt h , \wt\eta')$ is an embedding into $(\BB C , 0 , \infty)$ of a $\gamma$-quantum cone together with an independent whole-plane space-filling $\SLE_{\kappa'}$ from $\infty$ to $\infty$ parameterized by $\gamma$-quantum mass with respect to $\wt h$. 
\end{lem} 

For the proof of Lemma~\ref{lem-wpsf-to-chordal}, 
we recall the definitions of the quantum surfaces $\mcl S_{a,b}$ for $a,b\in\BB R\cup \{-\infty , \infty\}$ and the times $\sigma(t) = \sigma_{0,\infty}(t)$ and $\tau(t) = \tau_{0,\infty}(t)$ for $t \geq 0$ from above, and define the quantum surfaces $\wt{\mcl S}_{a,b}$ for $a < b$ obtained by restricting $\wt h$ to $\wt\eta'([a,b])$, as in~\eqref{eqn-increment-surface}. 

We start by defining the homeomorphism $\Phi : \BB C \rta \BB C$ appearing in condition~\ref{item-wpsf-char-homeo} in Theorem~\ref{thm-wpsf-char}. 
Let $\Phi_{\mcl C} : \mcl C \rta \wt{\mcl C}$ be the homeomorphism which is given by the identity map on the sub-surface of $\mcl C$ parameterized by $\mcl C\setminus \eta_{\mcl C}'([\ul T , \ol T])$ and which is given by the homeomorphism $\Phi^\bead$ from Lemma~\ref{lem-bead-char-conformal} (viewed as a map between quantum surfaces) on the sub-surface of $\mcl C$ parameterized by $ \eta_{\mcl C}'([\ul T , \ol T])$. Let $\Phi : \BB C\rta\BB C$ be the map corresponding to $\Phi_{\mcl C}$ for the embeddings $h$ and $\wt h$. 
The above construction of $\wt\eta'$ implies that $\Phi\circ \eta' = \wt\eta'$. 
By condition~\ref{item-bead-char-homeo} in Theorem~\ref{thm-bead-char} and our construction of $\wt h$, it is a.s.\ the case that $\Phi$ pushes forward the $\gamma$-quantum length measure with respect to $h^\bead$ on the boundary of the unbounded connected component of $\BB H\setminus \wt\eta^\bead([0,t])$ to the  $\gamma$-quantum length measure with respect to $\wt h^\bead$ on the boundary of the unbounded connected component of $\BB H\setminus \wt\eta^\bead([0,t])$ for each $t\in [0,\frk a] \cap \BB Q$. Thus condition~\ref{item-wpsf-char-homeo} in Theorem~\ref{thm-wpsf-char} is satisfied.
 
It remains to check condition~\ref{item-wpsf-char-wedge} in Theorem~\ref{thm-wpsf-char}.
We first observe that, to our choice of the triple $(h^\bead , \eta^\bead , \Phi^\bead)$ from Lemma~\ref{lem-bead-char-conformal} and the conformal removability of $\bdy \eta'([\ul T , \ol T])$, the map $\Phi$ is in fact conformal on $\BB C\setminus \eta_{0,\infty}( [\ul T , \ol T] ) $, with $\eta_{0,\infty}$ the future chordal $\SLE_{\kappa'}$ curve obtained by skipping the bubbles filled in by $\eta'$ during the time interval $[0,\infty)$ as above and pushes forward $h|_{  \BB C\setminus \eta_{0,\infty}( [\ul T , \ol T] ) }$ to the corresponding restriction of $\wt h$ via the $\gamma$-LQG coordinate change formula~\eqref{eqn-lqg-coord}. In particular, we have the equalities of curve-decorated quantum surfaces
\allb 
&\left( \mcl S_{-\infty,\ul T} , \eta'_{\mcl S_{-\infty,\ul T}} \right) = \left( \wt{\mcl S}_{-\infty,\ul T} , \wt\eta'_{\wt{\mcl S}_{-\infty,\ul T}} \right), \quad 
\left( \mcl S_{\ol T, \infty} , \eta'_{\mcl S_{\ol T ,\infty}} \right) = \left( \wt{\mcl S}_{\ol T ,\infty} , \wt\eta'_{\wt{\mcl S}_{\ol T ,\infty}} \right) , 
\quad \op{and} \notag \\
&\left( \mcl S_{\sigma(t) , \tau(t) } , \eta'_{\mcl S_{\sigma(t) , \tau(t) }} \right) = \left( \wt{\mcl S}_{\sigma(t) , \tau(t) } , \wt\eta'_{\sigma(t) , \tau(t) } \right) ,\quad \forall t \geq 0 . \label{eqn-most-surface-agree}
\alle

\begin{lem} \label{lem-bubble-filtration-simplify}
Define the $\sigma$-algebras $\wt{\mcl F}_{-\infty,t}$ and $\wt{\mcl F}_{t,\infty}$ as in Theorem~\ref{thm-wpsf-char} for $(\wt h , \wt\eta')$ as above. 
For each $t\geq 0$, we have $\wt{\mcl F}_{t,\infty} = \sigma\left(  (Z-Z_t)|_{[t,\infty)} \right) $. Furthermore, $\ol T$ is a stopping time for $Z$ and for each $t\geq 0$, $\wt{\mcl F}_{-\infty, t\wedge \ol T} = \sigma\left(   (Z-Z_{t\wedge \ol T} )|_{(-\infty , t\wedge \ol T]}   \right)$.  
\end{lem}
\begin{proof}
For $t\geq 0$, each $\pi/2$-cone interval $[v_Z(s) , s]$ for $Z$ which is contained in $[t,\infty)$ is contained in the maximal $\pi/2$-cone interval $[\sigma(r) , \tau(r)]$ in $[0,\infty)$ for any $r\in (v_Z(s) , s)$. By~\eqref{eqn-most-surface-agree}, a.s.\ $\mcl S_{v_Z(s) , s }  =  \wt{\mcl S}_{v_Z(s) , s } $ 
for each such $\pi/2$-cone interval $[v_Z(s) , s]$. 
Since the peanosphere Brownian motion determines the quantum surfaces corresponding to segments of $\eta'$ in a local manner, we infer that the quantum surfaces $ \wt{\mcl S}_{v_Z(s) , s } $ as $s$ ranges over all $\pi/2$-cone intervals contained in $[t,\infty)$ are a.s.\ determined by $(Z-Z_t)|_{[t,\infty)}$, which gives the first statement of the lemma. 

The time $\ol T$ is the smallest $t\geq 0$ such that the following is true: $L$ and $R$ attain a simultaneous running infimum relative to time $0$ at time $t$ and if $t'$ denotes the last time strictly before $t$ at which $L$ and $R$ attain a simultaneous running infimum, then the vector $(t-t', L_{t'} - L_t , R_{t'}  -R_t ) $ belongs to $\frk A$. Hence $\ol T$ is a stopping time for $Z$ (and hence also for $\{\wt{\mcl F}_{-\infty,t}\}_{t\geq 0}$). 

To prove the formula for $\wt{\mcl F}_{-\infty, t\wedge \ol T}$, let $[v_Z(s) , s]$ be a $\pi/2$-cone interval for~$Z$ in $(-\infty,t \wedge \ol T)$. Since~$L$ and~$R$ attain a simultaneous running infimum relative to time $0$ at time~$\ul T$, it follows that $\ul T$ is a $\pi/2$-cone time for $Z$ with $v_Z(\ul T) < 0$. Since $\ol T$ is the next time after $\ul T$ at which $L$ and $R$ attain a simultaneous running infimum relative to time 0, it cannot be the case that $s \in (\ul T , t\wedge \ol T)$ and $v_Z(s) < 0$. 
Since two $\pi/2$-cone intervals for $Z$ are either nested or disjoint, either $[v_Z(s) , s] \subset (-\infty,\ul T]$ or $[v_Z(s) , s] \subset [\ul T , t \wedge \ol T]$. In the latter case, $[v_Z(s) , s]$ is contained in a maximal $\pi/2$-cone time for $Z$ in $[0,\infty)$. By~\eqref{eqn-most-surface-agree}, a.s.\ $\mcl S_{v_Z(s) , s }  =  \wt{\mcl S}_{v_Z(s) , s } $ 
for each such $\pi/2$-cone interval $[v_Z(s) , s]$ and we conclude as in the case of $\wt{\mcl F}_{t,\infty}$.  
\end{proof}

\begin{lem} \label{lem-bead-wpsf-char-ind}
For each $t\in\BB R$, the $\sigma$-algebras $\wt{\mcl F}_{-\infty,t}$ and $\wt{\mcl F}_{t,\infty}$ from the statement of Theorem~\ref{thm-wpsf-char} are independent.  
\end{lem}
\begin{proof}
If $t < 0$, then our definition of $\wt h , \wt\eta'$ shows that $(\wt{\mcl S}_{-\infty ,  t} , \wt\eta'_{\wt{\mcl S}_{-\infty,t}}) =  ( {\mcl S}_{-\infty,t} ,  \eta'_{ {\mcl S}_{-\infty,t}})$ is independent from the pair consisting of $(h^\bead , \eta^\bead)$ and the future quantum wedge $ ( {\mcl S}_{t,\infty} ,  \eta'_{ {\mcl S}_{t,\infty}})$. Since this curve-decorated quantum surface and this pair of curve-decorated quantum surfaces determine $\wt{\mcl F}_{-\infty,t}$ and $\wt{\mcl F}_{t,\infty}$, respectively, this gives the lemma statement for $t < 0$. 
Hence we can restrict attention to the case when $t\geq 0$, which is more involved since to construct $\wt h , \wt\eta'$ we modified $  h|_{ \eta'([0,\infty)}$ and $\eta'|_{[0,\infty)}$. 

By Lemma~\ref{lem-bubble-filtration-simplify}, for $t\geq 0$ the conditional law of $(Z-Z_t)|_{[t ,\infty)}$ given $\wt{\mcl F}_{-\infty,t}$ on the event $\{t \leq \ol T\}$ is the same as its marginal law. 
Since $(\wt{\mcl S}_{\ul T , \ol T} , \wt\eta'_{\wt{\mcl S}_{\ul T , \ol T}})$ is conditionally independent from $(h,\eta')$ (and hence from $Z$) given $(\mcl S_{\ul T , \ol T} , \eta'_{\mcl S_{\ul T , \ol T}})$, the conditional law of $(Z-Z_t)|_{[t,\infty)}$ given $(Z-Z_t)|_{(-\infty,t]}$ and $(\wt{\mcl S}_{\ul T , \ol T} , \wt\eta'_{\wt{\mcl S}_{\ul T , \ol T}})$ on the event $\{ t > \ol T\}$ is the same as its marginal law. 
The $\sigma$-algebra $\wt{\mcl F}_{-\infty,t}$ is clearly contained in the $\sigma$-algebra generated by $(Z-Z_t)|_{(-\infty,t]}$ and $(\wt{\mcl S}_{\ul T , \ol T} , \wt\eta'_{\wt{\mcl S}_{\ul T , \ol T}})$. Hence the conditional law of $(Z-Z_t)|_{[t,\infty)}$ given $\wt{\mcl F}_{-\infty,t}$ is the same as its marginal law. By Lemma~\ref{lem-bubble-filtration-simplify}, $\wt{\mcl F}_{t,\infty} = \sigma\left(  (Z-Z_t)|_{[t,\infty)} \right) $ so the statement of the lemma follows. 
\end{proof}

We now check the last remaining condition of Theorem~\ref{thm-wpsf-char}.  

\begin{lem} \label{lem-bead-wpsf-char-wedge}
For each $t\in\BB R$, the future quantum surface $\wt{\mcl S}_{t,\infty}$ has the law of a $\frac{3\gamma}{2}$-quantum wedge and is independent from $\wt{\mcl F}_{-\infty,t}$.
\end{lem}
\begin{proof}
It is clear that the statement of the lemma holds for $t < 0$. 
To check the condition for $t > 0$, we define for $r\in\BB R$ the $\sigma$-algebra 
\eqb \label{eqn-partial-sigma-algebra}
\wt{\mcl G}_r := \wt{\mcl F}_{-\infty , r} \vee \sigma(P_{r,\infty})  
\eqe 
as in Lemma~\ref{lem-partial-filtration}. Then both $\ul T$ and $\ol T$ are $\wt{\mcl G}_0$-measurable. Also define for $t\geq 0$
\eqb \label{eqn-ul-tau}
\ul \tau(t) := 
\begin{cases}
\ul T ,\quad &t < \ul T \\
\tau(t) ,\quad &t \in [ \ul T , \ol T] \\
t ,\quad &t > \ol T .
\end{cases} 
\eqe 
where $\tau(t ) = \tau_{0,\infty}(t)$ is as above. Since $\ul T$ and $\ol T$ are determined by $P_{0,\infty}$ and, whenever $t \leq \ol T$, $\tau(t)$ is the left endpoint of the interval on which $P_{t,\infty}$ is constant whose right endpoint is $\ol T$, we have $\ul\tau(t) \in \wt{\mcl G}_t$.  

The beaded quantum surface $\wt{\mcl S}_{t,\infty}$ is the concatenation of the beaded quantum surfaces $ \wt{\mcl S}_{t , \ul\tau(t)}$, $\wt{\mcl S}_{\ul\tau(t) , \ol T \vee t}$, and $\wt{\mcl S}_{\ol T\vee t, \infty}$, with $\ul\tau(t)$ as in~\eqref{eqn-ul-tau} (note that the first two of these quantum surfaces are degenerate if $t > \ol T$). We will now consider the conditional laws of these quantum surfaces given the $\sigma$-algebra $\wt{\mcl G}_t$ of~\eqref{eqn-partial-sigma-algebra}. 

Since $\wt\eta'([t,\ul\tau(t)])$ is either a single point or is contained in either $\wt\eta'([t, \ul T])$ or in one of the bubbles cut out by $\wt\eta_{0,\infty}$,~\eqref{eqn-most-surface-agree} implies that a.s.\ $ \wt{\mcl S}_{t , \ul\tau(t)}   =  \mcl S_{t , \ul\tau(t)} $. 
The conditional law of this quantum surface given $\wt{\mcl G}_t$ is that of an ordered collection of beads of a $\frac{3\gamma}{2}$-quantum wedge with areas and left/right boundary lengths specified by $P_{t,\infty}|_{[t,\ul\tau(t)]}$.    

The quantum surface $\wt{\mcl S}_{\ul\tau(t) , \ol T \vee t}$ is equal to $\wt{\mcl S}_{\ul T , \ol T}$ if $t < \ul T$, is equal to the quantum surface obtained by restricting $\wt h^\bead$ to the unbounded connected component of $\BB H\setminus \wt\eta^\bead([0,t])$ if $t\in [\ul T ,\ol T]$, and is equal to a single point if $t > \ol T$. 
By condition~\ref{item-bead-char-wedge} in Theorem~\ref{thm-bead-char} and Lemma~\ref{lem-bubble-filtration-simplify}, we infer that the conditional law of $\wt{\mcl S}_{\ul\tau(t) , \ol T \vee t}$ given $\wt{\mcl G}_{t}$ and $ \wt{\mcl S}_{t , \ul\tau(t)}  $ is that of a single bead of a $\frac{3\gamma}{2}$-quantum wedge with area and left/right boundary lengths given by the value of $P_{t,\infty}$ on $[\ul\tau(t) , \ol T]$ (or a single point, if this interval is empty).  
By combining this with the previous paragraph, we find that the conditional law of $\wt{\mcl S}_{t , \ol T \vee t} $ given $\wt{\mcl G}_{t}$ is that of an ordered collection of beads of a $\frac{3\gamma}{2}$-quantum wedge with areas and left/right boundary lengths specified by $P_{t,\infty}|_{[t, \ol T \vee t]}$.     

We have $\wt{\mcl S}_{\ol T \vee t , \infty} = \mcl S_{\ol T \vee t , \infty}$, and the conditional law of this quantum surface given $\wt{\mcl G}_t$ and $\wt{\mcl S}_{t , \ol T \vee t} $ is that of a collection of independent beads of a $\frac{3\gamma}{2}$-quantum wedge with areas and left/right boundary lengths specified by $P_{t,\infty}|_{[\ol T\vee t , \infty)}$. Hence the conditional law of $\wt{\mcl S}_{t,\infty}$ given $\wt{\mcl G}_t$ is that of an ordered collection of beads of a $\frac{3\gamma}{2}$-quantum wedge with areas and left/right boundary lengths specified by $P_{t,\infty}$. This is the same as the conditional law of $\mcl S_{t,\infty}$ given $P_{t,\infty}$. Averaging over all possible realizations of $P_{t,\infty}$ shows that the conditional laws of $\wt{\mcl S}_{t,\infty}$ and $\mcl S_{t,\infty}$ given $\wt{\mcl F}_{-\infty,t}$ agree. That is, $\wt{\mcl S}_{t,\infty}$ is a $\frac{3\gamma}{2}$-quantum wedge independent from $\wt{\mcl F}_{-\infty,t}$, as required.  
\end{proof}

\begin{proof}[Proof of Lemma~\ref{lem-wpsf-to-chordal}]
Combining the above results shows that the conditions of Theorem~\ref{thm-wpsf-char} are satisfied for the pairs $(h,\eta')$ and $(\wt h ,\wt\eta')$ defined above, and the last statement of Lemma~\ref{lem-wpsf-to-chordal} follows from Theorem~\ref{thm-wpsf-char}. 
\end{proof}

We can now conclude the proof of Theorem~\ref{thm-bead-char}. The main remaining obstacle is to transfer from the case where $(\frk a , \frk l^L , \frk l^R)$ is random (which we have been considering throughout this subsection) to the case when $(\frk a , \frk l^L , \frk l^R)$ is deterministic. 

\begin{proof}[Proof of Theorem~\ref{thm-bead-char}]
Let $\BB m$ be the infinite measure on beads of a $\frac{3\gamma}{2}$-quantum wedge and let $\frk A\subset (0,\infty)^3$ be a Borel set such that $\BB m(\frk A)$ is finite and positive, as in the beginning of this subsection. 
Lemma~\ref{lem-wpsf-to-chordal} together with Theorem~\ref{thm-wpsf-char} immediately imply the variant of Theorem~\ref{thm-bead-char} where $(\frk a , \frk l^L , \frk l^R)$ is sampled from $\BB m(\frk A)^{-1} \BB m|_{\frk A}$ (i.e., the setting of most of this section). 

Since $\BB m(\frk A)^{-1} \BB m|_{\frk A}$ is mutually absolutely continuous with respect to Lebesgue measure on $\frk A$ (see the proof of Lemma~\ref{lem-bead-sle-as}), we find that the theorem statement remains true if instead $(\frk a , \frk l^L , \frk l^R)$ is sampled from any probability measure $P$ on $(0,\infty)^3$ which is mutually absolutely continuous with respect to Lebesgue measure restricted to some Borel subset of $(0,\infty)^3$. 

Suppose now that $(\frk a , \frk l^L, \frk l^R)$ is deterministic, as in the theorem statement. The idea of the proof is to conformally map out a small initial segment of the curve $\wt\eta'$ in order to obtain a new curve-decorated quantum surface which satisfies the hypotheses of the theorem statement for a random choice of area and left/right boundary lengths (so we can apply the preceding paragraph). We will then send the size of the small initial segment to zero to conclude. 

 For $t\in [0,\frk a]$, let $\wt{\mcl W}^\bead_t$ be the doubly marked quantum surface parameterized by the unbounded connected component of $\BB H\setminus \wt\eta^\bead([0,t])$, with marked points $\wt\eta^\bead(t)$ and $\infty$, so that by condition~\ref{item-bead-char-wedge} in the theorem statement, $\wt{\mcl W}_t^\bead$ is a bead of a $\frac{3\gamma}{2}$-quantum wedge. Let $\mcl W^\bead_t$ be defined analogously with the pair $(h^\bead,\eta^\bead)$ in place of the pair $(\wt h^\bead , \wt\eta^\bead)$. 
If $\ep >0$ and we condition on $Z^\bead|_{[0,\ep]}$ and $\tau^\bead(\ep)$, then the conditions in the theorem statement are satisfied with embeddings into $(\BB H , 0, \infty)$ of the pair of curve-decorated quantum surfaces $(\wt{\mcl W}^\bead_\ep , \wt\eta^\bead_{\wt{\mcl W}^\bead_\ep})$ and $(\mcl W^\bead_\ep , \eta^\bead_{\mcl W^\bead_\ep})$ in place of $(h^\bead,\eta^\bead)$ and $(\wt h^\bead , \wt\eta^\bead)$; the restriction of $\Phi^\bead$ to the unbounded connected component of $\BB H\setminus \eta^\bead([0,t])$, post-composed and pre-composed with appropriate conformal maps, in place of $\Phi^\bead$; and the triple $(\frk a - \tau^\bead(\ep) , L^\bead_\ep , R^\bead_\ep)$ in place of $(\frk a , \frk l^L , \frk l^R)$. 
By local absolute continuity of $h$ with respect to an embedding into $(\BB H , 0 , \infty)$ of a bead of a $\frac{3\gamma}{2}$-quantum wedge with random area and boundary length, we infer that the joint law of $(\frk a - \tau^\bead(\ep) , L^\bead_\ep , R^\bead_\ep)$ is mutually absolutely continuous with respect to Lebesgue measure on $(0,\frk a - \ep) \times (0,\infty)^2$. 
Applying the theorem in the case when $(\frk a , \frk l^L ,\frk l^R)$ is random and sending $\ep \rta 0$ now yields the theorem in the case when $(\frk a , \frk l^L, \frk l^R)$ is deterministic. 
\end{proof}

\section{ Characterizations of $\SLE_6$ on a Brownian surface }
\label{sec-metric-char}

\subsection{Space-filling and chordal statements}
\label{sec-metric-char-result}
 
In the special case when $\kappa'=6$ (so $\gamma = \sqrt{8/3}$), the $\sqrt{8/3}$-LQG metric on a $\sqrt{8/3}$-LQG surface is well-defined (recall Section~\ref{sec-lqg-metric}).  In this case, we can re-phrase Theorems~\ref{thm-wpsf-char} and~\ref{thm-bead-char} in terms of metric measure spaces rather than quantum surfaces.

For the theorem statements in this subsection, we recall the definition of the internal metric: if $(X,d)$ is a metric space and $Y\subset X$, the internal metric $d_Y$ of $d$ on $Y$ is defined by setting $d_Y(y_1,y_2)$ for $y_1,y_2 \in Y$ to be the infimum of the $d$-lengths of all paths in $Y$ from $y_1$ to $y_2$.  Note that $d_Y(y_1,y_2)$ may be infinite.

In order to make sense of random metric measure spaces, we endow the space of compact finite metric measure spaces with the \emph{Gromov-Prokhorov} or \emph{Gromov-weak} topology~\cite{gpw-metric-measure}, which is the same topology used in~\cite{tbm-characterization}. This is the topology whereby two metric measure spaces are close if they can be isometrically embedded into a common metric space such that the corresponding measures are close in the Prokhorov distance. In the case of locally compact, locally finite metric length spaces (such as the Brownian plane), we instead use the local variant of the Gromov-Prokhorov topology, the \emph{Gromov-vague} topology, which is defined in~\cite{alw-gromov-vague}.
 
We define a \emph{singly (resp.\ doubly) marked Brownian disk} to be a Brownian disk together with one (resp.\ two) marked points sampled uniformly from its natural boundary measure. A doubly marked Brownian disk has a  notion of left and right quantum boundary lengths, corresponding to the lengths of the clockwise and counterclockwise boundary arcs between the two marked points (we take the marked points to be ordered). In the special case when $\gamma = \sqrt{8/3}$, a single bead of a $(\frac{3\gamma}{2} = \sqrt 6)$-quantum wedge is the same as a quantum disk with two marked boundary points (c.f.\ Section~\ref{sec-wedge}), which in turn is equivalent as a metric measure space to a doubly marked Brownian disk~\cite[Corollary~1.5]{lqg-tbm2}. 
We also recall that by~\cite[Corollary~1.5]{lqg-tbm2}, the $\sqrt{8/3}$-quantum cone is equivalent to the Brownian plane~\cite{curien-legall-plane}. 
 
We first state a metric space version of Theorem~\ref{thm-wpsf-char}, which characterizes a coupling of an instance $(\wt X ,\wt d , \wt\mu , \wt x)$ of the Brownian plane (equipped with its natural metric, area measure, and marked point), a curve $\wt\eta' : \BB R\rta \wt X$ with $\wt\eta'(0) = 0$, and a correlated two-dimensional Brownian motion $Z$ with variances and covariance as in~\eqref{eqn-bm-cov} for $\kappa'=6$.  For the statement of the theorem, we need to define $\sigma$-algebras analogous to the $\sigma$-algebras $\wt{\mcl F}_{a,b}$ in Theorem~\ref{thm-wpsf-char} (but in terms of the metric measure space rather than quantum surface structure).  For $a, b \in \BB R \cup \{-\infty,\infty\}$, with $a<b$, let $\wt X_{a,b}$ be the interior of $\wt\eta'([a,b])$ and let $\wt d_{a,b}$ be the internal metric of $\wt d$ on $\wt X_{a,b}$. Also let $\wt{\mcl F}_{a,b}^{\op{m}}$ be the $\sigma$-algebra generated by $(Z-Z_a)|_{[a,b]}$ (or $(Z-Z_b)|_{[a,b]}$ if $a = - \infty$) and the pointed metric measure spaces $(\wt X_{v_Z(s) , s} , \wt d_{v_Z(s) , s} , \wt\mu|_{\wt X_{v_Z(s) , s}}, \wt\eta' (s) )$ where $s$ ranges over all $\pi/2$-cone times for $Z$ which are maximal in some interval contained in $(a,b)$ with rational endpoints. Here we recall that $v_Z(s)$ is the start time of the $\pi/2$-cone excursion (Definition~\ref{def-cone-time}).
 
\begin{thm}[Characterization of whole-plane space-filling SLE$_6$ on the Brownian half-plane] \label{thm-wpsf-mchar}
Let $\kappa'=6$ and $\gamma = \sqrt{8/3}$. Suppose that $(\wt X ,\wt d , \wt\mu , \wt x)$ is a coupling of an instance of the Brownian plane (equipped with its natural metric, area measure, and marked point), a curve $\wt\eta' : \BB R\rta \wt X$ with $\wt\eta'(0) = 0$, and a correlated two-dimensional Brownian motion $Z$ with variances and covariance as in~\eqref{eqn-bm-cov} for $\kappa'=6$.  Assume that the following conditions are satisfied.
\begin{enumerate}
\item \label{item-wpsf-mchar-wedge} (Markov property) For each $t\in\BB R$, the $\sigma$-algebras $\wt{\mcl F}_{-\infty,t}^{\op{m}}$ and $\wt{\mcl F}_{t,\infty}^{\op{m}}$ defined just above are independent. Furthermore, for each $t\in\BB R$ the ordered collection of connected components of the metric measure space $(\wt X_{t,\infty} ,\wt d_{t,\infty} ,  \wt\mu|_{\wt X_{t,\infty}})$ (in the order they are filled in by $\wt\eta'$) has the same law as the beads of a $\sqrt 6$-quantum wedge equipped with their $\sqrt{8/3}$-LQG metrics and area measures, i.e.\ its law is that of a Poissonian collection of doubly marked Brownian disks, lying at infinite internal distance from each other, with areas and left/right boundary lengths specified by the increments between the times when $L$ and $R$ attain a simultaneous running infimum relative to time $t$ and the increments between the values of $L$ and $R$ at these times. Furthermore, the ordered collection of connected components of the metric measure space $(\wt X_{t,\infty} ,\wt d_{t,\infty} ,  \wt\mu|_{\wt X_{t,\infty}}   )$ is independent from $\wt{\mcl F}_{-\infty,t}^{\op{m}}$. 
\item \label{item-wpsf-mchar-homeo}   (Topology and consistency) The curve-decorated topological space $(\wt X ,  \wt\eta')$ is equivalent to the infinite-volume peanosphere generated by $Z$. Equivalently, if $((\BB C , h , 0, \infty) ,\eta')$ is the pair consisting of a $\gamma$-quantum cone and an independent space-filling $\SLE_{\kappa'}$ from $\infty$ to $\infty$ parameterized by $\gamma$-quantum mass with respect to $h$ which is determined by $Z$ via~\cite[Theorem~1.11]{wedges}, then there is a homeomorphism $\Phi : \BB C\rta\wt X$ with $\Phi\circ \eta' = \wt\eta'$. Moreover, $\Phi$ a.s.\ pushes forward the $\sqrt{8/3}$-quantum length measure on $\bdy\eta'([t,\infty))$ with respect to $h$ to the natural length measure on $\bdy\wt\eta'([t,\infty))$ (which is defined via the length measures on the Brownian disks which are the connected components of $\wt X_{t,\infty}$).
\end{enumerate}
Then $(\wt X , \wt d , \wt \mu , \wt \eta')$ is equivalent (as a curve-decorated metric measure space) to a $\sqrt{8/3}$-quantum cone decorated by an independent whole-plane space-filling $\SLE_{\kappa'}$ from $\infty$ to $\infty$ parameterized by $\sqrt{8/3}$-quantum mass with respect to $h$. In fact, the map $\Phi$ of condition~\ref{item-wpsf-char-homeo} is a.s.\ an isometry.
\end{thm}

We next state a metric space version of Theorem~\ref{thm-bead-char}, which is a precise version of Theorem~\ref{thm-bead-mchar-intro}. 

\begin{thm}[Characterization of chordal SLE$_6$ on the Brownian disk] \label{thm-bead-mchar}
Suppose $\kappa' = 6$ and $\gamma = \sqrt{8/3}$. 
Let $(\frk a , \frk l^L, \frk l^R )\in (0,\infty)^3$ and suppose we are given a coupling of a doubly-marked Brownian disk $(\wt X^\bead , \wt d^\bead , \wt \mu^\bead , \wt x , \wt y)$ with area $\frk a$ and left/right boundary lengths $\frk l^L$ and $\frk l^R$, a random continuous curve $\wt\eta^\bead : [0,\frk a] \rta \wt X^\bead$ from $\wt x$ to $\wt y$, parameterized by the $\wt\mu^\bead$-mass it disconnects from $\infty$ (Definition~\ref{def-chordal-parameterization}), and a random discontinuous process $Z^\bead = (L^\bead, R^\bead)$ as described in Section~\ref{sec-chordal} for $\kappa'=6$.  Assume that the following conditions are satisfied.
\begin{enumerate}
\item \label{item-bead-mchar-wedge}  (Laws of complementary connected components) For $t\in [0,\frk a]$ let~$\wt{\mcl U}^\bead_t$ be the collection of singly marked metric measure spaces of the form $(U , \wt d_U^\bead , \wt\mu^\bead|_U , \wt x_U^\bead)$ where~$U$ is a connected component of~$\wt X^\bead \setminus \wt\eta^\bead([0,t])$, $\wt d_U^\bead$ is the internal metric of~$\wt d^\bead$ on $U$, and $\wt x_U^\bead$ is the point where~$\wt\eta^\bead$ finishes tracing $\bdy U$.  If we condition on $Z^\bead|_{[0,t]}$ and the time $\tau^\bead(t)$ from~\eqref{eqn-bead-bubble-process}, then the conditional law of $\wt{\mcl U}^\bead_t$ is that of a collection of independent singly marked Brownian disks with areas and boundary lengths specified as follows. The elements of $\wt{\mcl U}^\bead_t$ corresponding to the connected components of $\wt X^\bead \setminus \wt\eta^\bead([0,t])$ which do not have the target point $\wt\eta^\bead(\frk a)$ on their boundaries are in one-to-one correspondence with the intervals of time in $[0,\tau^\bead(t)]$ on which $Z^\bead$ is constant, and their areas and boundary lengths are determined by $Z^\bead$ as in~\eqref{eqn-bead-bubble-process}. The element of $\wt{\mcl U}^\bead_t$ corresponding to the connected component of $\wt X^\bead \setminus \wt\eta^\bead([0,t])$ with $\wt\eta^\bead(\frk a)$ on its boundary has area $\frk a - \tau^\bead(t)$ and boundary length $L^\bead_t + R^\bead_t  $. 
\item \label{item-bead-mchar-homeo}  (Topology and consistency) 
There is a pair $((\BB H , h^\bead , 0, \infty) ,\eta^\bead)$ consisting of a doubly marked quantum disk with area $\frk a$ and left/right boundary lengths $\frk l^L$ and $\frk l^R$ and an independent chordal $\SLE_6$ from $0$ to $\infty$ in $\BB H$ parameterized by the $\mu_{h^\bead}$-mass it disconnects from~$\infty$ and a homeomorphism $\Phi^\bead : \BB H\rta \wt X^\bead$ with $\Phi^\bead\circ \eta^\bead  = \wt\eta^\bead $. Moreover, for each $t\in [0,\frk a] \cap \BB Q$, $\Phi^\bead$ a.s.\ pushes forward the $\sqrt{8/3}$-quantum length measure with respect to $h^\bead$ on the boundary of the unbounded connected component of $\BB H\setminus \eta^\bead([0,t])$ to the natural boundary length measure on the connected component of $\wt X^\bead \setminus \wt\eta^\bead([0,t])$ with $\wt\eta^\bead(\frk a)$ on its boundary (which is well-defined since we know the internal metric on this component is that of a Brownian disk).  
\end{enumerate}
Then $(\wt X^\bead , \wt d^\bead , \wt \mu^\bead , \wt\eta^\bead)$ is equivalent (as a curve-decorated metric measure space) to a doubly-marked quantum disk with area $\frk a$ and left/right boundary lengths $\frk l^L$ and $\frk l^R$ equipped with its $\sqrt{8/3}$-LQG area measure and metric together with an independent chordal $\SLE_{6}$ from $0$ to $\infty$ in $\BB H$ parameterized by the $\sqrt{8/3}$-quantum mass it disconnects from $\infty$. 
\end{thm}

In the remainder of this section we will prove Theorems~\ref{thm-wpsf-mchar} and~\ref{thm-bead-mchar}.   In Section~\ref{sec-bead-mchar-nat}, we will also deduce a variant of Theorem~\ref{thm-bead-mchar} where we parameterize by quantum natural time instead of disconnected quantum area. In fact, this is the precise characterization statement which we use in~\cite{gwynne-miller-perc}.

We know from~\cite[Theorem~1.4]{lqg-tbm3} that the metric measure space structure and the quantum surface structure of a $\sqrt{8/3}$-LQG surface a.s.\ determine each other, i.e.\ each is a.s.\ given by a deterministic measurable function of the other.
However, this together with Theorem~\ref{thm-wpsf-char} does not immediately imply Theorem~\ref{thm-wpsf-mchar} for the following reason. 
The deterministic function in~\cite[Theorem~1.4]{lqg-tbm3} is allowed to depend on the particular law of the quantum surface, so it is a priori possible, e.g., that two quantum surfaces with the same metric measure space structure are not equivalent if they have different laws.  
The hypotheses of Theorem~\ref{thm-wpsf-mchar} only tell us about the laws of the metric measure space structures of certain quantum surfaces, so we cannot immediately conclude anything about the conformal structures of these quantum surfaces. Similar considerations apply in the setting of Theorem~\ref{thm-bead-mchar}.

To get around the above difficulty, we will consider in Section~\ref{sec-metric-function} a single metric ball in a $\sqrt{8/3}$-LQG surface and argue using local absolute continuity that the quantum surface structure of this metric ball is a.s.\ determined by its metric measure space structure \emph{in a manner that does not depend on the particular law of the quantum surface} (Proposition~\ref{prop-metric-function}). We will also need a result to the effect that the only isometry from a $\sqrt{8/3}$-LQG metric ball to itself is the identity map (Proposition~\ref{prop-unique-isometry}). This latter result will allow us to show that certain isometries between $\sqrt{8/3}$-LQG surfaces have to be conformal by producing a conformal map which relates the corresponding fields via the $\sqrt{8/3}$-LQG coordinate change formula (such a map is automatically an isometry). 

Using the above results, we will reduce Theorems~\ref{thm-wpsf-mchar} and~\ref{thm-bead-mchar}, respectively, to Theorems~\ref{thm-wpsf-char} and~\ref{thm-bead-char} by covering the various surfaces appearing in the theorem statements by metric balls (Section~\ref{sec-wpsf-char}).
In Section~\ref{sec-bead-mchar-nat}, we will state and prove the aforementioned variant of Theorem~\ref{thm-bead-mchar} where we parameterize by quantum natural time.

Throughout this section, we will use the following notation. For a GFF-type distribution $h$ on a domain $D\subset \BB C$ we write $\frk d_h$, $\mu_h$, and $\nu_h$, respectively, for the $\sqrt{8/3}$-LQG metric, area measure, and length measure induced by $h$. 
For $r > 0$ and $z\in D$, we write $B_r(z;\frk d_h)$ for the closed $\frk d_h$-ball of radius $r$ centered at $z$. 
If $D \not=\BB C$ and $B_r(z;\frk d_h)$ does not intersect $\bdy D$ (resp.\ if $D  = \BB C$), we define the filled metric ball $B_r^\bullet(z; \frk d_h)$ to be the union of $B_r(z;\frk d_h)$ and the set of points which it disconnects from $\bdy D$ (resp.\ $\infty$).

\subsection{Filled metric balls in $\sqrt{8/3}$-LQG surfaces}
\label{sec-metric-function}

In this subsection we will prove some general facts about the filled $\sqrt{8/3}$-LQG metric balls $B_r^\bullet(z ;\frk d_h)$ for $(D,h)$ a $\sqrt{8/3}$-LQG surface.
We recall from~\cite[Corollary~1.5]{lqg-tbm2} that $\sqrt{8/3}$-LQG surfaces are equivalent to Brownian surfaces: in particular, the Brownian map (resp.\ plane, disk) is equivalent as a metric measure space to the quantum sphere (resp.\ $\sqrt{8/3}$-quantum wedge, quantum disk) equipped with its $\sqrt{8/3}$-LQG metric and area measure.

The first main result of this section, whose proof will occupy most of the section, says that the quantum surface structure of a filled $\sqrt{8/3}$-LQG metric ball a.s.\ determines and is determined by its metric measure space structure in a manner which does not depend on the particular law of the quantum surface (see the end of Section~\ref{sec-metric-char-result} for a discussion of why we need this result).  

\begin{prop} \label{prop-metric-function}
For each $r > 0$, there exists a deterministic function $F_r$ from the set of pointed metric measure spaces to the set of singly marked quantum surfaces and a deterministic function $G_r$ going in the opposite direction such that the following is true. 
Let $(D , h )$ be a quantum surface with the following property. For each bounded open set $V\subset D$ lying at positive distance from $\bdy D$ and each $R>0$ such that $\ol V\subset B_R(0)$, the field $h|_V$ is absolutely continuous with respect to the corresponding restriction of a whole-plane GFF normalized so that its circle average over $\bdy B_R(0)$ is zero. 
Let $U$ be any open subset of $D$ such that $\mu_h(U) <\infty$ a.s.\ and let $z$ be sampled uniformly from $\mu_h|_U$, normalized to be a probability measure.
On the event $\{\frk d_h(z , \bdy U)  >  r\}$, a.s.\ 
\begin{align} \label{eqn-metric-function}
F_r\left( B_r^\bullet(z ; \frk d_h) , \frk d_{h |_{B_r^\bullet(z ; \frk d_h)}  } , \mu_h |_{B_r^\bullet(z ; \frk d_h)}  , z \right) &= \left( B_r^\bullet(z ; \frk d_h) ,  h |_{B_r^\bullet(z ; \frk d_h)} , z \right) \quad \op{and} \notag\\
G_r\left( B_r^\bullet(z ; \frk d_h) ,  h |_{B_r^\bullet(z ; \frk d_h)} , z \right) &= \left( B_r^\bullet(z ; \frk d_h) , \frk d_{h |_{B_r^\bullet(z ; \frk d_h)}  } , \mu_h |_{B_r^\bullet(z ; \frk d_h)}  , z \right) . 
\end{align}
\end{prop}

We emphasize that the functions $F_r$ and $G_r$ appearing in Proposition~\ref{prop-metric-function} do \emph{not} depend on the particular law of the quantum surface $(D,h)$. 

We will deduce Proposition~\ref{prop-metric-function} from a similar but less general statement which applies only in the case of a $\sqrt{8/3}$-quantum cone and which follows from~\cite[Theorem 1.4]{lqg-tbm3}, together with a local absolute continuity argument. 
In the next several lemmas, we let $(\BB C , h , 0, \infty)$ be a $\sqrt{8/3}$-quantum cone and to lighten notation we write
\eqbn
B_r^\bullet := B_r^\bullet(0 ; \frk d_h)  .
\eqen
The main input in the proof of Proposition~\ref{prop-metric-function} is the following lemma. 

\begin{lem} \label{lem-metric-function-cone}
For each $r>0$, the pointed metric measure space $\left( B_r^\bullet  , \frk d_{h|_{B_r^\bullet} } , \mu_h|_{B_r^\bullet}  ,0 \right)$ and the quantum surface $\left( B_r^\bullet  ,  h|_{B_r^\bullet} , 0 \right)$ a.s.\ determine each other.  
In other words, there exist deterministic functions $F_r$ and $G_r$ such that~\eqref{eqn-metric-function} holds a.s.\ with $(\BB C ,h , 0, \infty)$ a $\gamma$-quantum cone and $z = 0$.   
\end{lem}
 
For the proof of Lemma~\ref{lem-metric-function-cone}, we first need to collect some facts about filled metric balls which follow from results in~\cite{lqg-tbm1,lqg-tbm2,lqg-tbm3}.  
Our first lemma will be used in particular to relate $h|_{B_r^\bullet}$ to the internal metric of $\frk d_h$ on $B_r^\bullet$.

\begin{lem} \label{lem-internal-metric}
Almost surely, for each open set $U\subset \BB C$ the $\sqrt{8/3}$-LQG metric $\frk d_{h|_U}$ is equal to the internal metric of $\frk d_h$ on $U$. 
\end{lem} 
\begin{proof}
This is essentially obvious from the construction of the $\sqrt{8/3}$-LQG metric, but we give a careful justification for the sake of completeness. 
Let $\mcl V$ be the set of those open sets $V\subset \BB C$ which can be expressed as a finite union of open Euclidean balls with rational radii whose center has rational coordinates.
By~\cite[Lemma~2.2]{gwynne-miller-gluing}, it is a.s.\ the case that for each $V\in \mcl V$, the $\sqrt{8/3}$-LQG metric $\frk d_{h|_V}$ induced by $h|_V$ is equal to the internal metric of $h$ on $V$. 

Suppose now that this is the case (which happens with probability 1) and fix an open set $U\subset \BB C$. Almost surely, there is an increasing sequence $\{V_n\}_{n\in\BB N}$ of open sets belonging to $\mcl V$ whose closures are contained in $U$ and whose union is $U$. 
By definition (c.f.~\cite[Section~2.2.2]{gwynne-miller-gluing}), for each $z,w \in U$, 
\eqbn
\frk d_{h|_U}(w_1,w_2) = \lim_{n\rta\infty} \frk d_{h|_{V_n}}(z,w) .
\eqen
Each path from $z$ to $w$ contained in $U$ is contained in $V_n$ large enough $n\in\BB N$, in which case its $\frk d_{h|_{V_n}}$-length is the same as its $\frk d_{h|_U}$-length and its $\frk d_h$-length. The statement of the lemma follows.
\end{proof}

We next collect some facts about the laws of the quantum surfaces parameterized by the filled metric ball $B_r^\bullet$ and its complement. 

\begin{lem} \label{lem-inside-outside}
Let $r > 0$ and let $w$ be sampled uniformly from $\nu_h|_{\bdy B_r^\bullet}$, normalized to be a probability measure.
The quantum surfaces
\eqbn
\left( B_r^\bullet   ,  h|_{B_r^\bullet} , 0 , w \right) \quad \op{and} \quad \left( \BB C\setminus B_r^\bullet  ,   h|_{\BB C\setminus B_r^\bullet } , \infty , w \right)
\eqen
parameterized by a filled metric ball and its complement are conditionally independent given the quantum length $\nu_h(\bdy B_r^\bullet)$. 
\end{lem}
\begin{proof}
This follows from the construction of the metric $\frk d_h$ via QLE$(8/3, 0)$ in~\cite{lqg-tbm1,lqg-tbm2}.
\end{proof}

\begin{lem} \label{lem-tbm-determined}
Let $r > 0$ and let $w$ be sampled uniformly from $\nu_h|_{\bdy B_r^\bullet}$, normalized to be a probability measure.
Almost surely, $(\BB C, \mu_h , \frk d_h , 0)$ is the metric space quotient of 
$\left( B_r^\bullet  , \frk d_{ h|_{B_r^\bullet} } , \mu_h|_{B_r^\bullet}  ,0 ,w \right)$ and
$\left( \BB C\setminus B_r^\bullet ,   \frk d_{ h|_{\BB C\setminus B_r^\bullet} },  \mu_h|_{\BB C\setminus B_r^\bullet} , \infty ,w \right)$, 
glued together according to quantum length along their boundaries in such a way that the second marked points in the two metric measure spaces are identified.
In particular, $(\BB C, \mu_h , \frk d_h , 0, w)$ is a.s.\ determined by
$\left( B_r^\bullet  , \frk d_{ h|_{B_r^\bullet} } , \mu_h|_{B_r^\bullet}  ,0 ,w \right)$ and
$\left( \BB C\setminus B_r^\bullet ,   \frk d_{ h|_{\BB C\setminus B_r^\bullet} },  \mu_h|_{\BB C\setminus B_r^\bullet} , \infty ,w \right)$.
\end{lem}
\begin{proof}
The analogous statement in the case when $(\BB C , h , 0, \infty)$ is a doubly marked unit area quantum sphere rather than a $\gamma$-quantum cone, in which case $(\BB C , \mu_h , \frk d_h , 0)$ is equivalent to the Brownian map with two marked points sampled uniformly from its area measure, is explained just after the statement of~\cite[Proposition 2.7]{tbm-characterization}. 
The case of a $\gamma$-quantum cone follows since the $\gamma$-quantum cone is equivalent to the Brownian plane and the local behavior of the latter at its marked point is the same as the local behavior of the Brownian map at a uniformly random point sampled from its area measure~\cite[Proposition~4]{curien-legall-plane}. 
\end{proof}

\begin{proof}[Proof of Lemma~\ref{lem-metric-function-cone}]
To lighten notation, define the quantum surface and metric measure space
\eqb \label{eqn-metric-function-spaces}
\mcl S_r^\bullet := \left( B_r^\bullet  ,  h|_{B_r^\bullet} , 0   \right) \quad \op{and} \quad
\frk B_r^\bullet := \left( B_r^\bullet  , \frk d_{h|_{B_r^\bullet} } , \mu_h|_{B_r^\bullet}  ,0 \right) .
\eqe 
Since $h|_{B_r^\bullet}$ a.s.\ determines $\frk d_{h|_{B_r^\bullet} }$ and $\mu_{h|_{B_r^\bullet}}$ by construction, it is clear that $\mcl S_r^\bullet$ a.s.\ determines $\frk B_r^\bullet$. 

We now prove the reverse is true. In fact, in order to apply Lemma~\ref{lem-tbm-determined}, we will first prove a slightly different statement.  Namely, let $w$ be sampled uniformly from $\nu_h|_{\bdy B_r^\bullet}$ and define the quantum surface $\mcl S_r^\bullet(w)$ and the metric measure space $\frk B_r^\bullet(w)$ as in~\eqref{eqn-metric-function-spaces} but with the extra marked boundary point $w$.  We will show that the metric measure space $\frk B_r^\bullet(w)$ a.s.\ determines the quantum surface $\mcl S_r^\bullet(w)$.
 
By~\cite[Theorem~1.4]{lqg-tbm3}, the metric measure space $\left( \BB C  , \frk d_h , \mu_h , 0  , w \right)$ (which is a Brownian plane with an extra marked point) and the quantum surface $(\BB C , h , 0, \infty  ,w)$ a.s.\ determine each other. In particular, $\mcl S_r^\bullet(w)$ is a.s.\ determined by $\left( \BB C  , \frk d_h , \mu_h , 0 ,w \right)$. To complete the proof it suffices to show that $\mcl S_r^\bullet(w)$ is conditionally independent from $\left( \BB C  , \frk d_h , \mu_h , 0 ,w \right)$ given $\frk B_r^\bullet(w)$. 

By Lemma~\ref{lem-inside-outside}, the interior and exterior quantum surfaces $\mcl S_r^\bullet(w)$ and $\left( \BB C\setminus B_r^\bullet ,   h|_{\BB C\setminus B_r^\bullet} , \infty ,w \right)$ are conditionally independent given $\nu_h(\bdy B_r^\bullet )$.  In particular, the exterior metric measure space $ \left( \BB C\setminus B_r^\bullet ,   \frk d_{ h|_{\BB C\setminus B_r^\bullet} },  \mu_h|_{\BB C\setminus B_r^\bullet} , \infty ,w \right)$ is conditionally independent from $\mcl S_r^\bullet(w)$ given $\nu_h(\bdy B_r^\bullet )$. 
 
By Lemma~\ref{lem-tbm-determined}, the interior and exterior metric measure spaces $\frk B_r^\bullet(w)$ and $ \left( \BB C\setminus B_r^\bullet ,   \frk d_{ h|_{\BB C\setminus B_r^\bullet} },  \mu_h|_{\BB C\setminus B_r^\bullet} , \infty ,w \right) $ together a.s.\ determine the full metric measure space $\left( \BB C  , \frk d_h , \mu_h , 0 , w\right)$.  The discussion in~\cite[Section~3.6]{tbm-characterization} (together with local absolute continuity between the Brownian plane and Brownian map) implies that $\nu_h(\bdy B_r^\bullet)$ is a.s.\ determined by $\frk B_r^\bullet(w)$.   
Hence $\left( \BB C  , \frk d_h , \mu_h , 0  , w\right)$ is conditionally independent from $\mcl S_r^\bullet(w)$ given $\frk B_r^\bullet(w)$. 
Therefore, $\frk B_r^\bullet(w)$ a.s.\ determines $\mcl S_r^\bullet(w)$. 

The conditional law of $\frk B_r^\bullet(w)$ given $(\BB C , \frk d_h , \mu_h , 0)$ depends only on $\frk B_r^\bullet$ (since $w$ is just a uniform sample from the boundary measure on $\frk B_r^\bullet$).  In particular, $\frk B_r^\bullet(w)$ and $\mcl S_r^\bullet $ (which is a.s.\ determined by $(\BB C , \frk d_h , \mu_h , 0)$) are conditionally independent given $\frk B_r^\bullet$, so $\frk B_r^\bullet$ a.s.\ determines $\mcl S_r^\bullet$. 
\end{proof}

Proposition~\ref{prop-metric-function} concerns a filled metric ball centered at a uniformly random point, whereas Lemma~\ref{lem-metric-function-cone} concerns a filled metric ball centered at the fixed point $0$. The following lemma will be used to bridge the gap between these choices of center point. 

\begin{lem}
\label{lem-cone-abs-cont}
Let $h$ be any embedding of a $\gamma$-quantum cone into $(\BB C , 0 ,\infty)$.  Let $V\subset \BB C$ be a deterministic bounded open set and let $z$ be sampled uniformly from $\mu_h|_V$, normalized to be a probability measure.  Then the law of the quantum surface $(\BB C , h , z ,\infty)$ is absolutely continuous with respect to the law of $(\BB C , h , 0,\infty)$. 
\end{lem}
\begin{proof}
Let $\eta'$ be a whole-plane space-filling $\SLE_{\kappa'}$ from $\infty$ to $\infty$, sampled independently from $h$ and then parameterized by $\gamma$-quantum mass with respect to $h$ in such a way that $\eta'(0) = 0$.  Let $t$ be sampled uniformly from $[0,1]$.  If $A$ is a bounded subset of $\BB C$ with $\mu_h(A) > 0$, then $\BB P[A\subset \eta'([0,1])  \,|\, h ]  >0$ and hence $\BB P[ \eta'(t) \in A \,|\, h ] > 0$.  Therefore, the conditional law of $z$ given $h$ is a.s.\ absolutely continuous with respect to the conditional law of $\eta'(t)$ given $h$.  Hence the joint law of $(h,z)$ is absolutely continuous with respect to the joint law of $(h,\eta'(t))$.  By~\cite[Theorem~1.9]{wedges}, $(\BB C ,h , \eta'(t) , \infty) \eqD (\BB C , h , 0, \infty)$.  The conclusion of the lemma follows.
\end{proof}

\begin{proof}[Proof of Proposition~\ref{prop-metric-function}]
Fix $r>0$ and $\ep \in (0,1)$.  Let $(D,h)$ be a quantum surface as in the proposition statement, let $U$ be an open subset of~$D$ such that $\mu_h(U) <\infty$, and let~$z$ be sampled uniformly from $\mu_h|_U$, normalized to be a probability measure.  On the event $\{\frk d_h(z , \bdy U)  > r\}$, the filled metric ball $B_r^\bullet(z ; \frk d_h)$ a.s.\ lies at positive Euclidean distance from $\bdy U$.  Hence there exists a bounded open set $V\subset U$ such that 
\eqb \label{eqn-domain-approx}
\BB P\left[  B_r^\bullet(z ; \frk d_h)  \subset V ,\, \frk d_h(z , \bdy U)  > r \right] \geq (1-\ep) \BB P\left[ \frk d_h(z , \bdy U)  > r \right] .
\eqe 
By Lemma~\ref{lem-internal-metric}, on the event $\{ B_r^\bullet(z ; \frk d_h)  \subset V\}$, we have $B_r^\bullet(z ; \frk d_h) = B_r^\bullet(z ; \frk d_{h|_V})$.  By assumption, the law of the field $h|_V$ is absolutely continuous with respect to the corresponding restriction of a whole-plane GFF normalized so that its circle average over $\bdy B_R(0)$ is $0$, for $R>0$ such that $\ol V\subset B_R(0)$. This law, in turn, is absolutely continuous with respect to the law of $\wh h|_V$, where $\wh h$ is an appropriate choice of embedding of a $\sqrt{8/3}$-quantum cone into $\BB C$.

By Lemma~\ref{lem-cone-abs-cont}, if $w$ is sampled uniformly from $\mu_{\wh h}|_V$, normalized to be a probability measure, then the law of the quantum surface $(\BB C , \wh h , w, \infty)$ is absolutely continuous with respect to the law of a $\sqrt{8/3}$-quantum cone.  By definition, the metrics $\frk d_{h|_V}$ and $\frk d_{\wh h|_V}$ are obtained by applying the same deterministic functional to~$h$ and~$\wh h$.  Hence the law of the quantum surface $\left( B_r^\bullet(z ; \frk d_h) , h|_{B_r^\bullet(z ; \frk d_h)} , z \right)$ is absolutely continuous with respect to the law of $\left( B_r^\bullet(0 ; \frk d_{\wh h} ) , \wh h|_{B_r^\bullet(0 ; \frk d_{\wh h})} , 0 \right)$ on the event $\{ B_r^\bullet(z ; \frk d_h)  \subset V\}$.  By Lemma~\ref{lem-metric-function-cone}, we infer that~\eqref{eqn-metric-function} occurs a.s.\ on this event.  By~\eqref{eqn-domain-approx} and since $\ep \in (0,1)$ is arbitrary, we infer that~\eqref{eqn-metric-function} holds a.s.
\end{proof}

We end this section by recording a result to the effect that there is a.s.\ only one isometry from a filled metric ball in the Brownian plane ($\sqrt{8/3}$-quantum cone) to itself.

\begin{prop} \label{prop-unique-isometry}
Let $(\BB C , h , 0, \infty)$ be a $\sqrt{8/3}$-quantum cone and let $B_r^\bullet = B_r^\bullet(0 ; \frk d_h)$ be the filled $\sqrt{8/3}$-LQG metric ball of radius $r$ centered at 0. Almost surely, the only isometry from $(B_r^\bullet , \frk d_{h|_{B_r^\bullet}}   )$ to itself is the identity.
\end{prop}
\begin{proof}
We will deduce the proposition from the description of the law of a filled metric ball in the Brownian map established in~\cite[Proposition 4.4]{tbm-characterization}. 

Let $\mcl I$ be the set of isometries from $ (B_r^\bullet , \frk d_{h|_{B_r^\bullet}}   )$ to itself. 
We must show that a.s.\ $\mcl I$ contains only the identity map.  We first observe that a.s.\ $\phi(0) = 0$ for each $\phi \in \mcl I$ since $0$ is a.s.\ the only point in $B_r^\bullet$ whose distance to every point in $\bdy B_r^\bullet$ is equal to $r$. 

Let $\{\Gamma_t\}_{t\geq 0}$ be the metric net process started from $0$ and targeted at $\infty$, i.e.\ for each $t\geq 0$, 
\eqbn
\Gamma_t = \bigcup_{s\leq t} \bdy \mcl B_s^\bullet.
\eqen
Equivalently, $\{\Gamma_t\}_{t\geq 0}$ is the QLE$(8/3,0)$ process started from $0$ and parameterized by quantum distance time (see~\cite[Section~6]{lqg-tbm1} and~\cite[Section~2.2]{lqg-tbm2}). We also let $\mcl U_t$ be the set of bubbles cut out by $\Gamma_t $. 

By~\cite[Proposition~4.4]{tbm-characterization} and local absolute continuity between the Brownian plane and Brownian map (see, e.g.,~\cite[Proposition~4]{curien-legall-plane}), the law of $\{\Gamma_t\}_{t\in [0,r] }$ (viewed as an increasing family of topological spaces) is absolutely continuous with respect to the law of a $3/2$-stable L\'evy net modulo a deterministic affine change of time (the $3/2$-stable L\'evy net is defined in~\cite[Section~3]{tbm-characterization}).  
Furthermore, the conditional law of the bubbles of $\mcl U_r$ given their boundary lengths is absolutely continuous with respect to a collection of independent Brownian disks with a certain measure on areas, which depends only on the boundary lengths.\footnote{In fact, these bubbles are exactly Brownian disks if we condition on their areas and boundary lengths: by the analog of~\cite[Proposition~6.5]{lqg-tbm1} for the $\sqrt{8/3}$-quantum cone (which is proven in the same manner; c.f.~\cite[Section~2.2]{lqg-tbm2}), the joint law of the quantum surfaces parameterized by the bubbles cut out by $\{\Gamma_t\}_{t \geq 0}$ is the same as the joint law of the quantum surfaces parameterized by the bubbles cut out by a whole-plane $\SLE_6$ from $0$ to $\infty$ independent from $h$. 
By~\cite[Theorem~1.17]{wedges} these bubbles are independent quantum disks if we condition on their boundary lengths. 
By~\cite[Corollary~1.5]{lqg-tbm2}, quantum disks are equivalent to Brownian disks.}
In particular, it is a.s.\ the case the following is true.
\begin{enumerate}
\item \label{item-net-single} No two distinct bubbles in $\mcl U_r$ are disconnected from $\infty$ by $\{\Gamma_t\}_{t \in [0,r]}$ at the same time. 
\item \label{item-net-dense} Each point of $\Gamma_r$ is a limit of a sequence of sets in $\mcl U_r$ whose diameters tend to $0$.
\item \label{item-net-disk} The internal metric of $\frk d_h$ on each bubble $U$ in $\mcl U_r$ is geodesic and extends continuously to $\ol U$. 
\end{enumerate}

Since each map $\phi \in \mcl I$ is an isometry, each such map satisfies $\phi(\Gamma_t) =  \Gamma_t$ for each $t\in [0,r]$. In particular, for each $t\in [0,r]$ the image of each bubble in $\mcl U_t$ under $\phi$ must be another bubble in $\mcl U_t$. By property~\ref{item-net-single} above, it follows that a.s.\ $\phi(U) = U$ for each $U\in\mcl U_r$ and each $\phi \in \mcl I$. By combining this with property~\ref{item-net-dense}, we infer that a.s.\ $\phi(z) = z$ for each $z\in \Gamma_r$ and each $\phi \in \mcl I$. 
In particular, $\phi$ fixes the boundary of each $U\in \mcl U_r$ pointwise.
Since each $\phi\in \mcl I$ is an isometry, each such map restricts to an isometry from $U$ to $U$ with respect to the internal metric on $U$ for each $U\in\mcl U_r$. By~\cite[Proposition~2.3]{tbm-characterization} and property~\ref{item-net-disk} above, a.s.\ $\phi|_U$ is the identity map for each $U\in \mcl U_r$ and each $\phi\in \mcl I$. 
Therefore a.s.\ each $\phi\in\mcl I$ is the identity map. 
\end{proof}

\subsection{Proof of Theorems~\ref{thm-wpsf-mchar} and~\ref{thm-bead-mchar}}
\label{sec-wpsf-mchar}

Suppose we are in the setting of Theorem~\ref{thm-wpsf-mchar}.  We will prove the theorem by reducing it to Theorem~\ref{thm-wpsf-char}.  Let $(\BB C , \wt h , 0,\infty)$ be the $\gamma$-quantum cone whose associated $\sqrt{8/3}$-LQG metric measure space structure is given by $(\wt X  , \wt d , \wt \mu , \wt x)$, which a.s.\ exists and is unique by~\cite[Theorem~1.4]{lqg-tbm3}. In other words, if $\frk d_{\wt h}$ and $\mu_{\wt h}$ are the $\sqrt{8/3}$-LQG metric and measure induced by $\wt h$, respectively, then $(\BB C , \frk d_{\wt h } , \mu_{\wt h} , 0)$ and $(\wt X  , \wt d , \wt \mu , \wt x)$ are equivalent as pointed metric measure spaces. 

Fix some choice of embedding $\wt h$ and henceforth assume without loss of generality that in fact $(\BB C , \frk d_{\wt h } , \mu_{\wt h} , 0) = (\wt X  , \wt d , \wt \mu , \wt x)$, so that $\wt\eta' : \BB R \rta \BB C$ and $\wt\eta'(0) = 0$; and the homeomorphism $\Phi $ of condition~\ref{item-wpsf-mchar-homeo} takes $ \BB C$ to $\BB C$.  We will check that the pair $(\wt h , \wt\eta')$ satisfies the conditions of Theorem~\ref{thm-wpsf-char}.  The following lemma implies in particular the statement about the law of the future beaded quantum surfaces in condition~\ref{item-wpsf-char-wedge} of Theorem~\ref{thm-wpsf-char}. 

\begin{lem} \label{lem-mchar-wedge}
For $t\in\BB R$, let $\wt{\mcl S}_{t,\infty} = (\wt\eta'([t,\infty)) , \wt h|_{\wt\eta'([t,\infty))} , \wt\eta'(t) , \infty)$ be the future quantum surface defined as in Theorem~\ref{thm-wpsf-char} for the above choice of $(\wt h , \wt\eta')$.  Then $\wt{\mcl S}_{t,\infty}$ has the law of a $\sqrt 6$-quantum wedge.  Furthermore, the ordered collection of connected components of the metric measure space $(\wt X_{t,\infty} ,\wt d_{t,\infty} ,  \wt\mu|_{\wt X_{t,\infty}}  )$ from Theorem~\ref{thm-wpsf-mchar} a.s.\ determines the quantum surface $ \wt{\mcl S}_{t,\infty}  $.  
\end{lem} 
\begin{proof}
We will use Proposition~\ref{prop-unique-isometry} and~\ref{prop-metric-function} to show that $\wt{\mcl S}_{t,\infty}$ a.s.\ coincides with the unique $\sqrt 6$-quantum wedge with the same metric measure space structure as  $(\wt X_{t,\infty} ,\wt d_{t,\infty} ,  \wt\mu|_{\wt X_{t,\infty}}  )$. 
\medskip

\noindent\textit{Step 1: constructing a quantum wedge.}
We first construct a $\sqrt 6$-quantum wedge from the future metric measure space $(\wt X_{t,\infty} , \wt d_{t,\infty} , \wt\mu|_{\wt X_{t,\infty}})$, which we will eventually show is equivalent (as a quantum surface) to the quantum surface parameterized by $\wt\eta'([t,\infty))$. 

For $s \geq t$, let $\wt U_t^s$ be the connected component of the interior of $\eta'([t,\infty))$ containing $s$ (or let $\wt U_t^s = \{\eta'(s)\}$ if $\eta'(s) \in \bdy \eta'([t,\infty))$). Also let $\wt x_t^s$ (resp.\ $\wt y_t^s$) be the point of $\bdy \wt U_t^s$ where $\wt\eta'$ starts (resp.\ finishes) filling in $\wt U_t^s$. 
By Lemma~\ref{lem-internal-metric}, it is a.s.\ the case that the internal metric of $\frk d_{\wt h}$ on each $\wt U_t^s$ coincides with the $\sqrt{8/3}$-LQG metric $\frk d_{\wt h|_{\wt U_t^s}}$.  

By condition~\ref{item-wpsf-mchar-wedge} in Theorem~\ref{thm-wpsf-mchar}, the collection of doubly pointed metric measure spaces $\{ (\wt U_t^s , \frk d_{\wt h|_{\wt U_t^s}} , \mu_{h|_{\wt U_t^s}} , \wt x_t^s ,\wt y_t^s) \}_{s \geq t}$ agrees in law with the collection of beads of a $\sqrt 6$-quantum wedge, parameterized by the accumulated quantum mass of the previous beads. 
By~\cite[Theorem~1.4]{lqg-tbm3}, a $\sqrt 6$-quantum wedge is a.s.\ determined by its metric measure space structure. 
Let $\rng{\mcl S}_{t,\infty}$ be the $\sqrt 6$-quantum wedge determined by $\{ (\wt U_t^s , \frk d_{\wt h|_{\wt U_t^s}} , \mu_{h|_{\wt U_t^s}} , \wt x_t^s ,\wt y_t^s) \}_{s \geq t}$. 
For $s\geq t$, let $\rng{\mcl D}_t^s$ be the first bead of $\rng{\mcl S}_{t,\infty}$ with the property that the sum of the quantum areas of the previous beads is at least $s-t$, so that $\rng{\mcl D}_t^s$ is a doubly marked quantum surface. Let $\rng h_t^s$ be an embedding of $\rng{\mcl D}_t^s$ into $(\BB D ,  1, -1)$ (chosen in some manner which depends only on $\rng{\mcl S}_{t,\infty}$). By our construction of $\rng{\mcl S}_{t,\infty}$, for each rational $s\geq t$, there a.s.\ exists an isometry $f_t^s : (\BB D , \frk d_{\rng h_t^s} ) \rta (\wt U_t^s , \frk d_{\wt h|_{\wt U_t^s}})$ such that $(f_t^s)_* \mu_{\rng h_t^s} = \mu_{\wt h|_{\wt U_t^s}}$, $f_t^s(1) = \wt x_t^s$, and $f_t^s(-1) = \wt y_t^s$.
\medskip

\noindent\textit{Step 2: $f_t^s$ is conformal on metric balls.}
Now fix $s\geq t$. 
To prove the first statement of the lemma, we will show that the map $f_t^s$ defined just above is a.s.\ conformal, so that $(\BB D , \rng h_t^s ,  1, -1)$ and $(\wt U_t^s , \wt h|_{\wt U_t^s} , \wt x_t^s , \wt y_t^s)$ are equivalent as quantum surfaces. 
To this end, let $z \in \BB D$ be sampled uniformly from $\mu_{\rng h_t^s}$, normalized to be a probability measure. 
For $r>0$, define the metric balls 
\eqbn
\rng B_r^\bullet  := B_r^\bullet(z ; \frk d_{\rng h_t^s} ) \quad \op{and} \quad \wt B_r^\bullet  := B_r^\bullet(f_t^s(z) ; \frk d_{\wt h} ) ,
\eqen
so that (since $f_t^s$ is an isometry) $f_t^s(\rng B_r^\bullet) = \wt B_r^\bullet$. 
The main step of the proof is to show that $f_t^s$ is a.s.\ conformal on $\rng B_r^\bullet$. 

Define the event
\eqb \label{eqn-bdy-dist-event}
E_t^s(r) := \{ \frk d_{\rng h_t^s}(z ,\bdy \BB D) > r \}  =  \{ \frk d_{\wt h}( f_t^s(z) , \bdy \wt U_t^s) > r \}      ,
\eqe  
and note that the two definitions agree since $f_t^s$ is an isometry.

By Proposition~\ref{prop-metric-function}, if we let $F_r$ be the deterministic function from that proposition then on the event $E_t^s( r)$, it is a.s.\ the case that
\eqb \label{eqn-wedge-bead-function}
F_r\left( \rng B_r^\bullet   , \frk d_{\rng h_t^s|_{  \rng B_r^\bullet  )}} , \mu_{\rng h_t^s}|_{ \rng B_r^\bullet  } , z \right) =
\left( \rng B_r^\bullet ,  \rng h_t^s|_{  \rng B_r^\bullet   } , z \right) .
\eqe 
Since $f_t^s$ is measure-preserving, the point $f_t^s(z)$ is sampled uniformly from $\mu_{\wt h|_{\wt U_t^s}}$, normalized to be a probability measure. 
Since $\wt U_t^s$ is a.s.\ bounded, if we are given $\ep \in (0,1)$, we can find a deterministic bounded open set $V\subset \BB C$ such that
\eqbn
  \BB P\left[\wt U_t^s \subset V   \right] \geq 1-\ep  .
\eqen
On the event $\{\wt U_t^s \subset V\}$, the conditional law of $z$ given $\wt h$ is a.s.\ absolutely continuous with respect to a uniform sample from $\mu_{\wt h}|_V$, normalized to be a probability measure. Since $\ep \in (0,1)$ is arbitrary and $\wt h$ is an embedding of a $\sqrt{8/3}$-quantum cone, it follows from Proposition~\ref{prop-metric-function} that on the event $E_t^s(r)$, a.s.\ 
\eqb \label{eqn-wedge-bead-function'}
F_r\left( \wt B_r^\bullet  , \frk d_{\wt h|_{\wt B_r^\bullet }} , \mu_{\wt h}|_{ \wt B_r^\bullet } , f_t^s(z) \right) =
\left( \wt B_r^\bullet , \wt h|_{ \wt B_r^\bullet } , f_t^s(z) \right) .
\eqe 

Since $f_t^s$ is a measure-preserving isometry, on the event $E_t^s(r)$ the left sides of~\eqref{eqn-wedge-bead-function} and~\eqref{eqn-wedge-bead-function'} a.s.\ coincide, so we have the equality of quantum surfaces
\eqb \label{eqn-ball-surface-agree}
 \left( \rng B_r^\bullet  ,  \rng h_t^s|_{  \rng B_r^\bullet   } , z \right) =    
 \left( \wt B_r^\bullet , \wt h|_{ \wt B_r^\bullet } , f_t^s(z) \right) ,\quad \text{on $E_t^s(r)$} .
\eqe 
Hence on this event there a.s.\ exists a conformal map $g_z  :    \rng B_r^\bullet   \rta   \wt B_r^\bullet  $ such that 
\eqbn
\wt h |_{ \wt B_r^\bullet  }  = \rng h_t^s \circ g_z^{-1} + Q\log |(g_z^{-1})'|. 
\eqen
By the LQG coordinate change formula for $\frk d_h$ (recall Section~\ref{sec-lqg-metric}), a.s.\ $g_z$ is an isometry from $\left( \rng B_r^\bullet   , \frk d_{\rng h_t^s|_{  \rng B_r^\bullet  )}}   \right)$ to $\left( \wt B_r^\bullet  , \frk d_{\wt h|_{\wt B_r^\bullet }}  \right)$.
Since $f_t^s$ also restricts to an isometry between these metric spaces, it follows from Proposition~\ref{prop-unique-isometry} (applied to the isometry $(f_t^s)^{-1} \circ g_z$) and local absolute continuity that a.s.\ $g_z = f_t^s|_{ \rng B_r^\bullet}$.  In other words, $f_t^s$ is a.s.\ conformal on $\rng B_r^\bullet $ on the event $E_t^s(r)$. 
\medskip

\noindent\textit{Step 3: conclusion.}
Since $z$ was a uniform sample from the measure $\mu_{h_t^s}$, whose closed support is a.s.\ equal to all of $\BB D$, for each $w \in \BB D$ it a.s.\ holds with positive conditional probability given $(\wt h , \wt\eta')$, $\rng{\mcl S}_{t,\infty}$, and $f_t^s$ that $\frk d_{\rng h_t^s}(z,w) \leq \frac12 \frk d_{\rng h_t^s}(z , \bdy \BB D)$. If this is the case then the above discussion implies that $f_t^s$ is conformal on a neighborhood of $w$ and that the restriction of $h_t^s$ (resp.\ $\wt h$) to this neighborhood (resp.\ its image under $f_t^s$) are related by the $\sqrt{8/3}$-LQG coordinate change formula. Hence $f_t^s$ is a.s.\ conformal on all of $\BB D$ and $\wt h|_{\wt U_t^s} = \rng h_t^s\circ (f_t^s)^{-1} + Q\log |((f_t^s)^{-1})'|$.  Therefore, a.s.\ $\wt{\mcl S}_{t,\infty} = \rng{\mcl S}_{t,\infty}$ as quantum surfaces.

By construction $\rng{\mcl S}_{t,\infty}$ is a.s.\ determined by $(\wt X_{t,\infty} ,\wt d_{t,\infty} ,  \wt\mu|_{\wt X_{t,\infty}} , \wt\eta'(t) )$ and we showed above that $  \rng{\mcl S}_{t,\infty}   =  \wt{\mcl S}_{t,\infty}$ a.s., which yields the last statement of the lemma.
\end{proof}

We next check the independence conditions in Theorem~\ref{thm-wpsf-char}.  In light of Lemma~\ref{lem-mchar-wedge}, it suffices to prove that if $\wt{\mcl F}^{\op{m}}_{-\infty,t}$ and $\wt{\mcl F}_{t,\infty}^{ \op{m}}$ are the $\sigma$-algebras from Theorem~\ref{thm-wpsf-mchar}, then $\wt{\mcl F}^{\op{m}}_{-\infty,t}$ (resp.\ $\wt{\mcl F}_{t,\infty}^{\op{m}}$) contains the $\sigma$-algebra generated by the quantum surfaces $\wt{\mcl S}_{v_Z(s) , s} =  (\wt\eta'([v_Z(s) , s])  , \wt h|_{\wt\eta'([v_Z(s) , s]) } , \wt\eta'(s))$ as $s$ ranges over all $\pi/2$-cone intervals for $Z$ which are maximal in some interval in $(-\infty , t)$ (resp.\ $(t,\infty)$) with rational endpoints. By Lemma~\ref{lem-internal-metric}, in the notation of Theorem~\ref{thm-wpsf-mchar} it holds that
\eqbn
 \left(\wt X_{v_Z(s) , s} , \wt d_{v_Z(s) , s} , \wt\mu|_{\wt X_{v_Z(s) , s}}, \wt\eta'_{\wt X_{v_Z(s) , s}}(s) \right)  =  \left(\wt\eta'([v_Z(s) , s])  , \frk d_{\wt h|_{\wt\eta'([v_Z(s) , s]) } }  ,  \mu_{\wt h}|_{ \wt\eta'([v_Z(s) , s])  }  , \wt\eta'(s)\right)
\eqen
as pointed metric measure spaces. It therefore suffices to prove the following.

\begin{lem} \label{lem-mchar-ind}
Suppose we are in the setting of Theorem~\ref{thm-wpsf-char}. 
Let $[a,b] \subset \BB R$ be an interval with rational endpoints.
Also let $c \in (a,b)$ and let $s$ be the maximal $\pi/2$-cone time for $Z$ in $[a,b]$ with $c \in [v_Z(s) , s]$. 
The quantum surface $\wt{\mcl S}_{v_Z(s) , s} = (\wt\eta'([v_Z(s) , s])  , \wt h|_{\wt\eta'([v_Z(s) , s]) } , \wt\eta'(s))$ is a.s.\ determined by its corresponding metric measure space structure $(\wt\eta'([v_Z(s) , s])  , \frk d_{\wt h|_{\wt\eta'([v_Z(s) , s]) } }  ,  \mu_{\wt h}|_{ \wt\eta'([v_Z(s) , s])  }  , \wt\eta'(s))$.
\end{lem} 
\begin{proof}
To lighten notation, let $\wt U_s$ be the interior of the bubble $\wt\eta'([v_Z(s) , s]) $. 
Condition on the metric measure space $( \wt U_s  , \frk d_{ \wt h|_{\wt U_s} }  ,  \mu_{\wt h}|_{\wt U_s}  , \wt\eta'(s))$ and let $\wt{\mcl S}_{v_Z(s) , s}'$ be a singly marked quantum surface sampled from the conditional law of $\wt{\mcl S}_{v_Z(s) ,s}$ given $( \wt U_s  , \frk d_{ \wt h|_{\wt U_s} }  ,  \mu_{\wt h}|_{\wt U_s}  , \wt\eta'(s))$ in such a way that it is conditionally independent from $\wt{\mcl S}_{v_Z(s) ,s}$.  
Then $\wt{\mcl S}_{v_Z(s) , s}$ and $\wt{\mcl S}_{v_Z(s) , s}'$ are two quantum surfaces with the same law which a.s.\ generate the same metric measure space structure. We must show that a.s.\ $\wt{\mcl S}_{v_Z(s) , s} = \wt{\mcl S}_{v_Z(s) , s}'$  as quantum surfaces (we note that this is not immediate from the results of~\cite{lqg-tbm3} since we do not a priori know the exact law of these quantum surfaces).

Let $\wt h_s'$ be an embedding of $\wt{\mcl S}_{v_Z(s) ,s}'$ into $(\BB D , 1)$. 
By our choice of coupling, there a.s.\ exists an isometry $f : (\BB D , \wt h_s') \rta (\wt U_s , \frk d_{h|_{\wt U_s}}) $ such that $f_*  \mu_{\wt h_s'} = \mu_{\wt h}|_{\wt U_s}$ and $f(1) = \wt\eta'(s)$. 

If we sample $z$ uniformly from $\mu_{\wt h_s}$, then by Propositions~\ref{prop-unique-isometry} and~\ref{prop-metric-function} and an absolute continuity argument as in the proof of Lemma~\ref{lem-mchar-wedge} (here we use that $\wt{\mcl S}_{v_Z(s) , s}$ is a sub-surface of a $\sqrt{8/3}$-quantum cone and that $\wt{\mcl S}_{v_Z(s) , s}' \eqD \wt{\mcl S}_{v_Z(s) , s}$), we find that for $r>0$, a.s.\ 
\eqbn
\left(  B_r^\bullet ( z   ; \frk d_{\wt h}) , \wt h|_{ B_r^\bullet ( z   ; \frk d_{\wt h})}    \right)  =  \left(  B_r^\bullet ( z   ; \frk d_{\wt h_s'}) , \wt h_s'|_{ B_r^\bullet ( z   ; \frk d_{\wt h_s'})}   \right) 
\eqen
as quantum surfaces on the event $\{\frk d_{\wt h_s'}(z , \bdy \BB D)  > r\}$. 
By the same argument used in the proof of Lemma~\ref{lem-mchar-wedge}, this implies that on the event $\{\frk d_{\wt h_s'}(z , \bdy \BB D)  > r\}$, $f$ is a.s.\ conformal on $B_r^\bullet ( z   ; \frk d_{\wt h})$  and the restriction of $\wt h_s'$ (resp.\ $\wt h$) to $B_r^\bullet ( z   ; \frk d_{\wt h_s'})$ (resp.\ $B_r^\bullet ( z   ; \frk d_{\wt h})$) are related by the $\sqrt{8/3}$-LQG coordinate change formula. 
Since the closed support of $\mu_{\wt h}$ is a.s.\ equal to all of $\BB C$, it follows that the closed support of $\mu_{\wt h_s'}$ is a.s.\ equal to all of $\BB D$. From this, we infer that $f$ is a.s.\ conformal on all of $\BB D$ and that a.s.\ $\wt h|_{\wt U_s} = \wt h \circ f^{-1} + Q\log |(f^{-1})'|$. 
Hence a.s.\ $\wt{\mcl S}_{v_Z(s) , s} = \wt{\mcl S}_{v_Z(s) , s}'$.
\end{proof}

\begin{proof}[Proof of Theorem~\ref{thm-wpsf-mchar}]
As described at the beginning of this subsection, we let $(\BB C ,\wt h , 0,\infty)$ be the $\sqrt{8/3}$-quantum cone determined by $(\wt X ,\wt d , \wt \mu , \wt x)$ and view $\wt\eta'$ as a curve in $\BB C$. 
We claim that the pair $(\wt h , \wt\eta')$ satisfies the conditions of Theorem~\ref{thm-wpsf-char}.
Indeed, condition~\ref{item-wpsf-char-homeo} is immediate from the corresponding condition in Theorem~\ref{thm-wpsf-mchar}, and condition~\ref{item-wpsf-char-wedge} in Theorem~\ref{thm-wpsf-char} is implied by the corresponding condition in Theorem~\ref{thm-wpsf-mchar} together with Lemmas~\ref{lem-mchar-wedge} and~\ref{lem-mchar-ind}.  
The theorem statement now follows immediately from Theorem~\ref{thm-wpsf-char}.  
\end{proof}

\begin{proof}[Proof of Theorem~\ref{thm-bead-mchar}]
Since a bead of a $\frac{3\gamma}{2}$-quantum wedge is the same as a quantum disk for $\gamma=\sqrt{8/3}$ (recall Section~\ref{sec-wedge}) and a quantum disk is equivalent (as a metric measure space) to the Brownian disk by~\cite[Corollary~1.5]{lqg-tbm2}, this can be deduced from Theorem~\ref{thm-bead-char} and the results of Section~\ref{sec-metric-function} via essentially the same argument used to prove Theorem~\ref{thm-wpsf-mchar}. Note that condition~\ref{item-bead-mchar-homeo} determines the left/right boundary lengths of the doubly marked Brownian disk parameterized by the unbounded connected component of $\wt X^\bead \setminus \wt\eta^\bead([0,t])$, even though these boundary lengths are not specified in condition~\ref{item-bead-mchar-wedge}. 
\end{proof}

\subsection{$\SLE_6$ parameterized by quantum natural time}
\label{sec-bead-mchar-nat}

We now give a version of Theorem~\ref{thm-bead-mchar} where we parameterize the curve by quantum natural time and we only condition on the left/right boundary lengths of the Brownian disk, not its area. This is the statement that is used in~\cite{gwynne-miller-perc} to identify the law of a subsequential limit of percolation on random quadrangulations with boundary.

Suppose $(\BB H , h^\bead , 0, \infty)$ is a doubly-marked quantum disk and $\eta^\bead$ is an independent chordal $\SLE_6$ from $0$ to $\infty$. Roughly speaking, parameterizing $\eta^\bead$ by quantum natural time with respect to $h^\bead$ is equivalent to parameterizing by the ``quantum local time" at the set of times when $\eta^\bead$ disconnects a bubble from $\infty$. Formally, quantum natural time is defined for a chordal $\SLE_{\kappa'}$, $\kappa' \in (4,8)$, and an independent free-boundary GFF on $\BB H$ with a $\gamma/2$-log singularity at the origin in~\cite[Definition~6.23]{wedges}; and is defined for the pair $(h^\bead , \eta^\bead)$ via local absolute continuity.  We will typically denote an $\SLE_6$ curve parameterized by quantum natural time and its associated left/right boundary length process by a superscript $\nat$. 
 
Our characterization theorem for $\SLE_6$ parameterized by quantum natural time will be stated in terms of free Boltzmann Brownian disks, which are introduced in \cite[Section~1.5]{bet-mier-disk} and defined as follows.  For $\frk l > 0$, a \emph{free Boltzmann Brownian disk with boundary length $\frk l$} is the random metric measure space obtained by first sampling a random area $\frk a$ from the probability measure $  \frac{\frk l^3}{ \sqrt{2\pi a^5 } } e^{-\frac{\frk l^2}{2 a} } \BB 1_{(a\geq 0)} \, da$, then sampling a Brownian disk with boundary length $\frk l$ and area $\frk a$. Note that a free Boltzmann Brownian disk with boundary length $\frk l$ is obtained from a free Boltzmann Brownian disk with unit boundary length by scaling areas by $\frk l^2$, boundary lengths by $\frk l$, and distances by $\frk l^{1/2}$. 
A singly (resp.\ doubly) marked free Boltzmann Brownian disk is a free Boltzmann Brownian disk together with one (resp.\ two) points sampled uniformly from its natural boundary length measure. In the doubly marked case, the left and right boundary lengths are the lengths of the arcs between the two marked points.  A free Boltzmann Brownian disk is equivalent (as a metric measure space with a boundary length measure) to a quantum disk with fixed boundary length.

\begin{thm}[Characterization of $\SLE_6$ on a Brownian disk with quantum natural time] \label{thm-bead-mchar-nat} 
Let $( \frk l^L, \frk l^R )\in (0,\infty)^2$ and suppose we are given a coupling of a doubly-marked free Boltzmann Brownian disk $(\wt X^\nat , \wt d^\nat , \wt \mu^\nat , \wt x , \wt y)$ with left/right boundary lengths $\frk l^L$ and $\frk l^R$, a random continuous curve $\wt\eta^\nat : [0,\infty) \rta \wt X^\nat$ from $\wt x$ to $\wt y$, and a random process $ Z^\nat = ( L^\nat , R^\nat)$ which has the law of the left/right boundary length process of a chordal $\SLE_6$ on a doubly marked quantum disk with left/right boundary lengths $\frk l^L$ and $\frk l^R$, parameterized by quantum natural time (this process is defined at the beginning of Section~\ref{sec-chordal}).   
Assume that the following conditions are satisfied.
\begin{enumerate}
\item (Laws of complementary connected components) For $u\geq 0$, let $\wt{\mcl U}^\nat_u$ be the collection of singly marked metric measure spaces of the form $(U , \wt d^\nat_U  , \wt\mu^\nat|_U  , \wt x_U^\nat)$ where $U$ is a connected component of $\wt X^\nat \setminus \wt\eta^\nat ([0,u])$, $\wt d_U^\nat$ is the internal metric of $\wt d^\nat$ on $U$, and $\wt x_U^\nat$ is the point where $\wt\eta^\nat$ finishes tracing $\bdy U$.  If we condition on $Z^\nat |_{[0,u]}$, then the conditional law of $\wt{\mcl U}^\nat_u$ is that of a collection of independent singly marked free Boltzmann Brownian disks with boundary lengths specified as follows. The elements of $\wt{\mcl U}^\nat_u$ corresponding to the connected components of $\wt X^\nat \setminus \wt\eta^\nat ([0,u])$ which do not have the target point $\wt y$ on their boundaries are in one-to-one correspondence with the downward jumps of the coordinates of $Z^\nat|_{[0,u]}$, with boundary lengths given by the magnitudes of the corresponding jump. The element of $\wt{\mcl U}^\nat_u$ corresponding to the connected component of $\wt X^\nat \setminus \wt\eta^\nat ([0,u])$ with $\wt y$ on its boundary has boundary length $L_u^\nat   +   R_u^\nat   $. 
\label{item-bead-mchar-nat-wedge}  
\item (Topology and consistency) 
There is a pair $((\BB H , h^\nat , 0, \infty) ,\eta^\nat)$ consisting of a doubly marked quantum disk with left/right boundary lengths $\frk l^L$ and $\frk l^R$ and an independent chordal $\SLE_6$ from $0$ to $\infty$ in $\BB H$ parameterized by quantum natural time and a homeomorphism $\Phi^\nat : \BB H\rta \wt X^\nat$ with $\Phi^\nat_* \mu_{h^\nat} = \wt\mu^\nat$ and $\Phi^\nat\circ \eta^\nat  = \wt\eta^\nat $. Moreover, for each $u\in [0,\infty)\cap \BB Q$, $\Phi^\nat$ a.s.\ pushes forward the $\sqrt{8/3}$-quantum length measure with respect to $h^\nat$ on the boundary of the unbounded connected component of $\BB H\setminus \eta^\nat([0,u])$ to the natural boundary length measure on the connected component of $\wt X^\nat \setminus \wt\eta^\nat([0,u])$ with $\wt y$ on its boundary.  \label{item-bead-mchar-nat-homeo} 
\end{enumerate}
Then $(\wt X^\nat , \wt d^\nat , \wt \mu^\nat , \wt\eta^\nat)$ is equivalent (as a curve-decorated metric measure space) to a doubly-marked quantum disk with left/right boundary lengths $\frk l^L$ and $\frk l^R$ equipped with its $\sqrt{8/3}$-LQG area measure and metric together with an independent chordal $\SLE_{6}$ between the two marked points, parameterized by quantum natural time. 
\end{thm}

Note that the process $Z^\nat$ does not encode the quantum areas of the bubbles disconnected from $\infty$ by $\eta^\nat$, so we need to assume that the map $\Phi^\nat$ of condition~\ref{item-bead-mchar-nat-homeo} is area-preserving.  

In the companion paper~\cite{gwynne-miller-sle6}, we prove that condition~\ref{item-bead-mchar-nat-wedge} is satisfied in the case when $(\wt X^\nat , \wt d^\nat , \wt \mu^\nat , \wt\eta^\nat)$ is actually a doubly-marked quantum disk equipped with its $\sqrt{8/3}$-LQG area measure and metric together with an independent chordal $\SLE_{6}$ between the two marked points, parameterized by quantum natural time; we do not need this statement for the proof of Theorem~\ref{thm-bead-mchar-nat}, but theorem statement is vacuous without it. We also provide a description of the law of the process $Z^\nat$ in terms of its Radon-Nikodym derivative with respect to the law of a pair of independent $3/2$-stable processes with no upward jumps. 

As we will now show, Theorem~\ref{thm-bead-mchar-nat} is an easy consequence of Theorem~\ref{thm-bead-mchar} (essentially it follows by re-parameterizing the curve $\wt\eta^\nat$).

\begin{proof}[Proof of Theorem~\ref{thm-bead-mchar-nat}]
Suppose we are in the setting of Theorem~\ref{thm-bead-mchar-nat}.
Let $\frk a = \wt\mu^\nat(\wt X^\nat)$ be the random area of the free Boltzmann Brownian disk $\wt X^\nat$.  For $t \in [0,\frk a]$, let $\sigma_t$ be the smallest $u\geq 0$ for which the $\mu_{\wt h^\nat}$-mass of the region disconnected from the target point $\wt y$ by $\wt\eta^\nat([0,u])$ is at least $t$. Also let $\sigma_t = \sigma_{\frk a}$ for $t\geq \frk a$.  Since the homeomorphism $\Phi^\nat$ in condition~\ref{item-bead-mchar-nat-homeo} satisfies $\Phi^\nat_*  \mu_{h^\nat} = \wt\mu^\nat$, $\sigma_t$ can equivalently be described as the smallest $u\geq 0$ for which the $\mu_{ h^\nat}$-mass of the region disconnected from $\infty$ by $ \eta^\nat([0,u])$ is at least $t$.

Let $\wt\eta^\bead(t) := \wt\eta^\nat(\sigma_t)$, $\eta^\bead(t) := \eta^\nat(\sigma_t)$, and $Z^\bead_t := Z^\nat_{\sigma_t}$.  Then $\wt\eta^\bead$ (resp.\ $\eta^\bead$) is parameterized by the $\mu_{\wt h^\bead}$- (resp.\ $\mu_{h^\bead}$-) mass it disconnects from~$\infty$ and $ Z^\bead$ is the left/right boundary length process for each of $\wt\eta^\bead$ and $\eta^\bead$.  We will check the hypotheses of Theorem~\ref{thm-bead-mchar} for the curve-decorated metric measure space $(\wt X^\nat , \wt d^\nat , \wt \mu^\nat ,\wt\eta^\bead)$, the field-curve pair $(h^\nat , \eta^\bead)$, and the process $Z^\bead$. 

It is clear from condition~\ref{item-bead-mchar-nat-homeo} of the present theorem that condition~\ref{item-bead-mchar-homeo} of Theorem~\ref{thm-bead-mchar} is satisfied in the present setting, with $\Phi^\bead = \Phi^\nat$. Hence we only need to check condition~\ref{item-bead-mchar-wedge} of Theorem~\ref{thm-bead-mchar}.

For $t\in [0,\frk a]$, let $\tau^\bead(t)$ be the right endpoint of the largest interval of times containing $t$ on which $Z^\bead$ is constant, as in the discussion just above Theorem~\ref{thm-bead-mchar}. Equivalently, $\tau^\bead(t)$ is the $\wt\mu^\nat$-mass of the region disconnected from $\wt y$ by $\wt\eta^\bead([0,t])$.  For $u \geq 0$, let $\wt{\mcl G}^\nat_u$ be the $\sigma$-algebra generated by $Z^\nat|_{[0,u]}$ and the ordered sequence of $\wt\mu^\nat$-masses of the bubbles disconnected from $\infty$ by $\wt\eta^\nat$ before time $u$.  Then for $t\geq 0$, $\sigma_t$ is a $\{\wt{\mcl G}^\nat_u\}_{u\geq 0}$-stopping time and $\wt{\mcl G}^\nat_{\sigma_t}$ is the same as the $\sigma$-algebra generated by $ Z^\bead|_{[0,t]}$ and $\tau^\bead(t)$.

For $u\geq 0$, let $\wt{\frk W}_u^\nat$ be the element of $\wt{\mcl U}_u^\nat$ which contains the target point $\wt y$ on its boundary.  Now fix $t\geq 0$ and for $n\in\BB N$ let $\sigma_t^n$ be the smallest integer multiple of $2^{-n}$ which is at least $\sigma_t$, so that each $\sigma_t^n$ is a $\{\wt{\mcl G}^\nat_u\}_{u\geq 0}$-stopping time which takes on only countably many possible values and $\sigma_t^n$ decreases to $\sigma_t$.  By condition~\ref{item-bead-mchar-nat-wedge} in the theorem statement, for each $n\in\BB N$ the conditional law given $\wt{\mcl G}^\nat_{\sigma_t^n}$ of the collection of metric measure spaces $\wt{\mcl U}^\nat_{\sigma_t^n} $ corresponding to the connected components of $\wt X^\nat\setminus \wt\eta^\nat([0,\sigma_t^n])$ is as in condition~\ref{item-bead-mchar-nat-wedge} with $\sigma_t^n$ in place of $u$, except that we condition on the area of each of each element of $\wt{\mcl U}^\nat_{\sigma_t^n} \setminus \{\wt{\frk W}_{\sigma_t^n}^\nat\}$ in addition to its boundary length.

We now take a limit as $n\rta\infty$ to get the same statement with $\sigma_t$ in place of $\sigma_t^n$.  We have $\bigcap_{n=1}^\infty ( \wt{\mcl U}_{\sigma_t^n}^\nat \setminus \{\wt{\frk W}_{\sigma_t^n}^\nat\} )  =   \wt{\mcl U}_{\sigma_t }^\nat \setminus \{\wt{\frk W}_{\sigma_t }^\nat\} $ and by continuity of $\wt\eta^\nat$ a.s.\ the unexplored metric spaces $\wt{\frk W}^\nat_{\sigma_t^n}$ converge to $\wt{\frk W}^\nat_{\sigma_t}$ as $n\rta\infty$ in, e.g., the pointed Gromov-Prokhorov topology~\cite{gpw-metric-measure}.  Since $Z^\nat$ is right continuous, the boundary length $L_{\sigma_t^n}^\nat + R_{\sigma_t^n}^\nat$ of $\wt{\frk W}^\nat_{\sigma_t^n}$ converges to the boundary length $L_{\sigma_t}^\nat + R_{\sigma_t}^\nat$ of $\wt{\frk W}^\nat_{\sigma_t}$ as $n\rta\infty$.  By scaling, it is clear that the law of a singly free Boltzmann Brownian disk depends continuously on its boundary length in the Gromov-Prokhorov topology.

Combining the preceding two paragraphs shows that the conditional law given $\wt{\mcl G}^\nat_{\sigma_t}$ (equivalently, given $ Z^\bead|_{[0,t]}$ and $\tau^\bead(t)$) of $\wt{\mcl U}_{\sigma_t}^\nat$ is as in condition~\ref{item-bead-mchar-nat-wedge} with $\sigma_t $ in place of $u$, except that we condition on the area of each element of $\wt{\mcl U}_{\sigma_t }^\nat  \setminus \{\wt{\frk W}_{\sigma_t^\nat}\} $ in addition to its boundary length.  Further conditioning on $\frk a$ is equivalent to conditioning on the total area of $\wt{\frk W}^\nat_{\sigma_t}$. Hence condition~\ref{item-bead-char-wedge} in Theorem~\ref{thm-bead-mchar} holds for  the conditional law of $(\wt X^\nat , \wt d^\nat \wt\mu^\nat , \wt\eta^\bead)$, $(h^\nat , \eta^\bead)$ and $Z^\bead$ given $\frk a$. 

By Theorem~\ref{thm-bead-mchar}, we infer that $(\wt X^\nat , \wt d^\nat ,\wt\mu^\nat , \wt\eta^\bead)$ is equivalent (as a curve-decorated metric measure space) to a doubly-marked quantum disk equipped with its $\sqrt{8/3}$-LQG area measure and metric together with an independent chordal $\SLE_6$ between the two marked points parameterized by the $\sqrt{8/3}$-quantum mass it disconnects from $\infty$.  It follows from the definition of quantum natural time that the inverse time change $\sigma^{-1}(u)$ such that $Z^\nat_u = Z_{\sigma^{-1}(u)}^\bead$ and $\eta^\nat(u) = \eta^\bead(\sigma^{-1}(u))$ is a deterministic functional of $Z^\nat_u$ ($\sigma^{-1}$ is the inverse of the local time that $Z^\nat_u$ at the endpoints of the intervals where it is constant). Since $\wt\eta^\nat(u) = \wt\eta^\bead(\sigma^{-1}(u))$, we obtain the desired description of $(\wt X^\nat , \wt d^\nat, \wt\mu^\nat , \wt\eta^\nat)$.
\end{proof}

\appendix

\section{Index of notation}
\label{sec-index}
 
Here we record some commonly used symbols in the paper, along with their meaning and the location where they are first defined. Symbols used only locally are not included.

\begin{multicols}{2}
\begin{itemize}
\item $Z = (L,R)$: left/right boundary length process (peanosphere BM);~\eqref{eqn-bm-cov}.
\item $(\wt h , \wt\eta' )$: $\gamma$-quantum cone coupled with a curve we want to show is SLE$_{\kappa'}$; Theorem~\ref{thm-wpsf-char-intro} or~\ref{thm-wpsf-char}.
\item $(h,\eta')$: $\gamma$-quantum cone and independent space-filling SLE$_{\kappa'}$ determined by $Z$; Theorem~~\ref{thm-wpsf-char}.
\item $\eta_{\mcl S}$: segment of the curve $\eta$, viewed as a curve on the quantum surface $\mcl S$; Definition~\ref{def-surface-curve}. 
\item $v_Z(t)$: cone entrance time for a $\pi/2$-cone time $t$; Definition~\ref{def-cone-time}.
\item $\wt{\mcl F}_{a,b}$: $\sigma$-algebra generated by $(Z-Z_a)|_{[a,b]}$ and ``bubbles" filled in by $\wt\eta'$ during $[a,b]$; Theorem~\ref{thm-wpsf-char}.
\item $\mcl C$ and $\wt{\mcl C}$: the quantum surfaces $(\BB C ,h , 0,\infty)$ and $(\BB C , \wt h , 0, \infty)$, resp.;~\eqref{eqn-cone-def}. 
\item $\mcl S_{a,b}$ and $\wt{\mcl S}_{a,b}$: quantum surfaces parameterized by $\eta'([a,b])$ and $\wt\eta'([a,b])$, resp.;~\eqref{eqn-increment-surface}.
\item $\eta_{a,b}$ and $\wt\eta_{a,b}$: curve obtained by skipping the bubbles filled in by $\eta'$ (resp.\ $\wt\eta'$) during $[a,b]$; just above~\eqref{eqn-tau-def}. 
\item $[\sigma_{a,b}(t) , \tau_{a,b}(t)]$: maximal interval in $[a,b]$ containing $t$ during which $\eta'$ (equivalently $\wt\eta'$) fills in a bubble;~\eqref{eqn-tau-def}. Often the subscript $a,b$ is dropped.
\item $\mcl S_{a,b}^0$ and $\wt{\mcl S}_{a,b}^0$: sub-surfaces of $\mcl S_{a,b}$ and $\wt{\mcl S}_{a,b}$, resp., parameterized by the bubbles; just after~\eqref{eqn-tau-def}. 
\item $P_{t,\infty}$: function which encodes areas and left/right boundary lengths of beads of $\mcl S_{t,\infty}$ (equivalently, $\wt{\mcl S}_{t,\infty}$);~\eqref{eqn-bead-function}. 
\item $\mcl A_{C,t_0}(S,r,a,b)$: stability event; Section~\ref{sec-stability}.
\item $\rng{\mcl C}_{n,k}$: $\gamma$-quantum cone obtained by welding $\wt{\mcl S}_{a,\tau(t_{n,k})}$, $\wt{\mcl S}_{\tau(t_{n,k}),\infty}$, and an independent $\frac{3\gamma}{2}$-quantum wedge;~\eqref{eqn-surface-inc-function}.
\item $\rng\eta_{n,k}'$: space-filling curve on $\rng{\mcl C}_{n,k}$; just after~\eqref{eqn-surface-inc-function}.
\item $\rng\eta_{n,k}$: non-space-filling curve in $\rng\eta_{n,k}'([a,b])$; just after~\eqref{eqn-surface-inc-function}.
\item $\rng h_{n,k}$: embedding of $\rng{\mcl C}_{n,k}$ into $\BB C$; Section~\ref{sec-surface-inc-embed}.
\item $f_{n,k_1,k_2}$: map taking $\eta_{n,k_1}'$ to $\eta_{n,k_2}'$; Lemma~\ref{lem-inc-homeo}.
\item $\BB K_{n,k}$: set disconnected from $\infty$ by $\rng\eta'_{n,k}([a,b])$; just after~\eqref{eqn-inc-homeo-asymp}.
\item $\BB K = \BB K_{n,n}$, $\wt{\BB K} = \BB K_{n,0}$, and $\BB f = f_{n,n,0}$;~\eqref{eqn-end-hull}. 
\item $\frk a , \frk l^L$, $\frk l^R$; area and left/right boundary lengths of a bead of a $\frac{3\gamma}{2}$-quantum wedge; Theorem~\ref{thm-bead-char}.
\item $Z^\bead = (L^\bead,R^\bead)$: left/right boundary length process for SLE$_{\kappa'}$ on a bead of a $\frac{3\gamma}{2}$-quantum wedge; Theorem~\ref{thm-bead-char}.
\item $(h^\bead,\eta^\bead)$: embedding into $\BB H$ of a chordal SLE$_{\kappa'}$ on a bead of a $\frac{3\gamma}{2}$-quantum wedge; Theorem~\ref{thm-bead-char}.
\item $(\wt h^\bead , \wt\eta^\bead)$: field $\wt h^\bead \eqD h^\bead$ coupled with a curve we want to show is SLE$_{\kappa'}$; Theorem~\ref{thm-bead-char}.
\item $\wt{\mcl F}_{a,b}^{\op{m}}$: metric space analog of $\wt{\mcl F}_{a,b}$; Theorem~\ref{thm-wpsf-mchar}.
\item $\frk d_h$: $\sqrt{8/3}$-LQG metric induced by $h$; Section~\ref{sec-lqg-metric}.
\end{itemize}
\end{multicols}

\bibliography{cibiblong,cibib}
\bibliographystyle{hmralphaabbrv}

\end{document}